\documentclass[11pt, a4paper]{report}
\PassOptionsToPackage{usenames,dvipsnames}{xcolor}
\usepackage[margin=1in]{geometry}

\usepackage[%
  backend=biber,
  defernumbers=true,
  style=alphabetic,
  maxnames=20,
]{biblatex}
\usepackage{tcolorbox}
\renewcommand{\comment}[1]{}

\usepackage{float}
\usepackage{quiver}
\usepackage{appendix}
\usepackage{aliascnt}
\usepackage[colorlinks=true,linkcolor=blue]{hyperref}
\usepackage{amsmath, amsthm,  amsfonts, etoolbox, changepage, bigints, relsize, mathtools, tikz-cd, amssymb,   tikz, multirow ,adjustbox, tabularx, rotating, booktabs, lscape,colortbl}
\usepackage{aliascnt}
\usepackage{cleveref}
\usepackage{url}
\usepackage[english]{babel}
\usepackage[utf8]{inputenc}
\input xypic
\xyoption{all}
\usepackage[all,arc]{xy}
\usepackage{enumerate}
\usepackage{mathrsfs}
\usepackage{dsfont}

\usepackage{caption}
\usepackage{subcaption}

\usepackage{scalerel,stackengine}
\stackMath
\newcommand\reallywidehat[1]{%
\savestack{\tmpbox}{\stretchto{%
  \scaleto{%
    \scalerel*[\widthof{\ensuremath{#1}}]{\kern-.6pt\bigwedge\kern-.6pt}%
    {\rule[-\textheight/2]{1ex}{\textheight}}
  }{\textheight}%
}{0.5ex}}%
\stackon[1pt]{#1}{\tmpbox}%
}
\tikzset{node distance=2cm, auto}

\usetikzlibrary{arrows}

\usepackage[usenames,dvipsnames]{pstricks}
\usepackage{epsfig}
\usepackage{pst-grad} 
\usepackage{pst-plot} 
\usepackage[space]{grffile} 
\usepackage{etoolbox} 
\makeatletter 
\patchcmd\Gread@eps{\@inputcheck#1 }{\@inputcheck"#1"\relax}{}{}
\makeatother

  \theoremstyle{plain}
  \newtheorem{theorem}{Theorem}[section]
  \newtheorem*{theorem*}{Theorem}
  \newtheorem{lemma}[theorem]{Lemma}
  \newtheorem{proposition}[theorem]{Proposition}
  
  \newtheorem{corollary}{Corollary}[section]
  
  \theoremstyle{definition}
  \newtheorem{definition}[theorem]{Definition}
  \newtheorem{example}{Example}[section]

  \theoremstyle{remark}
  \newtheorem{remark}[theorem]{Remark}
  \newtheorem{note}{Note}[section]
  
  \newtheorem{notation}{Notation}[section]

  \numberwithin{equation}{section}
  \numberwithin{figure}{section}
\crefname{lemma}{Lemma}{Lemma}
  \crefname{corollary}{Corollary}{Corollary}
  \crefname{theorem}{Theorem}{Theorem}
  \crefname{definition}{Definition}{Definition}
   \crefname{proposition}{Proposition}{Proposition}
 \crefname{section}{Section}{Section} 
   \crefname{construction}{Construction}{Construction}
   \crefname{generalization}{Generalization}{Generalization}
  \crefname{construction}{Construction}{Construction}
  \crefname{notation}{Notation}{Notation}
   \crefname{example}{Example}{Example}
  \crefname{remark}{Remark}{Remark}
  \crefname{fact}{Fact}{Fact}
  \crefname{conjecture}{Conjecture}{Conjecture}
  \crefname{motivation}{Motivation}{Motivation}  

  \newcommand{\cone}{\text{cone}}

  \newcommand{\D}{\mathbb{D}}
  \newcommand{\X}{\mathbb{X}}
  \newcommand{\Y}{\mathbb{Y}}
  \renewcommand{\cH}{{\mathcal H}}
  
  \newcommand{\cA}{{\mathcal A}}
  \newcommand{\cB}{\mathcal{B}}
  
  \newcommand{\cM}{{\mathcal M}}
  
  \renewcommand{\cD}{{\mathcal D}}
  \newcommand{\cP}{{\mathcal P}}
  \newcommand{\cC}{{\mathcal C}}
  \newcommand{\cG}{{\mathcal G }}
  
  \renewcommand{\cR}{\mathcal{R}}

  \newcommand{\cK}{{\mathcal K }}
  \newcommand{\cF}{\mathcal{F}}
  \newcommand{\cS}{\mathcal{S}}
  \newcommand{\cV}{\mathcal{V}}
  \newcommand{\cX}{\mathcal{X}}
  \newcommand{\cZ}{\mathcal{Z}}  
  \newcommand{\cEnd}{\mathcal{E}nd}

  \newcommand{\id}{\text{id}}
  \newcommand{\Kom}{\text{Kom}}
  
  \newcommand{\Hom}{\text{Hom}}
  \newcommand{\HOM}{\mathcal{HOM}}
  \newcommand{\End}{\text{End}}

  \newcommand{\xra}{\xrightarrow}
  \newcommand{\ra}{\rightarrow}
  
  \newcommand{\Aut}{\text{Aut}}

  \newcommand{\C}{{\mathbb C}}
  \newcommand{\R}{{\mathbb R}}
  \newcommand{\Z}{{\mathbb Z}}
   
  \newcommand{\<}{\langle}
  \renewcommand{\>}{\rangle}

\newcommand\restr[2]{{
  \left.\kern-\nulldelimiterspace 
  #1 
  \vphantom{\big|} 
  \right|_{#2} 
  }}

\newcommand{\gTLJ}{g\mathcal{TLJ}}
\newcommand{\TLJ}{\mathcal{TLJ}}
\newcommand{\TLJe}{\mathcal{TLJ}^\text{even}}
\newcommand{\gTLJe}{g\mathcal{TLJ}^\text{even}}
\newcommand{\I}{\mathscr{I}}
\renewcommand{\1}{\mathds{1}}

\newcommand{\wt}{\widetilde}
\newcommand{\B}{\mathbb{B}}
\newcommand{\W}{\mathbb{W}}
\newcommand{\Com}{\text{Com}}
\newcommand{\Ob}{\text{Ob}}
\newcommand{\Stab}{\text{Stab}}

\newcommand{\GL}{\text{GL}}

\DeclarePairedDelimiter\ceil{\lceil}{\rceil}
\DeclarePairedDelimiter\floor{\lfloor}{\rfloor}

\usepackage{listings}

\newcommand{\supp}{\text{supp}}

\newcommand{\tri}{%
\mathrel{
\begin{tikz}[line cap=round, line join=round]%
  \draw[-] (0, 1ex) -- (0,2ex)
(0, 1ex) -- (1ex,0ex)
(0, 1ex) -- (-1ex,0ex);
\end{tikz}%
}
}

\newcommand{\itri}{%
\begin{tikz}[line cap=round, line join=round]%
  \draw[-] (0, 1ex) -- (0,0ex)
(0, 1ex) -- (1ex,2ex)
(0, 1ex) -- (-1ex,2ex);
\end{tikz}%
}

\tikzcdset{scale cd/.style={every label/.append style={scale=#1},
    cells={nodes={scale=#1}}}}

\usetikzlibrary{positioning, angles, quotes, decorations.pathreplacing}

\usepackage{freetikz}

\definecolor{CBF_blue}{HTML}{648FFF}
\definecolor{CBF_purple}{HTML}{492CCC}
\definecolor{CBF_pink}{HTML}{DC267F}
\definecolor{CBF_orange}{HTML}{FE6100}
\definecolor{CBF_yellow}{HTML}{FFC441}

\title{Algebras in Fusion Categories and Categorical Dynamics of Braid Groups}
\author{Edmund Heng}

\begin{document}
\begin{titlepage}

	\centering
	{\scshape A thesis submitted for the degree of \par Doctor of Philosophy in Mathematics of \par the Australian National University\par}
	\vspace{1.5cm}
	\includegraphics[scale=1.2]{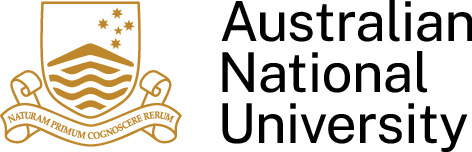}\par
	\vspace{1.5cm}
	{\Large\bfseries Categorification \par and \par Dynamics in Generalised Braid Groups \par}
	\vspace{2cm}
	{\large\itshape Edmund Xian Chen Heng\par}
	\vspace{2cm}
 	{Supervisor: \itshape Assoc. Prof. Anthony Licata}
	\vfill

	{ 
		\copyright Copyright by {\itshape Edmund Xian Chen Heng}\par
		January 23, 2022 (updated March 08, 2023) \par
		All Rights Reserved	\par
		}
\end{titlepage}

\chapter*{Abstract}
Recent developments in the theory of stability conditions and its relation to Teichmuller theory have revealed a deep connection between triangulated categories and surfaces.
Motivated by this, we prove a categorical analogue of the Nielsen-Thurston classification theorem for the rank two generalised braid groups by viewing them as (sub)groups of autoequivalences of certain triangulated categories.
This can be seen as a categorical generalisation of the classification known for the type $A$ braid groups when viewed as mapping class groups of the punctured discs.

Firstly, we realise the generalised braid groups as groups of autoequivalences through categorical actions that categorify the corresponding Burau representations.
These categorifications are achieved by constructing certain algebra objects in the tensor categories associated to the quantum group $U_q(\mathfrak{sl}_2)$, generalising the construction of zigzag algebras used in the categorical actions of simply-laced-type braid groups to include the non-simply-laced-types.

By viewing the elements of the generalised braid groups as autoequivalences of triangulated categories, we study their dynamics through mass growth (categorical entropy), as introduced by Dimitrov-Haiden-Katzarkov-Kontsevich.
Our classification is then achieved in a similar fashion to Bestvina-Handel's approach to the Nielsen-Thurston classification for mapping class groups.
Namely, our classification can be effectively decided through a given algorithm that also computes the mass growth of the group elements.
Moreover, it shows that the mass growth of the pseudo-Anosov elements are computable from certain rank two matrices.

\comment{
We construct algebra objects in some monoidal categories, which allows one to categorify the Burau representations of (non-simply laced type) generalised braid groups. 
These categorifications lead to the dynamical study of generalised braid groups through categorical dynamics, analogous to the dynamical study of classical braid groups viewed as mapping class groups of the punctured disks. 
Parallel to the theory of train-tracks by Bestvina-Handel, we show that the categorical entropies can be computed from the Perron-Frobenius eigenvalues of certain matrices obtained through stability conditions which are closely related to the root systems. \\

Talk title: A new (categorical) look on generalised braid groups \\
Abstract: 
The generalised braid groups (also known as Artin-Tits groups) are a family of groups which can be viewed as lifts of Coxeter groups.
The general approaches in understanding these groups usually involve generalising some known structure of the usual braid groups of (n+1) strands.
In particular, since the braid group of (n+1) strands are mapping class group of the n punctured disk, one could try and study the generalised braid groups as mapping class groups of certain surfaces.
Sadly, it was shown that not all generalised braid groups are ``nice'' subgroups of mapping class groups.
Not all hope is lost, however, as a striking analogy between the theory of surfaces and the theory of triangulated categories was established through a series paper by Bridgeland, Kontsevich, Dmitriov and many more.
Building on this analogy, we aim to offer a new perspective on generalised braid groups -- instead of viewing them as mapping class groups, we view them as autoequivalences of certain triangulated categories.

The first half of this talk will be an overview of this topic: we will introduce the generalised braid groups and look at some approaches used in studying them.
I will then briefly lay out the analogy between surfaces and triangulated categories, with the $n$ punctured disk as an explicit example.
Finally, I will talk about how categorifying certain representations of the generalised braid groups provide us with nice actions on triangulated categories, and state some results (about dynamics) obtained through these categorical actions.

The second half of the talk will be more interactive, where we can dive into the specific ideas from the first half, following the crowd's interest. This could range from categorifying representations through algebras in fusion categories, to understanding stability conditions and its relation to the dynamical study of the generalised braid groups.
}


\chapter*{Acknowledgements}
First and foremost, I would like thank my supervisor Tony Licata for his utmost guidance and his patience through out.
Thanks to his seemingly unbounded knowledge in many parts of mathematics, I have (hopefully) gained a much better overview of the interaction between different fields of mathematics.
I thank you for the amazing project you've offered and I could not have asked for a better supervisor.
I would also like to offer my gratitude to Anand Deopurkar and Asilata Bapat, whom -- despite all the craziness caused by the pandemic -- willingly spent so much time with me through Zoom to discuss mathematics.
Without them, the progress of this thesis would have come to a total halt during those tough times.
I would also like to specially thank Daniel Tubbenhauer, whom I met during the Kiola conference in 2019 and was the reason why I started thinking about the possibility of ``zigzag algebras for non-simply-laced graphs''.
I hope this thesis manage to provide you with a satisfying answer to your inspiring question.

I would like to take this opportunity to thank all my teachers, lecturers  and mentors alike, who have imparted me with knowledge and supported me in many different ways.
Especially to those who have taught me so much throughout my 8 years at the ANU,
I would like to give my thanks, in no particular order, to: Joan Licata for igniting my interest in algebra, Bregje Pauwells for teaching me Galois theory, Jim Borger for teaching me the functorial point-of-view in mathematics, Uri Onn for teaching me representation theory, Griff Ware for making analysis bearable and Amnon Neeman for providing me his wisdom in triangulated categories and algebraic geometry.
I would also like to thank Oded Yacobi for taking interest in my work and making the connection of my work to others possible.

To the whole PhD cohort, I thank you all for the company and laughter throughout, of which I can now only wish for more due to the untimely arrival of COVID-19.
I'd like to specially thank KieSeng Nge, with whom I frequently discuss mathematics; Sinead Wilson, who always willingly listen to me babbling on about  pieces of mathematics I find interesting; Ian Xiao for all things automaton that we discuss; Ivo Vekemans, for our random -- yet always interesting -- discussions on category theory down the corridor; Dominic Weiller, for teaching me fusion categories; Jack Brand, for all things geometry and coffee (related or not); and finally, Cale Rankin and Jaklyn Crilly, for all the great times during Friday beers.

On a more personal note, I would like to thank my parents who showered me with their love and never questioned my decision to go down the academic path. 
I'd like to thank my love Alanna Wu, for being the reason I stayed sane during this period and the reason I look forward to home every single day.
Finally I'd also like to extend my gratitude to all of my other friends and family for their support throughout.

Last but not least, I'd like to specially thank all who have offered to proof-read my thesis: Shirly Lee, Oded Yacobi, Anand Deopurkar, Asilata Bapat, Amnon Neeman, Cale Rankin, Wenqi Zhang and Noah White.
Without them this thesis would have been even more of a mess.

\chapter*{Declaration}
This thesis contains no material which has been accepted for the award of any other degree or diploma in any university. To the best of the author’s knowledge, it contains no material previously published or written by another person, except where due reference is made in the text.

\vspace{4cm}
\begin{flushright}
\noindent
Edmund Xian Chen Heng
\end{flushright}

\chapter*{Introduction}
In this thesis, we study the interplay between representation theory and Teichmuller theory from low-dimensional topology.
This interplay does not occur at the classical level, but rather after \emph{categorification}; so to be more precise, we study the interplay between \emph{categorical representation theory} and \emph{categorical ``Teichmuller theory''}.

Although the main focus of this thesis will be on the generalised braid groups (also known as Artin-Tits groups), the general idea should be applicable to a wider class of mathematical objects, which we lay out in the next section.

\comment{
One of the (many) far-reaching theorem of Thurston is the following dynamical classification of elements in the mapping class group:
\begin{theorem*}[Nielsen-Thurston classification \cite{FLP12}]
Every element in the mapping class group $\mathscr{M}CG(S)$ (for a compact, orientable surface $S$) is either
\begin{enumerate}
\item periodic;
\item reducible; or
\item pseudo-Anosov.
\end{enumerate}
\end{theorem*}

It is evident that ``categorification'', a philosophical idea coined by Crane and Frenkel \cite{Crane_1994}, is a powerful tool in recovering hidden symmetries and revealing surprising links between different mathematical fields.
The study of \emph{higher representation theory} is the incarnation of such an idea in the world of representation theory, which often involves categorifying known actions (of groups, rings, algebras etc.) on vector spaces into actions on categories.
Categorical actions have indeed proven to be fruitful, as seen from the result of \cite{chuang_rouquier_2008} proving the long standing Broue's abelian defect group conjecture, and the proof of the Kahzdan-Lusztig conjecture by \cite{EW_2014}, just to name a few.

Motivated by the recent developments on Bridgeland stability conditions, we explore yet another useful tool obtained from higher representations: a categorical ``Teichmuller theory''.
We show that having a group acting on a triangulated category can lead to a dynamical classification of its elements, analogous to the Nielsen-Thurston classification of surface diffeomorphisms.

The main focus of this thesis will be on the rank two generalised braid groups $\B(I_2(n))$, namely we will propose a (faithful) categorical action of $\B(I_2(n))$ for each $n \geq 3$, and prove a classification theorem for the elements in $\B(I_2(n))$ in terms of the dynamics of the categorical actions.
}

\section*{Generalities: the big picture}
It is now evident that ``categorification'', a philosophical idea coined by Crane and Frenkel \cite{Crane_1994}, is a powerful abstraction that recovers hidden symmetries and reveals surprising links between different mathematical fields.
The study of \emph{categorical group actions} is the incarnation of this idea in the world of representation theory, which often involves categorifying known group actions on vector spaces into actions on categories (see \cref{appen: categorification} for the definition of a weak categorical action).
Classically, the consideration of actions on vector spaces, which are sets with extra structures, allows one to access tools from linear algebra.
More often than not, the categories acted upon also come with extra structures -- such as being triangulated, where tools from triangulated categories can be used to analyse the group of interest.

One such tool we exploit here is the theory of Bridgeland's stability conditions \cite{bridgeland_2007}, which is the basic building block for a categorical version of Teichmuller theory.
\comment{
Categorical actions have indeed proven to be fruitful, as seen from the result of \cite{chuang_rouquier_2008} proving the long standing Broue's abelian defect group conjecture, and the proof of the Kahzdan-Lusztig conjecture by \cite{EW_2014}, just to name a few.

To explain why such a categorical Teichmuller theory exists, we will need to go back to Kontsevich's homological mirror symmetry.
}
The striking connection between Teichmuller theory and the theory of stability conditions was established through a series of papers \cite{gaiotto_moore_neitzke_2012}, \cite{bridgeland_smith_2014} and \cite{haiden_katzarkov_kontsevich_2017}.
These analogies, some of which can be explicitly realised, are summarised in \Cref{fig: analogy table}.
\begin{figure}[h]
\begin{center}
\begin{tabular}{ |c|c| } 
 \hline
 Surface & Triangulated category  \\
 \hline
 \hline
 mapping class group & group of autoequivalences \\
 simple curves $c$ & spherical objects $C$  \\ 
 intersection number of $c_1$ and $c_2$ & $\dim \Hom^*(C_1,C_2)$  \\ 
 flat metric $\mu$ & stability condition $\tau$ \\
 geodesics wrt. $\mu$ & $\tau$-semistable objects \\
 length${}_\mu (c)$ & mass${}_\tau (C)$ \\
 \hline
\end{tabular}
\end{center}
\caption{Analogy between theory of surfaces and theory of triangulated categories} 
\label{fig: analogy table}
\end{figure}

Since then, many have built on this connection to transfer ideas from Teichmuller theory to the context of triangulated categories.
For example,
\begin{enumerate}[-]
\item \cite{DHKK13} introduced the notion of categorical entropy and mass growth for endofunctors, which are the categorical analogues of entropy for diffeomorphisms;
\item \cite{bapat2020thurston} proposes a compactification of Bridgeland stability space, analogous to Thurston's compactification of Teichmuller space; and 
\item \cite{fan2020pseudoanosov} proposes a definition of pseudo-Anosov autoequivalence in terms of a filtration by triangulated subcategories.
\end{enumerate}
One of the unexplored territories is to ask questions about interesting autoequivalence groups of triangulated categories, given what we know about about mapping class groups, namely:
\begin{tcolorbox}
\begin{center}
``Which properties of mapping class groups can be extended to autoequivalence groups of triangulated categories?''
\end{center}
\end{tcolorbox}
A particular property that we explore in this thesis is the following far-reaching classification theorem due to Thurston:
\begin{theorem*}[Nielsen-Thurston classification \cite{FLP12}]
Every element in the mapping class group $\mathscr{M}CG(S)$ (for a compact, orientable surface $S$) can be classified into
\begin{enumerate}
\item periodic;
\item reducible; or
\item pseudo-Anosov.
\end{enumerate}
\end{theorem*}
\noindent
At the point of writing this thesis, the extent to which this theorem generalise in the context of triangulated categories remains unclear. 
Nonetheless, in working examples one can still use the analogies available as a guide and give a categorical interpretation of the above classification theorem -- even better, one might discover new analogies in the process!
Such is the goal of this thesis: a \emph{categorical} Nielsen-Thurston classification of the elements in the rank two generalised braid groups through categorical dynamics.

\section*{Categorical actions on triangulated module categories}
Let us now bring our focus back to the objects of interest in this thesis: the generalised braid groups.
Categorical representations of these groups have been studied by many and we shall list some of them here, all of which involve (faithful) actions on triangulated categories.
\begin{enumerate}
\item In \cite{khovanov_seidel_2001}, the authors categorified the Burau representation for the type $A_m$ (classical) braid groups, which led to several generalisations: type $ADE$ (simply laced) in \cite{licata2017braid}, extended affine type $\hat{A}$ in \cite{anno_logvinenko_2017} and type $B$ in the author's work with Nge \cite{heng2019curves}.
Note that the (first three) analogies given in \Cref{fig: analogy table} can be explicitly realised in these cases.
\item A more general result of Brav-Thomas \cite{brav_thomas_2010} proves that any triangulated category (with a dg enhancement) equipped with an $ADE$ configuration of spherical objects has a faithful $ADE$ braid group action on it via spherical twists.
\item Motivated by the result of Brav-Thomas, Jensen \cite{jensen_2016} proceeded to show that the categorical actions of spherical braid groups using Rouquier complexes (introduced in \cite{rouquier_2006}) are faithful, proving Rouquier's conjecture on the faithfulness of the 2-braid groups for all spherical types.
\end{enumerate}
The first main result in this thesis is the construction of new, faithful categorical actions of the rank two braid groups $\B(I_2(n))$ -- including the non-simply laced types when $n \geq 4$ -- which are categorifications of the corresponding Burau representations.

It is important to note that categorifying the Burau representation will inevitably involve categorifying a crucial (symmetric) bilinear pairing given by $\<\alpha_i, \alpha_j\> = 2 \cos \left( \frac{\pi}{m_{i,j}} \right)$.
In the simply laced cases, only $ m_{i,j} = 1, 2$ and $3$ are involved, hence $\<\alpha_i, \alpha_j\>$ will always be an integer.
As such, it is sufficient to lift the output of the bilinear pairing to vector spaces -- as was achieved in the construction given in \cite{khovanov_seidel_2001}.
For $m_{i,j} \geq 4$, it is only natural to consider objects with ``dimension'' $2 \cos \left( \frac{\pi}{m_{i,j}} \right)$ instead.
This can be attained through working in certain fusion categories instead, known as the Temperley-Lieb-Jones category $\TLJ_n$ of level $n$ (equivalent to the semi-simplified category of tilting modules of $U_q(\mathfrak{s}\mathfrak{l}_2)$ with $q= e^{i\frac{\pi}{n}}$).
To be more precise, instead of considering usual (graded) algebras, which are algebra objects (or monoid objects) in the category of (graded) vector spaces, we will be constructing algebra objects in the (graded) category $\TLJ_n$.

Moreover, our construction results in an action of $\B(I_2(n))$ on a category $\cD$ that is not only triangulated, but also a \emph{module category over $\TLJ_n$}.
This is a categorical analogue of the fact that the rank two Burau representations are defined as actions on rank two free modules over $\Z \left[ 2\cos \left( \frac{\pi}{n} \right) \right]$:
\[
\begin{tikzcd}[column sep=16ex, row sep=6ex]
	\text{Categorical: } 
		\ar[r, phantom, "\B(I_2(n)) \text{ acts on }\curvearrowright { }"] & 
	\cD
		\ar[r, phantom, "\text{module category over}"]
		\ar[d, dashed, "K_0"]& 
	\TLJ_n 
		\ar[d, dashed, "PF\dim \circ K_0"] \\
	\text{Burau: } 
		\ar[r, phantom, "\B(I_2(n)) \text{ acts on } \curvearrowright { }"] &
	K_0(\cD)
		\ar[r, phantom, "\text{module over}"] & 
	\Z \left[ 2\cos \left( \frac{\pi}{n} \right) \right]
\end{tikzcd}.
\]

\section*{Stability conditions on triangulated module category}
The consideration of such categorification has also led us to the study of a stricter class of stability conditions.
Since we will be looking at triangulated module categories over a fusion category $\cC$, we will be interested in stability conditions that interact with the $\cC$-module structure in a compatible way, which we call \emph{stability conditions respecting $\cC$} (see \cref{defn: q stab respect C}).
These stability conditions satisfy some nice properties, and in particular are also obtainable through stability functions on hearts which satisfy some extra requirements.

In our working example where the triangulated module category $\cD$ (acted upon by $\B(I_2(n))$) is considered, there is a specific stability condition $\tau_R$ (respecting $\TLJ_n$) that we will be particularly interested in, called the \emph{root stability condition of $I_2(n)$} (see \cref{sect: stab cond root}).
This naming convention comes from the fact that $\tau_R$ is reminiscent of the root system associated to $I_2(n)$.
For instance, certain $\tau_R$-semistable objects correspond to positive roots of the root system (see \cref{fig: root and central charge n=5} for an example corresponding to $I_2(5)$).
Moreover, the Coxeter element $s_2s_1$ which acts on the root system by a $\frac{2\pi}{n}$ rotation now has its positive lift $\sigma_2\sigma_1$ sending $\tau_R$-semistable objects of phase $\phi$ to $\tau_R$-semistable objects of phase $\phi+\frac{2}{n}$ (see \cref{gamma phase reducing}).
These extra structures will become useful when analysing the categorical dynamics of $\B(I_2(n))$.

\comment{
Our construction is largely motivated by the one in \cite{khovanov_seidel_2001}, so let us briefly recall the construction given there .
The crucial ingredient of the categorification is a finite dimensional (graded) algebra $\mathscr{A}_m$ associated to the $A_m$ graph, called the $A_m$ zigzag algebra.
For latter purposes, it will be beneficial to think of a finite dimensional algebra as an algebra object (also known as monoid object) in the category of finite dimensional vector spaces.
This algebra object $\mathscr{A}_m$ has exactly $m$ indecomposable projective (left) modules, which we will denote as $P_i$ for $1 \leq i \leq m$.
The categorical action is then given by certain complexes of $\mathscr{A}_m$-bimodules acting on the homotopy category of (graded) projective modules.
What is important to note here is that the dimension of the vector subspace ${}_iP\otimes_{\mathscr{A}_m} P_j \subset \mathscr{A}_m$ is what categorifies the crucial bilinear pairing $\<\alpha_i, \alpha_j\> = 2 \cos \left( \frac{\pi}{m_{i,j}} \right)$ in defining the Burau (or more precisely, the symmetrical geometric) representation of the associated braid group $\B(A_m)$.
Note that $2 \cos \left( \frac{\pi}{m_{i,j}} \right)$ is only integral when $m_{i,j} = 2$ and $3$, which was all that was needed for the simply laced types.
As such, working in the category of vector spaces (which categorifies $\Z$) turns out to be sufficient.
}

\section*{Towards a categorical Nielsen-Thurston classification}
Once we have obtained the faithful categorical action, we shall use it to work towards a categorical Nielsen-Thurston classification of the braid elements in $\B(I_2(n))$.
To understand our approach, it will be beneficial to explain our motivation from low-dimensional topology: the proof of the Nielsen-Thurston classification given by Bestvina-Handel in \cite{bestvina_handel_1995}.

Bestvina-Handel first introduced the theory of train tracks when they were studying automorphisms of free groups in \cite{bestvina_handel_free}, which was later on (due to Thurston's suggestion) used to provide an algorithmic classification of the Nielsen-Thurston classification.
This classification method is combinatorial in nature and essentially boils down to checking the irreducibility of some matrices (with integral entries) induced from the train track.
A rather crude version of their theorem can be stated as follows: 
\begin{tcolorbox}
\begin{center}
``A diffeomorphism is pseudo-Anosov if and only if the associated (train track) matrix is irreducible.''
\end{center}
\end{tcolorbox}
\noindent
Our idea is that if one replaces the above statement with an appropriate categorical analogue, one should obtain a categorical version of Bestvina-Handel result.
This turns out to be fruitful, whereby the matrices obtained from train tracks will be replaced by matrices obtained from \emph{mass automaton}, a combinatorial gadget introduced in \cite{bapat2020thurston} that can be viewed as a categorical analogue of a train track automaton.

A brief overview of our approach to the classification is as follows.
We start with a Coxeter-theoretic definition of periodic, reducible and pseudo-Anosov braids (see \cref{defn: braid types}).
This definition will probably seem meaningless at first, as a pseudo-Anosov braid is -- by definition -- a braid that is not periodic nor reducible.
However, the inclusion of categorical dynamics shall complement this definition.
More precisely, we show that the pseudo-Anosov braids are the only ones with positive mass growth (at $t=0$), whereas all the periodic and reducible braids have zero mass growth.
Moreover, similar to the result of Bestvina-Handel, the braids can be classified through an algorithm which assigns a certain matrix called \emph{mass matrix} to each braid.
A summary of the three equivalent definitions is provided in \Cref{fig: equiv defn braid types}.
We also conjecture that our definition(s) of pseudo-Anosov agrees with the definitions given in \cite{fan2020pseudoanosov} and \cite{bapat2020thurston}.
\begin{figure}[h]
\centering 
\resizebox{\columnwidth}{!}{
\begin{tabular}{ |c|c|c|c| } 
 \hline
 & Coxeter theoretic & Mass growth function $h_{\tau_R,t}$ & Mass matrix $\cM$  \\
 \hline 
 & & & \\
 $\beta$ is periodic 
 & $\beta^k$ is central
 & linear over $t$; 
 &	$\begin{bmatrix}
 	* & 0 \\
 	0 & *
	\end{bmatrix}$  \\
 & & $h_{\tau_R,0}(\beta) = 0$ & \\
 & & & \\
 \hline
 & & & \\
 $\beta$ is reducible 
 & $\beta$ is conjugate to $\sigma_i^k \Delta$ 
 & \emph{piece-wise} linear over $t$;  
 & $\begin{bmatrix}
 	* & * \\
 	0 & *
	\end{bmatrix}$ or 
	$\begin{bmatrix}
 	* & 0 \\
 	* & *
	\end{bmatrix}$ \\
 & & $h_{\tau_R,0}(\beta) = 0$ & \\
 & & & \\
 \hline 
 & & & \\
 $\beta$ is pseudo-Anosov & $\beta$ is not periodic nor reducible 
 & $h_{\tau_R,0}(\beta) > 0$ 
 & $\begin{bmatrix}
 	* & * \\
 	* & *
	\end{bmatrix}$ \\
 & & & \\
 \hline
\end{tabular}
}
\caption{Three equivalent definitions of the types of braids. The label $*$ represents a (strictly) positive real number.} 
\label{fig: equiv defn braid types}
\end{figure}

\section*{The main results}
This thesis starts with the construction of the (non-commutative) algebra object  $\I := \I_2(n)$ in the Temperley-Lieb-Jones category $\TLJ_n$ for each $3 \leq n < \infty$, which we call the \emph{zigzag algebra object} associated to the rank two Coxeter graph $I_2(n)$.
This then leads to the following theorem:
\begin{theorem*}[see \cref{weak braid action} and \cref{cor: faithful action}]
There is a (weak) faithful $\B(I_2(n))$-action on $\Kom^b(\I \text{-prmod})$, where each standard generator $\sigma_i$ of $\B(I_2(n))$ acts on a complex $M \in \Kom^b(\I \text{-prmod})$ via tensoring $M$ over $\I$ with a particular complex of $\I$-bimodules $\sigma_{P_i}$.
\end{theorem*}
\noindent
Note that the faithfulness of the categorical actions comes for free, since the induced actions on the Grothendieck groups (the Burau representations) are known to be faithful \cite{lehrer_xi_2001}.

We would like to mention here that although we only focus on the rank two braid groups, the construction of these zigzag algebra objects can be easily generalised to \emph{any} Coxeter graphs $\Gamma$, with a similarly defined categorical action of $\B(\Gamma)$.
Faithfulness however, if true, will no longer be deducible from the Grothendieck level -- the Burau representations of large ranks are known to be non-faithful, even in type $A$; see \cite{Bige_99} and \cite{LONG1993439}.

\comment{
The reason that we insisted $3 \leq n < \infty$ is because the cases where $n = 2, \infty$ (whose braid groups are $\Z\times \Z$ and the free group of two elements respectively) do not fit nicely into our dynamical study later on and are better treated separately.
For the free group of two elements case, we refer the reader to \cite{bapat2020thurston} (studied as $\widehat{A_1}$).
}

\comment{
Before we move on to the dynamical study, it is worth mentioning here that our categorical representation actually has more structure; the triangulated category that we act on is actually also a \emph{module category over a fusion category}.
This extra structure will be crucial in the dynamical study of this action, as we can impose more structure on the stability conditions that we consider (see \cref{defn: q stab respect C}).
In particular, for each $I_2(n)$, we will be considering the stability condition $\tau_R$, which is what we call the \emph{root stability condition associated to $I_2(n)$}.
This naming convention comes from the fact that $\tau_R$ is a reminiscent of the root system associated to $I_2(n)$.
For example, the $\tau_R$-semistable objects (in some sense) corresponds to positive roots of the root system.
Moreover, the coxeter element $s_2s_1$ which acts on the root system by a $\frac{2\pi}{n}$ rotation now has its positive lift $\sigma_2\sigma_1$ sending $\tau_R$-semistable objects of phase $\phi$ to semistable objects of phase $\phi+\frac{2}{n}$.
These extra structures will become useful in simplifying the analysis required in the dynamical study of the braid group.
See \cref{sect: root and semistable} for more details.
}

The next main theorem of this thesis is the categorical Nielsen-Thurston classification for the rank two braid groups:
\begin{theorem*}[see \cref{theorem: classification}, \cref{cor: categorical entropy t=0}]
There is an algorithm (pg.\pageref{sec: algorithm}) that decides whether a braid $\beta \in \B(I_2(n))$ is periodic, reducible or pseudo-Anosov.
Moreover, the mass growth function $h_t(\beta)$ completely determines the type of $\beta$, namely
\begin{enumerate}[(i)]
\item $\beta$ is periodic if and only if $h_0(\beta) = 0$ and $h_t(\beta)$ is linear over $t$;
\item $\beta$ is reducible if and only if $h_0(\beta) = 0$ and $h_t(\beta)$ is strictly piece-wise linear over $t$;
\item $\beta$ is pseudo-Anosov if and only if $h_0(\beta) = \log(PF(M))> 0$, where $M$ is a rank two matrix with non-negative entries in $\Z\left[2\cos\left(\frac{\pi}{n}\right)\right]$ and $PF(M)$ is its Perron-Frobenius eigenvalue.
\end{enumerate}
\end{theorem*}
Once again, the algorithm mentioned greatly depends on the existence of a mass automaton for each rank two braid group.
The existence of a ``good'' automaton allows one to prove, in addition to our result, a few other structural results about the associated space of stability conditions (see \cite{bapat2020thurston}); the problem lies in finding one.
However, the construction of good automata for other (higher rank) generalised braid groups is still a work in progress.

\section*{Outline of the thesis}
The first chapter of this thesis gives a brief overview of generalised braid groups and covers the necessary background on algebras and modules in general monoidal categories, fusion categories, Temperley-Lieb-Jones category $\TLJ_n$ and (triangulated) module categories.

The rest of this thesis is broken down into two parts: part I is mostly representation theoretic, dealing with categorification of the Burau representations; part II focuses on the categorical dynamics.

Part I contains a single chapter: Chapter 2, where we start by constructing the zigzag algebra object $\I:= \I_2(n)$ in the Temperley-Lieb-Jones category $\TLJ_n$ for each $I_2(n)$ and study some of the categories of modules associated to the algebra $\I$.
We show that certain complexes of bimodules satisfy the required $n$-braiding relation, and that they induce the action of the braid group $\B(I_2(n))$ on the homotopy category of modules.
At the end of this chapter, we relate our categorification to some other known categorifications.
This concludes part I of the thesis.

Part II contains two chapters.
The first chapter, Chapter 3, lays the foundation required to study categorical dynamics in a triangulated module ($\X$-)category. 
We recall the definition of stability conditions in triangulated categories, and generalise it to ($q$-)stability conditions on triangulated module ($\X$-)categories.
We then relate mass growth (with respect to the generalised stability condition above) to a generalised version of categorical entropy.
Finally we recall the definition of a mass automaton, which will be our main tool in the next chapter.

Chapter 4 contains the heart of this thesis (no pun intended): the dynamical classification of the braid elements in $\B(I_2(n))$.
We start with the construction of the stability condition $\tau_R$ associated to the root system of $I_2(n)$ using the heart of linear complexes.
We give the naive Coxeter-theoretic definition of periodic, reducible and pseudo-Anosov braids; and we compute the mass growth of periodic and reducible braids.
To deal with the pseudo-Anosov braids, we first construct a mass automaton for each braid group $\B(I_2(n))$.
We then provide the algorithm which classifies the braid elements and also computes the mass growth for any braids.
Two working examples ($n = 5$ and $4$) illustrating the algorithm can be found at the end of the chapter.

We end this thesis with an afterword comparing the mapping class representations and categorical representations of the generalised braid groups.

\section*{Notation and conventions}
We will fix $n$ to be an integer $\geq 3$ once and for all.

Categories we consider will always be (essentially) small.
Given a category $\cC$ we will use ob$(\cC)$ to denote its set of isomorphism classes of objects.
Subcategories will always mean strictly full subcategories, and when we say linear category, we always mean $\C$-linear unless stated otherwise.
When $\cC$ and $\cC'$ are both additive subcategories of an additive category $\cA$ such that $\cC \cap \cC' = 0$, we will write $\cC \oplus \cC'$ to denote the additive subcategory of $\cA$ generated by the object $C \oplus C'$ with $C \in \cC$ and $C' \in \cC'$.

We will reserve the letters $\X$ and $\Y$ to denote distinguished autoequivalences of a category $\cC$.
We will always assume that the distinguished autoequivalence has infinite order and that it respects any extra structure that $\cC$ comes equipped with: additive, monoidal, abelian, triangulated etc.
When $\cC$ is equipped with a fixed distinguished autoequivalence $\X$, we call $\cC$ an $\X$-category, denoted as the pair $(\cC, \X)$.
In our main examples, $\X = \<1\>$ will be some internal $\Z$-grading shift of the category (which will separate from the suspension functor of a triangulated category).
To distinguish $\X$ from the other endofunctors, we shall write $\X$ on the right hand side instead:
\[
E\cdot \X^k := \X^k(E), \quad \text{for all } k \in \Z.
\]
When considering endofunctors $\cF: (\cC, \X) \ra (\cC, \X)$, we will always require them to commute trivially with $\X$, i.e.
\[
\cF(A)\cdot \X = \cF(A\cdot \X).
\]

This thesis will involve two different, yet similarly named categories: a ``module category'' and a ``category of modules''.
The first one will always mean a module category over some monoidal category $\cC$, whereas the latter will always mean a category of modules over some algebra object $A$ in some monoidal category $\cC$.
The author takes no responsibility for any confusion caused by these naming conventions.

We will use $[1]$ to denote the triangulated shift (or suspension) functor of a triangulated category, where in our specific construction it will be the cohomological degree shift functor of (cochain) complexes.
For simplicity, the exact functors $\cF: \cD \ra \cD'$ on triangulated categories considered in this thesis will always commute trivially with the suspension functor, i.e. 
\[
\cF(A[1]_\cD) = (\cF(A))[1]_{\cD'}.
\]
As mentioned before, the triangulated categories we consider are usually $\X$-categories, where the distinguished autoequivalence $\X$ is given by some internal $\Z$-grading shift functor $\X = \<1\>$.
We warn the reader that this is not to be confused with the triangulated shift functor, which is denoted by a square bracket $[1]$.
Note that by construction, the functors $\<1\>$ and $[1]$ will always commute trivially.

The names ``Harder-Narasimhan'' will be used frequently throughout the categorical dynamics part of the thesis, which we shall often abbreviate as ``HN''.
With a stability condition fixed, we will use $\floor*{A}$ and $\ceil*{A}$ to denote the lowest and highest phase of the HN semistable pieces of $A$ respectively.

We use color-codes in our instructive examples (\cref{example n=5} and \cref{sect: example n=4}) to avoid unnecessary clutters in labelling, where we adopt the IBM colour pallete standard in hopes of being colour-blind-friendly.
The following colours will be used: {\color{CBF_pink} pink}, {\color{CBF_blue} blue}, {\color{CBF_purple} purple}, {\color{CBF_orange} orange} and {\color{CBF_yellow} yellow}.

\tableofcontents

\chapter{Background} \label{sect: prelim}
\comment{
In this section, we briefly describe the fusion categories (refer to \cref{subsection: TLJ} for definition) that we will be working with.
These fusion categories will serve as the categorifications of the quantum numbers $[k]_q$ with $q$ evaluated at the $2n$th root of unity -- in particular $ [2]_q = 2\cos \left( \frac{\pi}{n}\right)$.

We will also recall the definition of module categories (not to be confused with category of modules) as introduced in \cite{ostrik_2003}.
}

\section{Overview of Coxeter groups and generalised braid groups} \label{overview}
\begin{definition}
Let $I$ be a finite set.
A \emph{Coxeter matrix} is a symmetric square matrix $M = (m_{i,j})_{i,j \in I}$ with diagonal $m_{i,i} = 1$ for all $i\in I$ and $m_{i,j} =m_{j,i} \in \{2,3,..., \infty\}$ for $i \neq j$. 
The corresponding \emph{Coxeter graph} $\Gamma$ of a Coxeter matrix $M$ is the graph with set of vertices given by $I$ and labelled edges between each pair of vertices $i$ and $j$ whenever $m_{i,j} \geq 3$.
The edges will be left unlabelled whenever $m_{i,j} = 3$, and is labelled by the corresponding value $m_{i,j}$ otherwise.
The \emph{Coxeter group} $\W(\Gamma)$ of a Coxeter graph $\Gamma$ is the group with the group presentation:
\[
\W( \Gamma ) = \< s_i \text{ for each } i \in I  | (s_is_j)^{m_{i,j}} = 1\>
\]
%
The \emph{generalised braid group} (also known as \emph{Artin-Tits group}) $\B(\Gamma)$ of a Coxeter graph $\Gamma$ is the group with the group presentation:
\[
\B( \Gamma ) = \< \sigma_i \text{ for each } i \in I  | \underbrace{\sigma_i \sigma_j \sigma_i ...}_{m_{i,j} \text{ times}} = \underbrace{\sigma_j \sigma_i \sigma_j ...}_{m_{i,j} \text{ times}} \>.
\]
Note that by convention, $m_{i,j} = \infty$ means there are no relation between $s_i$ and $s_j$; similarly for $\sigma_i$ and $\sigma_j$.
\end{definition}
The Coxeter groups were first introduced by Coxeter in 1934 as abstractions of reflection groups, which were later on popularised by Tits in 1961.
They appear in many areas of mathematics -- for example they contain the Weyl groups, which served as a basis to the classification of semisimple Lie algebras.
These groups are fairly well understood and we refer the interested reader to \cite{bjorner_anders_brenti_2010} for the (vast) theory on them. 
What will be relevant to us is the following linear representation introduced by Tits (also known as the symmetrical geometric representation):
\begin{definition}\label{defn: sym geo rep}
Given a Coxeter group $\W(\Gamma)$ with $n$ generators $\{s_1, s_2, ..., s_k\}$, consider the symmetric bilinear form on $\R^k$ with basis $\{\alpha_1, \alpha_2, ..., \alpha_k\}$ given by
\[
\<\alpha_i, \alpha_j\> = 2\cos \left(\pi - \frac{\pi}{m_{i,j}} \right),
\]
where $\frac{\pi}{\infty} := 0$ by convention. \\
The faithful representation $\bar{\rho}: \W(\Gamma) \ra  \GL_n(\R)$ uniquely defined by
\[
\bar{\rho}(s_i)(v) = v - \<\alpha_i, v\>\alpha_i
\]
is called the \emph{symmetrical geometric representation} of $\W(\Gamma)$.
\end{definition}

The groups that we are interested in are actually the generalised braid groups.
These groups were introduced by Tits as ``lifts'' of Coxeter groups -- the presentation of $\B(\Gamma)$ can be obtained from the presentation of $\W(\Gamma)$ by deleting the self inverse relation.
Unlike their well understood cousins Coxeter groups, many questions about generalised braid groups remain unresolved -- we refer the interested reader to \cite{McCammond2017TheMG} and \cite{Charney2008PROBLEMSRT} for a survey on their developments and a wide range of open problems related to them.

We shall list a few examples of generalised braid groups and their corresponding Coxeter groups here:
\begin{example}
The type $A_m$ generalised braid groups are the classical braid groups of $m+1$ strands introduced by Artin. 
These groups appear in low dimensional topology in a few different forms, and have served as a starting point to some of the famous approaches to understanding generalised braid groups.
Their corresponding Coxeter groups are the symmetric groups on $m+1$ elements, which are also the Weyl groups of the Lie algebras $\mathfrak{s}\mathfrak{l}_{m+1}$.
\end{example}
\begin{example}
More generally, the \emph{spherical (or finite)} generalised braid groups are those whose corresponding Coxeter groups are finite (which has a complete classification in terms of Coxeter graphs).
These were studied extensively in \cite{deligne_1972} and \cite{EgbertKyoji_1972}.
As a matter of fact, most of the open problems about arbitrary generalised braid groups are generalisations of the results shown in these two papers.
\end{example}

The particular class of generalised braid groups that we focus on in this thesis are the rank two (spherical, connected) generalised braid groups.
They correspond to the Coxeter graph $I_2(n)$ for $3 \leq n < \infty$:
\begin{center}
\begin{tikzpicture}
\draw[thick] (0,0) -- (1,0) ;

\filldraw[color=black!, fill=white!]  (0,0) circle [radius=0.1];
\filldraw[color=black!, fill=white!]  (1,0) circle [radius=0.1];

\node at (0,-0.3) {$1$};
\node at (1, -0.3) {$2$};
\node at (0.5, 0.2) {$n$};
\end{tikzpicture}
\end{center}
and the generalised braid group $\B(I_2(n))$ can be presented as:
\[
\B( I_2(n) ) = \< \sigma_1, \sigma_2 | \underbrace{\sigma_1 \sigma_2 \sigma_1 ...}_{n \text{ times}} = \underbrace{\sigma_2 \sigma_1 \sigma_2 ...}_{n \text{ times}} \>.
\]
Their corresponding Coxeter groups are the dihedral groups, i.e. the isometry group of regular $n$-gons.
\comment{
\begin{remark}
Note that we have purposely left out two specific rank two cases: the spherical, right-angled one $I_2(2)$ and the non-spherical one $I_2(\infty)$.
These two edge cases do not fit nicely into the set up that we will be using in this thesis, and their braid groups are rather simple: $\B(I_2(2))$ is just the abelian group $\Z \times \Z$ and $B(I_2(\infty))$ is just the free group of two elements.
The latter case is also studied in \cite{bapat2020thurston} (as affine $\widehat{A_1}$) and (faithful) categorical actions for free groups in general can be found in \cite{licata_2017_free}.
\end{remark}
}

The representation of $\B(I_2(n)$ that we aim to categorify is the following, which can be seen as a $q$-deformed symmetrical geometric representation of $\W(I_2(n))$.
\begin{definition} \label{defn: Burau rep}
Let $\Z \left[2\cos \left(\frac{\pi}{n} \right) \right] \subset \R$ be the smallest subring containing $\Z$ and $2\cos \left(\frac{\pi}{n} \right)$.
The \emph{Burau representation} of $\B(I_2(n))$ 
\[
\rho: \B(I_2(n)) \ra \GL_2 \left( \Z\left[2\cos\left( \frac{\pi}{n}\right) , q^{\pm 1} \right] \right)
\]
is uniquely defined on the generators by
\[
\rho(\sigma_1) = 
	\begin{bmatrix}
	-q^2 & -2 \cos (\frac{\pi}{n}) q \\
	0      &  1
	\end{bmatrix}, \quad
\rho(\sigma_2) = 
	\begin{bmatrix}
	1                                   &  0  \\
	-2 \cos (\frac{\pi}{n}) q  &  -q^2
	\end{bmatrix}.
\]
\end{definition}
This representation is shown to be faithful in \cite{lehrer_xi_2001} (note that faithfulness of the Burau representations is a special phenomenon of the small ranks braid groups).
Note that if one evaluates $q=-1$ in the Burau representation, one recovers the symmetrical geometric representation of the corresponding Coxeter group $\W(I_2(n))$.

\section{Algebras and modules in monoidal categories} \label{sect: algebras and modules}
In this section we briefly recall the definitions of an algebra in a monoidal category and modules over an algebra.
For the readers who are not familiar with algebra objects in a monoidal category and modules over an algebra object, see \cite{pareigis_1980} or \cite{EGNO15}.

Throughout this section, $(\cC, \otimes, \1)$ will be a monoidal category.

\begin{definition}
An \emph{algebra} in $\cC$ is an object $A$ in $\cC$ equipped with two maps:
\begin{align*}
\mu &: A \otimes A \ra A \quad \text{(multiplication)} \\
\eta &: \1 \ra A \quad \text{(unit)}
\end{align*}
such that the following diagrams commute:
\begin{center}
\begin{tikzcd}
A\otimes A \otimes A 
	\ar[r, "\mu \otimes \id"] 
	\ar[d, "\id \otimes \mu"] 
& A\otimes A 
	\ar[d, "\mu"] \\
A \otimes A 
	\ar[r, "\mu"] 
& A
\end{tikzcd}
(associativity), \quad
\begin{tikzcd}
\1 \otimes A 
	\ar[r, "\eta \otimes \id"] 
	\ar[dr, "\cong", swap] 
& A\otimes A  
	\ar[d, "\mu"]
& A \otimes \1 
	\ar[l, "\id \otimes \eta", swap] 
	\ar[dl, "\cong"] \\
& A  
\end{tikzcd}
(unital).
\end{center}
\end{definition}

\begin{definition}
Let $(A, \mu, \eta)$ be an algebra in $\cC$.
A \emph{left module over $A$} is by definition an object $M$ in $\cC$ equipped with a map 
\[
\alpha : A \otimes M \ra M \quad \text{(action)}
\]
such that the following diagrams commute:
\begin{center}
\begin{tikzcd}
A\otimes A \otimes M 
	\ar[r, "\mu \otimes \id"] 
	\ar[d, "\id \otimes \alpha"] 
& A\otimes M 
	\ar[d, "\alpha"] \\
A \otimes M 
	\ar[r, "\alpha"] 
& M
\end{tikzcd}
(associativity), \quad
\begin{tikzcd}
\1 \otimes M 
	\ar[r, "\eta \otimes \id"] 
	\ar[dr, "\cong", swap] 
& A\otimes M  
	\ar[d, "\alpha"]
 \\
& M  
\end{tikzcd}
(left unital).
\end{center}
A \emph{morphism between two left $A$-modules} $(M,\alpha_M)$ and $(N, \alpha_N)$ is by definition a morphism in $\cC$ 
\[
\varphi: M \ra N
\]
such that the following diagram commutes:
\begin{center}
\begin{tikzcd}
A \otimes M 
	\ar[r, "\alpha_M"] 
	\ar[d, "\id \otimes \varphi", swap] 
& M
	\ar[d, "\varphi"] \\
A \otimes N
	\ar[r, "\alpha_N"] 
& N
\end{tikzcd}.
\end{center}
The notions of a right module over $A$ and a morphism between two right modules can be defined similarly.
\end{definition}

\begin{definition}
Let $(A, \mu, \eta)$ and $(B, \mu', \eta')$ be two algebras in $\cC$.
An \emph{$(A,B)$-bimodule} is an object $M$ in $\cC$ that is both a left $A$-module $(M,\alpha)$ and a right $B$-module $(M,\alpha')$, such that the following diagram commutes:
\begin{center}
\begin{tikzcd}
A\otimes M \otimes B 
	\ar[r, "\alpha \otimes \id"] 
	\ar[d, "\id \otimes \alpha'", swap] 
& M\otimes B 
	\ar[d, "\alpha'"] \\
A \otimes M 
	\ar[r, "\alpha"] 
& M
\end{tikzcd} (compatible actions).
\end{center}
A \emph{morphism of $(A,B)$-bimodules} is a morphism in $\cC$ which is both a morphism of left $A$-modules and right $B$-modules.
\end{definition}

\begin{definition} \label{defn: tensoring over an algebra}
Let $A,B$ and $C$ be algebra objects in a monoidal category $\cC$.
Recall that if $(M, \alpha, \theta)$ and $(N, \theta', \alpha')$ are $A$-modules-$B$ and $B$-modules-$C$ respectively, their \emph{tensor over $B$}, $M\otimes_B N$ is defined to be the coequaliser in $\cC$, if it exists:
\begin{center}
\begin{tikzcd}
M\otimes B \otimes N \ar[r, shift left=0.75ex, "\theta \otimes \id"] \ar[r, shift right=0.75ex, swap, "\id \otimes \theta'"] & M\otimes N \ar[r, "\Theta"] & M\otimes_B N.
\end{tikzcd}
\end{center}
\end{definition}

Note that if $\cC = \mathbb{K}$-vec is the category of (finite dimensional) vector spaces equipped with the usual notion of tensor product, all of the above said notions: algebras, modules and tensoring two modules over an algebra, coincide with the usual standard notions. 

The following simple general fact about modules over algebras will be needed later on, whose proof we leave to the reader:
\begin{proposition} \label{absorb I}
Let $(A, \mu, \eta)$ be an algebra object in a monoidal category $\cC$ and let $(M, \alpha)$ and $(N, \alpha')$ be left and right modules over $A$ respectively.
By viewing $A$ as a bimodule over itself, we have that $M$ and $N$ together with the maps $\alpha: A \otimes M \ra M$ and $\alpha': N\otimes A \ra N$ respectively, satisfy the coequaliser properties defining $A\otimes_A M$ and $N \otimes_A A$ respectively.
\end{proposition}

\section{Tensor and fusion categories} \label{sect: fusion}
We shall quickly review some properties of fusion categories and refer the reader to \cite{etingof_nikshych_ostrik_2005} for more details.

A (strict) \emph{tensor category} $\cC$ is a linear, abelian and (strictly) monoidal category with a bi-exact monoidal structure, satisfying the following properties:
\begin{enumerate}
\item (simplicity of unit) the monoidal unit is simple;
\item (semi-simple) every object is isomorphic to a finite direct sum of simple objects; and
\item (rigid) every object has left and right duals.
\end{enumerate}
If $\cC$ has only finitely many isomorphism classes of simple objects, we say that $\cC$ is a \emph{fusion category}.

It follows that the split Grothendieck group $K_0(\cC)$ of $\cC$ (which coincides with its exact Grothendieck group) is a free abelian group of finite rank, with basis given by the classes of the simple objects.
Furthermore, $K_0(\cC)$ can be endowed with a unital ring structure, with the class of the monoidal unit $[\1]$ as the multiplicative unit 1, and multiplication induced by the tensor product:
\[
[A] \cdot [B] = [A \otimes B].
\]
We call $K_0(\cC)$ (together with the ring structure above) the \emph{fusion ring} of $\cC$.
Moreover, $K_0(\cC)$ is a unital $\mathbb{N}$-ring, namely it has a finite $\mathbb{N}$-basis $\{[X_i]\}$ given by the isomorphism classes of simple objects such that
\[
[X_i]\cdot [X_j] = \sum_k c_{i,j,k} [X_k] \text{, with } c_{i,j,k} \in \mathbb{N}.
\]
As such, each matrix $C_i$ defined by $(C_i)_{j,k} = c_{i,j,k}$ has a unique maximal positive eigenvalue, called the \emph{Perron-Frobenius eigenvalue}, which we denote by $PF(C_i)$.
The \emph{Perron-Frobenius dimension of} $[X_i]$ (equivalently $X_i$) is defined by
\[
PF\dim([X_i]) := PF(C_i) \in \R_{>0}.
\]
One can check that this extends linearly to a ring homomorphism  (see \cite[Theorem 8.6]{etingof_nikshych_ostrik_2005}):
\[
PF\dim: K_0(\cC) \ra \R.
\]

\subsection{Chebyshev polynomials and quantum numbers}
We denote the second (normalised) $k$th Chebyshev polynomial by $\Delta_k(d)$, which can be defined recursively as follows:
\[
\Delta_0(d) = 1, \quad \Delta_1(d) = d, \quad \Delta_{k+1}(d) = d \Delta_k(d) - \Delta_{k-1}(d).
\]
For $k \in \Z_{\geq 0}$, denote the $k$th quantum number as $[k]_q := \frac{q^{k} - q^{-k}}{q - q^{-1}}$.
These numbers are known to satisfy the recurrence relation of the Chebyshev polynomial; namely
\[
[k+1]_q = \Delta_k(q+q^{-1}).
\]

Note that when $q$ is evaluated at the $2n$th root of unity, i.e. $d = q + q^{-1} = 2\cos(\frac{\pi}{n})$, we have that $[n]_q = \Delta_{n-1}(d) = 0$.
Furthermore, $\Delta_k(d) = \Delta_{n-2-k}(d)$ for all $0 \leq k \leq n-2$.

\subsection{Temperley-Lieb-Jones category at root of unity} \label{subsection: TLJ}
For each $n \geq 3$, we will be considering a strict, braided fusion category which categorifies the ring $\Omega_n := \Z[\omega]/\< \Delta_{n-1}(\omega) \>$.
Namely, we will consider a fusion category whose fusion ring is isomorphic to $\Omega_n$.
These are known as the fusion categories with type $A_{n-1}$ fusion rules (see \cite{ediemichell2017field}), which can be realised in several ways.

The explicit category that we will be using in this thesis is known as the \emph{Temperley-Lieb-Jones category evaluated at} $q = e^{i\pi/n}$, which we denote as $\TLJ_n$.
This is a diagramatic category and it is the additive completion of the category of Jones-Wenzl projectors that is semi-simplified by killing the negligible $(n-1)$th Jones-Wenzl projector.
We will only provide the minimal description of what is needed for the purpose of this thesis, where we follow the construction in \cite{chen2014temperleylieb} closely.
The interested readers could also refer to \cite{wang_2010} and \cite{turaev_2016} for further details.

\begin{remark}
For the representation-theoretic-inclined reader, the fusion category of type $A_{n-1}$ can also be realised as the semi-simplified category of tilting modules $Tilt(U_q(\mathfrak{s}\mathfrak{l}_2))$ with $q = e^{i\frac{\pi}{n}}$; see \cite[Chapter XII, Section 7]{turaev_2016}.
\end{remark}

The Temperley-Lieb-Jones category $\TLJ_n$ at $q = e^{i\pi/n}$ \cite[Definition 5.4.1]{chen2014temperleylieb} is generated additively and monoidally by the $n-1$ simple objects (the $k$th Jones-Wenzl projectors), denoted as
\[
\Pi_0, \Pi_1, ..., \Pi_{n-2},
\]
with $\Pi_0$ being the monoidal unit.
Every object in $\TLJ_n$ is self-dual.
In particular, we will use a cap (resp. cup) to denote the counit (resp. unit) of the self-dual structure on on each $\Pi_a$ for $0 \leq a \leq n-2$:
\begin{equation} \label{dual counit}
\begin{tikzpicture}[scale = 0.4]
	\node at (0,-0.5) {$\Pi_a$};
	\node at (2,-0.5) {$\Pi_a$};
	\node at (1, 2) {$\Pi_0$};
  \draw (0, 0) arc[start angle=180,end angle=0,radius=1];
\end{tikzpicture}, \quad
\begin{tikzpicture}[scale = 0.4]
	\node at (0,2) {$\Pi_a$};
	\node at (2,2) {$\Pi_a$};
	\node at (1, -0.5) {$\Pi_0$};
  \draw (0, 1.5) arc[start angle=-180,end angle=0,radius=1];
\end{tikzpicture},
\end{equation}
satisfying the usual isotopy of strings relation:
\[
\begin{tikzpicture}[scale = 0.4]
	\node at (4, 2.5) {$\Pi_a$};
  \draw (0, .5) arc[start angle=180,end angle=0,radius=1];
  \draw (4,.5) -- (4,2);
	\node at (0, 0) {$\Pi_a$};
	\node at (2, 0) {$\Pi_a$};
	\node at (4,0) {$\Pi_a$};
  \draw (0,-2) -- (0,-.5);
  \draw (2, -.5) arc[start angle=-180,end angle=0,radius=1];
	\node at (0, -2.5) {$\Pi_a$};
	\node at (6, 0) {$=$};
	\node at (8, 2.5) {$\Pi_a$};
  \draw (8,-2) -- (8,2);
  	\node at (8.5, 0) {id};
	\node at (8, -2.5) {$\Pi_a$};
	\node at (10, 0) {$=$};
	\node at (12, 2.5) {$\Pi_a$};
  \draw (12,.5) -- (12,2);
  \draw (14, .5) arc[start angle=180,end angle=0,radius=1];
	\node at (12, 0) {$\Pi_a$};
	\node at (14, 0) {$\Pi_a$};
	\node at (16,0) {$\Pi_a$};
  \draw (12, -.5) arc[start angle=-180,end angle=0,radius=1];
  \draw (16,-2) -- (16,-.5);
	\node at (16, -2.5) {$\Pi_a$};
\end{tikzpicture}.
\]
We describe here some of the morphisms needed in our calculation, together with the required relations.

\begin{definition}[\protect{\cite[Definition 5.3.4]{chen2014temperleylieb}}] \label{defn: q-admissible}
Let $a,b,c \in \{0, ..., n-2\}$. 
A triple $(a,b,c)$ is \emph{$q$-admissible} if:
	\begin{enumerate}
	\item $a+b+c \leq 2(n-2)$ and is even, and
	\item $a+b \geq c, b+c \geq a, c+a \geq b$.
	\end{enumerate}
\end{definition}
We have that
\[
\Hom_{\TLJ_n}(\Pi_a \otimes \Pi_b, \Pi_c) \cong 
\begin{cases}
\C, &\text{ for } (a,b,c) \text{ q-admissible}; \\
0,  &\text{ otherwise}.
\end{cases}
\]
When non-zero, the maps in $\Hom_{\TLJ_n}(\Pi_a \otimes \Pi_b, \Pi_c)$ will be scalar multiples of the \emph{$q$-admissible map}, which we will denote as the trivalent graph:
\begin{equation} \label{admissible triple}
\begin{tikzpicture}[scale = 0.4] 
  \node at (0,1.5) {$\Pi_c$};
  \node at (1,-1.5) {$\Pi_b$};
  \node at (-1,-1.5) {$\Pi_a$};
  \draw (0, 0) to (0,1);
  \draw (0, 0) to (1,-1);
  \draw (0, 0) to (-1,-1);
\end{tikzpicture}.
\end{equation}
Note that when $c=0$, it must be that $a=b$ for the triple $(a,b,c)$ to be $q$-admissible, and the $q$-admissible map in this case coincides with the cap map in \cref{dual counit}.

These $q$-admissible maps represented by the trivalent graphs will be ``rotational invariant'', namely under the dualisations
\[
\Hom_{\TLJ_n}(\Pi_a \otimes \Pi_b, \Pi_c) \xra{\cong} \Hom_{\TLJ_n}(\Pi_a,  \Pi_c \otimes \Pi_b) \xra{\cong} \Hom_{\TLJ_n}(\Pi_c \otimes \Pi_a, \Pi_b),
\]
the map agrees with the corresponding $q$-admissible map:
\[
\begin{tikzpicture}[scale = 0.4] 
  \node at (0,1.5) {$\Pi_b$};
  \node at (1,-1.5) {$\Pi_a$};
  \node at (-1,-1.5) {$\Pi_c$};
  \draw (0, 0) to (0,1);
  \draw (0, 0) to (1,-1);
  \draw (0, 0) to (-1,-1);
  \node at (2, 0) {$=$};
	\node at (10, 5) {$\Pi_b$};
  \draw (4, 3) arc[start angle=180,end angle=0,radius=1.5];
  \draw (10,3) -- (10,4.5);
	\node at (4, 2.5) {$\Pi_c$};
	\node at (7, 2.5) {$\Pi_c$};
	\node at (10, 2.5) {$\Pi_b$};
  \draw (4,.5) -- (4,2);
  \draw (7, 1.2) to (7,2);
  \draw (7, 1.2) to (8,0.5);
  \draw (7, 1.2) to (6,0.5);
  \draw (10,.5) -- (10,2);
	\node at (4, 0) {$\Pi_c$};
	\node at (6, 0) {$\Pi_a$};
	\node at (8, 0) {$\Pi_b$};
	\node at (10,0) {$\Pi_b$};
  \draw (4,-2) -- (4,-.5);
  \draw (6,-2) -- (6,-.5);
  \draw (8, -.5) arc[start angle=-180,end angle=0,radius=1];
	\node at (4, -2.5) {$\Pi_c$};
	\node at (6, -2.5) {$\Pi_a$};
\end{tikzpicture}.
\]
This is also true for the other rotation.

We shall use the inverted trivalent graphs:
\[
\begin{tikzpicture}[scale = 0.4] 
  \node at (0,-1.5) {$\Pi_a$};
  \node at (1,1.5) {$\Pi_b$};
  \node at (-1,1.5) {$\Pi_c$};
  \draw (0, 0) to (0,-1);
  \draw (0, 0) to (1,1);
  \draw (0, 0) to (-1,1);
  \node at (2, 0) {$:=$};
	\node at (5, 2.5) {$\Pi_c$};
	\node at (8, 2.5) {$\Pi_b$};
  \draw (5, 1.2) to (5,2);
  \draw (5, 1.2) to (6,0.5);
  \draw (5, 1.2) to (4,0.5);
  \draw (8,.5) -- (8,2);
	\node at (4, 0) {$\Pi_a$};
	\node at (6, 0) {$\Pi_b$};
	\node at (8,0) {$\Pi_b$};
  \draw (4,-2) -- (4,-.5);
  \draw (6, -.5) arc[start angle=-180,end angle=0,radius=1];
	\node at (4, -2.5) {$\Pi_a$};
\end{tikzpicture}
\qquad,
\begin{tikzpicture}[scale = 0.4] 
  \node at (0,-1.5) {$\Pi_b$};
  \node at (1,1.5) {$\Pi_c$};
  \node at (-1,1.5) {$\Pi_a$};
  \draw (0, 0) to (0,-1);
  \draw (0, 0) to (1,1);
  \draw (0, 0) to (-1,1);
  \node at (2, 0) {$:=$};
	\node at (4, 2.5) {$\Pi_a$};
	\node at (7, 2.5) {$\Pi_c$};
  \draw (7, 1.2) to (7,2);
  \draw (7, 1.2) to (8,0.5);
  \draw (7, 1.2) to (6,0.5);
  \draw (4,.5) -- (4,2);
	\node at (4, 0) {$\Pi_a$};
	\node at (6, 0) {$\Pi_a$};
	\node at (8,0) {$\Pi_b$};
  \draw (8,-2) -- (8,-.5);
  \draw (4, -.5) arc[start angle=-180,end angle=0,radius=1];
	\node at (8, -2.5) {$\Pi_b$};
\end{tikzpicture}
\]
to denote the two dualisations of the $q$-admissible map in $\Hom_{\TLJ_n}(\Pi_a \otimes \Pi_b,  \Pi_c)$ to the $q$-admissible maps in $\Hom_{\TLJ_n}(\Pi_a,  \Pi_c \otimes \Pi_b)$ and $\Hom_{\TLJ_n}(\Pi_b,  \Pi_a \otimes \Pi_c)$ respectively.
It follows that for any two $q$-admissible triples $(a,b,c)$ and $(c,d,e)$, the rotational invariant property implies that we have the following relation between the trivalent maps:
\begin{equation} \label{frobenius rel}
\begin{tikzpicture}[scale = 0.4] 
	\node at (-2, 2.5) {$\Pi_b$};
	\node at (1, 2.5) {$\Pi_d$};
  \draw (-2, 1.2) to (-2,2);
  \draw (-2, 1.2) to (-1,0.5);
  \draw (-2, 1.2) to (-3,0.5);
  \draw (1,.5) -- (1,2);
	\node at (-3, 0) {$\Pi_a$};
	\node at (-1, 0) {$\Pi_c$};
	\node at (1,0) {$\Pi_d$};
  \draw (-3,-2) -- (-3,-.5);
  \draw (0, -1.3) to (-1,-.5);
  \draw (0, -1.3) to (1,-.5);
  \draw (0, -1.3) to (0,-2);
	\node at (-3, -2.5) {$\Pi_a$};
	\node at (0, -2.5) {$\Pi_e$};
  \node at (2.5, 0) {$=$};
	\node at (4, 2.5) {$\Pi_b$};
	\node at (7, 2.5) {$\Pi_d$};
  \draw (4,.5) -- (4,2);
  \draw (7, 1.2) to (7,2);
  \draw (7, 1.2) to (8,0.5);
  \draw (7, 1.2) to (6,0.5);
	\node at (4, 0) {$\Pi_b$};
	\node at (6, 0) {$\Pi_c$};
	\node at (8,0) {$\Pi_e$};
  \draw (5, -1.3) to (4,-.5);
  \draw (5, -1.3) to (6,-.5);
  \draw (5, -1.3) to (5,-2);
  \draw (8,-2) -- (8,-.5);
	\node at (5, -2.5) {$\Pi_a$};
	\node at (8, -2.5) {$\Pi_e$};
\end{tikzpicture}
\end{equation}
Finally, given two $q$-admissible triples $(a,b,c)$ and $(d,b,c)$, we have that
\begin{equation} \label{constant multiple identity}
\begin{tikzpicture}[scale = 0.4] 
  \node at (0,1.5) {$\Pi_d$};
  \node at (-1,-1.5) {$\Pi_b$};
  \node at (1,-1.5) {$\Pi_c$};
  \draw (0, 0) to (0,1);
  \draw (0, 0) to (1,-1);
  \draw (0, 0) to (-1,-1);
  
  \node at (0,-4.5) {$\Pi_a$};
  \draw (0, -3) to (0,-4);
  \draw (0, -3) to (1,-2);
  \draw (0, -3) to (-1,-2);
  
  \node at (3, -1.5) {$=$};
  
  \node at (8, -1.5) {$
  	\begin{cases}
  	0, &\text{ if } a\neq d; \\
  	\frac{\theta(a,b,c)}{[a+1]_q} \id, &\text{ if } a = d
  	\end{cases}
  	$};
\end{tikzpicture}
\end{equation}
with $\theta(a,b,c) \in \C$ some non-zero constant (see \cite[Definition 3.2.1]{chen2014temperleylieb}), whose actual value will not matter to us.

Throughout this thesis, we will also use the symbols `` $\cap$ '', `` $\cup$ '', `` $\tri{}$ '' and `` $\itri{}$ '' to represent these $q$-admissible maps during calculations whenever the domain and codomain are clear from context.

\comment{
\begin{definition}
The $n$-sum is defined as:
\[
a +_n b = \begin{cases}
	a+b, &\text{if } a+b < n-1; \\
	2n - (a+b) - 4, &\text{if } a+b \geq n-1
	\end{cases}.
\]
\end{definition}
}
The category $\TLJ_n$ satisfies the following fusion rule:
\begin{equation} \label{fusion rule}
\Pi_a \otimes \Pi_b \cong \Pi_b \otimes \Pi_a \cong 
	\begin{cases}
	\Pi_{|a-b|} \oplus \Pi_{|a-b| + 2} \oplus \cdots \oplus \Pi_{a+b}, &a+b \leq n-2 \\
	 \Pi_{|a-b|} \oplus \Pi_{|a-b| + 2} \oplus \cdots \oplus \Pi_{2n-(a+b)-4}, &a+b > n-2,
	\end{cases}
\end{equation}
where the isomorphism (from LHS to RHS) is given by the respective $q$-admissible maps into each summand.
In particular, the isomorphism
\[
\Pi_{n-2} \otimes \Pi_{n-2} \xra{\cong} \Pi_0
\]
is the counit map given in \cref{dual counit}.
Note that if we set $b=1$, one sees that the fusion ring of $\TLJ_n$ is isomorphic to $\Omega_n := \Z[\omega]/\< \Delta_{n-1}(\omega) \>$, given by the isomorphism $[\Pi_a] \mapsto \Delta_a(\omega)$.

Each simple object $\Pi_a$ for $0 \leq a \leq n-2$ has Perron-Frobenius dimension (which agrees with the categorical trace):
\[
PF\dim(\Pi_a) = \Delta_a \left( 2\cos \left(\frac{\pi}{n} \right) \right) \geq 1.
\]
In other words, $\Delta_{n-1}(2\cos\left( \frac{\pi}{n} \right)) = 0$ and the composition
\[
K_0(\TLJ_n) \xra{\cong} \Omega_n \xra{\text{evaluate } \omega = 2\cos\left( \frac{\pi}{n} \right)} \R
\]
agrees with taking the Perron-Frobenius dimension.
It follows that for each $0 \leq k \leq n-2$,  
\[
PF\dim(\Pi_k) = PF\dim(\Pi_{n-2-k}); 
\]
in particular $PF\dim(\Pi_0)= 1 = PF\dim(\Pi_{n-2})$.
Using \cref{fusion rule}, we see that
\[
\Pi_{n-2} \otimes \Pi_k \cong \Pi_{n-2-k},
\]
so one can think of $\Pi_{n-2}$ as the simple object that switches the pair of objects $\{ \Pi_k, \Pi_{n-2-k} \}$ that have the same Perron-Frobenius dimension.
\begin{remark}
Note that a ``middle object'' exists when $n$ is even; namely for $k = \frac{n-2}{2}$, we have $\Pi_k = \Pi_{n-2-k}$ and it is the only simple object with Perron-Frobenius dimension $PF\dim\left(\Pi_{\frac{n-2}{2}} \right)$.
\end{remark}

\begin{remark}
Using the fusion rules in \cref{fusion rule}, one can see that the subcategory $\TLJ_n^{\text{even}}$ generated by the evenly labelled objects forms a fusion category.
In particular, when $n = 3$, the subcategory obtained is equivalent to the category of finite dimensional vector spaces.
\end{remark}

\comment{
\begin{proposition} \label{fusion ring identification}
The fusion ring of $\TLJ_n$ is isomorphic to the ring $\Omega_n$. 
\end{proposition}
\begin{proof}
Firstly note that $\Delta_{n-1}(\omega)$ is by definition a degree $n-1$ monic polynomial, thus $\Z[\omega]/\<\Delta_{n-1}(\omega) \>$ is a free abelian group with basis $\{1, \omega, ..., \omega^{n-2} \}$.
Let $\varphi: \Z[\omega] \ra K_0(\TLJ_n)$ be the ring homomorphism uniquely determined by $\varphi(\omega) = [\Pi_1]$.
Using \cref{fusion rule}, we get
\[
\Pi_a \otimes \Pi_1 \cong 
	\begin{cases}
	\Pi_1, &\text{ for } a = 0; \\
	\Pi_{a-1} \oplus \Pi_{a+1}, &\text{ for } 1 \leq a \leq n-3; \\
	\Pi_{n-3}, &\text{ for } a = n-2
	\end{cases}.
\]
Thus every class $[\Pi_a]$ (except for $[\Pi_0]$, which is already an image of $1 \in \Z[\omega]/ \< \Delta_{n-1} (\omega) \>$ anyway) is a polynomial over $[\Pi_1]$.
In particular, $\varphi(\Delta_{a}(\omega)) = [\Pi_a]$. 
This gives us the surjectivity of $\varphi$.
Since $\Pi_1 \otimes \Pi_{n-2} \cong \Pi_{n-3}$, we get that
\[
\varphi(\Delta_{n-1}(\omega)) = \varphi(\omega \Delta_{n-2}(\omega) - \Delta_{n-3}(\omega)) = [\Pi_1] \cdot [\Pi_{n-2}] - [\Pi_{n-3}] = 0
\]
and thus $\varphi$ factors through $\Z[\omega]/\<\Delta_{n-1}(\omega) \>$.
Since both $\Z[\omega]/\<\Delta_{n-1}(\omega) \>$ and $K_0(\TLJ_n)$ are free abelian groups of rank $n-1$, it follows that they are isomorphic.
\end{proof}
}

\section{Module categories and $\Z$-graded categories}
Throughout this subsection, $\cA$ will be an additive category and $\cC$ will be some additive monoidal category with its split Grothendieck group $K_0(\cC)$ a (unital, associative) ring induced by the monoidal structure.
We will use $\cC^{\otimes^\text{op}}$ to denote the category $\cC$ with the order of the monoidal tensor reversed; this is \emph{not} the opposite category which reverses all morphisms.
With $\cA$ as before, $\cEnd(\cA)$ will denote the category of additive endofunctors on $\cA$.
When $\cA$ has more structure: abelian or triangulated, we will require $\cEnd(\cA)$ to be the category of exact endofunctors on $\cA$ instead.
Note that $\cEnd(\cA)$ is always an additive monoidal category, with the monoidal structure given by the composition of functors.

\subsection{Module categories}\label{subsect: module categories}
\begin{definition}
A (right) module category over $\cC$ is an additive category $\cA$ together with an additive monoidal functor $\Psi: \cC^{\otimes^\text{op}} \ra \cEnd(\cA)$.
When $\cA$ has more structure, such as abelian or triangulated, we require $\Psi$ to map into the category of exact endofunctors $\cEnd(\cA)$.
Similarly when considering $\X$-categories, $\Psi$ will map to endofunctors which commutes trivially with the distinguished autoequivalence $\X$ (see the ``Notation and conventions'' section in the Introduction).
\end{definition}
\begin{remark}
Note that the module categories $\cA$ we consider in the main body of this thesis will not be semi-simple; moreover $\cA = \cD$ will be triangulated and not abelian.
As such, they will not be part of the classification studied in \cite{ostrik_2003}.
\end{remark}

\begin{example}
As a simple example, one can check that every additive category $\cA$ is naturally a module category over the category of finite dimensional vector spaces $\mathbb{K}$-vec, sending each (non-zero) vector space $V$ to the endofunctor of $\cA$ which takes $A\in \cA$ to 
$
V(A) := \bigoplus_{i=1}^{\dim (V)} A,
$
and every morphism $A \xra{f} A'$ in $\cA$ to 
$
V(A) \xra{\bigoplus_{i=1}^{\dim V} f} V(A').
$
\end{example}

Note that given a module category $(\cA, \Psi)$ over $\cC$, the functor $\Psi$ descends into a ring homomorphism on the Grothendieck groups (which are rings):
\[
K_0(\Psi): K_0(\cC^{\otimes^\text{op}}) \ra K_0(\cEnd(\cA)).
\]
It is easy to check that the map $K_0(\cEnd(\cA)) \ra \End(K_0(\cA))$ defined by $[\cF] \mapsto K_0(\cF)$ is a well-defined ring homomorphism.
Composing these two maps, we obtain a ring homomorphism $K_0(\cC^{\otimes^\text{op}}) \ra \End(K_0(\cA))$ defined by
\[
[C] \mapsto [\Psi(C)] \mapsto K_0(\Psi(C)),
\]
therefore giving $K_0(\cA)$ a (right) module structure over $K_0(\cC)$.
Given $[X] \in K_0(\cA)$, we will denote the (right) action of $K_0(\cC)$ by
\[
[X] \cdot [C] := [\Psi(C)(X)]
\] 
for all $C \in \cC$.
It follows that any endofunctor $\cG: \cA \ra \cA$  which respects the $\cC$-module structure on $\cA$ descends to a $K_0(\cC)$-module endomorphism on $K_0(\cA)$.
Note that when $\cA$ is also an $\X$-category, i.e. it is equipped with a distinguished autoequivalence $\X$, we have insisted that the endofunctors of $\cA$ in the image of $\Psi$ commute trivially with $\X$.
As such, given a module category $(\cA, \X, \Psi)$ equipped with a distinguished autoequivalence $\X$, we have that $K_0(\cA)$ is naturally a module over $K_0(\cC)[q^{\pm 1}]$, given by $q\cdot [X] = [X\cdot \X]$.

The same story follows for $\cA$ being abelian (resp. triangulated), where the Grothendieck group $K_0(\cA)$ is instead the exact Grothendieck group and all functors involved are required to be exact.

\subsection{Category of $\Z$-graded objects} \label{sect: graded categories}
Let $\cC$ be an additive monoidal category as before.
We can construct $g\cC$, the \emph{category of $\Z$-graded objects of $\cC$}, similar to how one constructs the category $\mathbb{K}$-gvec of $\Z$-graded finite dimensional vector spaces (with grading preserving morphisms) from the category $\mathbb{K}$-vec of finite dimensional vector spaces:
\begin{enumerate}
\item the objects are given by formal external sums of objects in $\cC$:
\[
C = \bigoplus_{i \in \Z} C_i,
\]
with $C_i \in \cC$ and all but finitely many $C_i$ are $0$; 
\item the morphisms $\Hom_{g\cC}(C, D)$ are $\bigoplus_{i \in \Z} \Hom_\cC(C_i, D_i)$.
\end{enumerate}
For each object $C$, we call the summand $C_i$ the \emph{homogeneous $i$-graded summand of} $C$.
Note that there are no morphisms between any of the homogeneous summands with different gradings by construction.
The category $g\cC$ is still additive and monoidal, with monoidal structure given by
\[
(C \otimes D)_i = \bigoplus_{a + b = i} C_a \otimes D_b
\]
and monoidal unit $\1_{g\cC}$ given by the object with $C_0 = \1_\cC$ and $C_i = 0$ for all other $i$.
Note that $\<1\>$ is naturally an (additive, monoidal) autoequivalence of $g\cC$, where it acts on the object $C \in g\cC$ by
\[
(C\<k\>)_i = C_{i - k}.
\]
If $E \in \cC$, we will abuse notation and use $E\<0\>$ to denote the object in $g\cC$ with $C_0 = E$ and $C_i=0$ otherwise.
As such, every homogeneous object of degree $k$ in $g\cC$ can be viewed as an object in $\cC\<k\>$.

\comment{
The following follows from the definition:
\begin{proposition}
A $\Z$-graded module category over $\cC$ is equivalent to a module category over $g\cC$.
\end{proposition}
Extending the example in the previous section, one can check that a $\Z$-graded additive category is naturally a module category over $gVec$.

As such, we shall not distinguish between $\Z$-graded module category over $\cC$ and module category over $g\cC$.
Given $(\cA, \Psi)$ a module category over $g\cC$, we shall abuse notation and use $\<1\>$ to denote the autoequivalence $\Psi(\1_{g\cC} \<1\>)$ of $\cA$.
}

\comment{
The Fibonacci category (or the golden category) is the Temperley-Lieb-Jones category with $n=5$, which we denote as $\TLJ_5^{even}$.
To simplify notation, we will denote the two Jones-Wenzl projectors $\Pi_0$ and $\Pi_2$ as $\1$ and $\tau$ respectively.
We shall construct the type $I(5)$ zigzag algebra, denoted as $\I$, as a graded algebra object in the Fibonacci category as follows:
as a graded object in $\TLJ_5$,
\[
\I = \1 \oplus \1\<2\> \oplus \tau\<1\> \oplus \1 \oplus \1\<2\> \oplus \tau\<1\> =: e_1 \oplus X_1\<2\> \oplus (2|1)\<1\> \oplus e_2 \oplus X_2\<2\> \oplus (1|2)\<1\>,
\]
We will mostly be using the latter notation so that it is clear which simple summand of $\I$ we are referring to.
As the notation suggests, we will think of these the simple summands as paths, with the the usual zigzag algebra in the back of our minds; namely $e_i$ is the constant path at $i$, $X_i$ is the loop based at $i$ and $(i|i\pm 1)$ is the one step path.
We will refer to these simple objects as the \emph{simple paths}.

Recall that the morphism space between $\Pi_a \otimes \Pi_b$ and $\Pi_c$ is a one dimensional $\C$ vector space, with the canonical basis given by the $q$-admissible map between them.
We shall define the (graded) multiplication map $\mu : \I \otimes \I \ra \I$ using path concatenation, where all length 2 paths except the loops $X_i$ are zero.
To spell it out, if the end points don't match or the concatenation is a path with length $\geq$ 2 that is not a loop, then the multiplication map is zero; else, it is given by the $q$-admissible map into the simple path in $\I$ representing the concatenation.
For example, $(1|2)\<1\> \otimes (2|1)\<1\> \xra{\mu} X_1\<2\>$ is given by the $q$-admissible map of $\tau \otimes \tau \ra \1$, which we draw as the arc \begin{center}
\begin{tikzpicture}[scale = 0.4]
	\node at (0,-0.4) {$\tau$};
	\node at (2,-0.4) {$\tau$};
  \draw (0, 0) arc[start angle=180,end angle=0,radius=1];
\end{tikzpicture} ,
\end{center}
whereas the multiplication map on $(1|2) \otimes X_2$ is given by 0.
It is easy to check the associativity of $\mu$, since any concatenation of three simple paths is zero unless one of them is the constant path.

The unit $\eta: \1 \ra \I$ is given by the following (grading preserving) morphism:
\[
\1 \xra{\begin{bmatrix}
	\id \\
	\id
	\end{bmatrix}} e_1 \oplus e_2 \hookrightarrow \I.
\]
For the unital condition, note that for each $L$ a simple path in $\I$, $\mu|_{e_i \otimes L}$ is non-zero if and only if $L$ is a path starting at $i$.
Moreover, being non-zero implies that $\mu|_{e_i \otimes L} : e_i \otimes L = L \ra L$ is the identity map by definition. 
Since every path have exactly one starting point, $\mu|_{e_i \otimes L}$ is the identity map for exactly one $i \in \{1,2\}$.
Thus $\mu|_{(e_1 \oplus e_2)\otimes \I} : (e_1 \oplus e_2)\otimes \I = \I \oplus \I \ra \I$ is the identity map.
The case for $L\otimes e_i$ follows similarly.

Now that we have a graded algebra object $(\I, \mu, \eta)$, we can consider left and right $\I$-modules.
In particular, we can define the following left $\I$-modules: as objects in $\TLJ_5^\text{even}$, we have
\[
P_1 := e_1 \oplus X_1\<2\> \oplus (2|1)\<1\> , \quad P_2:= e_2 \oplus X_2\<2\> \oplus (1|2)\<1\>.
\]
Note that $P_1$ and $P_2$ are subobjects of $\I$, and furthermore, $I = P_1 \oplus P_2$.
The left $\I$-module action $\I \otimes P_i \ra P_i$ for $i = 1,2$ are both given by the restriction of $\mu$, namely $\mu|_{I \otimes P_i}$, where it again follows from the path concatenation definition of $\mu$ that they have the correct codomain.
The left module conditions also follows easily from the structure of the algebra object $\I$.
Similarly, we can also define the right $\I$-modules below in a similar fashion:
\[
{}_1P := e_1 \oplus X_1\<2\> \oplus (1|2) \<1\>, \quad {}_2P := e_2 \oplus X_2 \<2\> \oplus (2|1) \<1\>.
\]

Recall that if $(M, \theta)$ and $(N, \theta')$ are right and left $\I$-modules respectively, we define $M\otimes_\I N$ as the coequaliser in the category $\TLJ_5^\text{even}$:
\begin{center}
\begin{tikzcd}
M\otimes \I \otimes N \ar[r, shift left=0.75ex, "\theta \otimes \id"] \ar[r, shift right=0.75ex, swap, "\id \otimes \theta'"] & M\otimes N \ar[r, "\Theta"] & M\otimes_\I N.
\end{tikzcd}
\end{center}
\begin{notation}
For $A$ and $B$ both subobjects of $\I$, we will use $\restr{\mu}{A\otimes B}$ to denote that map $\mu$ with its domain restricted to $A\otimes B$ and its codomain restricted to the image of the restricted domain.
\end{notation}
\begin{proposition}
Denote ${}_iP_j$ as the subobject of $\I$ consisting of all the simple paths in $\I$ starting with $i$ and ending with $j$. 
Together with the morphism $\mu|_{{}_i P \otimes P_j}: {}_iP \otimes P_j \ra {}_iP_j$, we have that ${}_iP_j \cong {}_iP \otimes_I P_j$.
\end{proposition}
\begin{proof}
We shall prove for the cases ${}_1P\otimes_\I P_2$ and ${}_1P \otimes_\I P_1$, where the other two cases follows from a symmetrical argument.
Let's start with ${}_1P\otimes_\I P_2$.
We shall organise ${}_1P\otimes P_2$ in the following order:
\begin{align*}
{}_1P\otimes P_2 =& (e_1 \otimes (1|2))\<1\> \oplus ((1|2) \otimes e_2)\<1\> \oplus \\
	& (e_1 \otimes e_2) \oplus (e_1 \otimes X_2)\<2\> \oplus \\
	& ((1|2) \otimes X_2)\<3\> \oplus ((1|2) \otimes (1|2))\<2\> \oplus \\
	& (X_1 \otimes e_2)\<2\> \oplus (X_1 \otimes X_2)\<4\> \oplus (X_1 \otimes (1|2))\<3\>.
\end{align*}
Note that only the first two direct summand of ${}_1P\otimes P_2$ from above are non-zero under $\restr{\mu}{{}_1P\otimes P_2}: {}_1P \otimes P_2 \ra {}_1P_2 = (1|2)\<1\>$, i.e.
\[
\restr{\mu}{{}_1P \otimes P_2} = \begin{bmatrix}
	\id_{\Pi_0} & \id_{\Pi_0} & 0 & \cdots & 0
	\end{bmatrix}
\]

On the other hand, we shall organise the direct sumamnds of ${}_1P \otimes \I \otimes P_2$ in the following order:
\begin{align*}
{}_1P \otimes \I \otimes P_2 = 
	&((1|2)\otimes e_2 \otimes e_2)\<1\> \oplus (e_1 \otimes (1|2) \otimes e_2)\<1\> \oplus (e_1 \otimes e_1 \otimes (1|2))\<1\> \oplus \\
	&(e_1 \otimes e_1 \otimes e_2) \oplus (e_1 \otimes e_1 \otimes X_2)\<2\> \oplus \\
	&(e_1 \otimes (1|2) \otimes X_2)\<3\> \oplus (e_1 \otimes (1|2) \otimes (1|2))\<2\> \oplus \\
	&(e_1 \otimes X_1 \otimes e_2)\<3\> \oplus (e_1 \otimes X_1 \otimes X_2)\<4\> \oplus \\
	&(e_1 \otimes X_1 \otimes (1|2))\<3\> \oplus U,
\end{align*}
where $U$ represents the rest of the summands in ${}_1P\otimes \I \otimes P_2$.
Note that the first three summands of ${}_1P \otimes \I \otimes P_2$ are the only summands in ${}_1P \otimes \I \otimes P_2$ such that under both of the following maps:
\[
\mu|_{ {}_1P \otimes \I} \otimes \id, \quad \id \otimes \mu|_{\I \otimes P_2} : {}_1P \otimes \I \otimes P_2 \ra {}_1P \otimes P_2
\]
they are non-zero into the first two summands of ${}_1P \otimes P_2$.
A simple computation shows that
\[
\restr{\mu}{ {}_1P \otimes \I } \otimes \id - \id \otimes \restr{\mu}{ \I \otimes P_2} = 
	\begin{bmatrix}
	0 & -1 & 0 & 0   & 0 \\
	0 &  1 & 0 & 0   & 0 \\
	0 &  0 & 0 & I_7 & * 
	\end{bmatrix},
\]
where ``1'' denotes the identity map of the respective componenent in  $\TLJ_3^\text{even}$ and $I_7$ is the $7\times 7$ identity matrix.

Let $\mathscr{E}$ be an object in $\TLJ_3^\text{even}$ with $\Theta: {}_1P \otimes P_2 \ra \mathscr{E}$ satisfying 
\begin{equation} \label{coeq 1P2}
\Theta \circ \restr{\mu}{ {}_1P \otimes \I } \otimes \id = \Theta\circ \id \otimes \restr{\mu}{ \I \otimes P_2}.
\end{equation}
Given the ordering of the summands of ${}_1P\otimes P_2$ as before, we may write $\Theta$ as
\[
\Theta = 
	\begin{bmatrix}
	f & g & h_1 & h_2 & \cdots & h_7
	\end{bmatrix}.
\]
We wish to find the unique map such that the following diagram commutes:
\begin{center}
\begin{tikzcd}[column sep = 1.5cm]
{}_1P \otimes \I \otimes P_2 
	\ar[r, shift left=0.75ex, "\restr{\mu}{ {}_1P \otimes \I } \otimes \id"] 
	\ar[r, shift right=0.75ex, swap, "\id \otimes \restr{\mu}{ \I \otimes P_2}"] 
& 
{}_1P\otimes P_2 
	\ar[d, "\Theta"]  \ar[r, "\restr{\mu}{{}_1P \otimes P_2}"]
& 
{}_1P_2 = (1|2)\<1\>. \ar[dashed, dl, "\exists !"] \\

&
\mathscr{E}
\end{tikzcd}
\end{center}
To do so, let us see what \cref{coeq 1P2} says about $\Theta$:
\begin{align*}
&\Theta \circ \restr{\mu}{ {}_1P \otimes \I } \otimes \id = \Theta\circ \id \otimes \restr{\mu}{ \I \otimes P_2} \\
\iff &\Theta \circ (\restr{\mu}{ {}_1P \otimes \I } \otimes \id - \id \otimes \restr{\mu}{ \I \otimes P_2}) = 0 \\
\iff
	&\begin{bmatrix}
	f & g & h_1 & h_2 & \cdots & h_7
	\end{bmatrix}
	\begin{bmatrix}
	0 & -1 & 0 & 0   & 0 \\
	0 &  1 & 0 & 0   & 0 \\
	0 &  0 & 0 & I_7 & * 
	\end{bmatrix} = 
	\begin{bmatrix}
	0 & \cdots & 0
	\end{bmatrix} \\
\iff
	&\begin{cases}
	f = g; \\
	h_j = 0 \text{ for all } j.
	\end{cases}
\end{align*}
Thus $\Theta = 
	\begin{bmatrix}
	f & f & 0 & \cdots & 0
	\end{bmatrix}$.
\\
Since $f = g$ are maps from $(e_1 \otimes (1|2)) \<1\> = ((1|2) \otimes e_2)\<1\> = (1|2)\<1\>$ to $\mathscr{E}$, it follows that the map $f: (1|2)\<1\> \ra \mathscr{E}$ is the unique map that makes the diagram commutative.

Now consider ${}_1P \otimes_\I P_1$. 
We shall organise the summands of ${}_1P \otimes P_1$ in the following order:
\begin{align*}
{}_1P \otimes P_1 = & 
(e_1 \otimes e_1) \oplus (e_1 \otimes X_1) \oplus (X_1 \otimes e_1) \oplus ((1|2)\otimes (2|1)) \oplus \\
& ((1|2) \otimes e_1) \oplus (e_1 \otimes (1|2)) \oplus ( (1|2) \otimes X_1) \oplus (X_1 \otimes (2|1)) \oplus (X_1 \otimes X_1).
\end{align*}
Note that the first summand is the only non-zero maps into $e_1$ and the next three following summands are the only non-zero maps into $X_1$ under the map $\restr{\mu}{{}_1P\otimes P_1}: {}_1P \otimes P_1 \ra {}_1P_1 = e_1 \oplus X_1\<2\>$, i.e.
\[
\restr{\mu}{{}_1P \otimes P_1} = 
	\begin{bmatrix}
	\id_\1 & 0      & 0      & 0    & 0 & \cdots & 0 \\
	0      & \id_\1 & \id_\1 & \cap & 0 & \cdots & 0 
	\end{bmatrix}
\]
On the other hand, we shall organise the direct sumamnds of ${}_1P \otimes \I \otimes P_1$ in the following order:
\begin{align*}
{}_1P \otimes \I \otimes P_1 = 
	&(e_1 \otimes X_1 \otimes e_1)\<2\> \oplus (e_1 \otimes (1|2) \otimes (2|1))\<2\> \oplus ((1|2) \otimes \otimes (2|1) \otimes e_1)\<2\> \oplus \\
	&(e_1 \otimes (1|2) \otimes e_1)\<1\> \oplus (e_1 \otimes e_2 \otimes (2|1))\<1\> \oplus \\
	&((1|2) \otimes X_1 \otimes e_1)\<3\> \oplus (X_1 \otimes (2|1) \otimes e_1)\<3\> \oplus \\
	&(X_1 \otimes X_1 \otimes e_1)\<4\> \oplus U,
\end{align*}
where $U$ represents the rest of the summands in ${}_1P\otimes \I \otimes P_1$ not listed out.
Note that the first three summands of ${}_1P \otimes \I \otimes P_1$ are the only summands in ${}_1P \otimes \I \otimes P_1$ that have non-zero maps under the map
\[
\mu|_{ {}_1P \otimes \I} \otimes \id - \id \otimes \mu|_{\I \otimes P_1} : {}_1P \otimes \I \otimes P_1 \ra {}_1P \otimes P_1
\]
into the second to fourth summands of ${}_1P \otimes P_1$ (under the map above there is no non-zero map into the first summand of ${}_1P \otimes P_1$).
A simple computation shows that
\[
\restr{\mu}{ {}_1P \otimes \I } \otimes \id - \id \otimes \restr{\mu}{ \I \otimes P_2} = 
	\begin{bmatrix}
	0  &  0    & 0    & 0   & 0    & 0 \\
	-1 & -\cap & 0    & 0   & 0    & 0 \\
	1  &  0    & \cap & 0   & 0    & 0 \\
	0  &  1    & -1   & 0   & 0    & 0 \\
	0  &  0    & 0    & 1   & 0    & * \\
	0  &  0    & 0    & 0   & -I_4 & * \\
	\end{bmatrix},
\]
where $I_4$ is the $4\times 4$ identity matrix and ``1'' denotes the identity map of the respective componenent in  $\TLJ_3^\text{even}$.

Let $\mathscr{E}$ be an object in $\TLJ_3^\text{even}$ with $\Theta: {}_1P \otimes P_1 \ra \mathscr{E}$ satisfying 
\begin{equation} \label{coeq 1P1}
\Theta \circ \restr{\mu}{ {}_1P \otimes \I } \otimes \id = \Theta\circ \id \otimes \restr{\mu}{ \I \otimes P_1}.
\end{equation}
Given the ordering of the summands of ${}_1P\otimes P_1$ as before, we may write $\Theta$ as
\[
\Theta = 
	\begin{bmatrix}
	f & g & h & k & l_1 & \cdots & l_5
	\end{bmatrix}.
\]
We wish to find the unique map such that the following diagram commutes:
\begin{center}
\begin{tikzcd}[column sep = 1.5cm]
{}_1P \otimes \I \otimes P_1 
	\ar[r, shift left=0.75ex, "\restr{\mu}{ {}_1P \otimes \I } \otimes \id"] 
	\ar[r, shift right=0.75ex, swap, "\id \otimes \restr{\mu}{ \I \otimes P_1}"] 
& 
{}_1P\otimes P_1 
	\ar[d, "\Theta"]  \ar[r, "\restr{\mu}{{}_1P \otimes P_1}"]
& 
{}_1P_1 = e_1 \oplus X_1 \<2\>. \ar[dashed, dl, "\exists !"] \\

&
\mathscr{E}
\end{tikzcd}
\end{center}
To do so, let us see what \cref{coeq 1P1} says about $\Theta$:
\begin{align*}
&\Theta \circ (\restr{\mu}{ {}_1P \otimes \I } \otimes \id - \id \otimes \restr{\mu}{ \I \otimes P_1}) = 0 \\
\iff
	&\begin{bmatrix}
	f & g & h & k & l_1 & \cdots & l_5
	\end{bmatrix}
	\begin{bmatrix}
	0  &  0    & 0    & 0   & 0    & 0 \\
	-1 & -\cap & 0    & 0   & 0    & 0 \\
	1  &  0    & \cap & 0   & 0    & 0 \\
	0  &  1    & -1   & 0   & 0    & 0 \\
	0  &  0    & 0    & 1   & 0    & * \\
	0  &  0    & 0    & 0   & -I_4 & * \\
	\end{bmatrix} = 
	\begin{bmatrix}
	0 & \cdots & 0
	\end{bmatrix} \\
\iff
	&\begin{cases}
	h = g; \\
	k = g\circ \cap \\
	l_j = 0 \text{ for all } j.
	\end{cases}
\end{align*}
Thus $\Theta = 
	\begin{bmatrix}
	f & g & g & g\circ \cap & 0 & \cdots & 0
	\end{bmatrix}$.
It now follows that the unique map from $e_1 \oplus X_1 \<2\>$ to $\mathscr{E}$ making the diagram commutative is given by 
$\begin{bmatrix}
f & g
\end{bmatrix}$.
\end{proof}

\begin{note}
{} The above proposition shows objects that satisfy the universal property that is required when tensoring certain right $\I$-modules with left $\I$-modules over $\I$, i.e. ${}_iP \otimes_\I P_j$. But because tensoring over $\I$ was defined as some object satisfying a universal property instead of an explicit construction as in the usual modules, it is only unique up to isomorphism.
So speaking of functors that is tensoring over $\I$, ${}_i P \otimes_\I -$ for example, requires a bit more care.
On the object level, I should really be specifying where each object is mapped to under the functor.
In particular, to make things ``nice'', I should construct it in a way that that the tensor product will satisfy the associativity property, for example $(P_i \otimes {}_i P) \otimes_\I P_i$ should be \emph{equal} to $P_i \otimes ({}_i P \otimes_\I P_i)$.
\end{note}

\begin{proposition}
The following maps are both grading preserving $(\I, \I)$-bimodules maps: 
\begin{enumerate}
\item $\beta_i : P_i \otimes {}_iP \ra \I$ given the the restriction of $\mu$, and 
\item $\gamma_i : \I \ra P_i \otimes {}_iP \<-2\>$ uniquely defined by $\1 \xra{\eta} e_i \oplus e_{i \pm 1} \xra{\varphi} (e_i \otimes X_i) \oplus (X_i \otimes e_i) \oplus (i \pm 1 | i) \otimes (i | i \pm 1))$, where $\varphi =
\begin{bmatrix}
\id_\1 & 0 \\
\id_\1 & 0 \\
0      & \cup 
\end{bmatrix}$.
\end{enumerate}
\end{proposition}
\begin{proof}
The fact that $\beta_i$ is a grading preserving $(\I,\I)$-bimodule map follows from $\beta_i$ being a restriction of the multiplication map $\mu$, the module structures on $P_i$ and ${}_iP$ were inherit from multiplication on $\I$ and that the multiplication is associative. 
{\color{red} CONTINUE HERE}
\end{proof}
\begin{proposition} \label{biadjoint pair}
The pair of objects $(P_i, {}_iP)$ is a Frobenius pair, i.e. the functors $P_i\otimes -$ and ${}_i P \otimes_\I -$ is a biadjoint pair.
\end{proposition}
\begin{proof}
{\color{red} CONTINUE HERE}
\end{proof}
Define the following complexes of $(\I,\I)$-bimodules:
\[
R_i := 0 \ra P_i \otimes {}_i P \xra{\beta_i} \I \ra 0
\]
and
\[
R_i' := 0 \ra \I \xra{\gamma_i} P_i\otimes {}_i P \ra 0,
\]
where both are normalised so that $\I$ is in cohomological degree 0.
We wish to check the following:
\begin{proposition}\label{n=5 relations}
We have the following isomorphisms in $\Kom^b((\I,\I)\text{-bimod})$:
\[
R_i \otimes_\I R_i' \cong \I \cong R_i' \otimes_\I R_i,
\]
and
\[
\sigma_{P_1} \otimes_\I \sigma_{P_2} \otimes_\I \sigma_{P_1} \otimes_\I \sigma_{P_2} \otimes_\I \sigma_{P_1} \cong \sigma_{P_2} \otimes_\I \sigma_{P_1} \otimes_\I \sigma_{P_2} \otimes_\I \sigma_{P_1} \otimes_\I \sigma_{P_2}.
\]
\end{proposition}
The fact that $R_i$ and $R_i'$ are inverses of each other follows easily from a direct calculation similar to \cref{biadjoint pair}. {\color{red} Say more here.}
We shall show the second statement in greater generality in the next section.
}
\comment{
The aim of all these was to categorify the number $2\cos(\frac{\pi}{n})$ by looking for an object with quantum dimension $2\cos(\frac{\pi}{n}) = q + q^{-1}$ for $q = e^{\frac{i\pi}{n}}$.
With our construction, there are exactly two objects in $\TLJ_n$ that has dimension $2\cos(\frac{\pi}{n})$, namely $\Pi_1$ and $\Pi_{n-3}$.
So one of these two objects could be a good guess for the summand in the zigzag algebra representing $(1|2)$ and $(2|1)$.
When $n$ is odd, we chose to consider only the even labelled objects and so $\Pi_1$ is already out of the question. 
In a sense, this is great because by only considering the even labelled objects, we leave ourselves with a single simple object from each pair of objects with the same dimension, so there is no more ambiguity in saying ``the simple object with quantum dimension $1, q+ q^{-1}$'' etc.

The main obstruction is in showing the braiding relation.
If we were to use the same construction and proof method as in the case for $n$ is odd, we can verify the braid relation by showing that the following holds:
\[
\underbrace{\sigma_{P_1} \otimes_\I R_2 \otimes_\I \sigma_{P_1} \otimes_\I \cdots \otimes_\I \sigma_{P_1}}_{n-1 \text{ times}}\otimes_\I P_2 \cong P_2.
\]
Sadly this isn't true anymore. 
Whether we have chosen $\Pi_1$ or $\Pi_{n-3}$ for $(1|2)$ and $(2|1)$ (when $n=4$ they are the same), we get instead
\[
\underbrace{\sigma_{P_1} \otimes_\I R_2 \otimes_\I \sigma_{P_1} \otimes_\I \cdots \otimes_\I \sigma_{P_1}}_{n-1 \text{ times}}\otimes_\I P_2 \cong P_2 \otimes {\color{red}\Pi_{n-2}}.
\]
Recall that in section 1, we said $\Pi_{n-2}$ is the object other than the monoidal unit $\Pi_0$ with dimension 1, and that when tensored with, it switches between the two pairs of objects that have the same quantum dimension.
So the decategorified version of this equation the right one, we just somehow ended up with the other simple object of dimension 1 tagging along instead.

Note that I have yet to verify the braiding relation directly even for the minimally interesting case $n=4$ without going through the biadjoint trick that we have been using, so the braiding relation could still very well be true.
But even if it true by a direct verification for $n=4$ (I will sit down and calculate it some time this week), generalising it to arbitrary even $n$ might be hard.
}

\newpage
\part{Categorification}
\chapter{Categorifying Burau representations}
Throughout, we fix $n$ to be an integer $\geq 3$.
In this chapter we shall define the zigzag algebra associated to $I_2(n)$ as a (non-commutative) algebra (also known as monoid) object in the category of $\Z$-graded objects $\gTLJ_n$ of the Temperley-Lieb-Jones category.
Following the idea in \cite{khovanov_seidel_2001}, we will construct a (weak) faithful action of $\B(I_2(n))$ on the homotopy category of (an additive full subcategory of) modules over the zigzag algebra, with the action given by tensoring with complexes of bimodules.
Finally we show that this action categorifies the corresponding Burau representation.

\section{The $I_2(n)$ zigzag algebra and its modules}
Recall from \cref{subsection: TLJ} that the Temperley-Lieb-Jones category at $2n$th root of unity $\TLJ_n$ is a fusion category with set of simple objects $\{ \Pi_0, \Pi_1, ... , \Pi_{n-2} \}$ and its fusion ring $K_0(\TLJ_n)$ is isomorphic to the ring $\Omega_n := \Z[\omega]/\< \Delta_{n-1}(\omega) \>$.
We shall construct our algebra object in $\gTLJ_n$, the category of $\Z$-graded objects of $\TLJ_n$ (see \cref{sect: graded categories}).
Note that $\gTLJ_n$ is no longer a fusion category (it is a tensor category); it satisfies all the  properties of a fusion category except having finitely many classes of simple objects.

\subsection{Constructing $\I$, the $I_2(n)$ zigzag algebra}
This subsection is devoted to the construction of the zigzag algebra associated to $I_2(n)$.
It will be constructed as an algebra object in the tensor category $\gTLJ_n$.

Before we start the construction, let us introduce some suggestive notations.
\begin{enumerate}[-]
\item Define $e_i := \Pi_0 = \1_{\gTLJ_n}$ for each $i \in \{1,2\}$. We call this the \emph{constant path object on vertex $i$}.
\item Define $X_i := \Pi_0 \<2\> = \1_{\gTLJ_n} \<2\>$ for each $i \in \{1,2\}$. We call this the \emph{length two loop object on vertex $i$ (starts and ends at $i$)}.
\item Define $(i|j) := \Pi_1 \<1\>$ for each $|i-j|=1$.
We call this the \emph{length one path object that starts at $i$ and ends at $j$}.
\end{enumerate}
We construct the $I_2(n)$ zigzag algebra, denoted as $\I := \I_2(n)$, as an algebra object in $\gTLJ_n$ as follows.
As an object in $\gTLJ_n$, it is given by
\begin{align*}
\I &:= e_1 \oplus X_1 \oplus (2|1) \oplus e_2 \oplus X_2 \oplus (1|2).
\end{align*}
As before, we should think of these simple summands as paths (objects) graded by their path lengths, which we call them \emph{simple path summands}.

We define the multiplication map $\mu : \I \otimes \I \ra \I$ on $\I$ as if we were multiplying paths, where we kill all the paths of length $\geq 3$.
More precisely, we define it on each tensor of two simple path summands $A \otimes A'$ as follows.
\begin{enumerate}
\item If the ending vertex of the simple path summand $A$ does not match with the starting vertex of $A'$, or the grading of $A\otimes A'$ is greater than or equal to 3, then $\mu|_{A\otimes A'} = 0$.
\item Otherwise, the restriction that the (homogeneous) grading of $A\otimes A'$ is less than or equal to 2 tells us that either:
\begin{enumerate}
\item  at least one of $A$ or $A'$ must be $e_i$ for some $i$, or 
\item $A$ and $A'$ are both length one path objects.
\end{enumerate}
\item Suppose first that $A$ or $A'$ is equal to $e_i = \Pi_0$.
By definition of the tensor product we have
\[
A'' := A \otimes A' =
	\begin{cases}
	A, &\text{ if } A' = e_i; \\
	A', &\text{ if } A = e_i
	\end{cases}.
\]
In this case we just define $\mu|_{A\otimes A'}: A\otimes A' \ra A''$ to be the identity map on $A''$.
\item Now suppose instead that $A$ and $A'$ are both length one path objects, i.e. $A = (i|j)$ and $A' = (j|i)$ as the end points must match.
Then $\mu|_{(i|j)\otimes (j|i)}: (i|j)\otimes (j|i) \ra X_i$  is given by the admissible map:
\begin{center}
\begin{tikzpicture}[scale = 0.5]
	\node at (1.5, 3)    {$X_i = \Pi_0 \<2\>$};
	\node at (-0.5,-0.4) {$(i|j) = \Pi_1 \<1\>$};
	\node at (5.5,-0.4) {$\Pi_1 \<1\> = (j|i)$};
  \draw (0, 0) arc[start angle=180,end angle=0,radius=2];
\end{tikzpicture}
\end{center}
Note that this is the map $\cap \<2\>$, where $\cap$ is the counit map that makes $\Pi_1$ self-dual.
\end{enumerate}
It is clear that the only ``interesting'' morphism within the multiplication map is the multiplication of two length one paths, where the rest of them are either the corresponding identity map or the zero map.
As such, it is easy to check the associativity of $\mu$ due to the path length restriction, since the multiplication of three simple path objects is zero except when at least one of them is the constant path $e_i$ and moreover, multiplication with $e_i$ only ends up being either the identity map (when the endpoints match) or the zero map otherwise.

Let $A$ be a simple path summand in $\I$.
Note that $\mu|_{e_i \otimes A}$ is non-zero if and only if $A$ is a path object with starting vertex $i$.
Moreover, when $\mu|_{e_i \otimes A} : e_i \otimes A = A \ra A$ is non-zero, it must be the identity map on $A$ by definition.
Similarly $\mu|_{A \otimes e_i}$ is non-zero if and only if $A$ is a path object with ending vertex $i$, and when it is non-zero it must be the identity map.
Let us denote $P_i$ (resp. ${}_iP$) to be the direct sum of all the simple path summands of $\I$ ending (resp. starting) with vertex $i$:
\[
P_i = e_i \oplus X_i \oplus (j|i), \quad {}_iP = e_i \oplus X_i \oplus (i|j).
\]
Note that as an object in $\gTLJ_n$, we can decompose $\I$ as follows:
\[
\I = P_1 \oplus P_2 = {}_1 P \oplus {}_2 P.
\]
It follows that $\mu|_{e_i \otimes \I}$ only maps into the summand ${}_iP$ of $\I$; in particular 
\[
\mu|_{e_i \otimes {}_jP} = 
	\begin{cases}
	\id_{{}_iP}, &\text{ if i = j}; \\
	0          , &\text{ otherwise}.
	\end{cases}
\]
Similarly $\mu|_{\I \otimes e_i}$ only maps into the summand $P_i$ of $\I$:
\[
\mu|_{P_j \otimes e_i} = 
	\begin{cases}
	\id_{P_i}, &\text{ if j = i}; \\
	0          , &\text{ otherwise}.
	\end{cases}
\]

As in usual path algebras, we define the unit map $\eta: \1_{\gTLJ_n} \ra \I$ as the composition
\[
\Pi_0 \xra{ \left[\id \quad  \id \right] }  e_1 \oplus e_2 \hookrightarrow \I,
\]
where the last map is the canonical inclusion of the summands.
We now use the decomposition $\I = {}_1P \oplus {}_2P$ to obtain the commutative diagram
\[
\begin{tikzcd}[row sep = large]
\Pi_0 \otimes \left( {}_1P \oplus {}_2P \right)
	\ar{rr}{\left[\id \quad  \id \right] \otimes \id_{\I}} &
{}
	{} &
\left( e_1 \oplus e_2 \right) \otimes \left( {}_1P \oplus {}_2P \right)
	 \ar[equal]{d} 
	 \ar[hook]{dr}\\
\Pi_0 \otimes \I
	\ar[equal]{u} &
{}
	{} &
\bigoplus_{i,j \in \{1,2\}} e_i \otimes {}_jP
	\ar{d}{\left[\mu|_{e_i \otimes {}_jP } \right]_{i,j\in \{1,2\}}}  &
\I \otimes \I
	\ar{d}{\mu} &\\
\I
	\ar[equal]{u} 
	\ar[equal]{rr}&
{}
	{} &
{}_1P \oplus {}_2P &
\I,
	\ar[equal]{l} &
\end{tikzcd}
\]
which shows the left unital condition.
The right unital condition follows similarly.
Therefore we conclude the following:
\begin{proposition}
$(\I, \mu, \eta)$ as defined above defines a (non-commutative) algebra object in $\gTLJ_n$.
\end{proposition}
From now on, we call $(\I, \mu, \eta)$ \emph{the zigzag algebra of $I_2(n)$}, where for notation simplicity will mostly be denoted as $\I$.

\begin{remark}
Note the emphasis on $\I$ being non-commutative, which in particular says that $\I$ does not fall into classification of the \emph{commutative} algebra objects in $\TLJ_n \cong \cR ep\left(U_{e^{i \pi/n}}(\mathfrak{s}\mathfrak{l}_2)\right)$, developed as the $q$-analogue of the McKay correspondence in \cite{kirillov_ostrik_2002}.
\end{remark}
\begin{remark}
Recall that 
\[
PF\dim(\Pi_1) = PF\dim(\Pi_{n-3}),
\]
and so $\Pi_{n-3}\<1\>$ is another natural candidate for the length one simple path object $(i|j)$.
The choice $(i|j) := \Pi_1\<1\>$ made here was just to simplify the calculations and the interested reader can check that a similar construction using $(i|j) := \Pi_{n-3}\<1\>$ would equally work.
In the specific case for $n=3$, this variant of $\I$, denoted as $\wt{\I}$, which uses
\[
(i|j) := \Pi_{n-3}\<1\> = \Pi_0\<1\>
\]
instead, coincides with the usual (path length graded) zigzag algebra associated to $A_2$ in the following sense:
the algebra $\wt{\I}$ for $n=3$ is the image of the $A_2$ zigzag algebra under the (additive, monoidal, $\Z$-graded) functor
\[
\C\text{-vec} \ra \gTLJ_3
\]
which sends $\C$ to $\Pi_0$.
\end{remark}

\comment{
We shall define the (graded) multiplication map $\mu : \I \otimes \I \ra \I$ using path concatenation, where we kill all length 2 paths except the loops $X_i$.
To spell it out, if the end points don't match or the concatenation is a path with length $\geq$ 2 that is not a loop, then the multiplication map is zero; else, it is given by the $q$-admissible map into the simple path in $\I$ representing the concatenation.
For example, $(1|2)\<1\> \otimes (2|1)\<1\> \xra{\mu} X_1\<2\>$ is given by the $q$-admissible map of $\Pi_1 \otimes \Pi_1 \ra \Pi_0$, which we draw as the arc \begin{center}
\begin{tikzpicture}[scale = 0.5]
	\node at (0,-0.4) {$\Pi_1$};
	\node at (2,-0.4) {$\Pi_1$};
  \draw (0, 0) arc[start angle=180,end angle=0,radius=1];
\end{tikzpicture} ,
\end{center}
whereas the multiplication map on $(1|2) \otimes X_2$ is the zero map.
It is easy to check the associativity of $\mu$, since any concatenation of three simple paths is zero unless one of them is the constant path, and multiplication with $e_i$ or $X_i$ only end up being either the identity map or the 0 map.

The unit $\eta: \Pi_0 \ra \I$ is given by the following (grading preserving) morphism:
\[
\Pi_0 \xra{\begin{bmatrix}
	\id \\
	\id
	\end{bmatrix}} e_1 \oplus e_2 \hookrightarrow \I.
\]
For the unital condition, note that for each $L$ a simple path in $\I$, $\mu|_{e_i \otimes L}$ is non-zero if and only if $L$ is a path starting at $i$.
Moreover, being non-zero implies that $\mu|_{e_i \otimes L} : e_i \otimes L = L \ra L$ is the identity map by definition. 
Since every path have exactly one starting point, $\mu|_{e_i \otimes L}$ is the identity map for exactly one $i \in \{1,2\}$.
Thus
\[
\Pi_0 \otimes \I \xra{\begin{bmatrix}
	\id \\
	\id
	\end{bmatrix} \otimes \id} 
	(e_1 \oplus e_2)\otimes \I = \I \oplus \I \xra{\mu|_{(e_1 \oplus e_2)\otimes \I}} \I
\]
is the identity map.
The case for $L\otimes e_i$ follows similarly.
}

\subsection{Modules over $\I$}
Now that we have an algebra object $(\I, \mu, \eta)$, we can consider the different types of modules over $\I$.
\begin{definition}
We denote the category of \emph{left (resp. right) modules over $\I$} by $\I$-mod (resp. mod-$\I$).
Similarly we denote the category of bimodules over $\I$ by $\I$-mod-$\I$.
\end{definition}
We remind the reader (see \cref{sect: algebras and modules}) that an object in $\I$-mod is by definition a pair $(M, \alpha)$, with $M$ an object in $\gTLJ_n$ and $\alpha : \I \otimes M \ra M$ a morphism in $\gTLJ_n$ satisfying the usual associativity and unital commutative diagram.
The objects in mod-$\I$ are defined similarly.
The morphisms in $\I$-mod and mod-$\I$ are morphisms in $\gTLJ_n$ that preserves the module structure.
Note that this implies that the morphisms are required to preserve the internal $\Z$-grading.
In other words, $\I$-mod and mod-$\I$ are naturally (abelian) subcategories of $\gTLJ_n$.
As an easy example, one can check that ($\I, \mu$) defines both left and right $\I$-module structure on $\I$.

A bimodule over $\I$ will be an object $M$ which is both a left module $(M, \alpha)$ and a right module $(M, \alpha')$ such that the left and right module structures are compatible.
We usually denote a bimodule as a triple $(M, \alpha, \alpha')$.
Once again, it is easy to see that $(\I, \mu, \mu)$ defines a $\I$-bimodule structure on $\I$.

Recall the objects $P_i$ and ${}_iP$ as defined in the previous section.
There are natural left $\I$-module actions induced by the restriction of $\mu$ to $\mu|_{\I \otimes P_i} : \I \otimes P_i \ra P_i$, where it again follows from the ``path concatenation'' definition of $\mu$ that the map does indeed map into $P_i$.
The left module conditions also follow easily from the structure of the algebra object $\I$.
Dually, ${}_iP$ can be endowed with a right $\I$-module structure using $\mu|_{{}_iP \otimes \I} : {}_iP \otimes \I \ra {}_iP$.

\comment{
Now that we have a graded algebra object $(\I, \mu, \eta)$, we can consider graded $\I$-modules.
In particular, consider the following ``projective'' left $\I$-modules:
\[
P_1 := e_1 \oplus X_1\<2\> \oplus (2|1)\<1\> , \quad P_2:= e_2 \oplus X_2\<2\> \oplus (1|2)\<1\>,
\]
where $\I = P_1 \oplus P_2$.
There are a natural left $\I$-module actions induced by the restriction of $\mu$, namely $\mu|_{\I \otimes P_i} : \I \otimes P_i \ra P_i$, where it again follows from the path concatenation definition of $\mu$ that the map does indeed map into $P_i$.
The left module conditions also follows easily from the structure of the algebra object $\I$.
Similarly, we can define the right $\I$-modules below in a similar fashion:
\[
{}_1P := e_1 \oplus X_1\<2\> \oplus (1|2) \<1\>, \quad {}_2P := e_2 \oplus X_2 \<2\> \oplus (2|1) \<1\>.
\]
}

Recall from \cref{defn: tensoring over an algebra} the definition of tensoring of two modules over $\I$.
We will be particularly interested in the following tensors:
\begin{proposition} \label{tensor iPj}
Denote ${}_iP_j$ as the subobject of $\I$ consisting of all the simple path summands in $\I$ starting with $i$ and ending with $j$. 
Together with the morphism $\mu|_{{}_i P \otimes P_j}: {}_iP \otimes P_j \ra {}_iP_j$, we have that ${}_iP_j$ is the coequaliser of 
\begin{tikzcd}[column sep=large]
{}_iP \otimes \I \otimes P_j \ar[r, shift left=0.75ex, "\mu|_{{}_iP\otimes \I} \otimes \id"] \ar[r, shift right=0.75ex, swap, "\id \otimes \mu|_{\I\otimes P_j}'"] & {}_iP\otimes P_j
\end{tikzcd}
in $\gTLJ_n$.
\end{proposition}
\begin{proof}
Throughout this proof, $\hookrightarrow$ and $\twoheadrightarrow$ will be used to denote the obvious inclusion and projection map of the summands respectively.
Let $\varphi: {}_iP \otimes P_j \ra \xi$ be a map satisfying 
\[
\varphi\circ \left( \mu|_{{}_iP \otimes \I} \otimes \id_{P_j} \right) =
\varphi\circ \left( \id_{{}_iP} \otimes \mu|_{\I \otimes P_j} \right).
\]
Consider the commutative diagram
\[
\begin{tikzcd}[column sep = 2.5cm, row sep = large]
{}_iP \otimes \I \otimes P_j 
	\ar[r, shift left=0.75ex, "\mu|_{{}_iP \otimes \I} \otimes \id_{P_j}"] 
	\ar[r, shift right=0.75ex, swap, "\id_{{}_iP} \otimes \mu|_{\I \otimes P_j}"] 
	\ar[d, hook] & 
{}_iP \otimes P_j 
	\ar[r, "\mu|_{{}_iP \otimes P_j}"] 
	\ar[d, hook] &
{}_iP_j 
	\ar[d, hook]\\
\I \otimes \I \otimes \I 
	\ar[r, shift left=0.75ex, "\mu \otimes \id_{\I}"] 
	\ar[r, shift right=0.75ex, swap, "\id_{\I} \otimes \mu"] & 
\I \otimes \I 
	\ar[r, "\mu"]
	\ar[d, two heads]
& \I 
	\ar[ddl, dashed, "\exists !"] \\
{}
	{} &
{}_iP \otimes P_j
	\ar[d, "\varphi"] &
{}
	{} \\
{}
	{} &
\xi
\end{tikzcd}
\]
where the existence of the unique (dashed) map follows from $\I \otimes_{\I} \I \cong \I$.
It is easy to see that the induced map from ${}_sP_t$ to $\xi$ is indeed unique and so $({}_iP_j, \mu|_{{}_iP \otimes P_j})$ does satisfy the required universal property.
\end{proof}
\comment{
\begin{proof}
We shall prove for the cases ${}_1P\otimes_\I P_2$ and ${}_1P \otimes_\I P_1$, where the other two cases follows from a symmetrical argument.
Let's start with ${}_1P\otimes_\I P_2$.
We shall organise ${}_1P\otimes P_2$ in the following order:
\begin{align*}
{}_1P\otimes P_2 =& (e_1 \otimes (1|2))\<1\> \oplus ((1|2) \otimes e_2)\<1\> \oplus \\
	& (e_1 \otimes e_2) \oplus (e_1 \otimes X_2)\<2\> \oplus \\
	& ((1|2) \otimes X_2)\<3\> \oplus ((1|2) \otimes (1|2))\<2\> \oplus \\
	& (X_1 \otimes e_2)\<2\> \oplus (X_1 \otimes X_2)\<4\> \oplus (X_1 \otimes (1|2))\<3\>.
\end{align*}
Note that only the first two direct summand of ${}_1P\otimes P_2$ from above are non-zero under $\restr{\mu}{{}_1P\otimes P_2}: {}_1P \otimes P_2 \ra {}_1P_2 = (1|2)\<1\>$, i.e.
\[
\restr{\mu}{{}_1P \otimes P_2} = \begin{bmatrix}
	\id_{\Pi_0} & \id_{\Pi_0} & 0 & \cdots & 0
	\end{bmatrix}
\]

On the other hand, we shall organise the direct sumamnds of ${}_1P \otimes \I \otimes P_2$ in the following order:
\begin{align*}
{}_1P \otimes \I \otimes P_2 = 
	&((1|2)\otimes e_2 \otimes e_2)\<1\> \oplus (e_1 \otimes (1|2) \otimes e_2)\<1\> \oplus (e_1 \otimes e_1 \otimes (1|2))\<1\> \oplus \\
	&(e_1 \otimes e_1 \otimes e_2) \oplus (e_1 \otimes e_1 \otimes X_2)\<2\> \oplus \\
	&(e_1 \otimes (1|2) \otimes X_2)\<3\> \oplus (e_1 \otimes (1|2) \otimes (1|2))\<2\> \oplus \\
	&(e_1 \otimes X_1 \otimes e_2)\<3\> \oplus (e_1 \otimes X_1 \otimes X_2)\<4\> \oplus \\
	&(e_1 \otimes X_1 \otimes (1|2))\<3\> \oplus U,
\end{align*}
where $U$ represents the rest of the summands in ${}_1P\otimes \I \otimes P_2$.
Note that the first three summands of ${}_1P \otimes \I \otimes P_2$ are the only summands in ${}_1P \otimes \I \otimes P_2$ such that under both of the following maps:
\[
\mu|_{ {}_1P \otimes \I} \otimes \id, \quad \id \otimes \mu|_{\I \otimes P_2} : {}_1P \otimes \I \otimes P_2 \ra {}_1P \otimes P_2
\]
they are non-zero into the first two summands of ${}_1P \otimes P_2$.
A simple computation shows that
\[
\restr{\mu}{ {}_1P \otimes \I } \otimes \id - \id \otimes \restr{\mu}{ \I \otimes P_2} = 
	\begin{bmatrix}
	0 & -1 & 0 & 0   & 0 \\
	0 &  1 & 0 & 0   & 0 \\
	0 &  0 & 0 & I_7 & * 
	\end{bmatrix},
\]
where ``1'' denotes the identity map of the respective objects in  $g\TLJ_n$ and $I_7$ is the $7\times 7$ identity matrix.

Let $\mathscr{E}$ be an object in $\TLJ_3^\text{even}$ with $\Theta: {}_1P \otimes P_2 \ra \mathscr{E}$ satisfying 
\begin{equation} \label{coeq 1P2}
\Theta \circ \restr{\mu}{ {}_1P \otimes \I } \otimes \id = \Theta\circ \id \otimes \restr{\mu}{ \I \otimes P_2}.
\end{equation}
Given the ordering of the summands of ${}_1P\otimes P_2$ as before, we may write $\Theta$ as
\[
\Theta = 
	\begin{bmatrix}
	f & g & h_1 & h_2 & \cdots & h_7
	\end{bmatrix}.
\]
We wish to find the unique map such that the following diagram commutes:
\begin{center}
\begin{tikzcd}[column sep = 1.5cm]
{}_1P \otimes \I \otimes P_2 
	\ar[r, shift left=0.75ex, "\restr{\mu}{ {}_1P \otimes \I } \otimes \id"] 
	\ar[r, shift right=0.75ex, swap, "\id \otimes \restr{\mu}{ \I \otimes P_2}"] 
& 
{}_1P\otimes P_2 
	\ar[d, "\Theta"]  \ar[r, "\restr{\mu}{{}_1P \otimes P_2}"]
& 
{}_1P_2 = (1|2)\<1\>. \ar[dashed, dl, "\exists !"] \\

&
\mathscr{E}
\end{tikzcd}
\end{center}
To do so, let us see what \cref{coeq 1P2} says about $\Theta$:
\begin{align*}
&\Theta \circ \restr{\mu}{ {}_1P \otimes \I } \otimes \id = \Theta\circ \id \otimes \restr{\mu}{ \I \otimes P_2} \\
\iff &\Theta \circ (\restr{\mu}{ {}_1P \otimes \I } \otimes \id - \id \otimes \restr{\mu}{ \I \otimes P_2}) = 0 \\
\iff
	&\begin{bmatrix}
	f & g & h_1 & h_2 & \cdots & h_7
	\end{bmatrix}
	\begin{bmatrix}
	0 & -1 & 0 & 0   & 0 \\
	0 &  1 & 0 & 0   & 0 \\
	0 &  0 & 0 & I_7 & * 
	\end{bmatrix} = 
	\begin{bmatrix}
	0 & \cdots & 0
	\end{bmatrix} \\
\iff
	&\begin{cases}
	f = g; \\
	h_j = 0 \text{ for all } j.
	\end{cases}
\end{align*}
Thus $\Theta = 
	\begin{bmatrix}
	f & f & 0 & \cdots & 0
	\end{bmatrix}$.
\\
Since $f = g$ are maps from $(e_1 \otimes (1|2)) \<1\> = ((1|2) \otimes e_2)\<1\> = (1|2)\<1\>$ to $\mathscr{E}$, it follows that the map $f: (1|2)\<1\> \ra \mathscr{E}$ is the unique map that makes the diagram commutative.

Now consider ${}_1P \otimes_\I P_1$. 
We shall organise the summands of ${}_1P \otimes P_1$ in the following order:
\begin{align*}
{}_1P \otimes P_1 = & 
(e_1 \otimes e_1) \oplus (e_1 \otimes X_1) \oplus (X_1 \otimes e_1) \oplus ((1|2)\otimes (2|1)) \oplus \\
& ((1|2) \otimes e_1) \oplus (e_1 \otimes (1|2)) \oplus ( (1|2) \otimes X_1) \oplus (X_1 \otimes (2|1)) \oplus (X_1 \otimes X_1).
\end{align*}
Note that the first summand is the only non-zero maps into $e_1$ and the next three following summands are the only non-zero maps into $X_1$ under the map $\restr{\mu}{{}_1P\otimes P_1}: {}_1P \otimes P_1 \ra {}_1P_1 = e_1 \oplus X_1\<2\>$, i.e.
\[
\restr{\mu}{{}_1P \otimes P_1} = 
	\begin{bmatrix}
	\id_{\Pi_0} & 0      & 0      & 0    & 0 & \cdots & 0 \\
	0      & \id_{\Pi_0} & \id_{\Pi_0} & \cap & 0 & \cdots & 0 
	\end{bmatrix}
\]
On the other hand, we shall organise the direct sumamnds of ${}_1P \otimes \I \otimes P_1$ in the following order:
\begin{align*}
{}_1P \otimes \I \otimes P_1 = 
	&(e_1 \otimes X_1 \otimes e_1)\<2\> \oplus (e_1 \otimes (1|2) \otimes (2|1))\<2\> \oplus ((1|2) \otimes \otimes (2|1) \otimes e_1)\<2\> \oplus \\
	&(e_1 \otimes (1|2) \otimes e_1)\<1\> \oplus (e_1 \otimes e_2 \otimes (2|1))\<1\> \oplus \\
	&((1|2) \otimes X_1 \otimes e_1)\<3\> \oplus (X_1 \otimes (2|1) \otimes e_1)\<3\> \oplus \\
	&(X_1 \otimes X_1 \otimes e_1)\<4\> \oplus U,
\end{align*}
where $U$ represents the rest of the summands in ${}_1P\otimes \I \otimes P_1$.
Note that the first three summands of ${}_1P \otimes \I \otimes P_1$ are the only summands in ${}_1P \otimes \I \otimes P_1$ that have non-zero maps under the map
\[
\mu|_{ {}_1P \otimes \I} \otimes \id - \id \otimes \mu|_{\I \otimes P_1} : {}_1P \otimes \I \otimes P_1 \ra {}_1P \otimes P_1
\]
into the second, third and fourth summands of ${}_1P \otimes P_1$ (under the map above there is no non-zero map into the first summand of ${}_1P \otimes P_1$).
A simple computation shows that
\[
\restr{\mu}{ {}_1P \otimes \I } \otimes \id - \id \otimes \restr{\mu}{ \I \otimes P_2} = 
	\begin{bmatrix}
	0  &  0    & 0    & 0   & 0    & 0 \\
	-1 & -\cap & 0    & 0   & 0    & 0 \\
	1  &  0    & \cap & 0   & 0    & 0 \\
	0  &  1    & -1   & 0   & 0    & 0 \\
	0  &  0    & 0    & 1   & 0    & * \\
	0  &  0    & 0    & 0   & -I_4 & * \\
	\end{bmatrix},
\]
where ``1'' denotes the identity map of the respective objects in  $\gTLJ_n$ and $I_4$ is the $4\times 4$ identity matrix.

Let $\mathscr{E}$ be an object in $\TLJ_3^\text{even}$ with $\Theta: {}_1P \otimes P_1 \ra \mathscr{E}$ satisfying 
\begin{equation} \label{coeq 1P1}
\Theta \circ \restr{\mu}{ {}_1P \otimes \I } \otimes \id = \Theta\circ \id \otimes \restr{\mu}{ \I \otimes P_1}.
\end{equation}
Given the ordering of the summands of ${}_1P\otimes P_1$ as before, we may write $\Theta$ as
\[
\Theta = 
	\begin{bmatrix}
	f & g & h & k & l_1 & \cdots & l_5
	\end{bmatrix}.
\]
We wish to find the unique map such that the following diagram commutes:
\begin{center}
\begin{tikzcd}[column sep = 1.5cm]
{}_1P \otimes \I \otimes P_1 
	\ar[r, shift left=0.75ex, "\restr{\mu}{ {}_1P \otimes \I } \otimes \id"] 
	\ar[r, shift right=0.75ex, swap, "\id \otimes \restr{\mu}{ \I \otimes P_1}"] 
& 
{}_1P\otimes P_1 
	\ar[d, "\Theta"]  \ar[r, "\restr{\mu}{{}_1P \otimes P_1}"]
& 
{}_1P_1 = e_1 \oplus X_1 \<2\>. \ar[dashed, dl, "\exists !"] \\

&
\mathscr{E}
\end{tikzcd}
\end{center}
To do so, let us see what \cref{coeq 1P1} says about $\Theta$:
\begin{align*}
&\Theta \circ (\restr{\mu}{ {}_1P \otimes \I } \otimes \id - \id \otimes \restr{\mu}{ \I \otimes P_1}) = 0 \\
\iff
	&\begin{bmatrix}
	f & g & h & k & l_1 & \cdots & l_5
	\end{bmatrix}
	\begin{bmatrix}
	0  &  0    & 0    & 0   & 0    & 0 \\
	-1 & -\cap & 0    & 0   & 0    & 0 \\
	1  &  0    & \cap & 0   & 0    & 0 \\
	0  &  1    & -1   & 0   & 0    & 0 \\
	0  &  0    & 0    & 1   & 0    & * \\
	0  &  0    & 0    & 0   & -I_4 & * \\
	\end{bmatrix} = 
	\begin{bmatrix}
	0 & \cdots & 0
	\end{bmatrix} \\
\iff
	&\begin{cases}
	h = g; \\
	k = g\circ \cap \\
	l_j = 0 \text{ for all } j.
	\end{cases}
\end{align*}
Thus $\Theta = 
	\begin{bmatrix}
	f & g & g & g\circ \cap & 0 & \cdots & 0
	\end{bmatrix}$.
It now follows that the unique map from $e_1 \oplus X_1 \<2\>$ to $\mathscr{E}$ making the diagram commutative is given by 
$\begin{bmatrix}
f & g
\end{bmatrix}$.
\end{proof}
}

\comment{
\begin{note}
{} The above proposition shows objects that satisfy the universal property that is required when tensoring certain right $\I$-modules with left $\I$-modules over $\I$, i.e. ${}_iP \otimes_\I P_j$. But because tensoring over $\I$ was defined as some object satisfying a universal property instead of an explicit construction as in the usual modules, it is only unique up to isomorphism.
So speaking of functors that is tensoring over $\I$, ${}_i P \otimes_\I -$ for example, requires a bit more care.
On the object level, I should really be specifying where each object is mapped to under the functor.
In particular, to make things ``nice'', I should construct it in a way that that the tensor product will satisfy the associativity property, for example $(P_i \otimes {}_i P) \otimes_\I P_i$ should be \emph{equal} to $P_i \otimes ({}_i P \otimes_\I P_i)$.
\end{note}
}

Now let $X$ and $Y$ be objects in $\I$-mod and mod-$\I$ respectively, and let $L$ be an object in $\TLJ_n$.
It is easy to see that $X \otimes L$ is again a left $\I$-module; similarly $L \otimes Y$ is again a right $\I$-module.
Moreover, $\I$-mod and mod-$\I$ are closed under the grading shift functor $\<1\>$.
Therefore, ($\I$-mod, $\<1\>, -\otimes ?)$ and (mod-$\I$, $\<1\>, ?\otimes -)$ are both module categories over $\TLJ_n$, equipped with the distinguished autoequivalence $\X = \<1\>$.

\begin{definition}
We define $\I$-prmod (resp. prmod-$\I$) as the \emph{smallest additive submodule category} of $\I$-mod (resp. mod-$\I$) over $\TLJ_n$, which contains $P_1$ and $P_2$ (resp. ${}_1P$ and ${}_2P$) and is closed under the grading shift $\<1\>$.
\end{definition}
\comment{
\begin{remark}
Note that both categories above are naturally module categories over $\TLJ_n$ and are equipped with the distinguished autoequivalence $\X = \<1\>$.
\end{remark}
}
By definition, all objects in $\I$-prmod and prmod-$\I$ are direct sums of the objects in 
\[
\{P_i \otimes \Pi_a \<k\>:
					i \in \{1,2\},
					0 \leq a \leq n-2,
					k \in \Z
					\}
\]
and
\[
\{\Pi_a\<k\> \otimes {}_iP :
					i \in \{1,2\},
					0 \leq a \leq n-2,
					k \in \Z
					\}
\]
respectively.
As we shall see in \cref{graded hom space}, these will be indecomposable as well.

Every object $X$ in $\I$-prmod induces an additive functor 
\[
X \otimes - : \gTLJ_n \ra \I \text{-prmod}.
\]
On the other hand, we shall fix 
\begin{equation} \label{iP as functor}
(\Pi_a \otimes {}_iP) \otimes_\I (P_j \otimes \Pi_b) := \Pi_a \otimes ({}_iP \otimes_\I P_j) \otimes \Pi_b = \Pi_a \otimes {}_iP_j \otimes \Pi_b
\end{equation}
using \cref{tensor iPj}, so that every object $Y$ in prmod-$\I$ also induces an additive, graded functor 
\[
Y \otimes_\I - : \I \text{-prmod} \ra \gTLJ_n.
\]
Furthermore, with $X$ and $Y$ as above, we have that $X \otimes Y$ is naturally an $\I$-bimodule, with left and right module structure that of $X$ and $Y$ respectively.
Using \cref{iP as functor}, with $X, X' \in \I$-prmod and $Y, Y' \in $ prmod-$\I$ we shall fix 
\begin{equation} \label{U_i tensor defn}
(X \otimes Y) \otimes_\I (X' \otimes Y') := X \otimes (Y \otimes_\I X') \otimes Y'.
\end{equation}
\comment{
Given $\I$-bimodules $P_i \otimes L \otimes {}_jP$ and $P_k \otimes L' \otimes {}_lP$, we can take their tensor product over $\I$, giving us
\[
P_i \otimes L \otimes {}_jP \otimes_\I P_k \otimes L' \otimes {}_lP = P_i \otimes L \otimes {}_jP_k \otimes L' \otimes {}_lP,
\]
which is again an $\I$-bimodule.
}
\begin{definition} \label{defn: bimodule category}
With $- \otimes_\I -$ defined as in \cref{U_i tensor defn}, the $\I$-bimodules 
\[
\{ X \otimes Y : 
	X \in \I\text{-prmod}, Y \in \text{ prmod-}\I
	\}
			\cup \{ \I \<k\> :  k \in \Z\}
\]
generates an additive, strictly monoidal subcategory within $\I$-mod-$\I$, with $\I$ as the monoidal unit (cf. \cref{absorb I}).
We denote this category as $\mathbb{U}_n$.
\end{definition}

Using \cref{U_i tensor defn}, every object $U = X \otimes Y$ in $\mathbb{U}_n$ induces a functor $U \otimes_\I - : \I$-prmod $\ra \I$-prmod defined by 
\[
U \otimes_\I X' := (X \otimes Y) \otimes_\I X' := X \otimes (Y \otimes_\I X').
\]
where the functor $\I \otimes_\I - : \I$-prmod $\ra \I$-prmod is by definition equal to the identity functor.
%
All of these can be easily extended to the homotopy categories of complexes using the standard tensor structure of complexes: every object $U^\bullet$ in $\Kom^b(\mathbb{U}_n)$ induces an exact endofunctor
\[
U^\bullet \otimes_\I -: \Kom^b(\I\text{-prmod}) \ra \Kom^b(\I\text{-prmod}).
\]
\comment{
\begin{remark}
Using the tensor product $- \otimes -$ in $\gTLJ_n$, we have a bifunctor $- \otimes - : \I$-prmod $\times$ prmod-$\I \ra \mathbb{U}_n$.
Furthermore, this functor is surjective, namely every object in $\mathbb{U}_n$ is an image (not just isomorphic) of an object in $\I$-prmod $\times$ prmod-$\I$.
In particular, any object $U$ in $\mathbb{U}_n$ induces a functor $X\otimes_\I - : \I$-prmod $\ra \I$-prmod since $U = X \otimes Y$ for $X \in \I$-prmod and $Y \in $ prmod-$\I$.
\end{remark}
}
\subsection{Morphisms in $\I$-prmod and some adjunctions}
Let us now understand some of the important morphisms in $\I$-prmod.
Let $(a,b, 1)$ be a $q$-admissible triple (cf. \cref{defn: q-admissible}), or equivalently, $|a-b|=1$. 
We shall abuse notation and use 
\[
(i | i \pm 1): P_i \otimes \Pi_a  \ra P_{i \pm 1} \otimes \Pi_b \<-1\>
\] 
to denote the map defined as follows:
it is the zero map on the simple summand $X_i \otimes \Pi_a$, and on the other two simple summand it is defined by
\begin{center}
\begin{tikzpicture}[scale = 0.8]
  \node (tl) at (-1,1)  {$(i | i \pm 1) \otimes$};
  \node (tr) at (1,1)   {$\Pi_b\<-1\>$};
  \coordinate (m)  at (0,0);
  \node (d)  at (0,-1)  {$\Pi_a$};
  \node[left = 0.5cm of d] {$e_i \otimes$};
  \draw[thick,black] (d)--(m);
  \draw[thick,black] (m)--(tr);
  \draw[thick,black] (m)--(tl);
\end{tikzpicture} \quad and  \quad
\begin{tikzpicture}[scale = 0.8]
  \node (tl) at (-1.5,1)  {$X_j \otimes$};
  \node (tr) at (0,1)   {$\Pi_b\<-1\>$};
  \coordinate (m)  at (0,0);
  \node (dr)  at (1,-1)  {$\Pi_a$};
  \node (dl) at (-1,-1){$(i \pm 1|i)\otimes$};
  \draw[thick,black] (tr)--(m);
  \draw[thick,black] (m)--(dl);
  \draw[thick,black] (m)--(dr);
\end{tikzpicture}.
\end{center}
It is easy to check that the map $(i|i \pm 1)$ is indeed a map of left $\I$-modules.

On the other hand, let $a = b$.
We shall again abuse notation and use 
\[
X_i: P_i \otimes \Pi_a \ra P_i \otimes \Pi_b \<-2\>
\]
to denote the map defined as follows: it is given by the $0$ map on both the summands $(i \pm 1 | i) \otimes \Pi_a$ and $X_i \otimes \Pi_a$, and the summand $e_i \otimes \Pi_a$ is mapped to the summand $X_i \otimes \Pi_b \<-2\>$ via the identity map $\id_{\Pi_a \<2\>}$.

\begin{remark}
These maps should be thought of as the generalised version of the right multiplication map by $(i|i \pm 1)$ (resp. $X_i$) in the usual zigzag algebra defined in the simply-laced cases, but requires appropriate simple objects attached on the side.
\end{remark}

Before we fully describe the morphism spaces in $\I$-prmod, we shall prove some adjunctions between $\I$-prmod and $\gTLJ_n$:
\begin{proposition} \label{biadjoint pair}
Consider the additive functors 
\[
(P_i \otimes \Pi_a) \otimes - : \gTLJ \ra \I\text{-prmod}
\]
and
\[
(\Pi_a \otimes {}_i P \<-2\>) \otimes_\I -, (\Pi_a \otimes {}_i P) \otimes_\I - : \I\text{-prmod} \ra \gTLJ_n.
\]
Then
\begin{enumerate}
\item  $(P_i \otimes \Pi_a) \otimes - $ is left adjoint to $(\Pi_a \otimes {}_i P) \otimes_\I - $, and 
\item $(P_i \otimes \Pi_a) \otimes - $ is right adjoint to $(\Pi_a \otimes {}_i P \<-2\>) \otimes_\I -$.
\end{enumerate}
\end{proposition}

We will need the following lemma:
\begin{lemma}
Both of the following morphisms live in $\mathbb{U}_n$: 
\begin{enumerate}
\item $\beta_i : P_i \otimes {}_iP \ra \I$ given by $\mu|_{P_i \otimes {}_iP}$, and 
\item $\gamma_i : \I \ra P_i \otimes {}_iP \<-2\>$ uniquely defined by
\[
e_i \oplus e_{i \pm 1} \xra{\varphi} (e_i \otimes X_i \<-2\>) \oplus (X_i \otimes e_i \<-2\>) \oplus \left( (i \pm 1 | i) \otimes (i | i \pm 1) \<-2\>\right),
\]
where $\varphi =
\begin{bmatrix}
\id_{\Pi_0} & 0 \\
\id_{\Pi_0} & 0 \\
0      & \cup 
\end{bmatrix}$.
\end{enumerate}
In particular, these two maps induce natural transformations between the two endofunctors of $\I$-prmod: ${(P_i\otimes {}_iP) \otimes_\I -}$ and $\I \otimes_\I - = \id_{\I\text{-prmod}}$.
\end{lemma}
\begin{proof}
The fact that $\beta_i$ is a grading preserving $\I$-bimodules map follows from the associativity of the multiplication map $\mu$. 
The proof that $\gamma_i$ is a $\I$-bimodule map is just an elementary but tedious verification on each of the summand of $\I$, which we shall omit and leave to the reader.
\end{proof}
\begin{proof}[Proof of \cref{biadjoint pair}]
Since all $\Pi_a$ are self dual and $\Pi_a \otimes \Pi_0 = \Pi_0 \otimes \Pi_a = \Pi_a$ by definition, we are reduced to proving the case where $a=0$.

Define the following maps in $\gTLJ_n$:
\[
\alpha_i : \Pi_0 \xra{
	\begin{bmatrix}
	\id \\
	0
	\end{bmatrix}
	}
	(e_i \oplus X_i) = {}_iP \otimes_\I P_i
\]
\[
\omega_i : {}_iP \otimes_\I P_i \<-2\> = (e_i \<-2\> \oplus X_i \<-2\>) \xra{
	\begin{bmatrix}
	0 & \id
	\end{bmatrix}		 
	}
	\Pi_0
\]
To show that $P_i \otimes \Pi_0 = P_i$ is left adjoint to $\Pi_0 \otimes {}_iP= {}_iP$, consider the following counit and unit maps:
\[
P_i \otimes {}_iP \xra{\beta_i} \I, \quad
\Pi_0 \xra{\alpha_i} e_i \oplus X_i = {}_iP \otimes_\I P_i.
\]
Let us first show that the following composition:
\[
P_i = P_i \otimes \Pi_0 
	\xra{ \id \otimes \alpha_i }
P_i \otimes {}_iP \otimes_\I P_i 
	\xra{ \beta_i \otimes_\I \id} 
\I \otimes_\I P_i = P_i
\]
is the identity map.
To do so, let us first compute the map $\beta_i \otimes_\I \id$, which is the unique map rendering the following diagram commutative:
\[
\begin{tikzcd}[column sep = 2.5cm]
P_i \otimes {}_iP \otimes P_i  
	\ar[r, "\beta_i \otimes \id"] 
	\ar[d, "\id \otimes \mu|_{ {}_iP\otimes P_i }"] & 
\I \otimes P_i \ar[d, "\mu|_{\I \otimes P_i}"] 
\\
P_i \otimes (e_i \oplus X_i) = P_i \otimes P_i\<2\> 
	\ar[r, dashed] & 
P_i
\end{tikzcd}
\]
However, choosing $\mu|_{P_i \otimes (e_i \oplus X_i)} = \left[ \id_{P_i} \quad 0 \right]$ works since the commutativity just follows from the commutativity of the multiplication map.
Thus $\beta_i \otimes_\I \id = [\id_{P_i} 0]$ and it follows that $(\beta_i \otimes_\I \id)\circ (\id \otimes \alpha_i) = [\id_{P_i} \quad 0] \begin{bmatrix}
	\id_{P_i} \\
	0
	\end{bmatrix} = \id_{P_i}$ as required.

The other required composition
\[
{}_iP = \Pi_0 \otimes {}_iP 
	\xra{ \alpha_i \otimes\id }
{}_iP \otimes_\I P_i \otimes {}_iP 
	\xra{ \id \otimes_\I \beta_i}
{}_iP \otimes_\I \I  = {}_iP
\]
is also the identity, following a similar argument as before.
This completes the proof that $P_i \otimes \Pi_a \otimes -$ is left adjoint to $\Pi_a \otimes {}_iP \otimes_\I -$.

Now to show that $P_i \otimes - $ is right adjoint to $ {}_i P \<-2\> \otimes_\I - $, we define the following maps as the counit and unit respectively:
\[
{}_iP \otimes_\I P_i \<-2\>\ = (e_i\<-2\> \oplus X_i\<-2\>)
	\xra{\omega_i} \Pi_0, \quad
\I \xra{\gamma_i} P_i\otimes {}_iP \<-2\> .
\]
Let us show that the composition
\[
{}_iP \<-2\> = {}_iP \<-2\> \otimes_\I \I 
	\xra{\id \otimes_\I \gamma_i}
{}_iP \<-2\> \otimes_\I P_i \otimes {}_iP \<-2\>  
	\xra{\omega_i \otimes \id}
\Pi_0 \otimes {}_iP \<-2\> = {}_iP \<-2\>
\]
is the identity map.
To do so we will need to compute the map $\id \otimes_\I \gamma_i$, which by definition is the unique map rendering the follow diagram commutative:
\begin{equation} \label{id otimes over I cup gamma}
\begin{tikzcd}[column sep = large]
{}_iP \<-2\> \otimes \I 
	\ar[d, "\mu |_{{}_iP \otimes \I} \<-2\>"] 
	\ar[r, "\id \otimes \gamma_i"]& 
{}_iP \<-2\> \otimes P_i \otimes {}_iP  \<-2\>
	\ar[d,"\mu |_{{}_iP \otimes P_i} \<-2\>\otimes \id"] \\
{}_iP \<-2\>
	\ar[r, dashed] & 
(e_i \<-2\> \oplus X_i\<-2\>) \otimes {}_iP \<-2\> = {}_iP\<-4\> \oplus {}_iP\<-2\>
\end{tikzcd}.
\end{equation}
Recall that $\gamma_i$ is a map of $\I$-bimodules, so in particular, it makes the following diagram commutative,
\[
\begin{tikzcd}[column sep = 1.5cm]
\I \otimes \I
	\ar[d, "\mu"] 
	\ar[r, "\id \otimes \gamma_i"]& 
\I \otimes P_i \otimes {}_iP \<-2\>
	\ar[d,"\mu|_{\I \otimes P_i} \otimes \id"] \\
\I 
	\ar[r, "\gamma_i"] & 
P_i \otimes {}_iP \<-2\>
\end{tikzcd}.
\]
Restricting to the proper domain and codomains, we have the following induced commutative diagram:
\[
\begin{tikzcd}[column sep = 1.5cm]
{}_iP \otimes \I
	\ar[d, "\mu|_{ {}_iP \otimes \I}"] 
	\ar[r, "\id \otimes \gamma_i"]& 
{}_iP \otimes P_i \otimes {}_iP \<-2\>
	\ar[d,"\mu|_{ {}_iP \otimes P_i} \otimes \id"] \\
{}_iP 
	\ar[r, "\gamma_i|_{ {}_iP }"] & 
(e_i \oplus X_i) \otimes {}_iP \<-2\> = {}_iP \<-2\> \oplus {}_iP
\end{tikzcd}.
\]
It follows that $\gamma_i|_{{}_iP} \<-2\>$ is the map that renders the diagram in \cref{id otimes over I cup gamma} commutative.
Note that $\gamma_i|_{{}_iP}$ is defined on each summand by:
\begin{align*}
\gamma_i|_{e_i} = 
	\begin{bmatrix}
	\id_{\Pi_0} \\
	\id_{\Pi_0} 
	\end{bmatrix}
 &: e_i \ra (e_i \otimes X_i)\<-2\> \oplus (X_i \otimes e_i ) \<-2\> = \Pi_0 \oplus \Pi_0; \\
\gamma_i|_{X_i} = \id_{\Pi_0\<2\>}
 &: X_i \ra (X_i \otimes X_i ) \<-2\> = \Pi_0 \<2\>; \\
\gamma_i|_{(i|i\pm 1)} = \id_{\Pi_1 \<1\>}
 &: (i|i\pm 1) \ra (X_i \otimes (i|i\pm 1))\<-2\> = \Pi_1 \<1\>.
\end{align*}
A simple computation shows that 
\[
(\omega_i \otimes \id ) \circ ( \id \otimes_\I \gamma_i) = (\omega_i \otimes \id ) \circ \gamma_i|_{{}_iP}\<-2\> = \id|_{{}_iP\<-2\>}.
\]
The fact that the other composition $(\id \otimes \omega_i) \circ (\gamma_i \otimes_\I \id)$ is the identity on $P_i$ follows from a symmetric argument.
This completes the proof that $(P_i \otimes \Pi_a) \otimes - $ is right adjoint to $(\Pi_a \otimes {}_i P \<-2\> ) \otimes_\I - $.
\end{proof}

\comment{
\begin{definition}
We define the graded morphism space as the graded vector space given by
\[
\HOM_{\I\text{-prmod}}(P_i \otimes \Pi_a , P_j \otimes \Pi_b) := \bigoplus_{k \in \Z} \Hom_{\I\text{-prmod}}(P_i \otimes \Pi_a , P_j \otimes \Pi_b \<-k\>)\<k\>
\]
Similarly, the bigraded morphism space is defined by
\[
\HOM_{\Kom^b( \I\text{-prmod} )}(X , Y) :=
\bigoplus_{k \in \Z, \ell \in \Z} \Hom_{\Kom^b( \I\text{-prmod} )}(X , Y \<-k\>[-\ell]) \<k\>[\ell]
\]
as a bigraded vector space.
\end{definition}
}

\begin{proposition} \label{graded hom space}
The morphism spaces between $P_i \otimes \Pi_a$ and $P_j \otimes \Pi_b\<k\>$ in $\I$-prmod are given by:
\[
\Hom_{\I\text{-prmod}}(P_i \otimes \Pi_a , P_j \otimes \Pi_b\<k\>) =
\begin{cases}
\C\{\id\} , &\text{for } a = b, i = j, k = 0; \\
\C\{X_i\}, &\text{for } a = b, i = j, k = -2; \\
\C\{(i|j)\}, &\text{for } |a-b|=1, |i - j| = 1, k = -1; \\
0, &\text{otherwise}.
\end{cases}
\]
\end{proposition}
\begin{proof}
Denote $H_k := \Hom_{\I\text{-prmod}}(P_i \otimes \Pi_a , P_j \otimes \Pi_b \<k\>)$.
Using the adjunction between $P_i$ and ${}_iP$, we get that 
\[
H_k \cong \Hom_{\gTLJ_n}(\Pi_a , {}_iP_j \otimes \Pi_b \<k\>).
\]
We now proceed by looking at the possible cases of $i$ and $j$:
\begin{enumerate}
\item when $i = j$, we have ${}_iP_j = e_i \oplus X_i = \Pi_0 \oplus (\Pi_0\<2\>)$, hence
\[
{}_iP_j \otimes \Pi_b = \Pi_b \oplus \left(\Pi_b\<2\> \right).
\]
Since all $\Pi_s$ are simple objects, we deduce that for all $k, \ell \in \Z$, $H_k \cong 0$ if $a \neq b$.
On the other hand, if $a=b$, then $H_0 \cong H_{-2} \cong \C$ and is 0 for all other $k$.
Since the only non-zero cases are one-dimensional vector spaces, together with the non-zero maps $\id: P_i \otimes \Pi_a \ra P_i \otimes \Pi_a$ and $X_i: P_i \otimes \Pi_a\ra P_i \otimes \Pi_a \<-2\>$ as defined before, we must have that
\begin{align*}
H_{k} =
\begin{cases}
\C\{\id\}, &\text{for } a=b, k = 0; \\
\C\{X_i\}, &\text{for } a=b, k = -2; and\\
0, &\text{otherwise}.
\end{cases}
\end{align*}
\item When $i\neq j$, i.e. $|i-j|=1$, we have ${}_iP_j = (i|j) = \Pi_1 \<1\>$, and so
\[
{}_iP_j \otimes \Pi_b =  \Pi_1 \otimes \Pi_b \<1\>.
\]
By \cref{admissible triple}, we see that $\Hom_{\gTLJ_n}(\Pi_a , \Pi_1 \otimes \Pi_b) = \C \left\{ \itri \right\}$ when $(a,1,b)$ is an admissible triple (namely $|a-b| = 1$), and is $0$ otherwise.
Taking the $\Z$-grading into account we have that $H_{-1} \cong \C$ if $|a-b|=1$, and $H_k \cong 0$ for all $k \neq -1$.
Since we already have the non-zero $\I$-module morphism $(i|j): P_i \otimes \Pi_a \ra P_j \otimes \Pi_b \<-1\>$ defined before, we can conclude that
\begin{align*}
H_{k} =
\begin{cases}
\C \{ (i|j) \}, &\text{for } |a-b|=1, k = -1; \text{ and}\\
0, &\text{otherwise}.
\end{cases}
\end{align*}
\end{enumerate}
This concludes the proof.
\end{proof}
As a result of \cref{graded hom space}, we have the following:
\begin{proposition}\label{prop: indecomposable Krull-Sch}
The set of indecomposable objects in $\I$-prmod is given by
\[
\{P_i \otimes \Pi_a \<k\>:
					i \in \{1,2\},
					0 \leq a \leq n-2,
					k \in \Z
					\}.
\]
Moreover, $\I$-prmod is Krull-Schmidt, and so is $\Kom^b(\I$-prmod).
\end{proposition}
\section{Complexes of bimodules and spherical twists}
Following \cite{khovanov_seidel_2001}, the categorical action of $\B(I_2(n))$ on the homotopy category of (cochain) complexes $\Kom^b(\I \text{-prmod})$ will be obtained through an assignment of braid elements in $\B(I_2(n))$ to a compatible set of complexes of $\I$-bimodules, which then acts on $\Kom^b(\I \text{-prmod})$ through tensoring over $\I$.
With this said, however, we shall introduce the complexes required through the notion of spherical twists, which will simplify the proof of the $n$-braiding relation greatly.

\subsection{Spherical twists}
\begin{definition}[cf. \cite{anno_logvinenko_2017}]
Let $X\in \Kom^b(\I$-prmod) with left and right adjoints $X^\ell, X^r \in \Kom^b($prmod-$\I$) respectively.
We define:
\begin{enumerate}
\item the \emph{twist} of $X$ as the complex of $\I$-bimodules:
\[
\sigma_X:= \cone \left( X\otimes X^r \xrightarrow{\varepsilon} \I[0] \right) \in \Kom^b(\mathbb{U}_n),
\]
with $\varepsilon$ the counit of the adjunction $X \dashv X^r$; and
\item the \emph{dual twist} of $X$ as the complex of $\I$-bimodules:
\[
\sigma_X' := \cone \left(  \I[0] \xrightarrow{\nu} X\otimes X^\ell  \right)[-1] \in \Kom^b(\mathbb{U}_n),
\]
with $\nu$ the unit of the adjunction $X^\ell \dashv X$.
\end{enumerate}
The twist of $X$ is said to be \emph{spherical} if the twist and dual twist are inverses of each other (tensoring over $\I$).
\end{definition}

\begin{proposition}\label{isomorphic biadjoint}
Let $X$ and $Y$ be isomorphic objects in $\Kom^b(\I$-prmod) equipped with left and right adjoints.
Then $\sigma_X \cong \sigma_Y$ and $\sigma'_X \cong \sigma'_Y$ in $\Kom^b(\mathbb{U}_n)$.
\end{proposition}
\begin{proof}
We will prove that the twists $\sigma_X$ and $\sigma_Y$ are isomorphic; the argument for their dual twists follows in a similar fashion.
Let $\varphi: X \rightarrow Y$ be an isomorphism of $X$ and $Y$ in $\Kom^b(\I$-prmod) and $\varphi^{-1}: Y \rightarrow X$ its inverse.
Using the adjunctions $X \dashv X^r$ and $Y \dashv Y^r$, we obtain two maps $(\varphi^{-1})^* : X^r \rightarrow Y^r$ and $\varphi^*: Y^r \rightarrow X^r$, which are inverses of each other and also render the first square in the following diagram commutative:
\[
\begin{tikzcd}[row sep = large, column sep = large, ampersand replacement=\&]
X\otimes X^r \ar{r}{\epsilon_X} \ar{d}{\varphi \otimes (\varphi^{-1})^*}\& 
\I \ar{r}{} \ar[equal]{d}\& 
\sigma_X \ar{r}{} \ar[dashed]{d}{\cong} \& {}\\
Y\otimes Y^r \ar{r}{\epsilon_Y} \ar[bend left = 60]{u}{\varphi^{-1}\otimes \varphi^*} \& 
\I \ar{r}{} \& 
\sigma_Y \ar{r}{} \& {}.
\end{tikzcd}
\]
It follows from the property of distinguished triangles that we have an induced isomorphism between $\sigma_X$ and $\sigma_Y$.
\end{proof}

\begin{remark} \label{biadjoint shifts}
Note that the converse of this statement is not true.
In particular, for any object $X$ with internal grading shift $\<n\>$ or cohomological grading shift $[m]$, we have that 
\[
X\<n\> \otimes X^r\<-n\> = X\otimes X^r = X[m]\otimes X^r[-m]
\]
and similarly
\[
X\<n\> \otimes X^\ell \<-n\> = X\otimes X^\ell = X[m] \otimes X^\ell [-m]
\]
for any $m, n  \in\Z$.
Thus $X^r[-m]\<-n\>$ and $X^\ell[-m]\<-n\>$ are still right and left adjoints of $X[m]\<n\>$, with the same unit and counit maps as the adjunctions of $X$, and so $\sigma_{X[m]\<n\>} = \sigma_X$ and $\sigma'_{X[m]\<n\>} = \sigma'_X$.
But clearly $X[m]\<n\> \not\cong X$ for $m \neq 0 \neq n$.
\end{remark}

\begin{proposition}\label{spherical twist relation}
Let $\Sigma, \Sigma^{-1} \in \Kom^b(\mathbb{U}_n)$ such that 
\[
\Sigma \otimes_\I \Sigma^{-1} \cong \I[0] \cong \Sigma^{-1} \otimes_\I \Sigma,
\]
and let $X$ be an object equipped with right and left adjoints $X^r$ and $X^\ell$ respectively.
Then $\Sigma \otimes_\I X \in \Kom^b(\I\text{-prmod})$ also has  right and left adjoints  $X^r \otimes_\I \Sigma^{-1}$ and $X^\ell \otimes_\I \Sigma^{-1}$ respectively.
Moreover, the twist $\sigma_{\Sigma \otimes_\I X}$ is isomorphic to $\Sigma \otimes_\I \sigma_X \otimes_\I \Sigma^{-1}$ and  similarly, the dual twist $\sigma'_{\Sigma \otimes_\I X}$ is isomorphic to $\Sigma \otimes_\I \sigma'_X \otimes_\I \Sigma^{-1}$.
\end{proposition}
\begin{proof}
Let $\varepsilon$ be the counit of the adjunction $X \dashv X^r$ and $\phi : \Sigma \otimes_\I \Sigma \ra \I [0]$ be an isomorphism defining the invertibility of $\Sigma$.
It is to see that $\Sigma\otimes_\I X$ is left adjoint to $X^r \otimes_\I \Sigma^{-1}$, with the counit of the adjunction given by $\phi\circ \left(\id \otimes \varepsilon \otimes \id \right)$.
We start with the distinguished triangle in $\Kom^b(\mathbb{U}_n)$ defining the twist of $X$:
\begin{equation}\label{sigma_i}
(X \otimes X^r) \xrightarrow{\varepsilon} \I[0] \rightarrow \sigma_{X} \rightarrow .
\end{equation}
Tensoring with $\Sigma\otimes_\I -$ and $ - \otimes_\I \Sigma^{-1}$, we get the distinguished triangle
\[
\Sigma\otimes_\I (X \otimes X^r) \otimes_\I \Sigma^{-1}
\xrightarrow{\id \otimes \varepsilon \otimes \id} \Sigma \otimes_\I \I[0] \otimes_\I \Sigma^{-1}
\rightarrow \Sigma \otimes_\I \sigma_{P_i} \otimes_\I \Sigma^{-1} \rightarrow .
\]
By the definition of the tensor product in $\mathbb{U}_n$, we have
\[
\Sigma\otimes_\I (X \otimes X^r) \otimes_\I \Sigma^{-1} = (\Sigma\otimes_\I X) \otimes (X^r \otimes_\I \Sigma^{-1})
\]
and
\[
\Sigma \otimes_\I \I[0] \otimes_\I \Sigma^{-1} = \Sigma \otimes_\I \Sigma^{-1} \xra[\cong]{\phi} \I[0].
\]
Thus, this distinguished triangle is isomorphic to
\[
(\Sigma\otimes_\I X) \otimes (X^r \otimes_\I \Sigma^{-1})
	\xrightarrow{\phi\circ \left(\id \otimes \varepsilon \otimes \id \right)} 
\I[0]
	\rightarrow 
\Sigma \otimes_\I \sigma_X \otimes_\I \Sigma^{-1} \rightarrow,
\]
By definition of the twist $\sigma_{\Sigma\otimes_\I X}$, we get that
$\sigma_{\Sigma\otimes_\I X} \cong  \Sigma\otimes_\I \sigma_{X} \otimes_\I \Sigma^{-1}$ as required.

The dual statements for left adjointness and dual twist follow similarly.
\end{proof}

\subsection{The braiding relations}
Recall that $P_i \otimes \Pi_a$ have right and left adjoints $\Pi_a \otimes {}_i P $ and $\Pi_a \otimes {}_i P \<-2\>$ respectively.
We are particularly interested in the twist and dual twist of the objects $P_i$ :
\[
\sigma_{P_i} = 0 \ra P_i \otimes {}_i P \xra{\beta_i} \I \ra 0
\]
and
\[
\sigma'_{P_i} = 0 \ra \I \xra{\gamma_i} P_i\otimes {}_i P\<-2\> \ra 0,
\]
with $\I$ in cohomological degree 0.
In what follows, we shall show that the corresponding twists and dual twists are inverses of each other (therefore spherical) and moreover, the twists satisfy the $n$-braiding relation.

\begin{proposition}\label{inverse relation}
The twist and dual twist of $P_i$ are inverses of each other, i.e. we have the following isomorphisms in $\Kom^b(\mathbb{U}_n)$:
\[
\sigma_{P_i} \otimes_\I \sigma_{P_i}' \cong \I[0] \cong \sigma_{P_i}' \otimes_\I \sigma_{P_i},
\]
In particular, $\sigma_{P_i}$ are spherical twists.
\end{proposition}
\begin{proof}
Note that $\sigma_{P_i} \otimes_\I \sigma_{P_i}'$ is given by the cone of the following map of complexes:
\[
\begin{tikzcd}[row sep = large, column sep = 2cm, ampersand replacement=\&]
P_i \otimes {}_iP \otimes_\I \I 
	\ar{r}{\id\otimes \id \otimes_\I \gamma_i} 
	\ar{d}{\beta_i \otimes_\I \id} \&
P_i \otimes {}_iP \otimes_\I P_i \otimes {}_iP \<-2\>
	\ar{d}{\beta_i \otimes_\I \id \otimes \id}
\\
\I \otimes_\I \I
	\ar{r}{\id \otimes_\I \gamma_i} \& 
\I \otimes_\I P_i \otimes {}_iP \<-2\>,
\end{tikzcd}
\]
with $\I \otimes_\I \I$ in cohomological degree 0.
Using the definition of the tensor $-\otimes_\I-$, the square reduces to
\[
\begin{tikzcd}[row sep = large, column sep = 2cm, ampersand replacement=\&]
P_i \otimes {}_iP 
	\ar{r}{\id\otimes \phi} 
	\ar{d}{\beta_i} \&
P_i \otimes ( e_i \oplus X_i ) \otimes {}_iP \<-2\>
	\ar{d}{\beta_i|_{(e_i \oplus X_i)} \otimes \id}
\\
\I
	\ar{r}{\gamma_i} \& 
P_i \otimes {}_iP \<-2\>,
\end{tikzcd}
\]
where $\phi: {}_iP \ra (e_i \oplus X_i) \otimes {}_iP \<-2\> = \left( e_i \otimes {}_iP \<-2\>\right) \oplus  \left( X_i \otimes {}_iP \<-2\>\right)$ is defined on each simple summand of ${}_iP$ by:
\begin{align*}
e_i = \Pi_0 
	&\xra{
		\begin{bmatrix}
		\id_{\Pi_0} \\
		\id_{\Pi_0}
		\end{bmatrix} 
		} 
	\Pi_0 \oplus \Pi_0 = 
	\left( e_i \otimes X_i\<-2\> \right) \oplus 
	\left( X_i \otimes e_i \<-2\> \right) \\
X_i = \Pi_0\<2\> 
	&\xra{ \id_{\Pi_0} } 
	\Pi_0 \<2\> = X_i \otimes X_i \<-2\> \\
(i|i\pm 1) = \Pi_1 \<1\> 
	&\xra{ \id_{\Pi_1} } 
	\Pi_1 \<1\> = X_i \otimes (i|i\pm 1) \<-2\>.
\end{align*}
In particular, we have that $\phi = 
	\begin{bmatrix}
	X_i \\
	\id
	\end{bmatrix}
$ as a map from ${}_iP$ to $ \left( e_i \otimes {}_iP \<-2\>\right) \oplus  \left( X_i \otimes {}_iP \<-2\>\right) = {}_iP\<-2\>\oplus {}_iP$.
Thus the square further reduces to
\[
\begin{tikzcd}[row sep = large, column sep = 2cm, ampersand replacement=\&]
P_i \otimes {}_iP 
	\ar{r}{ 
		\begin{bmatrix}
		\id\otimes X_i \\
		\id
		\end{bmatrix}
	} 
	\ar{d}{\beta_i} \&
\left( P_i \otimes {}_iP \<-2\>\right) \oplus  \left( P_i \otimes {}_iP \right)
	\ar{d}{
		\begin{bmatrix}
		\id 
			& X_i \otimes \id
		\end{bmatrix} 
	}
\\
\I
	\ar{r}{\gamma_i} \& 
P_i  \<-2\>\otimes {}_iP.
\end{tikzcd}
\]
The result now follows from applying gaussian elimination (\cref{gaussian elimination}).
\end{proof}

\begin{remark} \label{P_i involution}
Since the cap and cup mappings between $\Pi_{n-2} \otimes \Pi_{n-2}$ and $\Pi_0$ are isomorphisms, it is  easy to see that $\sigma_{P_i} \cong \sigma_{P_i \otimes \Pi_{n-2}}$.
Furthermore, it turns out that the twists $\sigma_{P_i \otimes \Pi_a}$ are \emph{not} autoequivalences for $a \neq 0, n-2$.
In particular, this means that the twists of $P_i \otimes \Pi_a$ are not spherical twists, even though the endomorphism space of $P_i \otimes \Pi_a$ are sphere-like. 
\end{remark}

\begin{proposition}\label{braid relation}
The $n$-braiding relation holds, i.e. we have the following isomorphism in $\Kom^b(\mathbb{U}_n)$:
\[
\underbrace{\cdots \otimes_\I \sigma_{P_1} \otimes_\I \sigma_{P_2} \otimes_\I \sigma_{P_1}  }_{n \text{ times}}
	\cong 
	\underbrace{\cdots \otimes_\I \sigma_{P_2} \otimes_\I \sigma_{P_1} \otimes_\I \sigma_{P_2} }_{n \text{ times}}.
\]
\end{proposition}

The following key lemma will come in handy throughout this whole thesis:
\begin{lemma}
\label{lemma for braid relation}
Let $a \geq 1$ and $k,\ell \in \Z$ be arbitrary, and denote the complex
\[
C := P_{i\pm 1} \otimes \Pi_a\<k\>[\ell] 
		\xra{ (i \pm 1| i) } 
	P_i \otimes \Pi_{a-1} \<k-1\>[\ell-1].
\]
We have that
\begin{align*}
\sigma_{P_i} \otimes_\I  C 
\cong 
\begin{cases}
P_i \otimes \Pi_{a+1}\<k+1\>[\ell+1] 
	\xra{ (i | i \pm 1) } 
	P_{i\pm 1} \Pi_a\<k\>[\ell]; &\text{for } a\neq n-2, \\
P_{i\pm 1} \otimes \Pi_a\<k\>[\ell]; &\text{for } a=n-2.
\end{cases}
\end{align*}
\end{lemma}
\begin{proof}
Since $\sigma_{P_i}$ commutes with both shift functors $\<-\>$ and $[-]$, we may assume that $k = \ell = 0$.
We have that $\sigma_{P_i} \otimes_\I C$ is the cone of the following square
\[
\begin{tikzcd}[row sep = large, column sep = 2cm, ampersand replacement=\&]
P_i \otimes (i|i\pm 1) \otimes \Pi_a
	\ar{r}{ 
		\begin{bmatrix}
		0 \\
		\id \otimes \tri{}
		\end{bmatrix}
	} 
	\ar{d}{(i|i\pm 1) \otimes \id} \&
\left( P_i \otimes e_i \otimes \Pi_{a-1} \<-1\> \right) \oplus
	\left( P_i \otimes X_i \otimes \Pi_{a-1} \<-1\> \right)
	\ar{d}{
		\begin{bmatrix}
		\beta_i|_{P_i \otimes e_i} \otimes \id 
			& \beta_i|_{P_i \otimes X_i} \otimes \id
		\end{bmatrix} 
	}
\\
P_{i\pm 1} \otimes \Pi_a
	\ar{r}{(i\pm 1|i)} \& 
P_i \otimes \Pi_{a-1} \<-1\>,
\end{tikzcd}
\]
with $P_{i\pm 1} \otimes \Pi_a$ sitting in cohomological degree 0.

When $a \neq n-2$, we have the isomorphism
\[
\Pi_{a-1}\<1\> \oplus \Pi_{a+1}\<1\> \underset{\cong}{
	\xra{
	\begin{bmatrix}
	\itri & \itri
	\end{bmatrix}
	}
}
\Pi_1 \otimes \Pi_a \<1\> = 
(i|i+1) \otimes \Pi_a.
\]
Note that the composition
\[
\Pi_{a+1}\<1\> \xra{\itri} \Pi_1 \otimes \Pi_a\<1\> \xra{\tri} \Pi_{a-1}
\]
is 0, whereas the composition
\[
\Pi_{a-1}\<1\> \xra{\itri} \Pi_1 \otimes \Pi_a\<1\> \xra{\tri} \Pi_{a-1}
\]
is given by $c \cdot \id$ for some $c \neq 0 \in \C$ (see \cref{constant multiple identity}).
Together with $\beta_i|_{P_i \otimes e_i}$ being the identity map on $P_i$, we reduce to square to
\[
\begin{tikzcd}[row sep = large, column sep = 2cm, ampersand replacement=\&]
\left( P_i \otimes \Pi_{a+1} \<1\> \right) \oplus 
	\left( P_i \otimes \Pi_{a-1} \<1\> \right)
	\ar{r}{ 
		\begin{bmatrix}
		0 & 0 \\
		0 & \id \otimes c\cdot \id 
		\end{bmatrix}
	} 
	\ar{d}{\left( (i|i\pm 1) \otimes \id\right) \circ 	
		\begin{bmatrix} 
		\id \otimes \itri & 
		\id \otimes \itri   
		\end{bmatrix} 
	} \&
\left( P_i \otimes e_i \otimes \Pi_{a-1} \<-1\> \right) \oplus
	\left( P_i \otimes X_i \otimes \Pi_{a-1} \<-1\> \right)
	\ar{d}{
		\begin{bmatrix}
		\id 
			& \beta_i|_{P_i \otimes X_i} \otimes \id
		\end{bmatrix} 
	}
\\
P_{i\pm 1} \otimes \Pi_a
	\ar{r}{(i\pm 1|i)} \& 
P_i \otimes \Pi_{a-1} \<-1\>.
\end{tikzcd}
\]
Taking the cone and then applying gaussian elimination (\cref{gaussian elimination}), we obtain the complex
\[
P_i \otimes \Pi_{a+1}[1]\<1\> 
	\xra{\left( (i|i\pm 1) \otimes \id\right) \circ 
		\left(\id \otimes \itri \right)} 
P_{i\pm 1} \otimes \Pi_a[0].
\]
One can show that the composition $\left( (i|i\pm 1) \otimes \id\right) \circ 
		\left(\id \otimes \itri \right)$
is just the map $(i|i\pm 1)$ by computing the resulting map on each summand of $P_i \otimes \Pi_{a+1}\<1\>$ as follows:
\begin{enumerate}
\item for the summand $e_i \otimes \Pi_{a+1}\<1\>$, the composition is:
\begin{center}
\begin{tikzpicture}[scale = 0.8]
  \node (tl) at (-2.5,1) {$e_i$};
  \node (t) at (-1,1)  {$(i | i \pm 1)$};
  \node (tr) at (1,1)   {$\Pi_a$};
  \node (d)  at (0,-1)  {$\Pi_{a+1} \<1\>$};
  \coordinate (m1)  at (0,0);
  \node (dl) at (-2.5,-1) {$e_i$};
  \draw[thick,black] (d)--(m);
  \draw[thick,black] (m1)--(tr);
  \draw[thick,black] (m1)--(t);
  \node (t2) at (-1,3) {$(i|i\pm 1)$};
  \node (tr2) at (1,3)   {$\Pi_a$};
  \draw[thick,black] (t2)--(t);
  \draw[thick,black] (tr2)--(tr);
\end{tikzpicture}
\end{center}
giving us the map required on this summand;
\item for the summand $(i\pm 1|i) \otimes \Pi_{a+1}\<1\>$, the composition is: 
\begin{center}
\begin{tikzpicture}[scale = 0.8]
  \node (tl) at (-3,1) {$(i\pm 1|i)$};
  \node (t) at (-1,1)  {$(i|i\pm 1)$};
  \node (tr) at (1,1)   {$\Pi_a$};
  \node (d)  at (0,-1)  {$\Pi_{a+1} \<1\>$};
  \coordinate (m1)  at (0,0);
  \node (dl) at (-3,-1) {$(i\pm 1|i)$};
  \draw[thick,black] (d)--(m);
  \draw[thick,black] (m1)--(tr);
  \draw[thick,black] (m1)--(t);
  \draw[thick,black] (dl)--(tl);
  \node (t2) at (-2,3) {$X_{i\pm 1}$};
  \node (tr2) at (1,3)   {$\Pi_a$};
  \draw (-3, 1.2) arc[start angle=180,end angle=0,radius=1];
  \draw[thick,black] (tr2)--(tr);
\end{tikzpicture}
\end{center}
where applying the relation \cref{frobenius rel} gives us the required map on this summand;
\item the map is 0 on the summand $X_i \otimes \Pi_{a+1}\<1\>$ as required, since $(i|i\pm 1)\otimes \id$ is 0 on the summand $X_i \otimes (i|i\pm 1) \otimes \Pi_a$.
\end{enumerate}

The case where $a = n-2$ is left to the reader since it follows from a simpler computation, applying instead the isomorphism
\[
\Pi_{n-1}\<1\> \underset{\cong}{
	\xra{
	\itri
	}
}
\Pi_1 \otimes \Pi_{n-2} \<1\> = 
(i|i+1) \otimes \Pi_{n-2}.
\]
\end{proof}
We are now ready to prove \cref{braid relation}:
\begin{proof}[Proof of \cref{braid relation}]
Denote $\Sigma$ as the object $\underbrace{ \cdots \otimes_\I \sigma_{P_1} \otimes_\I \sigma_{P_2} \otimes_\I \sigma_{P_1}}_{n-1 \text{ times}}$ in $\Kom^b(\mathbb{U}_n)$ with $\Sigma^{-1}$ its respective inverse object, whose existence follows from \cref{inverse relation}.
We aim to show the relation:
\[
\Sigma \otimes_\I \sigma_{P_2} \cong 
	\begin{cases}
	\sigma_{P_2} \otimes_\I \Sigma.; \text{ when $n$ is even} \\
	\sigma_{P_1} \otimes_\I \Sigma.; \text{ when $n$ is odd}.
	\end{cases}
\]
\cref{spherical twist relation} tells us that
$
\Sigma \otimes_\I \sigma_{P_2} \otimes_\I \Sigma^{-1} \cong \sigma_{\Sigma \otimes_\I P_2}
$.
Since $\sigma_{P_i \otimes \Pi_{n-2}} \cong \sigma_{P_i}$ (see \cref{P_i involution}), together with \cref{isomorphic biadjoint} it is sufficient to show that 
\[
\Sigma \otimes_\I P_2 \cong 
	\begin{cases}
	P_2 \otimes \Pi_{n-2}; \text{ when $n$ is even} \\
	P_1 \otimes \Pi_{n-2}; \text{ when $n$ is odd}
	\end{cases}
\]
up to grading shifts $\<-\>$ and $[-]$ (see \cref{biadjoint shifts}).
This follows from a repeated application of \cref{lemma for braid relation}.
\end{proof}

\comment{
\begin{remark} {\color{red} maybe this can be a subsection itself as a miscellaneous subsection}
Note that this construction is slightly different from the usual zigzag algebra when $n=3$, since $(i|i \pm 1)$ and $e_i, X_i$ are non-isomorphic simple objects in $\TLJ_3$ -- the usual zigzag algebra is constructed in the category of $\Z$-graded vector space (which is equivalent to $\gTLJ_2$), where there is only one class of simple object.
Nevertheless, there is a different zigzag algebra construction specifically for odd $n$, where it will coincides with the usual zigzag algebra in the category of graded vector spaces when $n=3$.
We shall roughly outline the construction here.
Firstly, one can construct the zigzag algebras using $\Pi_{n-3}$ for the length one path objects $(1|2)$ and $(2|1)$ instead of $\Pi_1$, which will still induce a braid group action in a similar way.
One of the main difference happens when $n$ is odd; the zigzag algebra will actually be an algebra object in the subcategory of $\gTLJ$ consisting of only the even label objects.
In particular, when $n=3$, the subcategory of even labelled objects in $\gTLJ_3$ is just the fusion category generated by the monoidal unit, which is equivalent to the category of graded vector space.
In this sense, this zigzag algebra construction when $n=3$ coincides with the usual zigzag algebra construction in the category of graded vector space.
\end{remark}
}

\section{Categorification of the Burau representation}
Here we collect our results from previous sections to show that there is a faithful categorical action of $\B(I_2(n))$ on $\Kom^b(\I \text{-prmod})$ and it categorifies the corresponding Burau representation.
\subsection{Categorical action of $\B(I_2(n))$ on $\Kom^b(\I \text{-prmod})$}
Having proven the required relations as in \cref{inverse relation} and \cref{braid relation}, all that is left is to define the action of $\B(I_2(n))$ on $\Kom^b(\I \text{-prmod})$.

Before we do that, let us state the salient properties of $\Kom^b(\I \text{-prmod})$.
Firstly, it is a triangulated $\X$-category, with the triangulated shift given by the cohomological degree shift functor $[1]$ and the distinguished autoequivalence $\X$ given by the internal grading shift $\<1\>$.
Moreover, the $\TLJ_n$-module structure on $\I$-prmod extends to a $\TLJ_n$-module structure on $\Kom^b(\I \text{-prmod})$, where each $A \in \TLJ_n$ acts on the complex $X^\bullet \in \Kom^b(\I \text{-prmod})$ by tensoring with $-\otimes A$ on each component, namely
\[
(X^\bullet \otimes A)^i = X^i \otimes A
\]
for each $i \in \Z$.
Hence, $(\Kom^b(\I \text{-prmod}), \<1\>, -\otimes ?)$ is a triangulated module $\X$-category over $\TLJ_n$ with distinguished autoequivalence $\X:= \<1\>$.

Now let $U^\bullet \in \Kom^b(\mathbb{U}_n)$.
By the definition of our tensor product $- \otimes_\I -$, it follows that the exact endofunctors $U^\bullet \otimes_\I - : \Kom^b(\I \text{-prmod}) \ra \Kom^b(\I \text{-prmod})$ commute trivially with the grading shift functor $\<1\>$ and the functors 
\[
-\otimes A: \Kom^b(\I\text{-prmod}) \ra \Kom^b(\I\text{-prmod})
\]
for all $A$ objects in $\TLJ_n$.
As such, we have that:
\begin{proposition}
Every object $U^\bullet$ in $\Kom^b(\mathbb{U}_n)$ induces an endofunctor $U^\bullet \otimes_\I -$ which commutes with $\<1\>$ and respects the $\TLJ_n$-module structure on $\Kom^b(\I$-prmod), namely $U^\bullet \otimes_\I -$ is an endofunctor on ($\Kom^b(\I$-prmod), $\<1\>$, $-\otimes ?$).
\end{proposition}

We can then use the spherical twists $\sigma_{P_i} \in \Kom^b(\mathbb{U}_n)$ defined in the previous section to obtain a $\B(I_2(n))$-action on $\Kom^b(\I \text{-prmod})$.
Together with the relations showns in \cref{inverse relation} and \cref{braid relation}, we arrive at our first main theorem:
\begin{theorem} \label{weak braid action}
There is a (weak) $\B(I_2(n))$-action on ($\Kom^b(\I$-prmod), $\<1\>$, $-\otimes ?$), where the action of each standard generator $\sigma_i$ of $\B(I_2(n))$ is given by the spherical twist $\sigma_{P_i}$:
\[
\sigma_i(C^\bullet) := \sigma_{P_i} \otimes_\I C^\bullet
\]
for each complex $C^\bullet$ in $\Kom^b(\I \text{-prmod})$.
\end{theorem}

\subsection{Categorification and faithfulness} \label{sect: categorification}
Recall the Burau representation of $\B(I_2(n))$ in \cref{defn: Burau rep}, which is defined on the generators as follows:
\[
\rho(\sigma_1) = 
	\begin{bmatrix}
	-q^2 & -2 \cos (\frac{\pi}{n}) q \\
	0      &  1
	\end{bmatrix}, \quad
\rho(\sigma_2) = 
	\begin{bmatrix}
	1                                   &  0  \\
	-2 \cos (\frac{\pi}{n}) q  &  -q^2
	\end{bmatrix}.
\]

Denote $\cK := \Kom^b(\I$-prmod) and let $K_0(\cK)$ be the triangulated Grothendieck group of $\cK$.
Since $\cK$ is a triangulated module category over $\TLJ_n$ equipped with a distinguished autoequivalence, $K_0(\cK)$ is a module over $\Omega_n[q,q^{-1}]:= \Z[\omega]/\< \Delta_{n-1}(\omega)\> \cong K_0(\TLJ_n)[q,q^{-1}]$ (see \cref{subsection: TLJ} and \cref{subsect: module categories}), where the module structure can be described explicitly as:
\[
\omega\cdot[C] = [C \otimes \Pi_1], \quad q\cdot[C] = [C\<1\>]. 
\]
It is easy to see that the two Grothendieck classes $[P_1]$ and $[P_2]$ generates $K_0(\cK)$ as a $\Omega_n[q,q^{-1}]$-module, and moreover $\{ [P_1], [P_2] \}$ forms a $\Omega_n[q,q^{-1}]$-basis.
Hence $K_0(\cK)$ is a rank two free module over $\Omega_n[q,q^{-1}]$.

Since the autoequivalences $\sigma_{P_i} \otimes_\I - $ respect the $\TLJ_n$-module structure on $\cK$ and commute with $\<1\>$, they descend to $\Omega_n[q,q^{-1}]$-module automorphisms on $K_0(\cK)$.
A simple calculation on the basis elements $[P_i]$ shows that 
\[
K_0(\sigma_{P_i} \otimes_\I -)|_{\omega = 2\cos \left(\frac{\pi}{n} \right)} = \rho(\sigma_i).
\] 
Hence we have the following theorem:
\begin{theorem}
The $\B(I_2(n))$-action on $\cK$ categorifies the Burau representation.
\end{theorem}

It is known that the Burau representation of $\B(\I_2(n))$ is faithful \cite{lehrer_xi_2001}, which implies that the categorical action of $\B(\I_2(n))$ on $\cK$ must be faithful.
\begin{corollary} \label{cor: faithful action}
The $\B(I_2(n))$-action on $\cK$ defined in \cref{weak braid action} is faithful.
\end{corollary}

\section{Relation to other categorical actions}
We end this chapter with this section, which relates our categorification of the Burau representation to some other known categorical action.

\subsection{A double $A_{n-1}$ configuration} \label{sect: double A configuration}
Let us denote the category of spherical, graded, projective  $\mathscr{A}_{n-1}$-modules defined by Khovanov and Seidel in \cite{khovanov_seidel_2001} as $\mathscr{A}_{n-1}\text{-prgrmod}^{>0}$ -- it is the additive category of all projectives that do not contain the projective $\mathscr{A}_{n-1}(e_0)$ over the $(0)$ vertex as a summand.
Our construction of the additive category $\I$-prmod can be viewed as two disjoint copies of $\mathscr{A}_{n-1}\text{-prgrmod}^{>0}$.
To spell it out, note that there are two (disjoint) $A_{n-1}$-configuration within $\I$-prmod, namely:
\begin{enumerate}[(i)]
\item the objects $P_i\otimes \Pi_a$ are $\<2\>$-spherical:
\[
\Hom(P_i\otimes \Pi_a \<k\>, P_i\otimes \Pi_a) 
\cong \begin{cases}
\C, &\text{when } k =0, 2 \\
0, &\text{otherwise}.
\end{cases}
\]
\item the morphism space $\Hom(P_{i}\otimes \Pi_{a} \<k\>, P_{j}\otimes \Pi_{b})$ for $i\neq j$ is one-dimensional and concentrated in degree $\<1\>$ whenever $|a-b| = 1$ or 0 otherwise:
\[
\Hom(P_{i}\otimes \Pi_{a} \<k\>, P_{j}\otimes \Pi_{b}) 
\cong \begin{cases}
\C, &\text{when } k =1 \text{ and } |a-b| = 1 \\
0, &\text{otherwise}.
\end{cases}
\]
\end{enumerate}
See \cref{fig: a double A configuration} for a graphical representation of this two disjoint copies of $A_{n-1}$ configuration, where each vertex represents a $\<2\>$-spherical object and an edge is drawn between two distinct vertices whenever the total morphism space between them is non-zero (which as shown above, must be one dimensional and concentrated in degree $\<1\>$, regardless of the order).
\begin{figure}[H]
\centering
\begin{tikzpicture}
\draw[thick] (0,0) -- (1,1) ;
\draw[thick] (1,1) -- (0,2) ;
\draw[thick] (0,2) -- (1,3) ;

\filldraw[color=black!, fill=white!]  (0,0) circle [radius=0.1];
\filldraw[color=black!, fill=white!]  (1,1) circle [radius=0.1];
\filldraw[color=black!, fill=white!]  (0,2) circle [radius=0.1];
\filldraw[color=black!, fill=white!]  (1,3) circle [radius=0.1];

\node[left] at (0,-0.1) {$P_1 \otimes \Pi_0$};
\node[right] at (1, 1-0.1) {$P_2 \otimes \Pi_1$};
\node[left] at (0, 2-0.1) {$P_1 \otimes \Pi_2$};
\node[right] at (1, 3-0.1) {$P_2 \otimes \Pi_3$};
\node[above] at (0.5, 3.5) {$\vdots$};

\draw[thick] (6,0) -- (5,1) ;
\draw[thick] (5,1) -- (6,2) ;
\draw[thick] (6,2) -- (5,3) ;

\filldraw[color=black!, fill=white!]  (6,0) circle [radius=0.1];
\filldraw[color=black!, fill=white!]  (5,1) circle [radius=0.1];
\filldraw[color=black!, fill=white!]  (6,2) circle [radius=0.1];
\filldraw[color=black!, fill=white!]  (5,3) circle [radius=0.1];

\node[right] at (6, -0.1) {$P_2 \otimes \Pi_0$};
\node[left] at (5,1-0.1) {$P_1 \otimes \Pi_1$};
\node[right] at (6, 2-0.1) {$P_2 \otimes \Pi_2$};
\node[left] at (5, 3-0.1) {$P_1 \otimes \Pi_3$};
\node[above] at (5.5, 3.5) {$\vdots$};
\end{tikzpicture}
\caption{The two disjoint $A_{n-1}$-configurations in $\I$-prmod.}
 \label{fig: a double A configuration}
\end{figure}
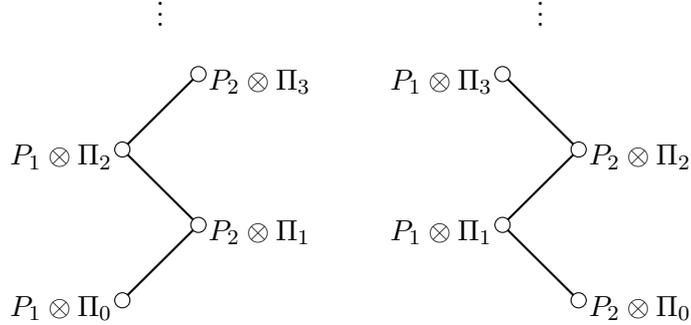
\begin{remark}
We note here that our definition of $A_{n-1}$-configuration is slightly different from the one given in \cite{brav_thomas_2010}.
This is mainly because of the fact that we have two separate $\Z$-gradings, one for the internal grading and one for the cohomological grading.
These two gradings can be identified if one considers the appropriate versions of dg algebras and dg modules instead.
\end{remark}
In particular, $\I$-prmod decomposes (as an additive category) into a direct sum of two additive subcategories:
\[
\I\text{-prmod} = \I\text{-prmod}^+ \oplus \I\text{-prmod}^-,
\]
each containing a separate $A_{n-1}$ configuration.
Given the configuration it is easy to see that each $\I$-prmod${}^\pm$ is equivalent (as additive categories) to $\mathscr{A}_{n-1}\text{-prgrmod}^{>0}$.

Similarly on the homotopy category, $\Kom^b(\I$-prmod) decomposes (as a triangulated category) into a direct sum of two triangulated subcategories:
\[
\Kom^b(\I\text{-prmod}) = \Kom^b(\I\text{-prmod}^+) \oplus \Kom^b(\I\text{-prmod}^-),
\]
and each $\Kom^b(\I\text{-prmod}^\pm)$ is equivalent to $\Kom^b(\mathscr{A}_{n-1}\text{-prgrmod}^{>0})$ (as triangulated categories).
Moreover, it is easy to see that the action of $\B(I_2(n))$ on $\Kom^b(\I$-prmod) restricts to actions on both $\Kom^b(\I$-prmod${}^+$) and $\Kom^b(\I$-prmod${}^-$) as triangulated categories.
Note however that the action is no longer on a triangulated module category -- neither of $\Kom^b(\I$-prmod${}^\pm$) are preserved under the $\TLJ_n$ action.

We can also relate our action of $\B(I_2(n))$ on $\Kom^b(\I$-prmod${}^\pm$) to the action of $\B(A_{n-1})$ on $\Kom^b(\mathscr{A}_{n-1}\text{-prgrmod}^{>0}) \subset D^b(\mathscr{A}_{n-1}\text{-mod})$ defined in \cite{khovanov_seidel_2001}.
First, we will need to recall an injection of $\B(I_2(n))$ into $\B(A_{n-1})$ introduced in \cite{crisp_1997}.
Let us use $d_i$ to denote the standard generators of $\B(A_{n-1})$ to avoid confusion (think of ``$d$'' as Dehn twist).
It was shown in \cite{crisp_1997} that the homomorphism $\varphi: \B(I_2(n)) \ra \B(A_{n-1})$ defined by
\[
\sigma_1 \mapsto d_1d_3..., \quad \sigma_2 \mapsto d_2d_4...
\]
is an injection.
Under this injection, one can check that the categorical action of $\B(I_2(n)) \overset{\varphi}{\subset} \B(A_{n-1})$ on $\Kom^b(\mathscr{A}_{n-1}\text{-prgrmod}^{>0})$ as defined in \cite{khovanov_seidel_2001} intertwines our $\B(I_2(n))$ action on $\Kom^b(\I$-prmod${}^\pm$) under the equivalence $\Kom^b(\mathscr{A}_{n-1}\text{-prgrmod}^{>0}) \cong \Kom^b(\I\text{-prmod}^\pm)$.

\subsection{Generalised Temperley-Lieb algebra and a quotient of Soergel bimodules}
The (generalised) Temperley-Lieb algebra TL$(\Gamma)$ of a (finite-type) Coxeter group $\W(\Gamma)$ is defined to be the quotient of the Hecke algebra of $\W(\Gamma)$ by the Kazhdan-Lusztig basis element $b_{w_0}$ corresponding to the longest element $w_0 \in \W(\Gamma)$.
In particular, TL$(\Gamma)$ is by definition a free $\Z[q,q^{-1}]$-module with basis $b_w$ for each $w \in \W(\Gamma) \setminus \{w_0\}$.

Recall the category of $\I$-bimodules $\mathbb{U}_n$ defined in \cref{defn: bimodule category}.
We shall consider the monoidal, thick additive subcategory of $\mathbb{U}_n$ that is closed under grading shift $\<1\>$ and is generated by the objects $U_i := P_i \otimes {}_iP \<-1\> \in \mathbb{U}_n$.
We claim that this category categorifies TL$(I_2(n))$ and shall therefore denote it as \textbf{TL}$(I_2(n))$.
We will only briefly spell out why \textbf{TL}$(I_2(n))$ categorifies TL$(I_2(n))$ here and leave the details to the reader.
Note that all the indecomposables (up to shifts) are given by $P_i \otimes \Pi_{2a} \otimes {}_iP \<-1\>$ and $P_i \otimes \Pi_{2a-1} \otimes {}_{i\pm 1}P \<-1\>$, where each of them first appears as the unique indecomposable summand in the expression below and not in any shorter alternating tensor product:
\begin{align*}
P_i \otimes \Pi_{2a} \otimes {}_iP \<-1\> &\overset{\oplus}{\subset} \underbrace{U_i \otimes_{\I} U_{i \pm 1} \otimes_{\I} ...}_{2a+1 \text{ times}} \\
P_i \otimes \Pi_{2a-1} \otimes {}_{i\pm 1}P \<-1\> &\overset{\oplus}{\subset} \underbrace{U_i \otimes_{\I} U_{i \pm 1} \otimes_{\I} ...}_{2a \text{ times}}.
\end{align*}
Moreover, the indecomposable listed above are all pairwise non-isomorphic (under all possible grading shifts) -- this can be checked on the morphism space by applying the adjunctions between (shifts of) $P_i \otimes \Pi_a$ and $\Pi_a \otimes {}_iP$, and the self-duality of $\Pi_a$.
As such, their Grothendieck class forms a $\Z[q,q^{-1}]$ basis of the Grothendieck ring $K_0(\textbf{TL}(I_2(n)))$.
We leave it to the reader to check that the required categorified relations hold, so that $K_0(\textbf{TL}(I_2(n)))$ is isomorphic to TL$(I_2(n))$ under the map:
\begin{align*}
[P_i \otimes \Pi_{2a} \otimes {}_iP \<-1\>] &\mapsto b_{{}_{s_i}\widehat{2a+1}} \\
[P_i \otimes \Pi_{2a-1} \otimes {}_{i\pm 1}P \<-1\>] &\mapsto b_{{}_{s_i}\widehat{2a}}
\end{align*}
where
\[
{}_{s_1}\widehat{k}: = \underbrace{s_1s_2s_1...}_{k \text{ times}}, \quad {}_{s_2}\widehat{k}: = \underbrace{s_2s_1s_2...}_{k \text{ times}} \in \W(I_2(n)).
\]
\begin{remark}
The fact that this category of bimodules categorifies the generalised Temperley-Lieb algebra is a rank two only feature -- the similar category of bimodules in \cite{khovanov_seidel_2001} only categorifies a (non-trivial) quotient of the Temperley-Lieb algebra for higher ranks.
\end{remark}

Just as the Temperley-Lieb algebra is a quotient of the Hecke algebra, this relation can also be expressed on the categorical level.
Namely our category \textbf{TL}$(I_2(n))$ is some quotient of the category of Soergel bimodules.
To spell it out, one can emulate the functor $G_0$ defined in \cite[Theorem 5.18]{jensen_master} (which maps from the category of Soergel bimodules for type $A$ to category of bimodules considered in \cite{khovanov_seidel_2001}) and define a similar functor from the category of Soergel bimodules for type $I_2(n)$ (the diagrammatic category $\cD_{n}$ defined in \cite{elias_2015}) to \textbf{TL}$(I_2(n))$.
We note in particular that this functor maps the $2n$-valent map to zero, where the ``dotting the vertex'' relation is trivial to check in our case since the ($n-1$)th (negligible) Jones-Wenzl projector is already killed in $\TLJ_n$.
Note however that this functor does not induce an equivalence between Elias' categorification of TL$(\Gamma)$ in \cite{elias_2015} -- defined as the quotient of $\cD_{n}$ by the $2n-$valent map, and our category \textbf{TL}$(I_2(n))$.
This can be seen on the monoidal unit, since in Elias' construction the total endomorphism space of the monoidal unit is infinite dimensional, whereas the total endomorphism space of our monoidal unit $\I$ is finite dimensional (one dimensional in degree zero and two dimensional in degree two).

\comment{
Throughout this section, we fix $n \geq 3$ odd.
Recall that when $n$ is odd, we can consider the sub fusion category of $\TLJ_n$ generated by just the evenly labelled simple objects, which we denote as $\TLJe_n$.
Note that by considering $\TLJe_n \subseteq \TLJ_n$, no two simple objects have the same quantum dimensions anymore.
As in the previous section, we impose a $\Z$-grading on $\TLJe_n$, so that the objects are now $\Z$-graded and morphisms are grading preserving; this category will be denoted as $\gTLJe_n$ and the grading shifts will be denoted by the bracket $\< - \>$.

The zigzag algebra object $\I$ will now live in $\gTLJe_n$ instead.
The construction of $\I$ is the same everywhere except we replace every occurence of $\Pi_1$ (which has quantum dimension $2\cos(\frac{\pi}{n})$) with $\Pi_{n-3}$, which is now the only object with quantum dimension $2\cos(\frac{\pi}{n})$ in the category $\gTLJe_n$.

The definitions of $P_i$ and ${}_iP$ are the same, and the fact that they form a biadjoint pair follows similarly.
Thus, we may define $\sigma_{P_i}$ and $\sigma_{P_i}'$ in the same way.
We wish to verify the following relations:
\begin{proposition}
We have the following isomorphism in $\Kom^b((\I,\I)\text{-bimod})$:
\[
\sigma_{P_i} \otimes_\I \sigma_{P_i}' \cong \I \cong \sigma_{P_i}' \otimes_\I \sigma_{P_i}
\]
\[
\underbrace{\sigma_{P_1} \otimes_\I \sigma_{P_2} \otimes_\I \sigma_{P_1} \otimes_\I \cdots \otimes_\I \sigma_{P_1}}_{n \text{ times}} \cong \underbrace{\sigma_{P_2} \otimes_\I  \otimes_\I \sigma_{P_2} \otimes_\I \cdots \otimes_\I \sigma_{P_2}}_{n \text{ times}}
\]
\end{proposition}

We will need the following slightly more complicated version of \cref{lemma for braid relation}:
\begin{lemma}
Let $0 \leq k \leq \frac{n-3}{2}$. For $k \neq \frac{n-3}{2}$,
\[
\sigma_i (P_{i\pm 1} \otimes \Pi_{n-3-2k} \xra{(i\pm 1|i)} P_i \otimes \Pi_{2k}) \cong P_i \otimes \Pi_{2k+2} \xra{(i|i\pm 1)} P_{i\pm 1} \otimes \Pi_{n-3-2k}
\] 
and for $k = \frac{n-3}{2}$,
\[
\sigma_i (P_{i\pm 1} \otimes \Pi_{n-3-2k} \xra{(i\pm 1|i)} P_i \otimes \Pi_{2k}) \cong P_{i\pm1}.
\] 
Similarly, for $1 \leq j \leq \frac{n-3}{2}$,
\[
\sigma_i (P_{i\pm 1} \otimes \Pi_{2j} \xra{(i\pm 1|i)} P_i \otimes \Pi_{n-1-2j}) \cong P_i \otimes \Pi_{n-3-2j} \xra{(i|i\pm 1)} P_{i\pm 1} \otimes \Pi_{2j}.
\] 
\end{lemma}
}

\newpage
\part{Categorical Dynamics}
\chapter{Dynamical systems in triangulated module categories}

In this chapter, we introduce the appropriate notion of dynamics and stability conditions required to study triangulated module $\X$-category over a fusion category.
We also give the definition of a mass automaton as introduced in \cite{bapat2020thurston}, which will be essential in organising the computation of mass into (partially) linear transformation in the main body of the thesis.
Throughout this chapter, we fix $\mathcal{D}$ to be a triangulated category.

\section{Stability conditions and mass growth}
In this section, we recall the notion of stability conditions and mass growth as introduced in \cite{bridgeland_2007}, \cite{DHKK13} and \cite{ikeda_2020}.

\subsection{$t$-structure and stability conditions}
We start with a quick review of $t$-structures on triangulated categories and its relation to Bridgeland stability conditions.
We refer the unfamiliar readers to \cite{neeman01} and \cite{arapura10} for an introduction to triangulated categories and $t$-structures; \cite{bayer11} and \cite{bridgeland_2007} for details on stability conditions.
\begin{definition}
A \emph{$t$-structure} on $\mathcal{D}$ consists of a pair of subcategories $(\cF, \cF^\perp)$ of $\mathcal{D}$ such that
\begin{enumerate}
\item $\cF \subset \cF[-1], \quad \cF^\perp[-1] \subset \cF^\perp$;
\item $\Hom(\cF, \cF^\perp) = 0$;
\item for all objects $A$ in $\mathcal{D}$, there exists $X \in \cF$ and $Y \in \cF^\perp$ such that 
\[
X \ra A \ra Y \ra X[1]
\]
is a distinguished triangle. 
\end{enumerate}
When $\mathcal{D} = \bigcup_{i,j \in \mathbb{Z}} \cF^\perp[i] \cap \cF[j]$, we say that $(\cF, \cF^\perp)$ is \emph{bounded}.
\end{definition}

\begin{definition}
Given a $t$-structure $(\cF, \cF^\perp)$, the \emph{heart} $\mathcal{H}$ of $(\cF,\cF^\perp)$ is the subcategory defined by $\mathcal{H}:= \cF \cap \cF^\perp[1]$. 
The heart of a $t$-structure is always an abelian category.
\end{definition}

\begin{definition} \label{defn: abelian stability}
A \emph{stability function} on an abelian category $\cA$ is a group homomorphism $Z: K_0(\cA) \ra \C$ such that for all non-zero objects $E$ in $\cA$, we have $Z(E) \neq 0$ lying in the strict upper-half plane
\[
\mathbb{H} = \{re^{i\pi\phi} : r \in \R_{> 0} \text{ and } 0 \leq \phi < 1\}.
\]
Writing $Z(E) = m(E)e^{i\pi\phi}$, we call $\phi$ the \emph{phase} of $E$, which we will denote as $\phi(E)$.
We say that $E$ is \emph{semi-stable} (resp. \emph{stable}) if every subobject $A \hookrightarrow E$ has $\phi(A) \leq$ (resp. $<$) $\phi(E)$ . \\
A stability function is said to have the \emph{Harder-Narasimhan (HN) property} if every object $A$ in $\cA$ has a filtration
\[
\begin{tikzcd}[column sep = 3mm]
0
	\ar[rr] &
{}
	{} &
A_1
	\ar[rr] \ar[dl]&
{}
	{} &	
A_2
	\ar[rr] \ar[dl]&
{}
	{} &	
\cdots
	\ar[rr] &
{}
	{} &
A_{m-1}
	\ar[rr] &
{}
	{} &
A_m = A
	\ar[dl] \\
{}
	{} &
E_1
	\ar[lu, dashed] &
{}
	{} &
E_2
	\ar[lu, dashed] &
{}
	{} &
{}
	{} &
{}
	{} &
{}
	{} &
{}
	{} &
E_m
	\ar[lu, dashed] &
\end{tikzcd}
\]
such that all $E_i$ are semi-stable and $\phi(E_i) > \phi(E_{i+1})$.
Note that this filtration, if it exists, is unique up to isomorphism.
Hence we shall denote $\phi_+(A) := \phi(E_1)$, $\phi_-(A) := \phi(E_m)$ and $\lceil A \rceil := E_1, \lfloor A \rfloor := E_m$.
\end{definition}

\begin{definition}
A \emph{slicing} on $\cD$ is a collection of full additive subcategories $\cP(\phi)$ for each $\phi \in \R$ such that
\begin{enumerate}[(1)]
\item $\cP(\phi + 1) = \cP(\phi)[1]$;
\item $\Hom( \cP(\phi_1), \cP(\phi_2)) = 0$ for all $\phi_1 > \phi_2$; and
\item for all objects $A$ in $\cD$, there exists $\phi_1 > \phi_2 > \cdots > \phi_m$ such that
\[
\begin{tikzcd}[column sep = 3mm]
0
	\ar[rr] &
{}
	{} &
A_1
	\ar[rr] \ar[dl]&
{}
	{} &	
A_2
	\ar[rr] \ar[dl]&
{}
	{} &	
\cdots
	\ar[rr] &
{}
	{} &
A_{m-1}
	\ar[rr] &
{}
	{} &
A_m = A
	\ar[dl] \\
{}
	{} &
E_1
	\ar[lu, dashed] &
{}
	{} &
E_2
	\ar[lu, dashed] &
{}
	{} &
{}
	{} &
{}
	{} &
{}
	{} &
{}
	{} &
E_m
	\ar[lu, dashed] &
\end{tikzcd}
\]
with each $E_i$ object in $\cP(\phi_i)$.
\end{enumerate}
Objects in $\cP(\phi)$ are said to be \emph{semistable} with \emph{phase} $\phi$ and the (unique upto isomorphism) filtration in (3) is called the \emph{Harder-Narasimhan (HN) filtration} of $A$.
We will also call the $E_i$ pieces the \emph{HN semistable pieces of} $A$.
\end{definition}

Given a slicing $\cP$, denote $\cP[\phi, \phi+1)$ as the smallest additive subcategory containing all objects $A$ such that the phases of the semistable pieces all live in $[\phi, \phi+1)$.
Then $\cP[\phi, \phi+1)$ is the heart of a bounded $t$-structure.
In particular, we call $\cP[0,1)$ the \emph{standard heart} of the slicing $\cP$.
As such, a slicing should be thought of as a refinement of a t-structure.

\begin{definition}[\cite{bridgeland_2007}]
A \emph{stability condition} $(\cP, \cZ)$ on $\cD$ is a slicing $\cP$ and a group homomorphism $\cZ: K_0(\cD) \ra \C$, called the \emph{central charge}, satisfying the following condition: for all non-zero semistable objects $E\in \cP(\phi)$,
\[
\cZ(E) = m(E)e^{i\pi\phi}, \text{ with } m(E) \in \R_{>0}.
\]
\end{definition}

\begin{theorem}[Proposition 5.3, \cite{bridgeland_2007}] \label{thm: stab function and stab cond}
Giving a stability condition on $\cD$ is equivalent to giving a bounded $t$-structure on $\cD$ with a stability function which satisfies the Harder-Narasimhan property.
\end{theorem}

It follows that if the bounded $t$-structure has a finite-length heart, then any stability function will satisfy the Harder-Narasimhan property (see \cite[Proposition 2.4]{bridgeland_2007}).
One of the surprising result about stability conditions is that they form a $\C$-manifold:
\begin{theorem}[Theorem 1.2, \cite{bridgeland_2007}]
The space of stability conditions $\Stab(\cD)$ is a complex manifold, given by the local homeomorphism 
\begin{align*}
\Stab(\cD) &\ra \Hom(K_0(\cD),\C) \\
(\cP, \cZ) &\mapsto \cZ.
\end{align*}
\end{theorem}

\subsection{Mass and mass growth}
We now give a quick review of mass and mass growth with respect to a stability condition.

\begin{definition}
Let $\tau = (\cZ, \cP)$ be a stability condition on $\cD$.
Let $A$ be a non-zero object in $\cD$ with Harder-Narasimhan filtration
\[
\begin{tikzcd}[column sep = 3mm]
0
	\ar[rr] &
{}
	{} &
A_1
	\ar[rr] \ar[dl]&
{}
	{} &	
A_2
	\ar[rr] \ar[dl]&
{}
	{} &	
\cdots
	\ar[rr] &
{}
	{} &
A_{m-1}
	\ar[rr] &
{}
	{} &
A_m = A
	\ar[dl] \\
{}
	{} &
E_1
	\ar[lu, dashed] &
{}
	{} &
E_2
	\ar[lu, dashed] &
{}
	{} &
{}
	{} &
{}
	{} &
{}
	{} &
{}
	{} &
E_m
	\ar[lu, dashed] &
\end{tikzcd}
\]
where $E_i \in \cP(\phi_i)$.
We define the \emph{mass of $A$} (with parameter $t$) as the (non-negative) real-valued function over $t$ given by
\[
m_{\tau,t}(A) := \sum_{i=1}^m |\cZ(E_i)|e^{\phi_i t} \in \R^\R.
\]
By convention, we set $m_{\tau,t}(A) = 0$ if $A \cong 0$.
\end{definition}
\begin{remark}
In this thesis, mass of $A$ will always mean the whole function with parameter $t$; we will specify $t=0$ when we want to refer to its specific value at $t=0$.
This convention differs from some literature, where ``mass of $A$'' sometimes refers to the mass of $A$ at $t=0$.
\end{remark}

The mass of an object satisfies the following ``triangle inequality'':
\begin{proposition}[Proposition 3.3, \cite{ikeda_2020}]
For each distinguished triangle $A\ra B \ra C \ra A[1]$ in $\cD$, we have
\[
m_{\tau,t}(B) \leq m_{\tau,t}(A) + m_{\tau,t}(C)
\]
for all $t\in \R$.
\end{proposition}

\begin{definition}
Let $\tau$ be a stability condition on $\cD$ and $\cF : \cD \ra \cD$ be an endofunctor.
The \emph{$\tau$-mass growth of $F$} (with parameter $t$) is the function $h_{\tau, t}(\cF) : \R \ra [-\infty, \infty]$ given by
\[
h_{\tau, t}(\cF) := \sup_{A \in \text{Obj}(\cD)} \left\{
	\limsup_{N\ra \infty} \frac{1}{N} \log  m_{\tau,t}(\cF^N(A))
	\right\}.
\]
\end{definition}

The following proposition shows that mass growths with respect to different stability conditions within the same connected component coincide:
\begin{proposition}[\cite{ikeda_2020}]
Let $\tau$ and $\tau'$ be two stability conditions on $\cD$ lying in the same connected component of $\Stab(\cD)$ and $\cF : \cD \ra \cD$ be an endofunctor.
Then
\[
h_{\tau, t}(\cF) = h_{\tau', t}(\cF).
\]
\end{proposition}

\section{Level-complexity, $q$-stability conditions and mass growth for $\X$-categories}
In the main body of this thesis, we will be mainly interested in mass growth of endofunctors on triangulated categories with more structures: a module structure over fusion categories and the existence of a distinguished autoequivalence $\X$.
We begin our study with the inclusion of a distinguished autoequivalence $\X$ in this section.

We start by introducing the notion of level-complexity and level-entropy, which are slight generalisations of complexity and entropy introduced in \cite{DHKK13}.
In particular, for a triangulated $\X$-category $(\cD, \X)$, we will only require $\cD$ to contain the more relaxed notion of a split generator called an $\X$-split generator and use level-complexity to show that we can compute mass growth of endofunctors on triangulated $\X$-categories from just the $\X$-split generator.
Most of the results in this section are a rederivation of the results in \cite{ikeda_2020} applied to this more general setting, so we will only provide a sketch to the proof whenever we think is necessary.

Throughout this section, $(\cD, \X)$ will denote a triangulated $\X$-category.

\subsection{Level entropy for $\X$-categories}
When dealing with triangulated $\X$-categories it is natural to ask for a split generator up to the action of $\X$ instead:
\begin{definition}
We say that an object $G$ is a \emph{$\X$-split generator} of $(\cD, \X)$ if for all objects $F \in \cD$, we have some object $F' \in \cD$ and a finite length filtration:
\[
\begin{tikzcd}[scale cd = .9, column sep = .5mm]
0= F_0
	\ar[rr] &
{}
	{} &
F_1
	\ar[rr] \ar[dl]&
{}
	{} &	
F_2
	\ar[rr] \ar[dl]&
{}
	{} &	
\cdots
	\ar[rr] &
{}
	{} &
F_{m-1}
	\ar[rr] &
{}
	{} &
F_m = F \oplus F'
	\ar[dl] \\
{}
	{} &
G\cdot \X^{k_1}[\ell_1]
	\ar[lu, dashed] &
{}
	{} &
G\cdot \X^{k_2}[\ell_2]
	\ar[lu, dashed] &
{}
	{} &
{}
	{} &
{}
	{} &
{}
	{} &
{}
	{} &
G\cdot \X^{k_m}[\ell_m]
	\ar[lu, dashed] &
\end{tikzcd}
\]
with $k_i, \ell_i \in \Z$.
\end{definition}

\comment{
Given some fixed object $E$, we will define the \emph{level grading} $\{-\}$ of $E\<k\>[\ell]$ as the difference of the triangulated grading and internal grading:
\[
\{ - \} := [-] - \< - \>.
\]
}
We shall define a variant of complexity introduced in \cite{DHKK13}, in which the coefficient of $t$ measures the difference between $\X$ and $[1]$ instead:
\begin{definition}
Let $E, F \in \Ob(\cD, \X)$.
The \emph{level complexity} of $F$ relative to $E$ is the function $\delta_t(E,F) : \R \ra \R_{\geq 0} \cup \{ \infty \}$ defined by
\[
\delta_t(E,F) := 
	\underset{F'}{\inf} \left\{ 
		\sum_{i=1}^{k} e^{\alpha_i t} : 
			\begin{tikzcd}[scale cd = .8, column sep = .5mm, row sep = 2mm]
			0
				\ar[rr] &
			{}
				{} &
			F_1
				\ar[dl]&
			\cdots
				{} &
			F_{m-1}
				\ar[rr] &
			{}
				{} &
			F \oplus F'
				\ar[dl] \\
			{}
				{} &
			E\cdot \X^{k_1}[\ell_1]
				\ar[lu, dashed] &
			{}
				{} &
			{}
				{} &
			{}
				{} &
			E\cdot \X^{k_m}[\ell_m]
				\ar[lu, dashed] &
			\end{tikzcd}
		\right\},
\] 
with $\alpha_i := \ell_i - k_i$, where $\delta_t(E,F) := 0$ if $F \cong 0$ and $\delta_t(E,F) := \infty$ if $F$ does not live in the smallest thick subcategory of $(\cD, \X)$ with $E$ as a $\X$-split generator.
\end{definition}

Level-complexity satisfies the following properties that are similar to the usual complexity:
\begin{proposition} \label{complexity lemma}
For $D, E_1,E_2,E_3 \in \Ob(\cD, \X)$,
\begin{enumerate}
\item $\delta_{t}(E_1, E_3) \leq \delta_{t}(E_1,E_2)\delta_{t}(E_2,E_3)$;
\item if $E_1 \ra E_2 \ra E_3 \ra$ is a distinguished triangle, then 
\[
\delta_{t}(D, E_2) \leq \delta_{t}(D,E_1) + \delta_{t}(D, E_3);
\]
\item for any endofunctor $\cF: (\cD,\X) \ra (\cD,\X)$,
\[
\delta_{t}(\cF(E_1), \cF(E_2)) \leq \delta_{t}(E_1,E_2).
\]
\end{enumerate}
\end{proposition}

\begin{definition}
Let $G$ be a $\X$-split generator for $(\cD, \X)$ and $\cF: (\cD, \X) \ra (\cD, \X) $ be an endofunctor.
The \emph{level entropy} of $\cF$ is the function $h_t(\cF): \R \ra [-\infty, \infty)$ defined by
\[
h_t(\cF) := \lim_{N \ra \infty} \frac{1}{N} \log \delta_t \left(G, \cF^N(G) \right).
\]
\end{definition}
As is the case for the usual categorical entropy, one can show that the limit exists and that the definition is independent of the choice of $\X$-split generator.
Once again we shall state some simple properties of level entropy:
\begin{proposition}\label{prop: entropy properties}
Let $\cF: (\cD, \X) \ra (\cD, \X)$ be an endofunctor.
Then
\begin{enumerate}
\item $h_t(\cF^k) = k h_t(\cF)$ for $k \in \mathbb{N}$; 
\item $h_t(\cF) = h_t(\cG\cF\cG^{-1})$ for $\cG: (\cD, \X) \ra (\cD, \X)$ an autoequivalence. 
\end{enumerate}
\end{proposition}
\begin{remark} \label{rem: entropy of inverse}
Note that even when $\cF$ is assumed to be an autoequivalence, in general it is \emph{not true} that $h_t(\cF^k) = k h_t(\cF)$ for non-positive $k$ -- we shall see later (\cref{rem: entropy sum of itself and inverse}) that there are autoequivalences $\cF$ such that
\[
h_t(\cF^{-1}) \neq -h_t(\cF).
\]
However, it is indeed true that for any commuting endofunctors $\cF$ and $\cG$, 
\[
h_t(\cF\circ\cG) \leq h_t(\cF) + h_t(\cG).
\]
In particular, we have that
\[
h_t(\id_{\cD}) \leq h_t(\cF) + h_t(\cF^{-1}).
\]
\end{remark}

\subsection{$q$-Stability conditions on $\X$-categories}
Given a triangulated $\X$-category $(\cD, \X)$, its triangulated Grothendieck group $K_0(\cD, \X)$ is naturally a module over $\Z[q,q^{-1}]$ with $q\cdot [X] = [X \cdot \X]$.
We shall equip $\C$ with a module structure over $\Z[q,q^{-1}]$ as well, by evaluating $q = -1 \in \C$; we denote this $\Z[q,q^{-1}]$-module structure on $\C$ by $\C_{-1}$.

\begin{definition}[cf. \cite{ikeda2020qstability}] \label{defn: module q stab}
A \emph{$q$-stability condition} (with $s=-1$) on $(\cD, \X)$ consists of a stability condition $(\cP, \cZ)$ satisfying:
\begin{enumerate}
\item $
\cP(\phi)\cdot \X = \cP(\phi -1)
$; and
\item $\cZ \in \Hom_{\Z[q,q^{-1}]}(K_0(\cD, \X), \C_{-1})$.
\end{enumerate}
\end{definition}

\comment{
\begin{remark}
A $q$-stability condition as defined in \cite{ikeda2020qstability} is just a $q$-stability condition over $g\cC$ with $\cC$ the (fusion) category of finite dimensional vector space, and $\<1\>$ as the distinguish autoequivalence.
\end{remark}
}
\begin{remark}
In \cite{ikeda2020qstability}, they allow the more general setting of evaluating $q$ at any choice of $s \in \C$, where the first condition is then replaced with
\[
\cP(\phi) \cdot \X = \cP(\phi + \mathfrak{R}\mathfrak{e}(s)).
\]
Here we have fixed the choice of $s=-1$, and we will not go into the details regarding the different choices of $s$ and how it affects the space of $q$-stability conditions.
The interested reader may refer to \cite{ikeda2020qstability} for more details on this.
\end{remark}
One should think of (our) $q$-stability condition as a compatibility requirement between the two autoequivalences $\X$ and $[1]$ of $(\cD, \X)$.

\subsection{Level-entropy and mass growth with respect to $q$-stability conditions}
This subsection will be a reformulation of many of the results in \cite{ikeda_2020} for a $\X$-category $(\cD, \X)$.
As such we will only state the results, referencing the analogous statement in \cite{ikeda_2020} and sketch the proof whenever we think is necessary.

First, we show that for $(\cD, \X)$ with a $\X$-split generator $G$, the mass growth with respect to a $q$-stability condition (with $s = -1$) can be achieved by $G$ and it is a lower bound for level entropy.
\begin{theorem}[\protect{compare with \cite[Theorem 3.5]{ikeda_2020}} ] \label{mass growth split generator}
Let $\tau = (\cP, \cZ)$ be a $q$-stability condition (with $s=-1$) on $(\cD, \X)$ and $\cF : (\cD, \X) \ra (\cD, \X)$ be an endofunctor.
Suppose $(\cD, \X)$ has a $\X$-split generator $G$.
Then the $\tau$-mass growth satisfies the following:
\begin{enumerate}
\item $h_{\tau,t}(\cF)$ is achieved by $G$:
\[
h_{\tau,t}(\cF) = \limsup_{N\ra \infty} \frac{1}{N} \log  m_{\tau,t}(\cF^N(G)).
\]
\item The $\tau$-mass growth is a lower bound for level entropy:
\[
h_{\tau,t}(\cF) \leq h_t(\cF) < \infty.
\]
\end{enumerate}

\end{theorem}
\begin{proof}[Proof sketch]
The assumption that $\tau$ is a $q$-stability condition gives us:
\[
m_{\tau,t}(A\<k\>[\ell]) = m_{\tau,t}(A\<k\>[k][\ell-k]) = m_{\tau,t}(A) e^{(\ell - k)t}.
\]
Using this, we obtain a bound on mass
\[
m_{\tau,t}(A) \leq m_{\tau,t}(B) \delta_{t}(B,A)
\]
as in \cite[Proposition 3.4]{ikeda_2020} with $\delta_{t}(B,A)$ denoting the level complexity instead.
The rest of the proof follows exactly as in \cite[Proposition 3.4]{ikeda_2020}.
\end{proof}

In \cite{ikeda_2020}, it was shown that the mass growth and entropy coincide for triangulated categories with algebraic hearts.
Since we are working with $\X$-categories, it will be sufficient to have the following type of heart:
\begin{definition}\label{defn: X1-algebraic heart}
Let $\cH$ be the heart of a bounded $t$-structure on $(\cD, \X)$.
We say $\cH$ is \emph{$\X[1]$-algebraic} if 
\begin{enumerate}[(i)]
\item $\cH$ is closed under the autoequivalence $\X[1]$, 
\item $\cH$ is finite length and
\item there is a finite set of simple objects $\{S_1, S_2, ..., S_m\}$ such that all simple objects in $\cH$ are isomorphic to $S_i\cdot \X^k[k]$ for some $1\leq i \leq m$ and some $k\in \Z$.
\end{enumerate}
\end{definition}

Note that if $(\cD, \X)$ has a $\X[1]$-algebraic heart $\cH$, then $G:= \oplus_{i=1}^m S_i$ is always a $\X$-split generator of $(\cD, \X)$.
Moreover, $K_0(\cH)$ is a free module over $\Z[x,x^{-1}]$, where the action is defined by
\[
x \cdot [A] = [A\cdot \X[1]]
\]
with $\{[S_1], [S_2], ..., [S_m] \}$ as a basis.
As such, any stability function $Z$ on $\cH$ obtained by specifying $Z([S_j]) = s_j \in \C$ and extended linearly, defines a $q$-stability condition (with $s=-1)$ on $(\cD, \X)$.
Note that the Harder-Narasimhan property comes for free since $\cH$ is finite length by definition.
As in \cite{ikeda_2020}, we consider the $q$-stability condition $\tau_0$ obtained by specifying
\[
Z([S_j]) = i \in \C
\]
for all $j$, which gives us the following theorem:
\begin{theorem}[\protect{compare with \cite[Theorem 3.14]{ikeda_2020}} ] \label{thm: entropy = mass growth}
Suppose $(\cD, \X)$ contains a $\X[1]$-algebraic heart $\cH$.
Let $\Stab^\circ(\cD) \subseteq \Stab(\cD)$ be the connected component containing the $q$-stability condition $\tau_0 = (\cP,\Z)$ as defined above.
Then for all $\tau \in \Stab^\circ(\cD)$ lying in the same connected component (not necessarily a $q$-stability condition),
\[
h_{\tau,t}(\cF) = h_{\tau_0, t}(\cF) = \lim_{N \ra \infty} \frac{1}{N} \log m_{\tau_0,t}(\cF^N(G)) = h_t(\cF).
\]
for any $\X$-split generator $G$.
\end{theorem}
\begin{remark}
One of the main purpose of relating mass growth and level-entropy in this thesis is to obtain the fact that ``mass growth is invariant under conjugation'' whenever mass growth and level-entropy agree (cf. \cref{prop: entropy properties}).
If one shows that the orbit of $\tau$ in $\Stab(\cD)$ under $\Aut(\cD)$ are all in the same connected component, then $h_{\tau,t}$ will always be invariant under conjugation. 
\end{remark}

\section{Stability conditions on triangulated module $\X$-categories over fusion categories}
In this section, we include the module category structure into our analysis.
We start with the introduction of stability conditions on triangulated module categories, and then combine it with the $q$-stability conditions defined in the previous section.
We then prove an equivalent definition of these stability conditions in terms of stability functions on the heart (equipped with some extra structure), similar to \cref{thm: stab function and stab cond}.

Throughout this section, $\cC$ denotes a fusion category,
$(\cD,\Psi)$ is a triangulated module category over $\cC$ and $(\cD,\X, \Psi)$ is a triangulated module $\X$-category over $\cC$.

\subsection{Stability conditions for module categories over fusion categories}
Recall that the Grothendieck group $K_0(\cD)$ of a triangulated module category $(\cD,\Psi)$ over $\cC$ has a natural module structure over $K_0(\cC)$.
On the other hand, we have a ring homomorphism from $K_0(\cC)$ into $\C$ induced by the Perron-Frobenius dimension (see \cref{sect: fusion}):
\[
PF\dim: K_0(\cC) \ra \R \subset \C,
\]
which provide $\C$ with a natural module structure over $K_0(\cC)$.
We shall use $\C_\cC$ to denote this module structure on $\C$.
\begin{definition}\label{defn: stab respects C}
Let $(\cD, \Psi)$ be a triangulated module category over $\cC$ and let $\tau = (\cP, \cZ)$ be a stability condition on $\cD$.
We say that $\tau$ \emph{respects $\cC$} if it satisfies:
\begin{enumerate}
\item $
\Psi(C)\cP(\phi) \subseteq \cP(\phi)
$ for all $C \in \cC$; and
\item $\cZ \in \Hom_{K_0(\cC)}(K_0(\cD), \C_\cC)$.
\end{enumerate}
\end{definition}
\begin{remark}
We expect the (sub)space of stability conditions respecting $\cC$ to form a $\C$-(sub)manifold as well.
\end{remark}

\begin{example}
Recall that every additive category is trivially a module category over $\mathbb{K}$-vec.
With $\cD$ a triangulated category equipped with this canonical module category structure over $\mathbb{K}$-vec, all stability conditions on $\cD$ respect $\mathbb{K}$-vec.
\end{example}

\begin{example} \label{eg: TLJ_5 stab respect C}
Let $\cC:= \TLJ_5$ be the Temperley-Lieb-Jones category evaluated at the 10th root of unity, which is a fusion category with four simple objects $\{\Pi_0, \Pi_1, \Pi_2, \Pi_3\}$.
Consider its (bounded) derived category $D^b(\cC)$, which is naturally a triangulated module category over $\cC$ given by tensoring with the objects in $\cC$.
Let $Z$ be the stability function $Z: K_0(\cC) \ra \C$ on the standard heart $\cC \subset D^b(\cC)$ defined by
\[
Z([\Pi_a]) = \begin{cases}
	1, &\text{ for } a=0,3; \\
	2\cos\left( \frac{\pi}{5} \right), &\text{ for } a=1,2.
	\end{cases}
\]
The stability condition on $\cD^b(\cC)$ induced by $Z$ is a stability condition that respects $\cC$.
\end{example}

Note that given a stability condition $\tau$ respecting $\cC$, the standard heart of $\tau$ given by $\cH:= \cP[0,1)$ must be an abelian category closed under endofunctors $\Psi(C)$ for all $C$ in $\cC$.
Namely, $(\cH, \Psi)$ is a $\cC$-module subcategory of $(\cD, \Psi)$.
Hence its (exact) Grothendieck group $K_0(\cH)$ has a module structure over $K_0(\cC)$.
With the $K_0(\cC)$-module structure on $\C$ as before, we see that the induced stability function $Z: K_0(\cH) \ra \C$ on $\cH$ is a $K_0(\cC)$-module homomorphism.

Motivated by this, we make the following definition:
\begin{definition} \label{defn: stab function over C}
Let $\cA$ be an abelian category and $(\cA, \Phi)$ be an (abelian) module category over $\cC$.
A \emph{stability function over $\cC$} on $(\cA, \Phi)$ is a stability function $Z: K_0(\cA) \ra \C$  satisfying the extra condition that $Z$ is a $K_0(\cC)$-module homomorphism, where 
\begin{enumerate}
\item the (exact) Grothendieck group $K_0(\cA)$ is a $K_0(\cC)$-module induced by the $\cC$-module structure on $(\cA, \Phi)$; and
\item $\C$ is a $K_0(\cC)$-module induced by $PF\dim: K_0(\cC) \ra \R \subset \C$.
\end{enumerate}
\end{definition}

Even though $\Phi(C)$ may not be invertible, it turns out that $\Phi(C)$ still sends semistable objects to semistable objects (in the abelian category sense):
\begin{proposition} \label{cC action preserve semistable}
Let $(\cA, \Phi)$ be a $\cC$-module abelian category equipped with a stability function $Z: K_0(\cA) \ra \C$ over $\cC$ satisfying the Harder--Narasimhan property.
Then for any object $A$ in $\cA$, the following are equivalent:
\begin{enumerate}
\item $A$ is semistable;
\item $\Phi(C)A$ is semistable for all $C \neq 0 \in \cC$; and
\item $\Phi(C)A$ is semistable for some $C \neq 0 \in \cC$.
\end{enumerate}
\end{proposition}
\begin{proof}
Note that $Z$ being an $K_0(\cC)$-module map dictates that 
\[
\phi \left(\Psi(C)A \right) = \phi(A)
\]
for any object $A \in \cA$ and any $C \in \cC$.

\comment{
Furthermore, $\cdot \Y^m$ is an autoequivalence and $-\otimes \Pi_a$ is a self biadjoint functor, hence
\begin{equation}
\Hom_\cH(U,X\otimes \Pi_a \cdot \Y^m) \cong \Hom_\cH(U\otimes \Pi_a \<-m\>[m], X)
\end{equation}
and similarly
\begin{equation}
\Hom_\cH(U,X\otimes \Pi_a \cdot \Y^m[m]) \cong \Hom_\cH(U\otimes \Pi_a \<-m\>[m], X)
\end{equation}
for any object $U\in \cH$.
}

Let us first prove $(1) \Rightarrow (2)$.
Let $A$ be a semistable object and let $\{[S_i]\}_{i = 1}^n$ be the finite set of isomorphism classes of simple objects in $\cC$, with $S_1 = \1$ the monoidal unit; this set is finite as $\cC$ is a fusion category.
Since every non-zero object in $\cC$ is (upto isomorphism) a direct sum of $S_i$'s, it is sufficient to prove that $\Psi(S_i) A$ is semistable for all $i$.

Denote $r:= \max \{\phi_+(\Psi(S_i) X) \}_{i=1}^n$, so that $r \geq \phi_+(\Psi(S_0) X) = \phi(X)$ (see \cref{defn: abelian stability} for the notations).
If indeed $r = \phi(X)$, then all $\Psi(S_i) X$ must be semistable and so we are done.
Assume to the contrary that $r > \phi(X)$, so there exists some $S:= S_k$ with $\phi_+(\Psi(S) X) = r$.
In particular, $\Psi(S) X$ is not semistable, so we have a semistable object $E:= \lceil \Psi(S) X \rceil \subseteq \Psi(S) X$ with $\phi(E) = \phi_+(\Psi(S) X) = r$.
With $S^\vee$ denoting the (right) dual of $S$, $\Psi(S)$ is then a left adjoint of $\Psi(S^\vee)$, hence we have the isomorphism  
\[   
\Hom_{\cA}(E, \Psi(S) X) \cong \Hom_{\cA}(\Psi(S^\vee) E, X).  
\]
Since $E$ is a (non-zero) subobject of $\Psi(S)X$, it must be that $\Hom_{\cA}(E, \Psi(S) X) \neq 0$, which implies that $\Hom_{\cA}(\Psi(S^\vee) E, X) \neq 0$.
Combining this with the fact that $\phi(\Psi(S^\vee)E) = \phi(E) = r > \phi(X)$, it must be the case that $\Psi(S^\vee)E$ is not semistable.
So we have a semistable (non-zero) subobject $E':= \lceil \Psi(S^\vee)E \rceil \subseteq \Psi(S^\vee)E$ with 
\[
\phi(E') = \phi_+(\Psi(S^\vee)E) > \phi(\Psi(S^\vee)E) = r.
\]
Since $\Psi(C)$ are all exact endofunctors, we have that
\[
E' \subseteq \Psi(S^\vee)E \subseteq \Psi(S \otimes S^\vee) X.
\]
Note that $\Psi(S \otimes S^\vee) X \cong \bigoplus_{i=1}^n \Psi(S_i^{\oplus m_i}) X$, and so $r \geq \phi_+(\Psi(S \otimes S^\vee) X)$.
Since $E'$ is semistable and $\phi(E') >r \geq \phi_+(\Psi(S \otimes S^\vee) X)$, we must have $\Hom_\cA(E', \Psi(S \otimes S^\vee) X) = 0$.
This contradicts the fact that $E'$ is a non-zero subobject of $\Psi(S \otimes S^\vee) X$, hence the initial assumption $r> \phi(X)$ must be false.

$(2) \Rightarrow (3)$ is immediate.

We now show $(3) \Rightarrow (1)$.
Suppose $\Psi(C)A$ is semistable for some $C \neq 0 \in \cC$.
Consider an arbitrary exact sequence in $\cA$:
\[
0 \ra U \ra A \ra V \ra 0.
\]
Applying $\Psi(C)$, we get the exact sequence
\[
0 \ra \Psi(C)U \ra \Psi(C)A \ra \Psi(C)V \ra 0.
\]
Once again the semistable property of $\Psi(C)A$ tells us that
\[
\phi(\Psi(C)U) \leq \phi(\Psi(C)A),
\]
and therefore
$
\phi(U) \leq \phi(A).
$
\end{proof}

This allows us to obtain the following variant of \cref{thm: stab function and stab cond}:
\begin{theorem} \label{thm: stab function over C and stab cond respecting C}
Giving a stability condition respecting $\cC$ on $(\cD, \Psi)$ is equivalent to giving a heart $\cH$ of $\cD$ such that $(\cH, \Psi)$ is a $\cC$-module subcategory, together with a stability function over $\cC$ on $(\cH, \Psi)$ which satisfies the Harder-Narasimhan property.
\end{theorem}
\begin{proof}
The proof for the implication $(\Rightarrow)$ follows from the discussion before \cref{defn: stab function over C}, whereas the proof for $(\Leftarrow)$ follows from \cref{cC action preserve semistable}.
\end{proof}

\subsection{$q$-Stability conditions respecting $\cC$}
The main focus of our thesis will be on triangulated module $\X$-categories $(\cD, \X, \Psi)$ over $\cC$.
Hence $K_0(\cD)$ will be a module over $K_0(\cC)[q,q^{-1}]$ and we again extend $\C$ to a $K_0(\cC)[q,q^{-1}]$-module by evaluating $q = -1$.
As such, we can consider stability conditions $\tau$ which are both $q$-stability conditions and also respect $\cC$.
Collecting the two definitions, we have the following:
\begin{definition}[cf. \cref{defn: module q stab}, \cref{defn: stab respects C}]\label{defn: q stab respect C}
A \emph{$q$-stability condition $\tau = (\cZ, \cP)$ (with $s=-1$) that respects $\cC$} is a stability condition that satisfies:
\begin{enumerate}
\item $\cP(\phi)\cdot \X[1] = \cP(\phi)$;
\item $
\Psi(C)\cP(\phi) \subseteq \cP(\phi)
$ for all $C \in \cC$; and
\item $\cZ \in \Hom_{K_0(\cC)[q,q^{-1}]}(K_0(\cD), \C_{\cC,-1})$.
\end{enumerate}
\end{definition}

Note that given a $q$-stability condition $\tau$ respecting $\cC$, the standard heart of $\tau$ given by $\cH:= \cP[0,1)$ is now an abelian category closed under endofunctors $\Psi(C)$ for all $C$ in $\cC$ and also under the autoequivalence $\X[1]$.
Namely, $(\cH, \X[1], \Psi)$ is a $\cC$-module subcategory of $(\cD, \X, \Psi)$ equipped with the distinguished autoequivalence $\X[1]$.
Hence its Grothendieck group $K_0(\cH)$ has a module structure over $K_0(\cC)[x,x^{-1}]$ defined by 
\[
x \cdot [A] = [A\cdot \X[1]] \in K_0(\cH).
\]
Viewing $\C$ as a module over $K_0(\cC)[x,x^{-1}]$ by sending $x \mapsto 1$, we see that the induced stability function $Z: K_0(\cH) \ra \C$ on $\cH$ is a $K_0(\cC)[x,x^{-1}]$-module homomorphism.
Similar to \cref{defn: stab function over C}, we make the following definition:
\begin{definition} \label{defn: x[1] stab function}
Let $\cA$ be an abelian category and $(\cA, \Y, \Phi)$ be an (abelian) module $\Y$-category over $\cC$.
An \emph{$\Y$-stability function over $\cC$} on $(\cA, \Y, \Phi)$ is a stability function $Z: K_0(\cA) \ra \C$  satisfying the extra condition that $Z$ is a $K_0(\cC)[x,x^{-1}]$-module homomorphism, where 
\begin{enumerate}
\item the abelian Grothendieck group $K_0(\cA)$ is a $K_0(\cC)[x,x^{-1}]$-module induced by the $\cC$-module structure on $(\cA, \Y, \Phi)$ together with $x\cdot [A] = [A\cdot \Y] \in K_0(\cA)$; and
\item $\C$ is a $K_0(\cC)[x,x^{-1}]$-module induced by $PF\dim: K_0(\cC) \ra \R \subset \C$ together with the evaluation $x \mapsto 1$.
\end{enumerate}
\end{definition}

\comment{
\begin{definition}
Let $\cH$ be a heart of some $t$-structure on $(\cD_\X, \Psi)$, such that $(\cH|_{\X[1]}, \Psi|_\cH)$ is a $\cC$-submodule category equipped with the distinguished autoequivalence $\X[1]$.
An \emph{$\X[1]$-stability function over $\cC$} on $(\cH|_{\X[1]}, \Psi|_\cH)$ is a stability function $Z: K_0(\cH) \ra \C$ on $\cH$ satisfying the extra condition that $Z$ is a $K_0(\cC)[x,x^{-1}]$-module homomorphism, where 
\begin{enumerate}
\item $K_0(\cH)$ is a $K_0(\cC)[x,x^{-1}]$-module induced by the $\cC$-module structure on $\cH_{\X[1]}$ together with $x\cdot [A] = [A\cdot \X[1]] \in K_0(\cH)$; and
\item $\C$ is a $K_0(\cC)[x,x^{-1}]$-module induced by $PF\dim: K_0(\cC) \ra \R \subset \C$ together with the evaluation $x = 1$.
\end{enumerate}
\end{definition}
}

It follows from the definition of a $\Y$-stability function over $\cC$ that $\Y$ and $\Phi(C)$ preserve the phase of objects in $\cA$ for all $C \in \cC$.
Since $\Y$ is an autoequivalence, it follows that $\Y$ must preserve (and reflect) semistable objects.
As such, we can easily extend \cref{thm: stab function over C and stab cond respecting C} to obtain the following variant of \cref{thm: stab function and stab cond}:
\comment{
On the other hand, $\Phi(C)$ is just an endofunctor, so one might not expect this property to hold.
However, it turns out that this property for $\Phi(C)$ is an implication of the fact that every object in $\cC$, a fusion category, is dualisable (as shown below).
\begin{proposition} \label{linear shift cC action semistable}
Let $(\cA, \Y, \Phi)$ be a $\cC$-module abelian $\Y$-category equipped with $Z: K_0(\cA) \ra \C$ a $\Y$-stability function over $\cC$.
Then for any object $A$ in $\cA$, the following are equivalent:
\begin{enumerate}
\item $A$ is semi-stable;
\item $\Phi(C)A \cdot \Y^m$ is semistable for all $C \neq 0 \in \cC$ and all $m\in \Z$; and
\item $\Phi(C)A \cdot \Y^m$ is semistable for some $C \neq 0 \in \cC$ and some $m\in \Z$.
\end{enumerate}
\end{proposition}
\begin{proof}
Note that $Z$ being an $K_0(\cC)[x,x^{-1}]$-module map dictates that 
\[
\phi \left(\Psi(C)A\cdot \Y^m \right) = \phi(A)
\]
for any object $A \in \cA$, any $m\in \Z$, and any $C \in \cC$.

\comment{
Furthermore, $\cdot \Y^m$ is an autoequivalence and $-\otimes \Pi_a$ is a self biadjoint functor, hence
\begin{equation}
\Hom_\cH(U,X\otimes \Pi_a \cdot \Y^m) \cong \Hom_\cH(U\otimes \Pi_a \<-m\>[m], X)
\end{equation}
and similarly
\begin{equation}
\Hom_\cH(U,X\otimes \Pi_a \cdot \Y^m[m]) \cong \Hom_\cH(U\otimes \Pi_a \<-m\>[m], X)
\end{equation}
for any object $U\in \cH$.
}

Let us first prove $(1) \Rightarrow (2)$.
Suppose $A$ is semistable and let $C \neq 0$ and $m$ be arbitrary.
Let
\[
0 \ra U \ra \Psi(C)A\cdot \Y^m \ra V \ra 0
\]
be an arbitrary exact sequence in $\cA$.
Recall that since $\cC$ is a fusion category, all objects in $\cC$ have duals.
So let $C^\vee$ be the (left) dual of $C$, so that there is a non-zero map $C^\vee \otimes C \ra \1_\cC$ as the counit of the adjunction.
As such, since $\cC$ is semisimple, this means that $C^\vee \otimes C$ contains $\1_\cC$ as a summand, namely
\[
C\otimes C^\vee \cong \1_\cC \oplus D
\]
for some $D \in \cC$.
Applying the exact functors $\Psi(C^\vee)$ and $\Y^{-m}$, we get the exact sequence
\[
0 \ra \Psi(C^\vee)U \cdot \Y^{-m} \ra \Psi(C^\vee \otimes C) A \ra \Psi(C)V \cdot \Y^{-m} \ra 0.
\]
Thus, we have an injection 
\[
\Psi(C^\vee)U\cdot \Y^{-m} \hookrightarrow \Psi(\1_\cC)A = A.
\]
By the assumption that $A$ is semistable, $\phi(\Psi(C^\vee)U\cdot \Y^{-m}) \leq \phi(A)$.
But  since
\[
\phi(U) = \phi(\Psi(C^\vee)U\cdot \Y^{-m}), \quad \phi(A) = \phi(\Psi(C)A\cdot \Y^m),
\]
we get
\[
\phi(U) \leq \phi(\Psi(C)A\cdot \Y^m)
\]
as required.

$(2) \Rightarrow (3)$ is immediate.

We shall show $(3) \Rightarrow (1)$.
Suppose $\Psi(C)A\cdot \Y^m[m]$ is semistable for some $C \neq 0 \in \cC$ and some $m\in \Z$.
Consider an arbitrary exact sequence in $\cA$:
\[
0 \ra U \ra A \ra V \ra 0.
\]
Applying $\Psi(C)$ and $\Y^m$, we get the exact sequence
\[
0 \ra \Psi(C)U\cdot \Y^m \ra \Psi(C)A\cdot \Y^m \ra \Psi(C)V\cdot \Y^m \ra 0.
\]
Once again the semi-stability of $\Psi(C)A\cdot \Y^m$ tells us that
\[
\phi(\Psi(C)U\cdot \Y^m) \leq \phi(\Psi(C)A\cdot \Y^m),
\]
and therefore
$
\phi(U) \leq \phi(A).
$
\end{proof}
}
\begin{theorem} \label{thm: x[1] stab function and q stab cond respecting C}
Giving a $q$-stability condition (with $s=-1$) respecting $\cC$ on $(\cD, \X, \Psi)$ is equivalent to giving a heart $\cH$ of $\cD$ such that $(\cH, \X[1], \Psi)$ is a $\cC$-module subcategory equipped with the distinguished autoequivalence $\Y:= \X[1]$, together with a $(\X[1])$-stability function over $\cC$ on $(\cH, \X[1], \Psi)$ which satisfies the Harder-Narasimhan property.
\end{theorem}
\comment{
\begin{proof}
The proof for the implication $(\Rightarrow)$ is as discussed before \cref{defn: x[1] stab function}, whereas the proof for $(\Leftarrow)$ follows from \cref{linear shift cC action semistable}.
\end{proof}
\begin{remark}
A similar result of \cref{thm: x[1] stab function and q stab cond respecting C} for stability conditions respecting $\cC$ on triangulated module category $(\cD, \Psi)$ over $\cC$ also holds, where one drops the ``$\X[1]$'' condition.
\end{remark}
}

\comment{
Moreover, the associated stability function on $\cH$ m

As with the usual stability conditions, we would like to construct $q$-stability conditions respecting $\cC$ by specifying stability functions on hearts (which also satisfy the Harder-Narasimhan property).

Let $\cH \subseteq \cD$ be a heart of some $t$-structure.
Suppose $\cH$ is a submodule category over $\cC$ and is closed under the shifts $\<1\>[1]$. 
Then $K_0(\cH)$ is a module over $K_0(\cC)[x,x^{-1}]$ with $x\cdot[A] = [A\<1\>[1]]$.
Recall that we can endow $\C$ with a module structure over $K_0(\cC)$ using $PF\dim: K_0(\cC) \ra \R \subset \C$.
We shall now extend $\C$ to a module over $K_0(\cC)[x,x^{-1}]$ by evaluating $x = 1$.
As for $\cD_\X$, one obvious requirement for the stability function 
\[
Z: K_0(\cH) \ra \C
\]
is that $Z$ must be a $K_0(\cC)[x,x^{-1}]$-module homomorphism, so that the induced central charge $\cZ$ on $\cD$ will be a $K_0(\cC)[q,q^{-1}]$-module homomorphism (exercise).
We are now left to show that the functors $\<1\>[1]$ and $\Psi(C)$ for all $C \in \cC$ sends semistable objects to semistable objects (of the same phase).
This fact follows easily for exact autoequivalences (hence for $\<1\>[1]$), but $\Psi(C)$ are in general \emph{not} autoequivalences.
Nevertheless, the following lemma gives us a stronger statement:
}

\section{$\tau$-Mass automaton} \label{sect: mass automaton}
We end this chapter by introducing the notion of a $\tau$-mass automaton in this section.

\comment{
In the main body of the thesis, we concern ourselves with the dynamics of endofunctors and in particular, their mass growth.
This means that for an endofunctor $\cF: \cD \ra \cD$ with a stability condition $\tau$ on $\cD$, we would like to compute the following limit:
\[
\lim_{N \ra \infty} m_{\tau, t}( \cF^N(X))
\]
for some given object $X \in \cD$.
}
Let us start with a short discussion.
Let $G$ be a monoid (this will be a submonoid of $\End(\cD)$ or a subgroup of $\Aut(\cD)$).
We shall denote the category with one object $*$ and morphism given by $G$ (a monoid-oid!) as $\underline{G}$.
Let $\Theta$ be a quiver (directed graph) with arrows labelled by elements in $G$.
With $\widehat{\Theta}$ denoting the free category generated by the quiver $\Theta$, we have a functor $\widehat{\Theta} \xra{\Theta} \underline{G}$, sending each labelled arrow to its corresponding element in $G$.
As such, we shall not differentiate between ``a functor from a category $\Theta$ to $\underline{G}$'' and ``a $G$-labelled quiver $\Theta$'', which we all call \emph{$G$-automata}.
We will also abuse notation and use $\Theta$ to denote both the underlying $G$-labelled quiver and the functor from the category to $\underline{G}$.

Given a $G$-automaton $\Theta$, we say that an expression (or word) $\overline{g} = w_k w_{k-1} ... w_1$ of an element $g \in G$ is \emph{recognised by $\Theta$} if there is a path $p = (e_1, e_2, ..., e_k)$ in $\Theta$ such that each arrow $e_i$ is labelled by $w_i \in G$ (note that our path reads from left to right).

A $G$-set is equivalent to a functor $\underline{G}\ra $Sets, and we similarly define a $\Theta$-set as a functor $\Theta \ra $Sets, and a $\Theta$-representation as a functor $\Theta \ra R$-mod for some ring $R$.
With a fix $G$-automaton, every $G$-set can be considered as a $\Theta$-set through the composition $\Theta \ra \underline{G} \ra $Sets.
When $G \subset \End(\cD)$, then ob$(\cD)$ is naturally a $G$-set, and hence a $\Theta$-set.

We are now ready to give the definition of a mass automaton:
\begin{definition}
Fix some $G \subset \End(\cD)$.
A \emph{$\tau$-mass automaton} of a triangulated category $\cD$ over $G$ is a $G$-automaton $\Theta$ together with:
\begin{enumerate}
\item a $\Theta$-subset $\cS$ of ob$(\cD)\circ \Theta$,
\item a $\Theta$-representation $\cM : \Theta \ra R$-mod,
\item a natural transformation $\iota : \cS \ra \cM$ with $\cM$ viewed as a $\Theta$-set, and
\item a group homomorphism $\mathfrak{m}_v : \cM(v) \ra \R^\R$ for each vertex $v$,
\end{enumerate} 
satisfying the condition
\[
m_{\tau,t}(X) = \mathfrak{m}_v(\iota_v(X))
\]
for all vertices $v$ and all $X \in \cS(v)$.
\end{definition} 

\begin{remark}
The closely related notion of a $\tau$-HN automaton is defined in a similar way, where for each vertex $v$ the homomorphism $\mathfrak{m}_v: \cM(v) \ra \R^\R$ is replaced with a homomorphism $HN_v : \cM(v) \ra \Z^\Sigma$  which satisfies
\[
HN_\tau(X) = HN_v(\iota_v(X)
\]
for all $X \in \cS(v)$, with $\Sigma \subset $ ob$(\cD)$ denoting the set of indecomposable semistable objects and $HN_\tau(X)$ is the multiplicity vector of indecomposable semistable objects in the $\tau$-HN filtration of $X$.
\end{remark}

Given a path $p = (e_1, e_2, ..., e_k)$ in a $\tau$-mass automaton, we call 
\[
\cM(p) := \cM(e_k)\cM(e_{k-1})...\cM(e_1)
\]
\emph{the $\tau$-mass matrix associated to $p$} and $\cS(p) := \cS(e_k)\cS(e_{k-1})...\cS(e_1)$ \emph{the endofunctor associated to $p$}.

Unpacking the definition, a $\tau$-mass automaton over $G$ can be thought of as a quiver with ``a $G$-labelling'' and a ``$R$-mod labelling'', namely
\begin{enumerate}
\item each vertex $v$ is labelled by a tuple $(\cS(v), \cM(v))$, with $\cS(v)$ a subset ob($\cD$) and $\cM(v)$ some $R$-module;
\item each arrow $a: v \ra v'$ is labelled by a tuple $(\cS(a), \cM(a))$, with $\cS(a) \in G$ such that $\cS(a) \cdot \cS(v) \subset \cS(v')$, and $\cM(a)$ is a $R$-linear homomorphism from $\cM(v)$ to $\cM(v')$.
\end{enumerate}
Each vertex $v$ comes with a map $\iota_v : \cS(v) \ra \cM(v)$ which captures a compatibility condition between the two labellings:
to every arrow $a: v \ra v'$, we require the following diagram to commute:
\begin{equation} \label{eqn: S M compatible}
\begin{tikzcd}
\cS(v) \ar[r, "\iota_{v}"] \ar[d, swap, "\cS(a)"] 
	& \cM(v) \ar[d, "\cM(a)"] \\
\cS(v') \ar[r, "\iota_{v'}"] 
	&\cM(v').
\end{tikzcd}
\end{equation}
Each vertex $v$ also comes with a homomorphism $\mathfrak{m}_v: \cM(v) \ra \R^\R$ which ``recovers'' the mass of the objects in  $\cS(v)$:
\begin{equation} \label{eqn: recovers mass}
m_{\tau, t}(X) = \mathfrak{m}_v(\iota_v(X))
, \quad \text{for all } X \in \cS(v).
\end{equation}

The definition of a $\tau$-mass automaton is undoubtedly complicated; on the flip side it will prove itself as a very powerful tool, where the existence of ``good'' $\tau$-mass automata provide means of studying categorical dynamics for a large class of endofunctors.
In particular, notice that \cref{eqn: S M compatible} and \cref{eqn: recovers mass} allow us to deduce the following: if $p$ is a closed path based at a vertex $v$ in the automaton, the mass growth of the endofunctor $\cS(p)$ with respect to any object $X$ in $\cS(v)$ can be understood via the linear map $\cM(p)$:
\begin{equation} \label{eqn: closed path transform linearly}
m_{\tau, t}\left( \cS(p)^N(X) \right) = \mathfrak{m}_v\left( \cM(p)^N(\iota_v(X)) \right).
\end{equation}
This will be crucial in the main theorem of this thesis in Chapter 4.
We shall provide some examples of mass automata below, where we will construct more in \cref{sect: mass automaton construction}.
\begin{example}
Some examples of mass automata associated to $A_2$ and $\hat{A_1}$ can be found in \cite{bapat2020thurston}.
One of the mass automaton associated to $A_2$ in said paper will be similar to the mass automaton $\Lambda_{\tau_R}(n)$ for $n=3$ that we will be constructing in \cref{sect: mass automaton construction}.
\end{example}

\begin{example}
Let $\cC:= \TLJ_5$ and let $\tau$ denote the stability condition on $D^b(\cC)$ defined in \cref{eg: TLJ_5 stab respect C}.
Let $H^*: D^b(\cC) \ra g\cC$ denote the cohomological functor (of the standard heart) defined by
\[
H^*(X^\bullet) = \bigoplus_{i \in \Z} H^{-i}(X^\bullet),
\] 
where each $H^{-i}(X^\bullet)$ is of homogeneous grading $-i$.
Note that the (abelian) Grothendieck group $K_0(g\cC)$ is isomorphic to $K_0(\cC)[q^{\pm 1}]$, where $q\cdot [A] = [A\<1\>]$.
Let $G \subset \End(D^b(\cC))$ be the monoid of endofunctors induced by the $\cC$-module structure on $D^b(\cC)$.
We shall construct a $\tau$-mass automaton of $D^b(\cC)$ over $G$ as follows. \\
\noindent
The underlying $G$-automaton $\Theta$ is given by the $G$-labelled quiver with one vertex $v$ and four loop arrows $a_i: v \ra v$ labelled by $\Pi_i$ for each $0 \leq i \leq 3$.
The $\Theta$-subset $\cS$ is exactly the $\Theta$-set ob$(D^b(\cC))\circ \Theta$ itself.
The $\Theta$-representation $\cM: \Theta \ra \Z[q^{\pm 1}]$-mod is given by:
\begin{enumerate}[(i)]
\item $\cM(v) = K_0(g\cC)$, which is a free $\Z[q^{\pm 1}]$-module with basis set $\{[\Pi_0], [\Pi_1], [\Pi_2], [\Pi_3]\}$.
\item $\cM(a_i) = K_0(\Pi_i \otimes -)$ for each $i$, which can be expressed as a matrix with respect to the basis above:
\begin{align*}
\cM(a_0) = I_4 ,
\cM(a_1) = \begin{bmatrix}
	0 & 1 & 0 & 0 \\
	1 & 0 & 1 & 0 \\
	0 & 1 & 0 & 1 \\
	0 & 0 & 1 & 0
	\end{bmatrix},
\cM(a_2) = \begin{bmatrix}
	0 & 0 & 1 & 0 \\
	0 & 1 & 0 & 1 \\
	1 & 0 & 1 & 0 \\
	0 & 1 & 0 & 0
	\end{bmatrix},
\cM(a_3) = \begin{bmatrix}
	0 & 0 & 0 & 1 \\
	0 & 0 & 1 & 0 \\
	0 & 1 & 0 & 0 \\
	1 & 0 & 0 & 0
	\end{bmatrix}.
\end{align*}
\end{enumerate}
The natural transformation $\iota: \cS \ra \cM$ is defined by $\iota_v(X^\bullet) = [H^*(X^\bullet)] \in K_0(g\cC)$ for each $X^\bullet \in \cS(v) =$ ob$(D^b(\cC))$.
The group homomorphism $\mathfrak{m}_v: K_0(g\cC) \ra \R^\R$ is defined by
\[
q \mapsto e^t, \quad [\Pi_i] \mapsto PF\dim(\Pi_i).
\]
The condition that
\[
m_{\tau, t}(X^\bullet) = \mathfrak{m}_v(\iota_v(X^\bullet))
, \quad \text{for all } X \in \cS(v) = \text{ob}(D^b(\cC))
\]
is easy to check.
\comment{
By construction, every path in this automaton is a closed path.
Although this is a bit of an overkill, using \cref{eqn: closed path transform linearly} we see that
\[
h_{\tau,t}(C \otimes -) = PF\dim(C)
\]
for each $C \in \cC$.
}
\end{example}

\comment{
We will introduce the notion of supports here, which will be our main tool in constructing mass automaton later on.
\begin{definition}
We denote $\cS\cS^\oplus_\tau$ as the smallest additive subcategory of $\cD$ containing all the $\tau$-semistable objects.
We call $\cS\cS^\oplus_{\tau}$ \emph{the category of $\tau$-supports of $\cD$}. \\
For $A$ an object in $\cD$ with $\tau$-HN filtration consisting of $\tau$-semistable pieces $E_i$, we define the \emph{$\tau$-support of} $A$ as
\[
\supp_\tau(A) := \bigoplus_i E_i \in ob(\cS\cS^\oplus_\tau).
\]
\end{definition}
\begin{remark}
We remind the reader that the $\tau$-HN filtration of an object is only unique up to isomorphism, and so the $\tau$-support of an object is defined as an element in the set of isomorphism classes.
\end{remark}

The following proposition showcases a few important properties of $\tau$-support
\begin{proposition} \label{support properties}
Let $\tau$ be a stability condition on $\cD$.
Then the $\tau$-support satisfies the following properties:
\begin{enumerate}
\item (additivity) $\supp_\tau(A\oplus B) = \supp_\tau(A) \oplus \supp_\tau(B)$;
\item (triangulated shift commuting) $\supp_\tau(A[m]) = \supp_\tau(A)[m]$ for all $m \in \Z$; \\
\end{enumerate}
If $\tau$ is a $q$-stability condition on $(\cD, \X,\Psi)$ respecting $\cC$, we have in addition to the above:
\begin{enumerate}
\item ($\X$ commuting) $\supp_\tau(A\cdot \X^m) = \supp_\tau(A)\cdot \X^m$ for all $m \in \Z$; and
\item ($\cC$-action commuting) $\supp_\tau(\Psi(C)(A)) = \Psi(C)\supp_\tau(A)$ for all $C \in \cC$.
\end{enumerate} 
\end{proposition}
\begin{proof}
The additivity property follows from the fact that the HN filtration of $A\oplus B$ is just the direct sum of the HN filtrations of $A$ and $B$.

For the triangulated shift commuting property, let us start by taking the HN filtration of $A$:
\[
\begin{tikzcd}[column sep = 3mm]
0
	\ar[rr] &
{}
	{} &
A_1
	\ar[rr] \ar[dl]&
{}
	{} &	
A_2
	\ar[rr] \ar[dl]&
{}
	{} &	
\cdots
	\ar[rr] &
{}
	{} &
A_{m-1}
	\ar[rr] &
{}
	{} &
A_m = A
	\ar[dl] \\
{}
	{} &
E_1
	\ar[lu, dashed] &
{}
	{} &
E_2
	\ar[lu, dashed] &
{}
	{} &
{}
	{} &
{}
	{} &
{}
	{} &
{}
	{} &
E_m
	\ar[lu, dashed] &
\end{tikzcd}.
\]
By definition of stability conditions, we see that $E_i[m]$ are still semistable and 
\[
\phi(E_i[m]) = \phi(E_i)+m > \phi(E_{i+1}) +m = \phi(E_{i+1}[m]).
\]
Hence applying $[m]$ to the HN filtration of $A$ results in the HN filtration of $A[m]$.
In particular, 
\[
\supp_\tau(A[m]) = \bigoplus_i E_i[m] = \supp_\tau(A)[m]
\]
as required.

The $\X$ and $\cC$ commuting property follows similarly, where we use the fact that $\tau$ is a $q$-stability condition respecting $\cC$.
\end{proof}

Using a similar proof to \cref{support properties} and looking at how the the image of the central charge $\cZ$ changes, we can deduce the following  similar properties of mass:
\begin{proposition}\label{mass properties}
Let $\tau$ be a stability condition on $\cD$.
Then the $\tau$-mass satisfies the following properties:
\begin{enumerate}
\item (additivity) $m_{\tau,t}(A\oplus B) = m_{\tau,t}(A) + m_{\tau,t}(B)$;
\item (triangulated shift commuting) $m_{\tau,t}(A[m]) = m_{\tau,t}(A)\cdot e^{mt}$ for all $m \in \Z$.
\end{enumerate}
If  $\tau$ is a $q$-stability condition on $(\cD, \X,\Psi)$ respecting $\cC$, then we have in addition to the above:
\begin{enumerate}
\item ($\X$ commuting) $m_{\tau,t}(A\cdot \X^m) = m_{\tau,t}(A)\cdot e^{-mt}$ for all $m \in \Z$; and
\item ($\cC$-action commuting) $m_{\tau,t}(\Psi(C)(A)) = PF\dim(C)m_{\tau,t}(A)$ for all $C \in \cC$.
\end{enumerate} 
\end{proposition}

It should come as no surprise that $\tau$-support and mass are related.
In particular, we have that by definition of $\tau$-mass,
\[
m_{\tau,t}(A) = m_{\tau,t}(\supp_\tau(A)).
\]
Moreover, any two isomorphic objects have the same mass, so we can deduce that
\[
\supp_\tau(A) = \supp_\tau(B) \Rightarrow m_{\tau,t}(A) = m_{\tau,t}(B).
\]
This means that computing mass of an object amounts to computing mass of its support.
As such, we can expect that computing the mass of an object will be sufficient at the level of the split Grothendieck group $K_0(\cS\cS^\oplus_\tau)$.

Before we do that, let us understand the category of supports a little more.
Firstly, it is clear from definition that $\cS\cS^\oplus_\tau$ decomposes as $\bigoplus_{\phi \in \R} \cP(\phi)$.
The uniqueness of HN filtrations implies that each $\cP(\phi)$ is actually a thick subcategory (closed under taking summands), and hence $\cS\cS^\oplus_\tau$ must also be a thick subcategory.

By definition of stability conditions, $\cS\cS^\oplus_\tau$ is closed under the triangulated shifts $[-]$.
As such, the split Grothendieck group $K_0(\cS\cS^\oplus_\tau)$ is naturally a $\Z[s,s^{-1}]$-module, with $s$-action 
\[
s\cdot[A] = [A[1]]\in K_0(\cS\cS_\tau).
\]
We warn the reader that $[A[1]] \neq -[A]$ in the split Grothendieck group $K_0(\cS\cS_\tau)$.

When $\tau$ is a $q$-stability condition on $(\cD, \X, \Psi)$ respecting $\cC$, we see that $(\cS\cS^\oplus_\tau, \Psi, \X)$ is also an additive $\cC$-module $\X$-category.
As such, the split Grothendieck group $K_0(\cS\cS^\oplus_\tau)$ is instead a module over $K_0(\cC)[q^{\pm 1}, s^{\pm 1}]$, with 
\[
[C] \cdot [A] = [\Psi(C)(A)]\in K_0(\cS\cS_\tau)
\]
for all $C \in g\cC$, with $q$-action
\[
q \cdot [A] = [A \cdot \X]\in K_0(\cS\cS_\tau).
\]

Recall that the mass of object is defined to be an element in $\R^\R$.
Motivated by \cref{mass properties}, let us view $\R^\R$ as a module over $\Z[s^{\pm 1}]$ by evaluating $s = e^t$.
Similarly, if one works over $q$-stability conditions respecting $\cC$, we define the ring homomorphism
\[
\overline{PF\dim}: K_0(\cC)[q^{\pm 1}, s^{\pm 1}] \ra \R^\R
\]
as the extension of $PF\dim: K_0(\cC) \ra \R$ by sending $q \mapsto e^{-t}$ and $s \mapsto e^t$, which provides $\R^\R$ with a $K_0(\cC)[q^{\pm 1}, s^{\pm 1}]$-module structure.
With this, we have the following result:
\begin{proposition} \label{prop: mass from support}
The map $\wt{m}_{\tau, t}: K_0(\cS\cS^\oplus_\tau) \ra \\R^\R$ defined by
\[
[A] \mapsto m_{\tau,t}(A)
\]
is a well-defined module homomorphism over $\Z[s,s^{-1}]$ (over $K_0(\cC)[q^{\pm 1}, s^{\pm 1}]$ when working with $q$-stability condition respecting $\cC$).
Moreover, we have that
\[
m_{\tau,t}(X) = \wt{m}_{\tau, t}([\supp_\tau(X)])
\]
for any object $X$ in $\cD$.
\end{proposition}

Finally, let $\cF$ be an endofunctor on $\cD$.
In general, we can not expect $\cF$ to restrict to an endofunctor on $\cS\cS^\oplus_\tau$, but on the split Grothendieck group $K_0(\cS\cS^\oplus_\tau)$ this is possible by taking the support:
\begin{proposition} \label{prop: induced support map on Grothendieck}
The map $[\supp_\tau \circ \cF] : K_0(\cS\cS^\oplus_\tau) \ra K_0(\cS\cS^\oplus_\tau)$ defined by
\[
[A] \mapsto [\supp_\tau(\cF|_{\cS\cS^\oplus_\tau} (A))]
\]
is a well-defined $\Z[s^{\pm 1}]$-module homomorphism.
When working over $(\cD, \X, \Psi)$ with $\tau$ a $q$-stability condition respecting $\cC$, the map is a $K_0(\cC)[q^{\pm 1}, s^{\pm 1}]$-module homomorphism.
\end{proposition}
\begin{remark}
Note that we are not claiming that there is a endofunctor $\supp_\tau\circ \cF$ defined on $\cS\cS^\oplus_\tau$ (the support does not take morphism to morphism), but merely a well-defined map on the split Grothedieck group.
\end{remark}
{\color{red}
Is the explanation below useful?
}

Although this is jumping ahead of ourself, it might be worth briefly describing how these notions allow us to construct mass automata.
Let $G \subset \End(\cD)$, $\Theta$ be some $G$-automaton and $\cS$ be some well-defined $\Theta$-subset.
To each vertex $v$, we let $\cM(v)$ be some $\Z[s^{\pm 1}]$-submodule of $K_0(\cS\cS^\oplus_\tau)$ (or submodule over $K_0(\cC)[q^{\pm 1}, s^{\pm 1}]$ when working with triangulated $\X$-categories), which contains $[\supp(X)]$ for all $X \in \cS(v)$.
This ensures that we have a well-defined map 
\[
\iota_v = [\supp_\tau(-)] : \cS(v) \ra \cM(v)
\]
for each vertex $v$.
Following \cref{prop: induced support map on Grothendieck}, for each arrow $a: v \ra v'$ we can define $\cM(a): \cM(v) \ra \cM(v')$ by 
\[
\cM(a) = [\supp_\tau\circ \cS(a)],
\]
which does map into $M(v')$ by the definition of $\cM(v')$ (and by the fact that $\cS(a)$ is a $\Theta$-subset).
This gives a well-defined functor $\cM$ from the category $\Theta$ to the category of modules.
Defining $\mathfrak{m}_v = \wt{m}_{\tau,t}$, the condition $\mathfrak{m}_v(\iota_v(X)) = m_{\tau,t}(X)$ comes for free following \cref{prop: mass from support}.
Up to now everything seems to work in general, with the exception that we need $\iota_v$ to be a natural transformation!
This means that for every arrow $a : v \ra v'$, we need
\[
\cM(a)\circ \iota_v = \iota_{v'} \circ \cS(a),
\]
which amounts to checking
\[
[(\supp_\tau \circ \cS(a) \circ \supp_\tau)(X)] = [(\supp_\tau \circ \cS(a))(X)].
\]
for each $X \in \cS(v)$.
This is not true in general and is the main non-trivial condition that needs to be check in this whole construction.
}

\comment{
Let $\cV \subseteq $ ob$(\cS\cS_\tau)$.
We shall use $\<\cV \>_{\cS\cS_\tau}$ to denote the smallest thick additive $\cC$-module subcategory of $\cS\cS_\tau$, closed under the two autoequivalences $\X$ and $[1]$. 
In particular, the split Grothendieck group $K_0(\<\cV \>_{\cS\cS_\tau})$ will be a $K_0(\cC)[q^{\pm 1},s^{\pm 1}]$-submodule of $K_0(\cS\cS_\tau)$.
A particularly nice choice of $\cV$ is the following:
\begin{definition}
Let $\cV \subseteq $ ob$(\cS\cS_\tau)$.
We say that $\cV$ is \emph{free over $(\cC,\X,[1])$} if $K_0(\<\cV \>_{\cS\cS_\tau})$ is a free $K_0(\cC)[q^{\pm 1}, s^{\pm 1}]$-submodule of $K_0(\cS\cS_\tau)$, with basis set given by the classes of the objects in $\cV$.
\end{definition}
{\color{red}
NOT TRUE:\\
Note that $\cV$ being free over $(\cC, \X, [1])$ implies that each object $V \in \cV$ is indecomposable.
}

Recall the ring homomorphism $PF\dim: K_0(\cC) \ra \R$ induced by the Perron-Frobenius dimension (see \cref{sect: fusion}).
We shall extend it to a ring homomorphism
\[
\overline{PF\dim}: K_0(\cC)[q^{\pm 1},s^{\pm 1}] \ra \R[e^t, e^{-t}]
\]
by sending $q \mapsto e^{-t}, s \mapsto e^t$.
As such, we make the following definition:
\begin{definition}
Let $\cV = \{V_j \}_{j=1}^m \subseteq $ ob$(\cS\cS_\tau)$ be free over $(\cC, \X, [1])$.
We say that $A \in ob(\cD)$ is \emph{$\tau$-supported} by $\cV$ if we have $\supp_\tau(A) \in ob(\<\cV \>_{\cS\cS_\tau})$.
When so, we have a unique expression
\[
[\supp_\tau(A)] = \sum_j a_j[V_j] \in  K_0(\<\cV \>_{\cS\cS_\tau}),
\]
with $a_j \in K_0(\cC)[q^{\pm 1},s^{\pm 1}]$ and we define the \emph{$\tau$-HN multiplicity vector of $A$ with respect to $\cV$} to be the vector given by
\[
HN^\cV_\tau(A) := 
\begin{bmatrix}
\overline{PF\dim}(a_1) \\
\overline{PF\dim}(a_2) \\
\vdots \\
\overline{PF\dim}(a_m)
\end{bmatrix} \in \R[e^t, e^{-t}]^{\oplus m}.
\]
\end{definition}
In particular, we have that by definition,
\[
HN^\cV_\tau(V_j) = e_j
\]
with $e_j$ denoting the vector that has 1 on the $j$th entry and 0 elsewhere.

Note that the mass of an object can be easily computed from its HN multiplicity vector as follows:
let $u = (u_j)$ be the vector defined by by $u_j =  m_{\tau, t}(V_j)$ for each $j$.
Then
\[
m_{\tau,t}(A) = u^T HN^\cV_\tau(A).
\]

We summarise some simple properties of HN multiplicity vector below, which follows directly from the properties of $\tau$-support:
\begin{proposition} \label{HN multiplicity properties}
The HN-multiplicity vector satisfies the following properties:
\begin{enumerate}
\item $HN_\tau^\cV(A\oplus B) = HN_\tau^\cV(A) + HN_\tau^\cV(B)$;
\item $HN_\tau^\cV(A[m]) = e^{mt}HN_\tau^\cV(A)$ for all $m \in \Z$;
\item $HN_\tau^\cV(A\cdot \X^m) = e^{-mt}HN_\tau^\cV(A)$ for all $m \in \Z$; and
\item $HN_\tau^\cV(\Psi(C)(A)) = PF\dim(C)(HN_\tau^\cV(A))$ for all $C \in \cC$.
\end{enumerate} 
\end{proposition}

\begin{definition} \label{defn: HN matrix}
Let $\cV = \{V_j\}$ and $\cV'=\{V'_{j'}\}$ be finite subsets of ob$(\cS\cS_\tau)$ that are free over $(\cC, \X, [1])$.
Suppose $\cF: (\cD, \X, \Psi) \ra (\cD, \X, \Psi)$ is an endofunctor such that 
\[
\supp(\cF(V_j)) \in ob(\<\cV'\>^\oplus_\tau)
\]
for all $V_j \in \cV$.
We define the \emph{HN matrix of $\cF$ with respect to $\cV \ra \cV'$} as the $\R[e^t, e^{-t}]$-linear matrix 
\[
M_\cF^{\cV \ra \cV'}: K_0(\< \cV \>^\oplus_\tau) \ra K_0(\< \cV' \>^\oplus_\tau),
\] 
uniquely defined by
\[
M_\cF^{\cV \ra \cV'}(HN^\cV_\tau(V_j)) = M_\cV^{\cV \ra \cV'}(e_j) = HN^{\cV'}_\tau(\cF(V_j)).
\]
\end{definition}
}

\comment{
{\color{red}
The rest of this might not be necessary
}
Given an object $A$ with $\supp(A) \in ob(\<\cV\>^\oplus_\tau)$, it is then natural to ask whether it is possible to compute the HN multiplicity vector of $\cF(A)$ linearly with respect the HN matrix of $\cF$, namely is it true that the following equation holds:
\[
HN^{\cV'}_\tau(\cF(A)) = M_\cF^{\cV \ra \cV'}(HN^\cV_\tau(A)).
\]
The following property on $\cF$ and $A$ provides a simple criterion for this to hold (see \cref{HN matrix computes}):
\begin{definition}
Let $\cF: (\cD, \X,\Psi) \ra (\cD, \X, \Psi)$ be an endofunctor.
We say that $\cF$ is \emph{$\tau$-HN preserving on} $A$ if
\[
\supp_\tau(\cF(A)) = \supp_\tau \left(\cF \left( \supp_\tau(A) \right) \right).
\]
\end{definition}

Note that by definition, all endofunctors are $\tau$-HN preserving on any object in $\cS\cS_\tau$, in particular for any $\tau$-semistable object.

\begin{proposition} \label{HN matrix computes}
Let $\cV = \{V_j\}$ and $\cV'=\{V'_{j'}\}$ be finite subsets of ob$(\cS\cS_\tau)$ which are free over $(\cC, \X, [1])$.
Let $\cF: (\cD, \X,\Psi) \ra (\cD, \X, \Psi)$ be an endofunctor such that
\[
\supp_\tau(\cF(V_j)) \in ob(\<V'\>^\oplus_\tau)
\]
for all $V_j \in \cV$.
Let $A$ be an object in $\cD, \X$ with
\[
\supp_\tau(A) \in ob(\<V\>^\oplus_\tau).
\]
If $\cF$ is HN preserving on $A$, then
\[
HN^{\cV'}_\tau(\cF(A)) = M^{\cV \ra \cV'}_\cF (HN^\cV_\tau(A)).
\]
\end{proposition}
\begin{proof}
By the assumptiont that $\cF$ is HN preserving on $A$, we get
\[
\supp_\tau(\cF(A)) = \supp_\tau \left(\cF \left( \supp_\tau(A) \right) \right).
\]
Using the definition of HN multiplicity vector, we see that
\[
HN^{\cV'}_\tau(\cF(A)) = HN^{\cV'}_\tau( \cF(\supp(A))).
\]
Since $\supp(A)$ is an object in $\<\cV\>^\oplus_\tau$, the result follows from
the definition of $M^{\cV \ra \cV'}_\cF$.
\end{proof}
}

\comment{
\begin{proposition} \label{composition of HN matrix}
Let $\cV_1, \cV_2$ and $\cV_3$ be $\tau$-semistable basis categories over $g\cC$, and let $\cF_1, \cF_2$ be endofunctors of $\cD$ such that $\cF_i(V)$ is $\tau$ supported by $\cV_{i+1}$ for all $V \in \cV_i$.
Suppose $\cF_2$ is HN preserving on $\cF_1(B)$ for all $B \in \cV_1$. 
Then
\[
M^{\cV_1 \ra \cV_3}_{\cF_2\cF_1} = M^{\cV_2 \ra \cV_3}_{\cF_2} M^{\cV_1 \ra \cV_2}_{\cF_1}.
\]
\end{proposition}
\begin{proof}
Let $B \in \cB_{\cV_1}$ be a basis element.
Then
\[
M^{\cV_1 \ra \cV_3}_{\cF_2\cF_1}[B] = [\cF_2\cF_1(B)]_{\cV_3}
\]
by definition.
Since $\cF_1(B)$ is $\tau$ supported by $\cV_2$, we have that
\[
[\cF_2\cF_1(B)]_{\cV_3} = M^{\cV_2 \ra \cV_3}_{\cF_2} [\cF_1(B)]_{\cV_2}
\]
using \cref{HN matrix computes}.
The result now follows from the definition of $M^{\cV_1 \ra \cV_2}_{\cF_1}$.
\end{proof}
}

\comment{
Let $\cF: \cD \ra \cD$ be an endofunctor such that $\cF$ is HN preserving on all objects $\tau$ supported by $\cV$ and $\cF(V)$ is again $\tau$ supported by $\cV$ for all $V \in \cV$.
By \cref{HN matrix computes}, we get that
\[
[\cF(A)]_{\cV} = M^{\cV \ra \cV}_\cF[A]_\cV.
\]
for any object $A$ $\tau$ supported by $\cV$.
In particular, since $\cF(A)$ is again $\tau$ supported by $\cV$, we get that
\[
[\cF^N(A)]_{\cV} = \left( M^{\cV \ra \cV}_\cF \right)^N [A]_\cV
\]
for all $N \in \mathbb{N}$.
As before, we shall use $\overline{M_\cF^{\cV \ra \cV'}}$ to denote $M_\cF^{\cV \ra \cV'}$ with $\wt{PF\dim}$ applied to each entry.
This combined with \cref{vector to mass} gives us
\begin{equation} \label{eqn: mass and matrix}
m_{\tau,t}\left( \cF^N(A) \right)
	= u^T \left( \overline{M^{\cV \ra \cV}_\cF} \right)^N \overline{[A]_\cV}.
\end{equation}
}

\comment{
\begin{proposition} \label{eqn: mass and matrix}
Let $w$ be an accepted word in a $\tau$-HN automaton which starts and ends at the same vertex $\cV = \{V_j\}$.
Let $u = (u_j)$ be the vector given by $u_j =  m_{\tau, t}(V_j)$ for each $j$.
Then for all object $A$ $\tau$-supported by $\cV$, we have
\[
m_{\tau,t}\left( \sigma_w^N(A) \right)
	= u^T \overline{M_w} ^N \left(\overline{HN_\tau(A)} \right),
\]
where $\overline{M_w}$ and $\overline{HN_\tau(A)}$ denote $M_w$ and $HN_\tau(A)$ with $\overline{PF\dim}$ applied to all the entries.
\end{proposition}
\begin{proof}
Let $w$ be an accepted word that starts and ends at the same vertex $\cV$.
By repeated application of \cref{HN matrix computes}, we get that for all objects $A$ with $\supp_\tau(A) \in ob(\<\cV\>^\oplus_{\tau})$, 
\[
HN_\tau(\sigma_w(A)) = M_{w} (HN_\tau(A)) \in \<[\cV]\>_{K_0(g\cC)[s,s^{-1}]}.
\]
In particular, since $\supp_\tau(\sigma_w(A)) \in ob(\<\cV\>^\oplus_\tau)$, by induction we have that 
\[
HN_\tau(\sigma_w^N(A)) = M_{w}^N (HN_\tau(A)).
\]
This combined with \cref{vector to mass} gives us the required result.
\end{proof}
}

\comment{
\begin{definition}
Let $\cD$ be a triangulated category.
A mass function with parameter $t$ on $\cD$ is a non-negative real valued function
\[
\mu: \Ob(\cD) \ra \R_{\geq 0}
\]
such that 
\begin{enumerate}
\item $\mu_t(B) \leq \mu_t(A) + \mu_t(C)$ for any distinguished triangle $A \ra B \ra C \ra $ in $\cD$ (triangle inequality).
\item $\mu_t(A[1]) = \mu_t(A)\cdot e^{t}$
\end{enumerate}
\end{definition}
}

\chapter{Stability conditions, root systems and the classification theorem}
In this chapter, we bring our focus back to the categorical action of $\B(I_2(n))$ on the triangulated $\TLJ_n$-module $\<1\>$-category $(\Kom^b(\I$-prmod)${}, \<1\>, -\otimes ?)$ constructed in Part I of this thesis.
We start with constructing a $q$-stability condition respecting $\TLJ_n$ on $\cK$, which we call the root stability condition, and we study its semistable objects.
We then use the tools set up in Chapter 3 to study the group $\B(I_2(n))$ through the categorical action defined in Chapter 2.
We proceed by first giving a naive Coxeter-theoretic classification of braid elements  in $\B(I_2(n))$ into three types: periodic, reducible and pseudo-Anosov.
We then compute the mass growth of braids with respect to the root stability condition, starting with the periodic and reducible braids.
To deal with the pseudo-Anosov braids, we construct a mass automaton over each $\B(I_2(n))$ and provide an algorithm which effectively decides the type of any given braids.
This algorithm also computes the mass growths and leads to the equivalent definitions of the three types of braids in terms of categorical dynamics.

Throughout this chapter, we will use $\cK:= \Kom^b(\I$-prmod) defined in Chapter 2 and treat $\cK$ as a triangulated module $\<1\>$-category over the fusion category $\TLJ_n$.

\section{Stability conditions associated to root systems}
The main aim of this section is to construct and study the root stability condition $\tau_R$ associated to each root system $I_2(n)$.
We start by recalling the standard heart $\cH$ of $\cK$, known as the heart of \emph{linear complexes}, and define the root stability condition by defining a $\X[1]$-stability function respecting $\TLJ_n$ on $\cH$.

We then study the relationship between the $\tau_R$-semistable objects and the (positive) roots of the associated root system.
We also show that the braid $\sigma_2\sigma_1$ restricts to an autoequivalence on the subcategory consisting of these $\tau_R$-semistable objects, which can be seen as a categorical version of the $\frac{2\pi}{n}$ rotation on the root system when acted upon by the corresponding Coxeter element $s_2s_1$.

\subsection{The heart of linear complexes}
To define a stability condition on $\cK$, we shall first construct a heart which is a finite-length abelian category, where defining a stability function on this heart induces a stability condition on $\cK$.

Before we define such a heart, let us quickly recall the definition of minimal complex:
\begin{definition}
Let $X$ be a complex in $\Com(\I$-prmod$)$.
We say $X$ is a \emph{minimal complex} if $X$ is indecomposable in the category $\Com(\I$-prmod$)$.
Similarly, a complex $Y$ in $\cK$ is \emph{minimal} if it is a minimal complex viewed as a complex in $\Com(\I$-prmod$)$.
\end{definition}
We note some important properties about minimal complexes:
\begin{enumerate}
\item every indecomposable complex $X$ in $\cK$ is isomorphic (in $\cK$) to a complex $X'$ which is minimal, i.e. it always has a minimal representative $X'$ where $X'$ has no summands which are homotopy equivalent to the 0 complex;
\item all minimal representatives of a complex are isomorphic in $\Com(\I-$prgrmod$)$, where in particular their underlying (graded) $\I$-modules are isomorphic.
\end{enumerate}

We shall view each complex in $\cK$ as a direct sum of $\Z$-graded (cohomological degree) $\I$-modules $P_i\otimes \Pi_a \<k\>[\ell]$, equipped with a differential map $\partial$, namely $\partial$ has degree +1 with respect the the cohomological grading and $\partial^2 = 0$.
Each object $P_i\otimes \Pi_a \< k \>[\ell]$ has a \emph{level grading} defined by $m := \ell-k$.
A minimal complex in $\cK$ is of \emph{homogeneous level} $m$ if all of its underlying summands $P_i\otimes \Pi_a \<k\> [\ell]$ have level $m$.
In general, a complex is of homogeneous level $m$ if it is isomorphic to a minimal complex of level $m$.
We define $\cK^{=m} \subset \cK$ to be the additive subcategory of $\cK$ that consists of complexes with homogeneous level $m$.
Similarly, we define $\cK^{\leq m}$ and $\cK^{\geq m}$ to be the (strict) additive subcategories of $\cK$ that consist of minimal complexes with summands that have level $\leq m$ and $\geq m$ respectively.
Note that 
\[
\cK^{\geq m} \subset \cK^{\geq m} [-1] = \cK^{\geq m -1}, \quad \text{and } \cK^{\leq m} \subset \cK^{\leq m} [1] = \cK^{\leq m +1}.
\]
Given a minimal complex $C$, denote $\cK^{\geq m}(C)$ as the complex consisting of all the $\I$-module summands $P_i\otimes \Pi_a \<k\>[\ell]$ of $C$ with level $\geq m$ equipped with the differential given by the restriction of the differential on $C$ -- one can check that this is indeed a complex since the degrees of the maps in the differential with respect to the level grading are always non-decreasing in a minimal complex.
Hence, we have an exact triangle
\[
\cK^{\geq m}(C) \xra{\iota} C \ra \cone(\iota) \ra
\]
for each minimal complex $C$, with $\iota$ the obvious inclusion map, where it follows from the definition that $\cone(\iota)$ is an object in $\cK^{\leq m-1}$.
Since every complex $X$ is isomorphic to a minimal one, we have a distinguished triangle with $E \in \cK^{\geq m} $ and $F \in \cK^{\leq m-1}$:
\[
E \ra X \ra F \ra
\]
for every $X \in \cK$.
Moreover, it is clear that there are no maps from $\cK^{\geq m}$ to $\cK^{\leq m'}$ for $m > m'$.
It follows that the pair $(\cK^{\geq 0}, \cK^{\leq -1})$ defines a $t$-structure on $\cK$, with heart 
\[
\cH = \cK^{\geq 0} \cap \cK^{\leq -1}[1] = \cK^{=0}
\]
consisting of complexes with homogeneous level $0$. 
We will call $\cH$ the \emph{linear heart} of $\cK$.
Note that complexes in $\cH$ are isomorphic to minimal complexes whose non-zero differentials consist of the map $(1|2)$ or $(2|1)$ (up to scalar multiples).

\subsection{The root stability condition $\tau_R$} \label{sect: stab cond root}
It is easy to see that the heart of linear complexes $\cH$ is by definition an (abelian) $\TLJ_n$-module subcategory of $\cK$, and is closed under the autoequivalence $\<1\>[1]$; in other words, $(\cH, {\<1\>[1]}, \Psi)$ is a $\TLJ_n$-module $\<1\>[1]$-category.
Moreover,  $\cH$ is a $\<1\>[1]$-algebraic heart (see \cref{defn: X1-algebraic heart}) with $\{ P_i \otimes \Pi_a : i \in \{1,2\}, 0 \leq a \leq n-2\}$ as the set of (isomorphism classes of) simple objects.
In particular, $\cH$ is a finite-length abelian category, and therefore any stability function on $\cH$ will satisfy the Harder-Narasimhan property.

Let us denote again the ring $\Omega := \Z[\omega]/ \< \Delta_{n-1}(\omega)\> \cong K_0(\TLJ_n)$.
Given the structure $(\cH, {\<1\>[1]}, \Psi)$ on $\cH$, recall that its abelian Grothendieck group $K_0(\cH)$ is a module over $\Omega[x^{\pm 1}]$, with action defined by
\[
\omega\cdot[A] = [A\otimes \Pi_1], \quad x\cdot [A] = [A\<1\>[1]].
\]
Similarly recall that $\C$ is a $\Omega[x^{\pm 1}]$-module induced by the ring homomorphism $\Omega[x^{\pm 1}] \ra \C$ defined by
\[
\omega \mapsto 2\cos\left( \frac{\pi}{n} \right), \quad x \mapsto 1.
\]
To obtain a $q$-stability condition respecting $\TLJ_n$ on $\cK$, \cref{thm: x[1] stab function and q stab cond respecting C} (together with the fact that $\cH$ is $\<1\>[1]$-algebraic) tells us that it is sufficient to specify a stability function $Z : K_0(\cH) \ra \C$ which is a $\<1\>[1]$-stability function over $\TLJ_n$; namely $Z$ is required to be a $\Omega[x^{\pm 1}]$-module homomorphism.
Since the Grothendieck group $K_0(\cH)$ is a rank two free $\Omega[x^{\pm 1}]$-module with basis $\{[P_1], [P_2]\}$, such a stability function is uniquely defined by specifying the image of each $[P_1]$ and $[P_2]$ under $Z$ in the half-open upper-half plane:
\[
Z([P_j]) \in \mathbb{H} \cup \R^{\geq 0} \subset \C.
\]
This will be our method of constructing a $q$-stability condition respecting $\TLJ_n$ on $\cK$.

Recall that with $\{\alpha_1,\alpha_2\} \subset \R^2$ as basis (simple roots), the symmetric geometric representation $\overline{\rho}$ of $\mathbb{W}(I_2(n))$ is uniquely defined by the bilinear pairing $(\alpha_1 | \alpha_2) = (\alpha_2 | \alpha_1) = 2\cos\left( \frac{\pi}{n} \right)$, where the action can be described explicitly on the generators by the following matrices:
\[
\overline{\rho}(s_1) = 
	\begin{bmatrix}
	-1 & 2 \cos (\frac{\pi}{n}) \\
	0  &  1
	\end{bmatrix}, \quad
\overline{\rho}(s_2) = 
	\begin{bmatrix}
	1                       &  0  \\
	2 \cos (\frac{\pi}{n})  &  -1
	\end{bmatrix}.
\]
Once again this is the Burau representation with $q = -1$.
With $\alpha_i$ denoting the simple roots, one can visualise the positive roots of the associated root system (not necessarily crystallographic) as in \Cref{fig: root system general}.
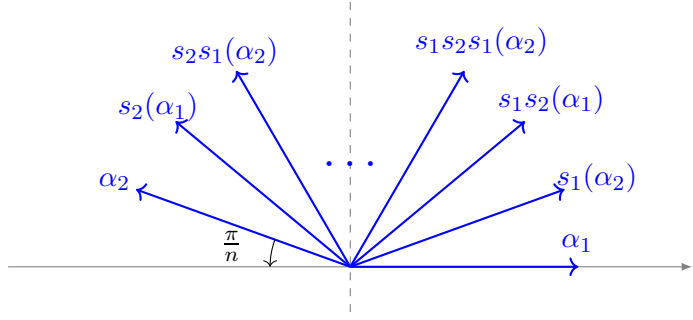
\begin{figure}[h]
  \centering
  \begin{tikzpicture}[scale = 3]
    \coordinate (Origin)   at (0,0);
    \coordinate (XAxisMin) at (-1.5,0);
    \coordinate (XAxisMax) at (1.5,0);
    \coordinate (YAxisMin) at (0,-.2);
    \coordinate (YAxisMax) at (0,1.2);
    \coordinate (na1)      at (-1,0);

    \draw [thin, gray,-latex] (XAxisMin) -- (XAxisMax); 
    \draw [thin, gray, dashed] (YAxisMin) -- (YAxisMax);
    \node[thick, blue] (cdots) at ($(Origin) + (0.015,0.45)$) {\LARGE $\cdots$};
	
	\draw[blue, thick, ->] (0,0) -- ++(0:1);
	\node[blue] (a1) at ($(Origin) + (6:1)$) {$\alpha_1$};
	
	\draw[blue, thick, ->] (0,0) -- ++(20:1);
	\node[blue] (s1sa2) at ($(Origin) + (20:1.15)$) {$s_1(\alpha_2)$};
	
	\draw[blue, thick, ->] (0,0) -- ++(40:1);
	\node[blue] (s1s2a1) at ($(Origin) + (40:1.15)$) {$s_1s_2(\alpha_1)$};
	
	\draw[blue, thick, ->] (0,0) -- ++(60:1);
	\node[blue] (s1s2s1a2) at ($(Origin) + (60:1.15)$) {$s_1s_2s_1(\alpha_2)$};
	
	\draw[blue, thick, ->] (0,0) -- ++(120:1);
	\node[blue] (s2s1a2) at ($(Origin) + (120:1.1)$) {$s_2s_1(\alpha_2)$};
	
	\draw[blue, thick, ->] (0,0) -- ++(140:1);
	\node[blue] (s2a1) at ($(Origin) + (140:1.1)$) {$s_2(\alpha_1)$};
	
	\draw[blue, thick, ->] (0,0) -- ++(160:1);
	\node[blue] (a2) at ($(Origin) + (160:1.1)$) {$\alpha_2$};

    \pic [draw, ->, "$\frac{\pi}{n}$", angle radius = 30, angle eccentricity=1.5] {angle=a2--Origin--na1};
  \end{tikzpicture}
  \caption{The positive roots of the $I_2(n)$ root system.}
  \label{fig: root system general}
\end{figure}
Under this representation, the Coxeter generators $s_i$ of $\mathbb{W}(I_2(n))$ act on the roots by reflecting along the perpendicular line of $\alpha_i$, and in particular $s_2s_1$ acts by a $\frac{2\pi}{n}$ clockwise rotation along the origin.

Note that all the positive roots can be obtained by rotating the two positive roots $\alpha_2$ and $s_2(\alpha_1)$ using the action of $s_2s_1$:
\begin{enumerate}
\item When $n$ is odd, the positive roots are
\begin{equation} \label{eqn: positive roots odd}
\begin{cases}
(s_2s_1)^{k}(\alpha_2), 
	& \text{for } 0 \leq k \leq \frac{n-1}{2}; \\
(s_2s_1)^{k'}s_2(\alpha_1), 
	& \text{for } 0 \leq k' \leq \frac{n-3}{2},
 \end{cases}
\end{equation}
where $\alpha_1 = (s_2s_1)^{\frac{n-1}{2}}(\alpha_2)$ and $s_1(\alpha_2) = (s_2s_1)^{\frac{n-3}{2}}s_2(\alpha_1)$ etc.
\item When $n$ is even, the positive roots are
\begin{equation} \label{eqn: positive roots even}
\begin{cases}
(s_2s_1)^{k}(\alpha_2), 
	& \text{for } 0 \leq k \leq \frac{n}{2}-1; \\
(s_2s_1)^{k'}s_2(\alpha_1), 
	& \text{for } 0 \leq k' \leq \frac{n}{2}-1,
 \end{cases}
\end{equation}
where $\alpha_1 = (s_2s_1)^{\frac{n}{2}-1}s_2(\alpha_1)$ and $s_1(\alpha_2) = (s_2s_1)^{\frac{n}{2}-1}(\alpha_2)$ etc.
\end{enumerate}
In particular, the total of $n$ positive roots $\Phi^+$ can be enumerated by the alternating expressions
\[
\Phi^+ = \{ \alpha_2, s_2(\alpha_1), s_2s_1(\alpha_2), ..., \underbrace{(s_2s_1s_2...s_j)}_{n-1 \text{ times}}(\alpha_i) \}, \quad \text{ where } i \neq j.
\]
Note that by definition $\underbrace{(s_2s_1s_2...s_j)}_{n-1 \text{ times}}(\alpha_i) = \alpha_1$.

For each $n$, we shall define a stability condition that reflects this particular geometry of the root system.
\begin{definition}\label{defn: root stability condition}
Consider the stability function $ Z : K_0(\cH) \ra \C$ uniquely defined as a $\Omega[x,x^{-1}]$-module homomorphism by
\[
Z([P_1]) = 1, \quad 
Z([P_2]) = e^{i(\pi - \frac{\pi}{n})}.
\]
This determines a $q$-stability condition on $\cK$ respecting $\TLJ_n$.
We denote this stability condition by $\tau_R$, and call it \emph{the root stability condition associated to $I_2(n)$}.
\end{definition}

Note that the central charge $\cZ: K_0(\cK) \ra \C$ of $\tau_R$ is then a $\Omega[q^{\pm 1}]$-module homomorphism, where $K_0(\cK)$ is a $\Omega[q^{\pm 1}]$-module with action defined by
\[
\omega\cdot [X] = [X\otimes \Pi_1], \quad q\cdot [X] = [X\<1\>];
\]
whereas $\C$ is a $\Omega[q^{\pm 1}]$-module induced by the homomorphism $\Omega[q^{\pm 1}] \ra \C$ defined by
\[
\omega \mapsto 2\cos\left( \frac{\pi}{n} \right), \quad q \mapsto -1.
\]

Let us view $\C \cong \R^2$ as $\R$-vector space naturally, where the positive roots in \Cref{fig: root system general} are now elements in $\C$, namely 
\[
\alpha_1 = 1 \in \C, \alpha_2 = e^{i\frac{\pi}{n}} \in \C \text{ etc.}
\]
As such, $\W(I_2(n))$ acts on $\C \cong \R^2$ through the symmetrical geometric representation.
Under this definition, the central charge $\cZ: K_0(\cK) \ra \C$ intertwines the actions of $\B(I_2(n))$ and $\W(I_2(n))$:
\[
\cZ( \sigma_k ([P_j]) ) = s_k \cdot \cZ([P_j]) = s_k(\alpha_j).
\]
In particular, we see that for any object $X \in \cK$, the central charge of $\sigma_k(X)$ is given by reflecting $\cZ([X])$ along the perpendicular line of $\cZ([P_k]) = \alpha_k$, and the central charge of $\sigma_2\sigma_1(X)$ is given by rotating $\cZ([X])$ by $\frac{2\pi}{n}$ clockwise along the origin.

\comment{
In this subsection, we collect some results relating semi-stable objects and roots.
All the results here are slight generalisations of the results in \cite{bdl_root}, so we will only state them without proofs.
The main result that we need is \cref{root lift semistable}.

As before, denote the ring $\Omega := \Z[\omega]/ \< \Delta_{n-1}(\omega)\> \cong K_0(\TLJ_n)$.
Throughout this subsection, we fix $\tau$ to be a $q$-stability condition  respecting $\TLJ_n$ defined by some $\<1\>[1]$-stability function $Z: K_0(\cH) \ra \C$ over $\TLJ_n$ on the linear heart $\cH$ (see \cref{sect: stab cond root}).
We shall study the relation between the semistable objects of the linear heart $\cH$, and the roots in the root system of $I_2(n)$ given by the (symmetric) geometric representation.

Recall that  $\R$ can be viewed as a module over $\Omega[q,q^{-1}]$ through the extension of
\[
PF\dim: \Omega \ra \R,
\]
where we send $\omega \mapsto 2\cos(\frac{\pi}{n})$ and $q \mapsto -1$.
This way, we may extend scalars on $K_0(\cK)$ to get $\mathbb{K} := K_0(\cK) \otimes_{\Omega[q,q^{-1}]} \R \cong \R^{2}$, which is a two dimensional real vector space.

Note that in this case, $[-\otimes \Pi_a]$ acts on $\mathbb{K}$ through multiplication by $\Delta_a(2\cos(\frac{\pi}{n})) \geq 1$, where $\Delta_a(2\cos(\frac{\pi}{n})) = \Delta_{n-2 - a}(2\cos(\frac{\pi}{n}))$ for all $0 \leq a \leq n-2$.
Moreover, the braid action on $\cK$ given by $\sigma_i$ descends to an action on $\mathbb{K}$ given by the corresponding Coxeter reflection $s_i$.

\begin{proposition}
Let $C$ be an object in $\cH$. 
Then $\sigma_i(C)$ (resp. $\sigma^{-1}_i(C)$) is in $\cH$ if and only if $P_i \otimes \Pi_a \<k\> [k]$ is not a quotient (resp. subobject) of $C$ for all $a \in \{0,1, ..., {n-2} \}$ and all ${k\in \mathbb{Z}}$. 
\end{proposition}
\begin{definition}
Let $S \subseteq K_0$ and $C$ an object in $\cH$.
We say that $S$ \emph{envelops} the quotients (resp. subobjects) of $C$ if the class in $K_0$ of every quotient (resp. subobject) of $C$ can be expressed as a non-negative $\R$-linear combination of elements in $S$. 
We say that $S$ \emph{weakly envelops} the quotients (resp. subobjects) of $C$ if this property holds modulo the class of $C$.
\end{definition}

\begin{proposition}
Let $S \subseteq K_0$ and $C$ an object in $\cH$.
Then $S$ weakly envelops the quotients of $C$ if and only if $-S$ weakly envelops the subobjects of $C$.
\end{proposition}

\begin{proposition}
Let $C$ be an object in $\cH$ and $S \subseteq K_0$ such that $S$ envelops the quotients (resp. subobjects) of $C \otimes \Pi_a$ for all $0 \leq a \leq n-2$. 
If $\sigma_i(C)$ lies in $\cH$, then $s_i(S) \cup \{ [P_i] \}$ envelops the quotients (resp. subobjects) of $\sigma_i(C) \otimes \Pi_a$ for all $0 \leq a \leq n-2$.
The analogous statement for weakly envelop holds true as well.
\end{proposition}

\begin{definition}
Let $C := \sigma_{v_m}^{\epsilon_m}\sigma_{v_{m-1}}^{\epsilon_{m-1}} ... \sigma_{v_1}^{\epsilon_1}(P_{v_0})$, with $\epsilon_{i} \in \{-1, 1\}$.
The \emph{root sequence} of $C$ is a sequence $R = (R_m, ..., R_0)$ of elements in $\mathbb{K}$ given by
\[
R_i := s_{v_m} s_{v_{m-1}} ... s_{v_{i+1}} ( [ P_{v_i} ] ).
\]
We call $R_0 = [C]$ the neutral root.
The \emph{positive subsequence} $R_+$ of $R$ is the subsequence $(R_i : \epsilon_{v_i} = 1)$.
Similarly the \emph{negative subsequence} $R_-$ of $R$ is the subsequence $(R_i : \epsilon_{v_i} = -1)$.
\end{definition}


\begin{proposition}
Let $C := \sigma_{v_m}^{\epsilon_m}\sigma_{v_{m-1}}^{\epsilon_{m-1}} ... \sigma_{v_1}^{\epsilon_1}(P_{v_0})$, with $\epsilon_{i} \in \{-1, 1\}$ and $R$ its corresponding root sequence. 
\begin{enumerate}
\item If the images of $R_+$ and $R_-$ in $\mathbb{K}/\< [C] \>$ can be separated by a hyperplane, then $C \otimes \Pi_a$ lies in $\cH$ for all $0 \leq a \leq n-2$.
\item If (1) holds, then the set $R_+ \cup -R_-$ (resp. $-R_+ \cup R_-$) weakly envelops the quotients (resp. subobjects) of all $C \otimes \Pi_a$.
\end{enumerate}
\end{proposition}

Let $(\cP, \cZ)$ be a stability condition on $\cK$ defined by a stability function $Z: K_0(\cH) \ra \C$ on the linear heart $\cH$.
Let $C$ be a non-zero object in $\cH$.
Set $L := \C/\R\cdot \cZ([C])$ and decompose $L$ into positive ray $L_+$ and negative ray $L_-$:
\begin{align*}
L_+ &:= \text{ Image of } \left\{ me^{i\theta} | \theta \in [0,\pi), \theta \geq \arg( \cZ([C]) ) \right\}, \\
L_-  &:= \text{ Image of } \left\{ me^{i\theta} | \theta \in [0,\pi), \theta \leq \arg( \cZ([C]) ) \right\}.
\end{align*}

\begin{proposition}\label{root lift semistable}
Let $C := \sigma_{v_m}^{\epsilon_m}\sigma_{v_{m-1}}^{\epsilon_{m-1}} ... \sigma_{v_1}^{\epsilon_1}(P_{v_0})$, with $\epsilon_{i} \in \{-1, 1\}$ and $R$ its corresponding root sequence. 
If the image of $\cZ(R_+)$ lies in $L_+$ and the image of $\cZ(R_-)$ lies in $L_-$, then $C \otimes \Pi_a$ is semi-stable for all $0 \leq a \leq n-2$.
\end{proposition} 
}

\comment{
Using \cref{lemma for braid relation}, we have some alternating composition of positive braid $\sigma_2 \sigma_1 ... $ such that $\sigma_2 \sigma_1 ... \sigma_{j}(P_i) = P_1 \otimes \Pi_{n-2}$ for some $i$ and $j$ such that $j \neq i$.
In particular, we have
\begin{align*}
(\sigma_2 \sigma_1)^{\frac{n-1}{2}} (P_2)                &= P_1\otimes \Pi_{n-2}, \text{ when $n$ is odd}; \\
(\sigma_2 \sigma_1)^{\frac{n-2}{2}}\sigma_2 (P_1) &= P_1\otimes \Pi_{n-2}, \text{ when $n$ is even}.
\end{align*}
Consider the sequence of objects 
\[
\mathfrak{R} = (P_2, \sigma_2(P_1), \sigma_2 \sigma_1 (P_2), \sigma_2 \sigma_1 \sigma_2 (P_1), ..., \sigma_2 \sigma_1 ... (P_i) = P_1\otimes \Pi_{n-2}).
\]
This sequence $\mathfrak{R}$ can be thought of as the positive lift of the root sequence of $P_1\otimes \Pi_{n-2} = \sigma_2 \sigma_1 ... (P_i)$.

\begin{proposition}
With the stability condition above, each of the object in the sequence $\mathfrak{R}$ lies in $\cH$ and is semi-stable.
\end{proposition}

\noindent
Let $\mathfrak{R}^{tot} = \left\{ A[k]  | A \in \mathfrak{R}, k \in \mathbb{Z} \right\}$, the set of objects in $\mathfrak{R}$ up to shifts.
Note that there is a strict total order on $\mathfrak{R}^{tot}$ according to their phases, i.e.
\[
... > P_1\otimes \Pi_{n-2}[1] > 
P_2 > \sigma_2(P_1) >
\sigma_2 \sigma_1 (P_2) >  \sigma_2 \sigma_1 \sigma_2 (P_1) >
 ... >  P_1 \otimes \Pi_{n-2} > 
 P_2[-1] > ...
\]
Thus we also have a bijection between $\mathfrak{R}^{tot}$ and the set of phases of objects in $\mathfrak{R}^{tot}$, given by $\{ \frac{k}{n} | k \in \mathbb{Z}\}$.

\begin{corollary}
Let $A$ be an object in $\mathfrak{R}^{tot}$ with phase $\phi$.
Then $\sigma_2 \sigma_1 (A)$ is the object in $\mathfrak{R}^{tot}$ two steps down from $A$ according to the ordering above, i.e. $\sigma_2 \sigma_1(A)$ is the unique object in $\mathfrak{R}^{tot}$ with phase $\phi - \frac{2}{n}$.
\end{corollary}
}

\subsection{The $\tau_R$-semistable objects}\label{sect: root and semistable}
In this subsection, we study the semistable objects for our category $\cK$ with respect to the root stability condition $\tau_R = (\cP, \cZ)$ constructed in \cref{sect: stab cond root}.
Throughout this subsection, $\cH = \cP[0,1)$ will denote the heart of linear complexes.

We will use a result from \cite{bdl_root} to identify some of the $\tau_R$-semistable objects.
Before we state it, we will need the following definition:
\begin{definition}
Let $C$ be an object in $\cH$ given by
\[
C := \sigma_{v_m}^{\epsilon_m}\sigma_{v_{m-1}}^{\epsilon_{m-1}} ... \sigma_{v_1}^{\epsilon_1}(P_{v_0}),
\]
with $\epsilon_{i} \in \{-1, 1\}, v_0 \in \{1,2\}$.
The \emph{root sequence} of $C$ is a sequence $R = (R_m, ..., R_0)$ of elements in $\C \cong \R^2$ given by
\[
R_i := s_{v_m} s_{v_{m-1}} ... s_{v_{i+1}} ( \alpha_{v_i} ) \in \C,
\]
where $\alpha_{v_0} = \cZ([P_{v_0}])$.
The \emph{positive subsequence} $R_+$ of $R$ is the subsequence $(R_i : \epsilon_{v_i} = 1)$ and
similarly the \emph{negative subsequence} $R_-$ of $R$ is the subsequence $(R_i : \epsilon_{v_i} = -1)$.
\end{definition}

Let $C$ be a non-zero object in $\cH$.
Set $L := \C/\R\cdot \cZ([C])$ and decompose $L$ into positive ray $L_+$ and negative ray $L_-$:
\begin{align*}
L_+ &:= \text{ Image of } \left\{ me^{i\theta} : \theta \in [0,\pi), \theta \geq \arg( \cZ([C]) ) \right\}, \\
L_-  &:= \text{ Image of } \left\{ me^{i\theta} : \theta \in [0,\pi), \theta \leq \arg( \cZ([C]) ) \right\}.
\end{align*}

\begin{proposition}[\cite{bdl_root}] \label{root lift semistable}
Let $C$ be an object in $\cH$ with 
\[
C := \sigma_{v_m}^{\epsilon_m}\sigma_{v_{m-1}}^{\epsilon_{m-1}} ... \sigma_{v_1}^{\epsilon_1}(P_{v_0}),
\]
where $\epsilon_{i} \in \{-1, 1\}, v_0 \in \{1,2\}$ and let $R$ be its corresponding root sequence. 
If the image of $R_+$ lies in $L_+$ and the image of $R_-$ lies in $L_-$, then $C$ is $\tau_R$-semistable.
\end{proposition} 

We shall use this proposition to identify the $\tau_R$-semistable objects, starting with the ones lying in the linear heart $\cH$.
Fix $\gamma := \sigma_2 \sigma_1$.
By applying \cref{lemma for braid relation} repeatedly, we see that the following set of objects are (indecomposable) linear complexes
\begin{enumerate}
\item for $n$ odd,
\begin{equation} \label{semi-stable forms odd}
\mathfrak{R}^+_n :=
	\left\{
		\gamma^{k}(P_2): \text{ for } 0 \leq k \leq \frac{n-1}{2}
	\right\}
\cup
	\left\{
		\gamma^{k'}\sigma_2(P_1) : \text{ for } 0 \leq k' \leq \frac{n-3}{2}
	\right\} \subset \cH
\end{equation}
\item for $n$ even,
\begin{equation} \label{semi-stable forms even}
\mathfrak{R}^+_n =
	\left\{
		\gamma^{k}(P_2) :  \text{for } 0 \leq k \leq \frac{n}{2}-1
	\right\}
\cup
	\left\{
		\gamma^{k'}\sigma_2(P_1) : \text{for } 0 \leq k' \leq \frac{n}{2}-1
	\right\} \subset \cH.
\end{equation}
\end{enumerate}
Using \cref{root lift semistable}, we see that the root sequences of the objects in $\mathfrak{R}_n^+$ and their corresponding positions of roots (as depicted in \cref{fig: root system general}) show that all the objects in $\mathfrak{R}_n^+$ are $\tau_R$-semistable.
Moreover, the stability function $Z$ (equivalently central charge $\cZ$) induced by $\tau_R$ gives a bijection between the set $\mathfrak{R}^+_n$ and the set of positive roots $\Phi^+$.
In particular,
\begin{align*}
P_2 &\in \cP\left(1-\frac{\pi}{n} \right), \\
\sigma_2 \left(P_1 \right) &\in \cP\left(1-\frac{2\pi}{n}\right), \\
\sigma_2\sigma_1(P_2) &\in \cP\left(1-\frac{3\pi}{n}\right), \\
&\vdots \\
P_1\otimes \Pi_{n-2} \<n-2\>[n-2] &\cong \sigma_2\sigma_1... \sigma_j(P_i) \in \cP\left( 0 \right).
\end{align*}
See \Cref{fig: root and central charge n=5} for a side-by-side comparison of the positive roots and the image of the corresponding $\tau_R$-semistable objects under the stability function for $n=5$.

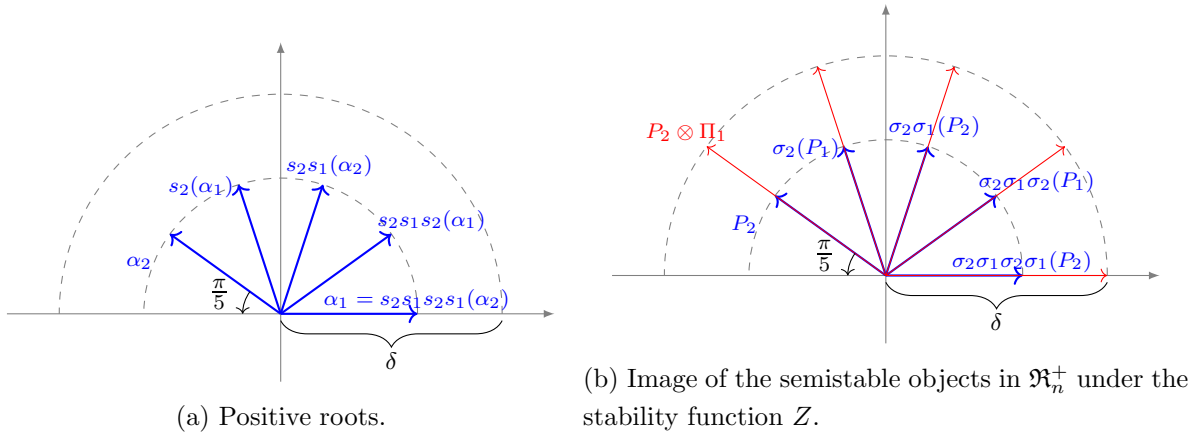
\begin{figure}[H]
\centering
\begin{subfigure}{.5\textwidth}
  \centering
  \begin{tikzpicture}[scale = 1.8]
    \coordinate (Origin)   at (0,0);
    \coordinate (XAxisMin) at (-2,0);
    \coordinate (XAxisMax) at (2,0);
    \coordinate (YAxisMin) at (0,-.5);
    \coordinate (YAxisMax) at (0,2);
    \coordinate (na1)      at (-1,0);
    \coordinate (xa2) at ($(Origin) + (144:1.8)$);

    \draw [thin, gray,-latex] (XAxisMin) -- (XAxisMax);
    \draw [thin, gray,-latex] (YAxisMin) -- (YAxisMax);
    \draw[gray, dashed] (1,0) arc (0:180:1);
    \draw[gray, dashed] (1.618,0) arc (0:180:1.618);
    	
	\draw [decorate,decoration={brace,amplitude=10pt,mirror}]
		($(Origin) + (0,-0.05)$) -- (1.618, -0.05) node [black,midway,yshift=-22] 
		{\footnotesize $\delta$};	
	
	\draw[blue, thick, ->] (0,0) -- ++(0:1);
	\node[blue] (a1) at ($(Origin) + (6:1)$) 
		{\scriptsize $\alpha_1= s_2s_1s_2s_1(\alpha_2)$};
	
	\draw[blue, thick, ->] (0,0) -- ++(36:1);
	\node[blue] (s2s1s2a1) at ($(Origin) + (32:1.3)$) 
		{\scriptsize $s_2 s_1 s_2(\alpha_1)$};
	
	\draw[blue, thick, ->] (0,0) -- ++(72:1);
	\node[blue] (s2s1a2) at ($(Origin) + (72:1.15)$) 
		{\scriptsize $s_2 s_1(\alpha_2)$};
	
	\draw[blue, thick, ->] (0,0) -- ++(108:1);
	\node[blue] (s2a1) at ($(Origin) + (122:1.1)$) 
		{\scriptsize $s_2(\alpha_1)$};
	
	\draw[blue, thick, ->] (0,0) -- ++(144:1);
	\node[blue] (a2) at ($(Origin) + (160:1.1)$) 
		{\scriptsize $\alpha_2$};
	
	
	
	
	
    
    \pic [draw, ->, "$\frac{\pi}{5}$", angle eccentricity=1.7] {angle=xa2--Origin--na1};
  \end{tikzpicture}
  \caption{
  		Positive roots.
  }
  \label{subfig: roots n=5}
\end{subfigure}%
\begin{subfigure}{.5\textwidth}
  \centering
\begin{tikzpicture}[scale = 1.8]
    \coordinate (Origin)   at (0,0);
    \coordinate (XAxisMin) at (-2,0);
    \coordinate (XAxisMax) at (2,0);
    \coordinate (YAxisMin) at (0,-.5);
    \coordinate (YAxisMax) at (0,2);
    \coordinate (na1)      at (-1,0);

    \draw [thin, gray,-latex] (XAxisMin) -- (XAxisMax);
    \draw [thin, gray,-latex] (YAxisMin) -- (YAxisMax);
    \draw[gray, dashed] (1,0) arc (0:180:1);
    \draw[gray, dashed] (1.618,0) arc (0:180:1.618);
    	
	\draw [decorate,decoration={brace,amplitude=10pt,mirror}]
		($(Origin) + (0,-0.05)$) -- (1.618, -0.05) node [black,midway,yshift=-22] 
		{\footnotesize $\delta$};	
	
	\draw[blue, thick, ->] (0,0) -- ++(0:1);
	\node[blue] (a1) at ($(Origin) + (6:1)$) 
		{\scriptsize $\sigma_2\sigma_1\sigma_2\sigma_1(P_2)$};
	
	\draw[blue, thick, ->] (0,0) -- ++(36:1);
	\node[blue] (s2s1s2a1) at ($(Origin) + (32:1.3)$) 
		{\scriptsize $\sigma_2\sigma_1\sigma_2(P_1)$};
	
	\draw[blue, thick, ->] (0,0) -- ++(72:1);
	\node[blue] (s2s1a2) at ($(Origin) + (72:1.15)$) 
		{\scriptsize $\sigma_2\sigma_1(P_2)$};
	
	\draw[blue, thick, ->] (0,0) -- ++(108:1);
	\node[blue] (s2a1) at ($(Origin) + (122:1.1)$) 
		{\scriptsize $\sigma_2(P_1)$};
	
	\draw[blue, thick, ->] (0,0) -- ++(144:1);
	\node[blue] (a2) at ($(Origin) + (160:1.1)$) 
		{\scriptsize $P_2$};
	
	\draw[red, ->] (0,0) -- ++(0  :1.618);
	
	\draw[red, ->] (0,0) -- ++(36 :1.618);
	
	\draw[red, ->] (0,0) -- ++(72 :1.618);
	
	\draw[red, ->] (0,0) -- ++(108:1.618);
	
	\draw[red, ->] (0,0) -- ++(144:1.618);
	\node[red] (xa2) at ($(Origin) + (144:1.8)$) {\scriptsize $P_2 \otimes \Pi_1$};
    
    \pic [draw, ->, "$\frac{\pi}{5}$", angle eccentricity=1.7] {angle=xa2--Origin--na1};
  \end{tikzpicture}
  \caption{
  		Image of the semistable objects in $\mathfrak{R}^+_n$ under the stability function $Z$.
  }
  \label{subfig: central charge n=5}
\end{subfigure}   
 \caption{
  		An example for $n = 5$. Compare the positive roots with their corresponding (positive) lifts to $\tau_R$-semistable objects.
  		Tensoring with $\Pi_1$ corresponds to multiplication by the scalar $\delta:= 2\cos(\frac{\pi}{5})$.
  }
  \label{fig: root and central charge n=5}
\end{figure}

Once we know that the objects in $\mathfrak{R}^+_n$ are $\tau_R$-semistable, the fact that $\tau_R$ is a $q$-stability condition respecting $\TLJ_n$ allows us to obtain other $\tau_R$-semistable objects by tensoring with $\Pi_a$ for any $0 \leq a \leq n-2$ and applying the shift functors $\<-\>$ and $[-]$.
As such, we introduce the following notation:
\begin{definition}\label{defn: notation for generating category}
Let $\mathfrak{S}$ be a set of $\tau_R$-semistable objects in $\cK$.
We use $\< \mathfrak{S} \>^\oplus_{\tau_R}$ to denote the additive subcategory of $\cK$ generated by the objects in $\mathfrak{S}$ which is closed under the $\TLJ_n$-action, the triangulated shifts $[-]$ and the internal grading shifts $\<-\>$.
\end{definition}
Using this notation, we have the following:
\begin{align*}
\< P_2 \>^\oplus_{\tau_R} &\subset \bigoplus_{k \in \Z} \cP\left(1-\frac{\pi}{n} +k\right),  \\
\<\sigma_2(P_1)\>^\oplus_{\tau_R} &\subset  \bigoplus_{k \in \Z}  \cP\left(1-\frac{2\pi}{n}+k\right),  \\
\<\sigma_2\sigma_1(P_2)\>^\oplus_{\tau_R} &\subset  \bigoplus_{k \in \Z} \cP\left(1-\frac{3\pi}{n}+k \right), \\
&\vdots  \\
\< \underbrace{\sigma_2\sigma_1... \sigma_j}_{n-1 \text{ times}}(P_i) \>^\oplus_{\tau_R} = \<P_1\>^\oplus_{\tau_R} &\subset \bigoplus_{k \in \Z} \cP\left(k \right), \quad j\neq i.
\end{align*}
\begin{remark}
Note that the last statement is not claiming that $\sigma_2\sigma_1... \sigma_j(P_i) \cong P_1$, but rather the categories they generate  through $\< -\>^\oplus_{\tau_R}$ are equal.
\end{remark}
Realise that none of the phases intersect, and so $\< S \>^\oplus_{\tau_R} \cap \< S' \>^\oplus_{\tau_R} = 0$ for any distinct $S$ and $S'$ in $\mathfrak{R}^+_n$.
In particular, we get that
\[
\< \mathfrak{R}^+_n \>^\oplus_{\tau_R} = \bigoplus_{S \in \mathfrak{R}^+_n} \<S\>^\oplus_{ \tau_R} \subset \bigoplus_{\phi \in \R} \cP(\phi).
\]
\comment{
\begin{remark}
Since $\tau_R$ is a $q$-stability condition respecting $\TLJ_n$, every object in $\<\mathfrak{S}\>_{\tau_R}^\oplus$ is by definition a direct sum of semistable objects (possibly with different phases).
In particular, since $\cK$ is Krull-Schmidt, $\<\mathfrak{S}\>_{\tau_R}^\oplus$ is also Krull-Schmidt and every indecomposable object in $\<\mathfrak{S}\>_{\tau_R}^\oplus$ is a $\tau_R$-semistable object.
\end{remark}
\noindent
It is easy to check the following:
\begin{enumerate}[(i)]
\item for each $S \in \mathfrak{R}^+_n$, the object $S \otimes \Pi_a \<k\>[\ell]$ is indecomposable (and semistable) for any $0\leq a \leq n-2$ and $k, \ell \in \Z$; and
\item $
S \otimes \Pi_a \<k\>[\ell] \cong S' \otimes \Pi_{a'} \<k'\>[\ell'] 
\iff
S\cong S', a=a', k=k' \text{ and } \ell = \ell'
$.
\end{enumerate}
Therefore, we see that $\< S \>^\oplus_{\tau_R} \cap \< S' \>^\oplus_{\tau_R} = 0$ for any distinct $S$ and $S'$ in $\mathfrak{R}^+_n$.

Since the set $\mathfrak{R}^+_n$ is in bijection with set of positive roots $\Phi^+$, let us denote $S_x \in \mathfrak{R}^+_n$ as the corresponding object of each positive root $x \in \Phi^+$ and denote $\cR_x := \<S_x\>^\oplus_{\tau_R}$.
As such we have an additive thick subcategory $\cR_{\Phi^+}$ of $\cK$ defined by
\begin{align*}
\cR_{\Phi^+} &:= \< \mathfrak{R}^+_n \>^\oplus_{\tau_R} \\
&=  \cR_{\alpha_2} \oplus \cR_{s_2(\alpha_1)} \oplus \cR_{s_2s_1(\alpha_2)} \oplus \cdots \oplus\cR_{s_2s_1s_2...s_j(\alpha_i)}, \quad \text{where } i \neq j.
\end{align*}
Note that the split Grothendieck group of each $\cR_x$ (and hence $\cR_{\Phi^+}$) has a natural $\Omega_n[q^{\pm 1}, s^{\pm 1}]$-module structure given by
\[
\omega \cdot [S] = [S \otimes \Pi_1], \quad q\cdot [S] = [S\<1\>], \quad s \cdot [S] = [S[1]].
\]
Moreover, properties (i) and (ii) above guarantee that each $K_0(\cR_x)$ free of rank one generated by the Grothendieck classes $[S_x]$ and $K_0(\cR_{\Phi^+})$ is free of rank $n$ with the following decomposition:
\[
K_0(\cR_{\Phi^+}) =K_0( \cR_{\alpha_2}) \oplus K_0(\cR_{s_2(\alpha_1)}) \oplus  \cdots \oplus K_0(\cR_{s_2s_1s_2...s_j(\alpha_i)}), \quad \text{where } i \neq j.
\]
Motivated by these, we shall make the following definition:
\begin{definition}\label{defn: category of root supports}
We define the \emph{category of root $\tau_R$-supports} as the additive subcategory of $\cK$ given by
\[
\cR_{\Phi^+} := \bigoplus_{x \in \Phi^+} \cR_x \subset \bigoplus_{\phi \in \R} \cP(\phi),
\]
where $\cP$ is the slicing with respect to $\tau_R$.
If $T$ is an object in $\cR_x$ such that $\cR_x = \< T \>^\oplus_{\tau_R}$ and , the Grothendieck class $[T]$ forms a $\Omega[q^{\pm 1}, s^{\pm 1}]$-basis of $K_0(\cR_x)$, we call $T$ a \emph{$\tau_R$-generator} for $\cR_x$.
\end{definition}
}
\begin{remark}\label{rmk: root semistable are all}
Although a priori we only have $\bigoplus_{S \in \mathfrak{R}^+_n} \<S\>^\oplus_{ \tau_R} \subset \bigoplus_{\phi \in \R} \cP(\phi)$, we expect one could show that they are equal, following the methods in \cite{thomas_2006} and \cite{bridgeland_2009}.
Nonetheless, we will not need this result for the purpose of computing entropy of braids in $\B(I_2(n))$ later on as all objects considered will be shown to have their $\tau_R$-HN semistable pieces given by objects in $\bigoplus_{S \in \mathfrak{R}^+_n} \<S\>^\oplus_{ \tau_R}$.
\end{remark}


Clearly, the autoequivalences of $\cK$ need not restrict to autoequivalences on $\bigoplus_{S \in \mathfrak{R}^+_n} \<S\>^\oplus_{ \tau_R}$; they need not send semistable objects to semistable objects!
However, just as $s_2s_1$ acts on the root system by a $\frac{2\pi}{n}$ (clockwise) rotation, we have the following categorified analogue:
\begin{proposition} \label{gamma phase reducing}
The autoequivalence $\gamma^{\pm 1}:= (\sigma_{P_2}\sigma_{P_1})^{\pm 1} : \cK \ra \cK$ restricts to an (additive) autoequivalence on $\bigoplus_{S \in \mathfrak{R}^+_n} \<S\>^\oplus_{ \tau_R}$.
Moreover, $\gamma$ (resp $\gamma^{-1}$) lowers (resp. increases) the phase of each $\tau_R$-semistable object in $\bigoplus_{S \in \mathfrak{R}^+_n} \<S\>^\oplus_{ \tau_R}$ by $\frac{2}{n}$.
\end{proposition}
\begin{proof}
We will only lay out the proof for $\gamma$; the argument for $\gamma^{-1}$ follows similarly.

First consider the case where $n$ is odd.
Note that every indecomposable object in $\bigoplus_{S \in \mathfrak{R}^+_n} \<S\>^\oplus_{ \tau_R}$ is obtainable by applying the functors $[-],\<-\>$ and $-\otimes \Pi_a$ to the objects $S \in \mathfrak{R}^+_n$ (cf. \cref{semi-stable forms odd}).
Since $\gamma$ commutes with all three functors $[-],\<-\>$ and $-\otimes \Pi_a$, to that show $\gamma$ restricted to $\bigoplus_{S \in \mathfrak{R}^+_n} \<S\>^\oplus_{ \tau_R}$ maps into $\bigoplus_{S \in \mathfrak{R}^+_n} \<S\>^\oplus_{ \tau_R}$, it is sufficient to check that $\gamma(S)$ is in $\bigoplus_{S \in \mathfrak{R}^+_n} \<S\>^\oplus_{ \tau_R}$ for each $S \in \mathfrak{R}^+_n$.
Moreover, using the fact that $\tau_R$ is a $q$-stability condition respecting $\TLJ_n$, we have that 
\[
\phi(X\otimes \Pi_a \<k\>[\ell]) = \phi(X) + \ell - k
\]
for any semistable object $X$.
Thus, to show that $\gamma$ reduces the phase of all semistable objects in $\bigoplus_{S \in \mathfrak{R}^+_n} \<S\>^\oplus_{ \tau_R}$ by $\frac{2}{n}$, it is sufficient to show $\phi(\gamma(S)) = \phi(S) - \frac{2}{n}$ for each $S$ in \cref{semi-stable forms odd}.

Suppose $S=\gamma^k(P_2)$.
When $ k \neq \frac{n-1}{2}$, $\gamma(S)$ is in $\mathfrak{R}^+_n$ by definition and $\phi(\gamma(S)) = \phi(S) - \frac{2}{n}$.
When $k = \frac{n-1}{2}$, a simple computation using \cref{lemma for braid relation} shows that
\begin{align*}
S &= P_1\<n-2\>[n-2] \text{ and} \\
\gamma(S) &= \sigma_2(P_1) \otimes \Pi_{n-2} \<n\>[n-1],
\end{align*}
which is in $\bigoplus_{S \in \mathfrak{R}^+_n} \<S\>^\oplus_{ \tau_R}$.
Furthermore, $\phi(S) = \phi(P_1\<n-2\>[n-2]) = 0$, and
\[
\phi(\gamma(S)) = \phi(\sigma_2(P_1)) - 1 = (1 - \frac{2}{n}) - 1 = -\frac{2}{n}
\]
as required.

Now suppose instead $S = \gamma^{k'}\sigma_2(P_1)$.
Once again $\gamma(S)$ is clearly in $\mathfrak{R}^+_n$ when $ k' \neq \frac{n-3}{2}$ and $\phi(\gamma(S)) = \phi(S) - \frac{2}{n}$.
For $k' = \frac{n-3}{2}$, we apply \cref{lemma for braid relation} again to show that
\[
\gamma(S) = P_2 \otimes \Pi_{n-2}\<n\>[n-1],
\]
which is in $\bigoplus_{S \in \mathfrak{R}^+_n} \<S\>^\oplus_{ \tau_R}$.
Moreover, 
\[
\phi(S) = \phi(\gamma^{\frac{n-3}{2}}\sigma_2(P_1)) = \phi(\sigma_2(P_1)) - \frac{n-3}{2}\cdot \frac{2}{n} = \frac{1}{n}
\]
and hence
\[
\phi(\gamma(S)) = \phi(P_2) - 1 = (1- \frac{1}{n}) - 1 =  - \frac{1}{n}
\]
as required.

The case where $n$ is even follows similarly, with $\mathfrak{R}^+_n$ as defined in \cref{semi-stable forms even}.
\comment{
\[
\gamma(\gamma^{\frac{n}{2}-1}(P_2)) = P_2 \otimes \Pi_{n-2}\<n\>[n-1],
\]
and
\[
\gamma(\gamma^{\frac{n}{2}-1}\sigma_2(P_1)) = \sigma_2(P_1) \otimes \Pi_{n-2}\<n\>[n-1].
\]
}
\end{proof}

\section{A (naive) Coxeter-theoretic classification}
Before we begin our journey on the dynamical study of $\B(I_2(n))$ and its classification, we shall first provide a Coxeter-theoretic definition of the three possible types of braids appearing in the classification.

Let $\gamma := \sigma_2 \sigma_1$.
Recall that the center of $\B(I_2(n))$ is generated by 
\[
\chi = \chi_n := \begin{cases}
\gamma^n; &n \text{ is odd}; \\
\gamma^{\frac{n}{2}}; &n \text{ is even},
\end{cases}
\]
see \cite[Section 3]{michel96}.
Consider the following classification of braid elements:
\begin{definition} \label{defn: braid types}
Let $\beta \in \B(I_2(n))$.
We say that $\beta$ is \emph{periodic} if some power of $\beta$ is central.
We say $\beta$ is \emph{reducible} if $\beta$ is conjugate to a braid of the form $\sigma_i^k \chi^\ell$ for some $i \in \{1,2\}, \ell \in \Z$ and $k \in \Z \setminus \{0\}$.
If $\beta$ is not any of the above, we say that $\beta$ is \emph{pseudo-Anosov}.
\end{definition}

\begin{remark}
When $n$ is odd, $\sigma_1$ and $\sigma_2$ live in the same conjugacy class, so it is sufficient to say that $\beta$ is reducible if and only if $\beta$ is conjugate to a braid of the form $\sigma_1^k \chi^\ell$ (or $\sigma_2^k \chi^\ell$).
This is \emph{not true} when $n$ is even.
\end{remark}

Note that by definition, the types of braid elements are invariant under conjugation and the three types are mutually exclusive.
As of now, this may seem like a meaningless definition for the pseudo-Anosov braids -- being pseudo-Anosov just means that it is neither periodic nor reducible.
As we will see in the later section, the pseudo-Anosov braids will turn out to be the only braids with interesting categorical dynamics: they are the only braids with positive mass growth at $t=0$.

\begin{remark}
The fact that the three types are mutually exclusive is a special feature of the rank two generalised braid groups.
In general, we do not expect the types in the classification to be mutually exclusive (between the periodic and reducible braids), as is the case in the (classical) Nielsen-Thurston classification for mapping class groups).
\end{remark}

\section{Computing mass growth for periodic and reducible braids}
In this section, we first deal with the mass growths of braids that are simple to compute: the periodic braids and the reducible braids.

Before we start, let us recall some salient properties of mass growth and level-entropy in our case.
Firstly, recall that the heart of linear complexes $\cH$ is a $\<1\>[1]$-algebraic heart of $\cK$ (cf. \cref{defn: X1-algebraic heart}) and $\tau_R$ lies in the same connected component of Stab$(\cK)$ that contains the stability condition $\tau_0$ given in \cref{thm: entropy = mass growth}.
As such, \cref{thm: entropy = mass growth} tells us that the level-entropy and the mass growth with respect to $\tau_R$ agree, and can be computed as
\[
h_{\tau_R, t}(\cF) = \lim_{N\ra \infty} \frac{1}{N} \log m_{\tau_R, t}(\cF^N(G)) = h_t(\cF)
\]
for any $\X$-split generator $G$.
It turns out that we can actually compute the mass growth (equivalently level-entropy) from even simpler objects:
\begin{proposition} \label{prop: compute entropy from simpler object}
Let $X$ be an object in $\cK$ such that $X \otimes A$ is a $\X$-split generator of $\cK$ for some object $A$ in $\TLJ_n$.
Then 
\[
h_{\tau_R, t}(\cF) = \lim_{N\ra \infty} \frac{1}{N} \log m_{\tau_R, t}(\cF^N(X)) = h_t(\cF).
\]
In particular, we can choose $X := \beta(P_1 \oplus P_2)$, where $A:= \bigoplus_{i=0}^{n-2} \Pi_i$ and $\beta$ is any braid in $\B(I_2(n))$.
\end{proposition}
\begin{proof}
Let $X$ and $A$ be as given.
By definition of $\tau_R$, we see that
\[
PF\dim(A) \cdot m_{\tau_R, t}(X) = m_{\tau_R, t}(X \otimes A).
\]
Hence we can compute the level entropy by taking
\[
h_t(\cF) = \lim_{N\ra \infty} \frac{1}{N} \log PF\dim(A) m_{\tau_R, t}(\cF^N(X))
\]
for any $\cF$ respecting the $\TLJ_n$-module structure of $\cK$.
Since the limit that defines $h_t$ exists, the limit above exists and it follows that
\[
\lim_{N\ra \infty} \frac{1}{N} \log m_{\tau_R, t}(\cF^N(X)) = h_t(\cF).
\]
\end{proof}
Lastly, recall from \cref{prop: entropy properties} that level-entropy is invariant under conjugation. 
Since level-entropy and mass growth agree in our case, the mass growth of a braid $\beta$ can be computed from the mass growth of any braid within the same conjugacy class of $\beta$.

We are now prepared to take on the computations.
For this section, we will be computing mass growth using $X = P_1\oplus P_2$:
\[
h_{\tau_R,t}(\cF) = \lim_{N\ra \infty} \frac{1}{N} \log m_{\tau_R, t}(\cF^N(P_1 \oplus P_2)) = h_t(\cF) .
\]
Since mass growth and level-entropy agree, we will sometimes use $h_t$ to denote mass growth with respect to $\tau_R$ to reduce clutter.

\subsection{Mass growth of periodic braids}
We shall start with the periodic braids.
Firstly, note that
\begin{align*}
h_t(\<\ell k\>[\ell (k-2)])
&= \lim_{N\ra \infty} \frac{1}{N} \log m_{\tau_R, t}((P_1 \oplus P_2)\<\ell kN\>[\ell N(k-2)]) \\
&= \lim_{N\ra \infty} \frac{1}{N} \log e^{2\ell N t} m_{\tau_R, t}(P_1 \oplus P_2) \\
&= -2 \ell t
\end{align*}
for all $\ell \in \Z$.
On the other hand, one can show that 
\begin{equation} \label{eqn: gamma and shift}
\gamma^n(P_i) = P_i\<2n\>[2n-2]
\end{equation}
by applying \cref{lemma for braid relation} repeatedly.
Note that \cref{eqn: gamma and shift} also tells us that
\[
\gamma^{-n}(P_i) = P_i\<-2n\>[-(2n-2)].
\]

Hence for all $\ell \in \Z$, the mass growth of $\gamma^{\ell n}$ can be computed as follows:
\begin{align*}
h_t(\gamma^{\ell n}) 
&= \lim_{N\ra \infty} \frac{1}{N} \log m_{\tau_R, t}((\gamma^{\ell n})^N(P_1 \oplus P_2)) \\
&= \lim_{N\ra \infty} \frac{1}{N} \log m_{\tau_R, t}((P_1 \oplus P_2)\<\ell (2n)N\>[\ell N(2n-2)]) \\
&= h_t(\<\ell(2n)\>[\ell(2n-2)]),
\end{align*}
which gives us
\[
h_t(\gamma^{\ell n}) = -2\ell t.
\]

Let $\beta \in \B(I_2(n))$ be a periodic braid.
This means that $\beta^k = \gamma^{\ell n}$ for some $k \in \mathbb{N}, \ell \in \Z$.
Combining the above calculation with property (1) in \cref{prop: entropy properties}, we get
\begin{align*}
h_t(\beta) 
&= \frac{1}{k} h_t(\beta^k) \\
&= \frac{1}{k} h_t(\gamma^{\ell n}) \\
&= \frac{\ell}{k} h_t(\gamma^{n}) \\
&= -\frac{\ell}{k} 2t.
\end{align*}
Thus we conclude the following:
\begin{proposition} \label{prop: entropy of periodic}
Let $\beta \in \B(I_2(n))$ be a periodic braid, such that $\beta^k = \gamma^{\ell n}$ for some $k \in \mathbb{N}$ and $\ell \in \Z$.
Then the mass growth is given by
\[
h_{\tau_R,t} = h_t(\beta) = -\frac{\ell}{k} 2t.
\]
\end{proposition}

\subsection{Mass growth of reducible braids} 
We begin with computing the mass growth of $\sigma_i^{\pm 1}$. 

\begin{lemma}\label{lemma: entropy of spherical twist}
The mass growth of $\sigma_i$ is given by
\[
h_t(\sigma_i) = \begin{cases}
0, & \text{for } t \geq 0; \\
-t, &\text{for } t < 0.
\end{cases}
\]
The mass growth of $\sigma_i^{-1}$ is given by
\[
h_t(\sigma_i^{-1}) = \begin{cases}
t, & \text{for } t \geq 0; \\
0, &\text{for } t < 0.
\end{cases}
\]

\end{lemma}
\begin{proof}
We will only present the proof for the case $\sigma_2^{-1}$; the other cases can be computed in a similar fashion and are left as exercises for the reader.

A simple calculation shows that
\[
\sigma_2(P_1 \ra P_2 \otimes \Pi_1 \<-1\>[-1]) = P_1.
\]
Using this, we can fit $\sigma_2^{-1}(P_1)$ into the following distinguished triangle:
\[
P_2 \otimes \Pi_1 \<-1\>[-1] \ra \sigma_2^{-1}(P_1) \ra P_1 \ra \quad,
\]
which is actually the $\tau_R$-HN filtration of $\sigma_2^{-1}(P_1)$.
Note that $\sigma_2^{-k}(P_2) = P_2\<-2k\>[-k]$ is always a $\tau_R$-semistable object.
Applying $\sigma_2^{-k}$ to the distinguished triangle above, we get
\[
P_2 \otimes \Pi_1 \<-2k-1\>[-k-1] \ra \sigma_2^{-k-1}(P_1) \ra \sigma_2^{-k}(P_1) \ra \quad.
\]
By induction, one can show that
\[
\ceil*{\sigma_2^{-k}(P_1)} \leq \phi(P_2) + k = \phi(P_2 \otimes \Pi_1 \<-2k-1\>[-k-1]).
\]
Hence we know that
\[
m_{\tau_R,t}(\sigma_2^{-k-1}(P_1)) = m_{\tau_R,t}(P_2 \otimes \Pi_1 \<-2k-1\>[-k-1]) + m_{\tau_R,t}(\sigma_2^{-k}(P_1)).
\]
Using induction again, we get that
\begin{equation} \label{eqn: reducible lemma sigma 2 inv}
m_{\tau_R,t}(\sigma_2^{-k-1}(P_1)) = 1+ \sum_{i=0}^k w e^{\frac{n-1}{n}\pi t} \cdot e^{(k-i)t}
\end{equation}
with $w := 2\cos(\frac{\pi}{n})$.
On the other hand, we have that
\begin{equation} \label{eqn: reducible lemma sigma 2 inv p2}
m_{\tau_R, t}(\sigma_2^{-k-1}(P_2)) = e^{\frac{n-1}{n}\pi t} \cdot e^{(k+1)t}.
\end{equation}

Let us first fix $t \leq 0$. 
It is easy to see that the sum of the RHS of \cref{eqn: reducible lemma sigma 2 inv} and \cref{eqn: reducible lemma sigma 2 inv p2} approaches a constant as $k \ra \infty$.
Hence
\begin{align*}
h_t(\sigma_2^{-1}) 
&= 
\lim_{k\ra \infty} \frac{1}{k+1} \log m_{\tau_R, t}(\sigma_2^{-k-1}(P_1 \oplus P_2)) \\
&= \lim_{k\ra \infty} \frac{1}{k+1} \log \left( m_{\tau_R, t}(\sigma_2^{-k-1}(P_1)) + m_{\tau_R, t}(\sigma_2^{-k-1}(P_2)) \right) \\
&= 0.
\end{align*}
Let us now consider the case where $t > 0$.
Adding \cref{eqn: reducible lemma sigma 2 inv} and \cref{eqn: reducible lemma sigma 2 inv p2} we get that
\begin{align*}
&m_{\tau_R, t}(\sigma_2^{-k-1}(P_1)) + m_{\tau_R, t}(\sigma_2^{-k-1}(P_2)) \\
&= e^{kt} 
	\left( 
	e^{-kt} + e^{(\frac{n-1}{n}\pi +1)t} + 
	\sum_{i=0}^k w e^{(\frac{n-1}{n}\pi -i)t}
	\right).
\end{align*}
Once again we see that $e^{-kt} + e^{(\frac{n-1}{n}\pi +1)t} + 
	\sum_{i=0}^k w e^{(\frac{n-1}{n}\pi -i)t}$ approaches a constant as $k \ra \infty$, hence 
\[
h_t(\sigma_2^{-1}) = \lim_{k\ra \infty}  \frac{1}{k+1} \log e^{kt} = t.
\]
\end{proof}
\begin{remark}
The entropy of spherical twists (in the sense of Seidel-Thomas \cite{thomas_seidel_2001}, \emph{not} our generalised version as defined in \cite{anno_logvinenko_2017}) have been computed in \cite{ikeda_2020} using mass growth with respect to the stability condition $\tau_0$ as in \cref{thm: entropy = mass growth}.
However, we note here that the method used there can indeed be adapted to our case as well.
\end{remark}

\begin{remark} \label{rem: entropy sum of itself and inverse}
We see that $h_t(\sigma_i) + h_t(\sigma_i^{-1}) = |t|$, which is non-zero for $t\neq 0$.
Hence it is \emph{not true} in general that $h_t(\cF^{-1}) = - h_t(\cF)$.
\end{remark}

\begin{lemma}
If $\cF$ is either the functor $\<1\>, [1]$ or $- \otimes C$, then
\[
\frac{m_{\tau_R, t}(\cF^N(A))}{m_{\tau_R, t}(A)} = \frac{m_{\tau_R, t}(\cF^N(B))}{m_{\tau_R, t}(B)}
\]
for any $N \in \mathbb{N}$ and any objects $A$ and $B$.
\end{lemma}
\begin{proof}
By definition of the stability condition $\tau_R$, we have that
\begin{enumerate}[(i)]
\item $m_{\tau_R, t}(A[1]) = m_{\tau_R, t}(A)e^t$; 
\item $m_{\tau_R, t}(A\<1\>) = m_{\tau_R, t}(A)e^{-t}$; and
\item $m_{\tau_R, t}(A\otimes C) = PF\dim(C)m_{\tau_R, t}(A)$, for any $C \in \TLJ_n$.
\end{enumerate}
We shall abuse notation and denote:
\[
m_{\tau_R, t}(\cF) := 
\begin{cases}
e^t, &\text{for }\cF = [1];\\
e^{-t}, &\text{for }\cF = \<1\>;\\
PF\dim(C), &\text{for }\cF  = - \otimes C.
\end{cases}
\]
It follows that
\begin{align*}
m_{\tau_R,t}(\cF^N(A)) 
&= \sum_i m_{\tau_R,t}(\cF^N(E_i)) \\
&= \sum_i m_{\tau_R, t}(\cF)^N \cdot m_{\tau_R,t}(E_i) \\
&= m_{\tau_R, t}(\cF)^N \cdot m_{\tau_R,t}(A).
\end{align*}
In particular,
\[
\frac{m_{\tau_R, t}(\cF^N(A))}{m_{\tau_R, t}(A)} = m_{\tau_R, t}(\cF)^N = \frac{m_{\tau_R, t}(\cF^N(B))}{m_{\tau_R, t}(B)}
\]
for any objects $A$ and $B$ as required.
\end{proof}

\begin{proposition} \label{prop: entropy additive}
Suppose $\cF$ is either $\<k\>, [\ell]$ or $- \otimes \Pi_a$.
Then for any endofunctor $\cG: \cK \ra \cK$ commuting with $\cF$, we get that
\[
h_t(\cF\circ \cG) = h_t(\cF) + h_t(\cG).
\]
\end{proposition}
\begin{proof}
Computing entropy with mass growth, we get
\[
h_t(\cF \circ \cG) = \lim_{N\ra \infty} \frac{1}{N} \log m_{\tau_R, t}(\cF^N \circ \cG^N (G))
\]
for some $\X$-split generator $G$.
Applying the lemma above, we get that
\[
h_t(\cF \circ \cG) = \lim_{N\ra \infty} \frac{1}{N} \log m_{\tau_R, t}(\cG^N (G)) \frac{m_{\tau_R, t}(\cF^N(G))}{m_{\tau_R, t}(G)}.
\]
The result now follows from a simple limit manipulation.
\end{proof}

For the following, recall that $\chi$ is the generator of the centre of $\B(I_2(n))$, namely $\chi = \gamma^n$ for $n$ odd and $\chi = \gamma^{\frac{n}{2}}$ for $n$ even.
We can compute the mass growths of reducible braids as follows:
\begin{proposition} \label{entropy reducible}
Let $\beta \in \B(I_2(n))$ be a reducible braid, i.e. $\beta = \sigma_i^k\chi^\ell$ for some $k \in \Z \setminus \{0\}, \ell \in \Z$.
Then the mass growth (equivalently level-entropy) of $\beta$ is given as follows:
\begin{enumerate}
\item for $n$ odd, $k \geq 1$
\[
h_t(\beta) = 
	\begin{cases}
	 - 2 \ell t, & t \geq 0;\\
	-kt - 2 \ell t, & t < 0,
	\end{cases}
\]
\item for $n$ odd, $k \leq -1$
\[
h_t(\beta) = 
	\begin{cases}
	 |k|t - 2 \ell t, & t \geq 0;\\
	 - 2 \ell t, & t < 0,
	\end{cases}
\]
\item for $n$ even, $k \geq 1$
\[
h_t(\beta) = 
	\begin{cases}
	 -  \ell t, & t \geq 0;\\
	-kt -  \ell t, & t < 0,
	\end{cases}
\]
\item for $n$ even, $k \leq -1$
\[
h_t(\beta) = 
	\begin{cases}
	 |k|t -  \ell t, & t \geq 0;\\
	 -  \ell t, & t < 0.
	\end{cases}
\]
\end{enumerate}
\end{proposition}
\begin{proof}
Let $X := P_1 \oplus P_2$.
The only subtle difference here is that when $n$ is even, $\chi = \gamma^{\frac{n}{2}}$ is the generator of the center, where we have that
\[
\gamma^{\frac{n}{2}}(P_i) = P_i \otimes \Pi_{n-2} \<n\> [n-1].
\]
This tells us that
\[
\chi^\ell(X) =  \begin{cases}
X\<\ell 2n\>[\ell(2n-2)], &\text{for $n$ odd}; \\
X \otimes \Pi_{n-2} \<\ell n\>[\ell(n-1)], &\text{for $n$ even, $\ell$ odd}; \\
X\<\ell n\>[\ell(n-1)], &\text{for $n$ even, $\ell$ even}.
\end{cases}
\]
Nevertheless, with $\tau_R$ respecting $\TLJ_n$ and $PF\dim(\Pi_{n-2}) = 1$, we have $m_{\tau_R,t}(Y \otimes \Pi_{n-2}) = m_{\tau_R,t}(Y)$ for any object $Y$.
Thus we get that
\[
h_t(\cG \circ \chi^\ell) = \begin{cases}
h_t(\cG \circ \<\ell 2n\>[\ell(2n-2)]), &\text{for $n$ odd}; \\
h_t(\cG \circ \<\ell n\>[\ell(n-1)]), &\text{for $n$ even},
\end{cases}
\]
for any $\cG \in \B(I_2(n))$.
Now the result follows from applying \cref{prop: entropy additive} with $\cG = \sigma_i^k$, where the mass growths of $\sigma_i^{\text{sgn}(k)}$ and $\chi^\ell$ can be computed from \cref{lemma: entropy of spherical twist} and \cref{prop: entropy of periodic} respectively.
\end{proof}

Combining the results from \cref{entropy reducible} and \cref{prop: entropy of periodic}, we deduce that all periodic braids and reducible braids have zero mass growth at $t=0$:
\begin{corollary} \label{periodic reducible entropy 0 when t=0}
If $\beta \in \B(I_2(n))$ is either periodic or reducible, then $h_{\tau_R,0}(\beta) = 0 = h_0(\beta)$.
\end{corollary}

\begin{remark}
As such, we see that mass growth at $t=0$ \emph{cannot} differentiate between periodic braids and reducible braids; however, mass growth $h_t$ as a whole function \emph{can}:  
$h_t(\beta)$ is a linear function (with respect to $t$) when $\beta$ is periodic (\cref{prop: entropy of periodic}) and $h_t(\beta)$ is \emph{strictly piecewise} linear when $\beta$ is reducible (\cref{entropy reducible}).
\end{remark}

\section{An algorithmic classification of braids and mass growth of arbitrary braids} \label{sect: dynamics}
This section is where we prove our main result, which is essentially an algorithm that does two things: given a braid $\beta \in \B(I_2(n))$,
\begin{enumerate}
\item the algorithm decides the type of $\beta$: periodic, reducible or pseudo-Anosov; and moreover
\item the algorithm provides a way to compute the mass growth of $\beta$: when $\beta$ is periodic (resp. reducible), it provides an expression of some conjugate of $\beta$ which fits definition of periodic (resp. reducible), where results from previous sections allow the computation of its mass growth. 
When $\beta$ is pseudo-Anosov, it provides a matrix whose (logarithm of its) Perron-Frobenius eigenvalue is the mass growth of $\beta$.  
\end{enumerate}
As such, by the end of this section, we will be able to compute the mass growth of any given braid in $\B(I_2(n))$.
This algorithm will rely on a mass automaton (see \cref{sect: mass automaton}) that we will be constructing for each $I_2(n)$.

As before, we shall fix $\tau_R$ to be the root stability condition associated to $I_2(n)$ and denote $\gamma := \sigma_2\sigma_1 \in \B(I_2(n))$ throughout this whole section. 

\subsection{Constructing $\tau_R$-mass automata} \label{sect: mass automaton construction}
We start by building the main tool that will be used in the algorithm, namely a $\tau_R$-mass automaton for $\B(I_2(n))$ for each $n \geq 3$.
The construction of the $\tau_R$-mass automaton for $\B(I_2(n))$ will depend on the parity of $n$ (odd or even).
Although they differ slightly from each other, the arguments and methods used to show that the construction satisfies the required properties of a mass automaton will be similar.
As such, we shall focus on the odd $n$ case, giving full proofs for the statements required; whereas in the even $n$ case we will provide the construction and comment on the proof whenever we feel is necessary.
The reader may wish to read both of the subsections simultaneously.

Recall from \cref{sect: root and semistable} the set of $\tau_R$-semistable corresponding to positive lift of positive roots 
\[
\mathfrak{R}^+_n = \{P_2, \sigma_2(P_1), \sigma_2\sigma_1\sigma_2(P_1), ..., \sigma_2\sigma_1...\sigma_j(P_k)\}.
\] 
For $\mathfrak{S}$ a set of $\tau_R$-semistable objects, we will use the notation $\<\mathfrak{S}\>^\oplus_{\tau_R}$ to denote the smallest additive subcategory of $\cK$ containing the objects in $\mathfrak{S}$ that is closed under the $\TLJ_n$-action, the triangulated shifts $[-]$ and the internal grading shifts $\<-\>$.

We will need the following notion of the support of an object:
\begin{definition}
Let $A$ be an object in $\cK$ with $\tau_R$-HN filtration consisting of $\tau_R$-semistable pieces $E_i$.
We define the \emph{${\tau_R}$-support of} $A$ as
\[
\supp_{\tau_R}(A) := \bigoplus_i E_i \in ob(\cK).
\]
\end{definition}
\begin{remark}
We remind the reader that the $\tau$-HN filtration of an object is only unique up to isomorphism, and so the $\tau$-support of an object is defined as an element in the set of isomorphism classes.
The objects we consider will always have their support in ob$\left( \bigoplus_{S \in \mathfrak{R}^+_n} \<S\>^\oplus_{ \tau_R} \right)$ (although we expect this to be true for all objects, see \cref{rmk: root semistable are all}).
\end{remark}
%
%
\noindent
The following properties of ${\tau_R}$-supports are immediate consequences of the definition of support and the fact that $\tau_R$ is a $q$-stability condition respecting $\TLJ_n$:
\begin{enumerate}[-]
\item (additivity) $\supp_{\tau_R}(A\oplus B) = \supp_{\tau_R}(A) \oplus \supp_\tau(B)$;
\item ($[1]$ equivariant) $\supp_{\tau_R}(A[m]) = \supp_{\tau_R}(A)[m]$ for all $m \in \Z$;
\item ($\<1\>$ equivariant) $\supp_{\tau_R}(A\<m\>) = \supp_{\tau_R}(A)\<m\>$ for all $m \in \Z$;
\item ($\TLJ_n$-action equivariant) $\supp_{\tau_R}(A \otimes C) = \supp_{\tau_R}(A) \otimes C$ for all $C \in \TLJ_n$; and
\item (support computes mass) $m_{{\tau_R},t}(A) = m_{{\tau_R},t}(\supp_{\tau_R}(A))$.
\end{enumerate} 

Clearly, it is not true in general that $\cF(\supp_{\tau_R}(A))$ agrees with $\supp_{\tau_R}(\cF(A))$ for endofunctors (or even autoequivalences) on $\cK$, but the braid functor $\gamma$ does satisfy this:
\begin{proposition}\label{gamma sends semistable pieces}
Suppose $A$ is an object in $\cK$ with $\supp_{\tau_R}(A) \in ob(\bigoplus_{S \in \mathfrak{R}^+_n} \<S\>^\oplus_{ \tau_R})$.
Applying $\gamma^{\pm 1}$ to the $\tau_R$-HN filtration of $A$ gives the $\tau_R$-HN filtration of $\gamma^{\pm 1}(A)$.
In particular,
\[
\supp_{\tau_R}(\gamma^{\pm 1}(A)) = \gamma^{\pm 1}(\supp_{\tau_R}(A)).
\]
\end{proposition}
\begin{proof}
Let $E_i$ be the $\tau_R$-HN semistable pieces of $A$ (which are objects in $\bigoplus_{S \in \mathfrak{R}^+_n} \<S\>^\oplus_{ \tau_R}$ by the assumption on $A$).
\cref{gamma phase reducing} shows that every semistable pieces of $A$ has its phase reduced (resp. increased) by the same amount $\frac{2}{n}$ when $\gamma$ (resp. $\gamma^{-1}$) is applied.
Hence the $\tau_R$-HN semistable pieces of $\gamma^{\pm 1}(A)$ are given by $\gamma^{\pm 1}(E_i)$ as required.
\end{proof}

\begin{definition}\label{defn: bracket vertex automaton}
Let $\{ X_1, X_2\}$ be a pair of $\tau_R$-semistable objects in $\cK$ such that $X_i \in \cP(\phi_i)$ 
and $\phi_1 - \phi_2  \not\in \Z$.
We define $[X_1, X_2] \subset ob(\cK)$ to be the subset consisting of objects $A$ such that $\supp_{\tau_R}(A) \in ob \left(\<\{X_1,X_2\}\>^\oplus_{\tau_R} \right)$.
The assumption on $\phi_i$ implies that we can write $m_{\tau_R,t}(A)$ as a unique $\R[e^{\pm t}]$-linear sum of $m_{\tau_R,t}(X_1)$ and $m_{\tau_R,t}(X_2)$, namely:
\[
m_{\tau_R,t}(A) \in \R[e^{\pm t}] \cdot m_{\tau_R,t}(X_1) \oplus \R[e^{\pm t}] \cdot m_{\tau_R,t}(X_2) \subset \R^\R
\]
for all $A \in [X_1, X_2]$.
\end{definition}

Finally, the following lemma will also be needed in the construction of mass automata (in particular, to apply \cref{sufficient condition for geodesic}), which shows that our objects are close cousins of objects in a 2-Calabi-Yau category.
\begin{lemma}\label{lemma: 2CY variant}
Let $X$ and $Y$ be objects in $\cK$ which lies in the orbit of $\{P_i \otimes \Pi_a : i \in \{1,2\}, 0 \leq a \leq n-2\}$ under the action of $\Aut(\cK)$.
Then
\[
\bigoplus_{k \in \Z} \Hom_\cK (X,Y\<k\>[k+1]) \cong \bigoplus_{k \in \Z} \Hom_\cK (Y,X\<k\>[k+1]).
\]
\end{lemma}
\begin{proof}
By applying the appropriate autoequivalence, we may assume without lost of generality that $X= P_i \otimes \Pi_a$ for some $i \in \{1,2\}$ and $0 \leq a \leq n-2$.
Then using the adjunctions between $P_i \otimes \Pi_a$ and shifts of ${}_iP \otimes \Pi_a$ (cf. \cref{biadjoint pair}), we get
\begin{align*}
\Hom_\cK (P_i \otimes \Pi_a, Y[1]) 
&\cong \Hom_{\TLJ_n} (\Pi_0, \Pi_a \otimes {}_iP \otimes_\I Y[1]) \\
&\cong \Hom_{\TLJ_n} (\Pi_a \otimes {}_iP \otimes_\I Y[1], \Pi_0)  \\
&\cong \Hom_\cK (Y, P_i \otimes \Pi_a\<-2\>[-1]). 
\end{align*}
Taking the total morphism space with respect to $\<1\>[1]$, we see that
\begin{align*}
\bigoplus_{k \in \Z}  \Hom_\cK (P_i \otimes \Pi_a, Y\<k\>[k][1]) 
&\cong \bigoplus_{k \in \Z} \Hom_\cK (Y\<k\>[k], P_i \otimes \Pi_a\<-2\>[-1]) \\
&\cong \bigoplus_{k \in \Z} \Hom_\cK (Y, P_i \otimes \Pi_a\<-k-2\>[-k-1]) \\
&\cong \bigoplus_{k \in \Z} \Hom_\cK (Y, P_i \otimes \Pi_a\<k\>[k+1]).
\end{align*}
\end{proof}

\subsubsection{Odd $n$ case} \label{odd n automaton}
In this subsection, fix $n$ to be odd.
We shall construct a $\tau_R$-mass automaton $\Lambda_{\tau_R}(n)$ for $\B(I_2(n))$ as follows.
The underlying $\B(I_2(n))$-automaton $\Theta$ is defined by the quiver with $n$ vertices $v_0, v_1, ..., v_{n-1}$.
For each $j$, all the incoming braid-labelled arrows into $v_j$ are defined as in \Cref{fig:HN automaton for odd n}.
One can refer to \cref{example n=5} for an example of the complete $\tau_R$-mass automaton for $n=5$.
\begin{figure}[H]
\[\begin{tikzcd}
	&& {v_j} &&& {v_{j+n-1}} \\
	{v_{j+1}} \\
	{v_{j+2}} &&&&& {v_{j+\frac{n+3}{2}}} \\
	&& {v_{j+\frac{n-3}{2}}} & {v_{j+\frac{n-1}{2}}} & {v_{j+\frac{n+1}{2}}}
	\arrow["{\gamma^{-1}}" description, from=2-1, to=1-3, curve={height=-20pt}]
	\arrow["{\gamma}" description, from=1-6, to=1-3, curve={height=20pt}]
	\arrow["\sigma_{\gamma^j(P_1)}" ', from=1-3, to=1-3, loop]
	\arrow["\sigma_{\gamma^j(P_1)}" description, from=2-1, to=1-3]
	\arrow["\sigma_{\gamma^j(P_1)}" description, from=3-1, to=1-3]
	\arrow["\sigma_{\gamma^j(P_1)}" description, from=4-3, to=1-3]
	\arrow["\sigma_{\gamma^j(P_1)}" description, from=4-5, to=1-3]
	\arrow["\sigma_{\gamma^j(P_1)}" description, from=1-6, to=1-3]
	\arrow["\sigma_{\gamma^j(P_1)}" description, from=3-6, to=1-3]
	\arrow["{\ddots}" description, from=3-1, to=4-3, phantom, no head]
	\arrow["{\vdots}" description, from=1-6, to=3-6, phantom, no head]
\end{tikzcd}\] 
\caption{The $\tau_R$-mass automaton $\Lambda_{\tau_R}(n)$ for odd $n$, showing only the incoming arrows of $v_j$. All the subscripts of the vertices are taken modulo $n$. Note that there is indeed no arrow with the label $\sigma_{\gamma^j(P_1)}$ from $v_{j+\frac{n-1}{2}}$ to $v_j$.}
\label{fig:HN automaton for odd n}
\end{figure}

The $\Theta$-subset $\cS: \Theta \ra $ Sets of $ob(\cK)\circ \Theta$ is defined by 
\[
\cS(v_j) := [\gamma^j P_1, \gamma^j P_2] \quad (\text{see \cref{defn: bracket vertex automaton}}).
\]
The $\Theta$-representation $\cM: \Theta \ra \R[e^{\pm t}]$-mod is defined as follows. 
Each $\cM(v_j)$ is a rank two free module over $ \R[e^{\pm t}]$ defined by
\[
\cM(v_j) 
:= \R[e^{\pm t}] \cdot m_{\tau_R,t}(X_1) \oplus \R[e^{\pm t}] \cdot m_{\tau_R,t}(X_2) \subset \R^\R \quad (\text{see \cref{defn: bracket vertex automaton}}).
\]
Note that elements in $\cM(v_j)$ will sometimes be written as column vectors, with the first entry representing the coefficient of $m_{\tau_R,t}(\gamma^j P_1)$ and the second representing the coefficient of $m_{\tau_R,t}(\gamma^j P_2)$.
For each arrow $a : v_j \ra v_k$ with label $\cS(a) \in \B(I_2(n))$, $\cM(a)$ is the $\R[e^{\pm t}]$-linear map uniquely defined on the basis element by
\[
m_{\tau_R,t}(\gamma^j P_i) ) \mapsto m_{\tau_R,t} \left( \cS(a)(\gamma^j P_i) \right).
\]
Similarly, $\cM(a)$ will sometimes be written as a two by two matrix with entries in $\R[e^{\pm 1}]$.
The natural transformation $\iota: \cS \ra \cM$ is given by 
\[
\iota_v (X) := m_{\tau_R,t}\left( X \right) \in \cM(v)
\]
for each vertex $v$ in $\Theta$ and each $X \in \cS(v)$.
The map $\mathfrak{m}_{v_j}: \cM(v_j) \hookrightarrow \R^\R$ is the natural inclusion for each vertex $v_j$.
\comment{
defined by
\[
\begin{bmatrix}
a_1 \\
a_2
\end{bmatrix} \mapsto 
a_1 m_{\tau_R, t} (\gamma^j P_1) + a_2 m_{\tau_R, t} (\gamma^j P_2)
\]
for each vertex $v_j$.
}

A priori these may not be well-defined and may not define a $\tau_R$-mass automaton; we shall make a check list of what is required here.
\begin{enumerate}[$\square$]
\item Firstly, it is easy to check that for each $j$, the pair $\{\gamma^j P_1, \gamma^j P_2\}$ is a pair of (indecomposable) semistable objects in $\bigoplus_{S \in \mathfrak{R}^+_n} \< S\>^\oplus_{\tau_R}$ with phases satisfying the required assumption in \cref{defn: bracket vertex automaton}. 
\item To show that $\cS$ is a $\Theta$-subset of $ob(\cK)\circ \Theta$, we need that for each arrow $a : v_j \ra v_k$ labelled by $\cS(a) \in \B(I_2(n))$, 
\begin{equation} \label{eqn: checklist}
\cS(a)(X) \in [\gamma^k P_1, \gamma^k P_2] 
\end{equation}
for all $X \in [\gamma^j P_1, \gamma^j P_2]$.
This follows from \cref{gamma loops V} for arrows labelled by $\gamma^{\pm 1}$ and \cref{cor: conditions of automaton odd} (ii) for all other arrows.
\item To show that $\cM$ is actually a well-defined functor, the only thing requiring a proof is that $m_{\tau_R,t} \left( \cS(a)(\gamma^j P_i) \right)$ actually lives in $\R[e^{\pm t}]^{ \{ [\gamma^k P_1], [\gamma^k P_2] \} }$ for each arrow $a : v_j \ra v_k$.
It suffices to show that
\[
\cS(a)(\gamma^j P_i) \in [\gamma^k P_1, \gamma^k P_2] 
\]
for both $i =1,2$, but this is a weaker statement than \cref{eqn: checklist} and is therefore also showed by \cref{gamma loops V} and \cref{cor: conditions of automaton odd} (ii).
\item To show that $\iota$ is a natural transformation between $\cS$ and $\cM$, we need to show that for each arrow $a : v_j \ra v_k$ labelled by $\cS(a)$, the following diagram commutes:
\[
\begin{tikzcd}
\cS(v_j) \ar[r, "\iota_{v_j}"] \ar[d, swap, "\cS(a)"] 
	& \cM(v_j) \ar[d, "\cM(a)"] \\
\cS(v_k) \ar[r, "\iota_{v_k}"] 
	&\cM(v_k).
\end{tikzcd}
\]
This follows from applying the properties of supports to the stronger statement \cref{cor: conditions of automaton odd} (i).
\item The fact that 
$\mathfrak{m}_{v_j} (\iota_{v_j}(X)) = m_{\tau_R,t}(X)$
is immediate from the definition.
\end{enumerate}
The series of (rather computational intensive) results to get to the required \cref{gamma loops V} and \cref{cor: conditions of automaton odd} shall occupy the rest of the subsection.
We shall start with the simple one:
\comment{
the underlying $H$-automaton $\Theta$ is generated by the following $H$-labelled quiver $\Gamma_n$:
\[
\begin{tikzcd}
v_0 
	\ar[r, "\gamma", yshift = 1ex] 
	\ar[rrrr, "\gamma^{-1}", bend right=60, yshift= -2ex, swap]&	 
v_1 
	\ar[l, "\gamma^{-1}", yshift = -1ex] 
	\ar[r, "\gamma", yshift = 1ex] & 
\cdots 
	\ar[l, "\gamma^{-1}", yshift = -1ex] 
	\ar[r, "\gamma", yshift = 1ex] &
v_{n-2} 
	\ar[l, "\gamma^{-1}", yshift = -1ex] 
	\ar[r, "\gamma", yshift = 1ex] & 
v_{n-1} 
	\ar[l, "\gamma^{-1}", yshift = -1ex] 
	\ar[llll, "\gamma", bend left=60, swap]
\end{tikzcd}
\]
The $\Theta$-subset $\cS$ of $ob(\cK) \circ \Theta$ is defined by by $\cS(v_j) = [\gamma^j P_1, \gamma^j P_2]$, where the face that it is a $\Theta$-subset follows from (2) and (3) above.
The $\Theta$-representation $\cM$ is given by $\cM(v_j) = \R[e^{\pm t}]^{\{\gamma^j P_1, \gamma^j P_2\}}$ and $\cM(a) = \id = I_2$ for all arrows $a$ that are not between $v_0$ and $v_{n-1}$, and
\[
\cM(v_{n-1}) \xra{\cM(\gamma) = e^{-2t}\id} \cM(v_0), \quad \cM(v_0) \xra{\cM(\gamma^{-1}) = e^{2t}\id} \cM(v_{n-1}).
\]
Finally, using the fact that
\[
\gamma^n P_i \cong P_i\<2n\>[2n-2]
\]
for both $i \in \{1,2\}$, we see that $\iota_{v_j} := i_{[\gamma^j P_1, \gamma^j P_2]}$ defines a natural transformation between $\cS$ and $\cM$ .
Moreover, we have $\widehat{m}_{v_j} := \begin{bmatrix}
m_{\tau_R, t}(\gamma^j P_1) & m_{\tau_R, t}(\gamma^j P_2)
\end{bmatrix}
$ satisfies
\[
m_{\tau_R, t}(A) = \widehat{m}_{v_j}( \iota_{v_j}(A))
\]
for each $A \in \cS(v_j)$.
}

\comment{
Recall the additive subcategory $\mathfrak{S}$ as defined in \cref{defn: semistable objects category} for $n$ odd.
Before we construct the automaton, we shall consider a nice set of additive subcategories of $\mathfrak{S}$ which are also $\tau_R$-semistable basis categories over $\gTLJ_n$.

Consider the subset $\cV_0 = \{P_1, P_2\}$ of ob($\mathfrak{S}$); note that ob$(\mathfrak{S})$ contains $\gamma^{\frac{n-1}{2}}(P_2) = P_1 \otimes \Pi_{n-2} \<n-2\>[n-2]$, hence it contains $P_1$.
For $0 \leq i \leq n-1$, take $\cV_i \subseteq $ ob$(\mathfrak{S})$ to be the image of $\gamma^i$ applied $\cV_0$.
It is easy to check that $\< [\cV_i] \>_{K_0(\gTLJ_n)[s,s^{-1}]} \subset K_0(\mathfrak{S})$ is a free $K_0(\gTLJ_n)[s,s^{-1}]$-module of rank two with basis $\{ [\gamma^i(P_1)], [\gamma^i(P_2)]\}$.
Hence $\cV_i$ are all finitely free over $K_0(\gTLJ_n)[s,s^{-1}]$.
}

\begin{proposition} \label{gamma loops V}
Suppose $A \in [\gamma^jP_1, \gamma^jP_2]$. Then 
\[
\gamma^{\pm 1}(A) \in [\gamma^{(j \pm 1 \mod n)} P_1, \gamma^{(j \pm 1 \mod n)} P_2].
\]
\end{proposition}
\begin{proof}
Applying \cref{lemma for braid relation} repeatedly, we obtain
\[
\gamma^n(P_i) = P_i \<2n\>[2n-2]
\]
for both $i = 1,2$.
The result is now a direct consequence of \cref{gamma sends semistable pieces}.
\end{proof}

The following lemma contains most of the computational result that we will require:
\begin{lemma} \label{lemma: phases of object odd}
For $n$ odd, we have
\begin{align*}
\sigma_1(\gamma^j(P_1) \otimes \Pi_a) 
&\in
	\begin{cases}
	\cP(-1), &\text{ for } j = 0; \\
	\cP([-1, -\frac{1}{n}]), &\text{ for } 1 \leq j \leq \frac{n-3}{2}\\
	\cP(-\frac{1}{n}), &\text{ for } j = \frac{n-1}{2}; \\
	\cP([-2, -\frac{1}{n}-1]), &\text{ for } \frac{n+1}{2} \leq j \leq n-1.
	\end{cases}
\\
\sigma_1(\gamma^j(P_2) \otimes \Pi_a)
&\in 
	\begin{cases}
	\cP([0, \frac{n-1}{n}]), &\text{ for } 0 \leq j \leq \frac{n-3}{2}; \\
	\cP(-1), &\text{ for } j = \frac{n-1}{2}; \\
	\cP([-1, -\frac{1}{n}]), &\text{ for } \frac{n+1}{2} \leq j < n-1; \\
	\cP(-\frac{1}{n}), &\text{ for } j = n-1.
	\end{cases}
\end{align*}
\end{lemma}
\begin{proof}
One can compute the following using \cref{lemma for braid relation}:
\begin{equation}\label{eqn: sigma1 on V_j generator 1}
\begin{split}
&\sigma_1(\gamma^j(P_1)) \\
&\cong 
	\begin{cases}
	P_1 \<2\>[1], & j = 0; \\
	P_1 \otimes \Pi_{2j} \<2j+2\>[2j+1] \ra P_2 \otimes \Pi_{2j-1} \<2j+1\>[2j], & 1 \leq j \leq \frac{n-3}{2}\\
	P_2\otimes \Pi_{n-2} \<n\>[n-1], &j = \frac{n-1}{2}; \\
	P_1 \otimes \Pi_{2n-2j-2} \<2j+2\>[2j] \ra P_2 \otimes \Pi_{2n-2j-1} \<2j+1\>[2j-1], & \frac{n+1}{2} \leq j \leq n-1.
	\end{cases};
\end{split}
\end{equation}
and
\begin{equation}\label{eqn: sigma1 on V_j generator 2}
\begin{split}
&\sigma_1(\gamma^j(P_2)) \\
&\cong 
	\begin{cases}
	P_1 \otimes \Pi_{2j+1} \<2j+1\>[2j+1] \ra P_2 \otimes \Pi_{2j} \<2j\>[2j], & 0 \leq j \leq \frac{n-3}{2}; \\
	P_1 \otimes \Pi_{n-2} \<n\>[n-1],  &j = \frac{n-1}{2}; \\
	P_1 \otimes \Pi_{2n-2j-3} \<2j+1\>[2j] \ra P_2 \otimes \Pi_{2n-2j-2} \<2j\>[2j-1], & \frac{n+1}{2} \leq j < n-1; \\
	P_2 \otimes \Pi_{2n-2j-2} \<2j\>[2j-1], & j = n-1.
	\end{cases}
\end{split}
\end{equation}
The result now follows from looking at the $\tau_R$-HN filtration of the objects, and the fact that tensoring $\cP(\phi) \otimes \Pi_a \subseteq \cP(\phi)$ for all $\phi \in \R$.
\end{proof}

The following technical lemma will be crucial in the proof of \cref{sigma1 HN preserving}.
\begin{lemma} \label{sigma1 phase non-overlapping}
Let $(E,E')$ be a pair of indecomposable objects in $\<\gamma^j P_1\>^\oplus_{\tau_R} \oplus \< \gamma^j P_2 \>^\oplus_{\tau_R}$.
Then one of the following must hold:
\begin{enumerate}[(1)]
\item $\phi(E) < \phi(E')$,
\item $\bigoplus_{k \in \Z} \Hom(\sigma_1(E'), \sigma_1(E)\<k\>[k][1]) \cong \bigoplus_{k \in \Z} \Hom(E', E\<k\>[k][1]) \cong 0$,
\item $\floor{\sigma_1(E)} \geq \ceil{\sigma_1(E')}$, or
\item $j = \frac{n-1}{2}$.
\end{enumerate}
\end{lemma}
\begin{proof}
Note that the statements above are equivalent up to shifting both $E$ and $E'$ by the same amount $\<m\>[m']$, namely
\begin{enumerate}[-]
\item $\phi(E) \geq \phi(E') \iff \phi(E\<m\>[m']) \geq \phi(E' \<m\>[m'])$;
\item $\bigoplus_{k \in \Z} \Hom(E', E\<k\>[k][1]) \cong 0 \iff \bigoplus_{k \in \Z} \Hom(E'\<m\>[m'], E\<m+k\>[m'+k][1]) \cong 0$; and
\item $\floor{\sigma_1(E)} \geq \ceil{\sigma_1(E')} \iff \floor{\sigma_1(E)\<m\>[m']} \geq \ceil{\sigma_1(E')\<m\>[m'}$.
\end{enumerate}
As such, we may assume without loss of generality that $E'$ is always $\gamma^j(P_i)\otimes \Pi_a$ for some $i\in \{1,2\}$ and some $0 \leq a \leq n-2$, by shifting the two gradings whenever necessary.
Moreover, if we now consider linear shifts $\<\ell\>[\ell]$, we have
\begin{enumerate}[-]
\item $\phi(E) = \phi(E)\<\ell\>[\ell]$;
\item $\bigoplus_{k \in \Z} \Hom(E', E\<k\>[k][1]) = \bigoplus_{k \in \Z} \Hom(E', E\<k+\ell\>[k+\ell][1])$; and
\item $\floor{\sigma_1(E)} = \floor{\sigma_1(E)\<\ell\>[\ell]}$.
\end{enumerate}
Thus, if the statement is true for some pair $(E,E')$, then it must also be true for all pairs $(E\<\ell\>[\ell], E')$ with different linear shifts of $E$.

The proof shall proceed as follows:
taking $E' = \gamma^j(P_i)\otimes \Pi_a$ for each $i\in\{1,2\}$, we shall look for all of the corresponding choices of $E$ (up to linear shifts) such that (2) is not true.
Then, we show that for all such choices of $E$ that do not satisfy (1) nor (4) must satisfy (3).

Since $\gamma$ is an autoequivalence, we can apply \cref{graded hom space} to show that there are only six possible pairs of $(E,E')$ (up to linear shifts of $E$) with $\bigoplus_{k \in \Z} \Hom(E', E\<k\>[k][1])$ non-zero:
\begin{enumerate}
\item $(E,E') = (\gamma^j(P_1) \otimes \Pi_a\<-1\> , \gamma^j(P_1) \otimes \Pi_a)$;
\item $(E,E') = (\gamma^j(P_2) \otimes \Pi_a\<-1\> , \gamma^j(P_2) \otimes \Pi_a)$;
\item $(E,E') = (\gamma^j(P_2) \otimes \Pi_{a\pm 1}, \gamma^j(P_1) \otimes \Pi_a)$;
\item $(E,E') = (\gamma^j(P_1) \otimes \Pi_{a\pm 1}, \gamma^j(P_2) \otimes \Pi_a)$;
\item $(E,E') = (\gamma^j(P_1) \otimes \Pi_a\< 1\> , \gamma^j(P_1) \otimes \Pi_a)$; and
\item $(E,E') = (\gamma^j(P_2) \otimes \Pi_a\< 1\> , \gamma^j(P_2) \otimes \Pi_a)$.
\end{enumerate}
Restricting to the case $j=0$, one may check explicitly that $\phi(E) \geq \phi(E')$ only for the first three pairs;
the rest of the cases $j\geq 1$ follows from the fact that $\gamma$ reduces the phase of all semistable objects in $\bigoplus_{S \in \mathfrak{R}^+_n} \<S\>^\oplus_{ \tau_R}$ by the same amount (cf. \cref{gamma phase reducing}).

With $j \neq \frac{n-1}{2}$, we use \cref{lemma: phases of object odd} to compare the phases, which shows that $\floor{\sigma_1(E)} \geq \ceil{\sigma_1(E')}$ for the three pairs $(E,E')$ with $\phi(E) \geq \phi(E')$, as required.
\comment{
Let us start with $j=0, \cV_0$.
In this case, $X_0 = P_1$, $Y_0 = P_2$.
We have that
\[
\sigma_1(P_1) = P_1 \<2\>[1], \quad
\sigma_1(P_2) = P_1\otimes \Pi_1\<1\>[1] \ra P_2.
\]
In particular, since $-\otimes \Pi_a$ does not change the phase of the semistable pieces, we get
\[
\floor{\sigma_1(P_i \otimes \Pi_a)} =
	\begin{cases}
	-1, &\text{ if } i =1; \\
	0. &\text{ if } i=2
	\end{cases}, \quad
\ceil{\sigma_1(P_i \otimes \Pi_a)} =
	\begin{cases}
	-1, &\text{ if } i =1; \\
	\frac{n-1}{n}. &\text{ if } i=2
	\end{cases}
\]
for any choice of $a$.
By definition, $\bigoplus_{k \in \Z} \Hom(E', E\<k\>[k][1]) \subseteq \HOM_{\I\text{-mod}}(E' , E[1])$.
With $E' = P_1 \otimes \Pi_a$ or $P_2\otimes \Pi_a$, we shall look for the possible choices of $E$ such that
\[
\bigoplus_{k \in \Z} \Hom(E', E\<k\>[k][1]) \cong \bigoplus_{k \in \Z} \Hom(\sigma_1(E'), \sigma_1(E)\<k\>[k][1]) \not\cong 0,
\] 
where we proceed to show that $\floor{\sigma_1(E)} \geq \ceil{\sigma_1(E')}$ whenever the assumption $\phi(E)\geq\phi(E')$ holds.
Using \cref{graded hom space}, it follows that $\bigoplus_{k \in \Z} \Hom(E', E\<k\>[k][1])$ is non-zero only for the following cases:
\begin{enumerate}
\item $(E,E') = (P_1 \otimes \Pi_a\<-1\> , P_1 \otimes \Pi_a)$ \\
In this case, $\phi(E) > \phi(E')$, so we see that
\[
\floor{\sigma_1(P_1 \otimes \Pi_a\<-1\>)} = 0 \geq -1 = \ceil{\sigma_1(P_1 \otimes \Pi_{a\pm 1})}
\]
as required.
\item $(E,E') = (P_2 \otimes \Pi_{a\pm 1} , P_1 \otimes \Pi_a)$\\
In this case, $\phi(E) > \phi(E')$, so we see that
\[
\floor{\sigma_1(P_2 \otimes \Pi_{a\pm 1})} = \frac{n-1}{n} \geq -1 = \ceil{\sigma_1(P_1 \otimes \Pi_a)}
\]
as required.
\item $(E,E') = (P_2 \otimes \Pi_a \<-1\>, P_2 \otimes \Pi_a)$\\
In this case, $\phi(E) > \phi(E')$, so we see that
\[
\floor{\sigma_1(P_2 \otimes \Pi_a\<-1\>)} = 1 \geq \frac{n-1}{n}  = \ceil{\sigma_1(P_2 \otimes \Pi_a)}
\]
as required.
\item $(E,E') = (P_1 \otimes \Pi_{a\pm 1} , P_2 \otimes \Pi_a)$\\
In this case, $\phi(E) < \phi(E')$.
\end{enumerate}
This concludes the case where $j=0$.

Now consider the case where $1 \leq j \leq \frac{n-3}{2}$.
Using \cref{lemma for braid relation} we get that
\begin{align*}
X_j 
&= (\sigma_2\sigma_1)^j(P_1) \\
&= (\sigma_2\sigma_1)^{j-1}\sigma_2(P_1\<2\>[1]) \\
&= \underbrace{\sigma_2\sigma_1 ... \sigma_2}_{2j-1 \text{ times}}(P_1)\<2\>[1] \\
&= \left( P_2 \otimes \Pi_{2j-1} \<2j-1\>[2j-1] \ra P_1 \otimes \Pi_{2j-2} \<2j-2\>[2j-2] \right) \<2\>[1] \\
\sigma_1(X_j) 
&= \left( P_1 \otimes \Pi_{2j} \<2j\>[2j] \ra P_2 \otimes \Pi_{2j-1} \<2j-1\>[2j-1] \right) \<2\>[1] \\
Y_j
&= (\sigma_2\sigma_1)^j(P_2) \\
&= P_2 \otimes \Pi_{2j} \<2j\>[2j] \ra P_1 \otimes \Pi_{2j-1} \<2j-1\>[2j-1] \\
\sigma_1(Y_j)
&= P_1 \otimes \Pi_{2j+1} \<2j+1\>[2j+1] \ra P_2 \otimes \Pi_{2j} \<2j\>[2j]
\end{align*}
Thus, we get
\begin{align*}
\floor{\sigma_1(X_j \otimes \Pi_a)} = -1, \quad
\floor{\sigma_1(Y_j \otimes \Pi_a)} = 0
\end{align*}
and
\begin{align*}
\ceil{\sigma_1(X_j \otimes \Pi_a)} = \frac{n-1}{n}-1 = -\frac{1}{n}, \quad
\ceil{\sigma_1(Y_j \otimes \Pi_a)} = \frac{n-1}{n}.
\end{align*}
Since $\gamma$ is an autoequivalence, with $E'$ either $X_j \otimes \Pi_a$ or $Y_j \otimes \Pi_a$, there is again four pairs $(E,E')$ such that $\bigoplus_{k \in \Z} \Hom(E', E\<k\>[k][1])$ is non-zero as before:
\begin{enumerate}
\item $(E,E') = (X_j \otimes \Pi_a\<-1\>, X_j \otimes \Pi_a)$. \\
In this case, $\phi(E) > \phi(E')$, so we see that
\[
\floor{\sigma_1(X_j \otimes \Pi_a\<-1\>)} = 0 \geq -\frac{1}{n} = \ceil{\sigma_1(X_j \otimes \Pi_a)} 
\]
as required.
\item $(E,E') = (Y_j \otimes \Pi_{a\pm 1}, X_j \otimes \Pi_a)$. \\
In this case, $\phi(E) > \phi(E')$, so we see that
\[
\floor{\sigma_1(Y_j \otimes \Pi_{a\pm 1})} 
= 0 
\geq -\frac{1}{n} = \ceil{\sigma_1(X_j \otimes \Pi_a)} 
\]
as required.
\item $(E,E') = (Y_j \otimes \Pi_a\<-1\>, Y_j \otimes \Pi_a)$. \\
In this case, $\phi(E) > \phi(E')$, so we see that
\[
\floor{\sigma_1(Y_j \otimes \Pi_a\<-1\>)} 
= 1 
\geq \frac{n-1}{n} = \ceil{\sigma_1(Y_j \otimes \Pi_a)} 
\]
as required.
\item $(E,E') = (X_j \otimes \Pi_{a\pm 1}, Y_j \otimes \Pi_a)$. \\
In this case,
\[
\phi(X_j \otimes \Pi_{a\pm 1}) = \phi(P_1 \otimes \Pi_{a\pm 1}) - \frac{2j}{n} < \phi(P_2 \otimes \Pi_a) - \frac{2j}{n} = \phi(Y_j \otimes \Pi_a),
\]
so $\phi(E) < \phi(E')$.
\end{enumerate}
This concludes the case where $1 \leq j \leq \frac{n-3}{2}$.

Finally consider the case where $\frac{n+1}{2} \leq j \leq n-1$.
Using \cref{lemma for braid relation} we obtain
\begin{align*}
X_j 
&= (\sigma_2\sigma_1)^j(P_1) \\
&= (\sigma_2\sigma_1)^{j-\frac{n+1}{2}} (\sigma_2\sigma_1)^{\frac{n-1}{2}}\sigma_2(P_1\<2\>[1]) \\
&= (\sigma_2\sigma_1)^{j-\frac{n+1}{2}} \sigma_2(\sigma_1\sigma_2)^{\frac{n-1}{2}}(P_1)\<2\>[1] \\
&= (\sigma_2\sigma_1)^{j-\frac{n+1}{2}} \sigma_2(P_2\<n-2\>[n-2])\<2\>[1] \\
&= (\sigma_2\sigma_1)^{j-\frac{n+1}{2}} (P_2\<n\>[n-1])\<2\>[1] \\
&= \underbrace{\sigma_2\sigma_1 ... \sigma_1}_{2j-n-1 \text{ times}}(P_2)\<n+2\>[n] \qquad (\text{note that } 0 \leq 2j-n-1 \leq n-3)\\
&= P_2 \otimes \Pi_{2j-n-1} \<2j+1\>[2j-1] \ra P_1 \otimes \Pi_{2j-n-2} \<2j\>[2j-2] \\
\sigma_1(X_j) 
&= P_1 \otimes \Pi_{2j-n} \<2j+2\>[2j] \ra P_2 \otimes \Pi_{2j-n-1} \<2j+1\>[2j-1] \\
Y_j 
&= (\sigma_2\sigma_1)^j(P_2) \\
&= (\sigma_2\sigma_1)^{j-\frac{n+1}{2}} (\sigma_2\sigma_1)(\sigma_2\sigma_1)^{\frac{n-1}{2}}(P_2) \\
&= (\sigma_2\sigma_1)^{j-\frac{n+1}{2}} (\sigma_2\sigma_1)(P_1\<n-2\>[n-2]) \\
&= (\sigma_2\sigma_1)^{j-\frac{n+1}{2}} \sigma_2(P_1\<n\>[n-1]) \\
&= \underbrace{\sigma_2\sigma_1 ... \sigma_2}_{2j-n \text{ times}}(P_1)\<n\>[n-1] \qquad (\text{note that } 1 \leq 2j-n \leq n-2)\\
&= P_2 \otimes \Pi_{2j-n} \<2j\>[2j-1] \ra P_1 \otimes \Pi_{2j-n-1} \<2j-1\>[2j-2] \\
\sigma_1(Y_j) 
&= 
	\begin{cases}
	P_2 \otimes \Pi_{2j-n} \<2j\>[2j-1], &\text{ when } 2j-n = n-2; \\
	P_1 \otimes \Pi_{2j-n+1} \<2j+1\>[2j] \ra P_2 \otimes \Pi_{2j-n} \<2j\>[2j-1], &\text{ when } 2j-n < n-2.
	\end{cases}
\end{align*}
Thus, we get
\[
\floor{\sigma_1(X_j \otimes \Pi_a)} = -2, \quad
\floor{\sigma_1(Y_j \otimes \Pi_a)} =
	\begin{cases}
	\frac{n-1}{n}-1=-\frac{1}{n}, &\text{ when } 2j-1 = n-2; \\
	-1, &\text{ when } 2j-1 < n-2
	\end{cases}
\]
and
\[
\ceil{\sigma_1(X_j \otimes \Pi_a)} = \frac{n-1}{n}-2 = -\frac{n+1}{n}, \quad
\ceil{\sigma_1(Y_j \otimes \Pi_a)} = -\frac{1}{n}.
\]
As before, with $E'$ either $X_j \otimes \Pi_a$ or $Y_j \otimes \Pi_a$, there is again four pairs $(E,E')$ such that $\bigoplus_{k \in \Z} \Hom(E', E\<k\>[k][1])$ is non-zero:
\begin{enumerate}
\item $(E,E') = (X_j \otimes \Pi_a\<-1\>, X_j \otimes \Pi_a)$. \\
In this case, $\phi(E) > \phi(E')$, so we see that
\[
\floor{\sigma_1(X_j \otimes \Pi_a\<-1\>)} = -1 
\geq -\frac{n+1}{n} = \ceil{\sigma_1(X_j \otimes \Pi_a)} 
\]
as required.
\item $(E,E') = (Y_j \otimes \Pi_{a\pm 1}, X_j \otimes \Pi_a)$. \\
In this case, $\phi(E) > \phi(E')$, so we see that
\[
\floor{\sigma_1(Y_j \otimes \Pi_{a\pm 1})} 
\geq \min\left\{-1, -\frac{1}{n} \right\} = -1 
\geq -\frac{n+1}{n} = \ceil{\sigma_1(X_j \otimes \Pi_a)} 
\]
as required.
\item $(E,E') = (Y_j \otimes \Pi_a\<-1\>, Y_j \otimes \Pi_a)$. \\
In this case, $\phi(E) > \phi(E')$, so we see that
\[
\floor{\sigma_1(Y_j \otimes \Pi_a\<-1\>)} 
\geq \min\left\{0, \frac{n-1}{n} \right\} = 0 
\geq -\frac{1}{n} = \ceil{\sigma_1(Y_j \otimes \Pi_a)} 
\]
as required.
\item $(E,E') = (X_j \otimes \Pi_{a\pm 1}, Y_j \otimes \Pi_a)$. \\
In this case,
\[
\phi(X_j \otimes \Pi_{a\pm 1}) = \phi(P_1 \otimes \Pi_{a\pm 1}) - \frac{2j}{n} < \phi(P_2 \otimes \Pi_a) - \frac{2j}{n} = \phi(Y_j \otimes \Pi_a),
\]
so $\phi(E) < \phi(E')$.
\end{enumerate}
This concludes the case where $\frac{n+1}{2} \leq j \leq n-1$, and hence all cases of $j \neq \frac{n-1}{2}$.
}
\end{proof}

\begin{proposition}\label{sigma1 HN preserving}
Suppose $A \in [\gamma^j P_1, \gamma^j P_2]$ for $j \neq \frac{n-1}{2}$ with indecomposable $\tau_R$-semistable pieces in the HN filtration given by $E_i$, i.e. $\supp_{\tau_R}(A) = \bigoplus_{i=1}^m E_i$.
Then
\begin{enumerate}[(i)]
\item $\supp_{\tau_R}(\sigma_1(A)) = \bigoplus_{i=1}^m \supp_{\tau_R}(\sigma_1(E_i))$, and
\item $\sigma_1(A) \in [P_1, P_2]$.
\end{enumerate}
\end{proposition}
\begin{proof}
From \cref{eqn: sigma1 on V_j generator 1} and \cref{eqn: sigma1 on V_j generator 2}, we see that 
\[
\sigma_1(\gamma^j(P_1)), \sigma_1(\gamma^j(P_2))\in [P_1,P_2].
\]
Using the property of supports together with the fact that all braids commute with the functors $-\otimes \Pi_a, \<k\>$ and $[\ell]$, we get
\[
\sigma_1(\gamma^j(P_1 \otimes \Pi_a \<k\>[\ell])) , \sigma_1(\gamma^j(P_2 \otimes \Pi_a \<k\>[\ell]))\in [P_1,P_2].
\]
Since each $E_i$ is of the form $\gamma^j(P_1 \otimes \Pi_a \<k\>[\ell])$ or $\gamma^j(P_2 \otimes \Pi_a \<k\>[\ell])$, statement (ii) follows from statement (i).

We shall now prove (i).
Let $A$ be as given in the proposition.
Then we have a filtration $F$ of $A$
\[
F := 
\begin{tikzcd}[column sep = 3mm]
0
	\ar[rr] &
{}
	{} &
A_1
	\ar[rr] \ar[dl]&
{}
	{} &	
A_2
	\ar[rr] \ar[dl]&
{}
	{} &	
\cdots
	\ar[rr] &
{}
	{} &
A_{m-1}
	\ar[rr] &
{}
	{} &
A_m = A,
	\ar[dl] \\
{}
	{} &
E_1
	\ar[lu, dashed] &
{}
	{} &
E_2
	\ar[lu, dashed] &
{}
	{} &
{}
	{} &
{}
	{} &
{}
	{} &
{}
	{} &
E_m
	\ar[lu, dashed] &
\end{tikzcd}
\]
obtainable from the $\tau_R$-HN filtration of $A$ by splitting up each semistable pieces into indecomposables.
Note that this means the phases of the $E_i$'s are only non-increasing: $\phi(E_i) \geq \phi(E_{i+1})$ (instead of strictly decreasing).
Such a filtration $F$ is called a weak HN filtration of A (see \cref{defn: weak HN filtration} in the appendix; this is \emph{not} the HN filtration).
Applying $\sigma_1$ to $F$ results in the filtration
\[
\sigma_1(F)=
\begin{tikzcd}[column sep = 3mm]
0
	\ar[rr] &
{}
	{} &
\sigma_1(A_1)
	\ar[rr] \ar[dl]&
{}
	{} &	
\cdots
	\ar[rr] &
{}
	{} &
\sigma_1(A_{m-1})
	\ar[rr] &
{}
	{} &
\sigma_1(A).
	\ar[dl] \\
{}
	{} &
\sigma_1(E_1)
	\ar[lu, dashed] &
{}
	{} &
{}
	{} &
{}
	{} &
{}
	{} &
{}
	{} &
\sigma_1(E_m)
	\ar[lu, dashed] &
\end{tikzcd}
\]
Using \cref{lemma: 2CY variant} and \cref{sigma1 phase non-overlapping}, we see that $\sigma_1(F)$  satisfies the assumption of \cref{sufficient condition for geodesic} and is therefore a geodesic filtration.
As such, a weak HN filtration polygon of $\sigma_1(A)$ can be obtained from concatenating weak HN filtrations of $\sigma_1(E_i)$'s up to some rearrangements of the filtration pieces (cf. \cref{cor: geodesic implies weak HN}).
We shall use the weak HN filtrations of the $\sigma_1(E_i)$'s obtained via splitting up each of the semistable pieces in their HN filtrations into indecomposables as before.
Up to rearranging (and applying the appropriate octahedral flips), this results in a weak HN filtration $F'$ of $\sigma_1(A)$ given by:
\[
F' =
\begin{tikzcd}[column sep = 3mm]
0
	\ar[rr] &
{}
	{} &
A_1'
	\ar[rr] \ar[dl]&
{}
	{} &	
\cdots
	\ar[rr] &
{}
	{} &
A_{r-1}'
	\ar[rr] &
{}
	{} &
A_r' = \sigma_1(A)
	\ar[dl] \\
{}
	{} &
X_1
	\ar[lu, dashed] &
{}
	{} &
{}
	{} &
{}
	{} &
{}
	{} &
{}
	{} &
X_r
	\ar[lu, dashed] &
\end{tikzcd}
\]
where each $X_j$ is an indecomposable semistable object that appears as a summand of some semistable pieces in the HN filtration of some $E_i$.

We claim that taking the strictification (see discussion after \cref{defn: weak HN filtration}) of this weak HN filtration $F'$ into the HN filtration, if necessary, will only involve taking direct sums.
This will be achieved by showing that if $\phi(X_i) = \phi(X_{i+1})$, then $\Hom(X_{i+1}, X_{i}[1]) = 0$, which implies that 
\[
\cone(A_{i-1}' \ra A_{i}' \ra A_{i+1}') = X_i \oplus X_{i+1}.
\]
In particular, this implies that
\[
\supp_{\tau_R}(\sigma_1(A)) = \bigoplus_{i=1}^m \supp_{\tau_R}(\sigma_1(E_i)).
\]

We shall now prove the claim above.
Recall from the start of the proof that $\sigma_1(E_j) \in [P_1, P_2]$ for all $j$.
As such, each $X_i$ (which is indecomposable and semistable) must be an indecomposable object in $\<P_1\>^\oplus_{\tau_R} \oplus \<P_2\>^\oplus_{\tau_R}$, namely 
\[
X_i = P_{s_i} \otimes \Pi_{a_i} \<k_i\>[\ell_i]
\]
for some $s_i \in \{1,2\}$, $0 \leq a_i \leq n-2$ and $k_i,\ell_i \in \Z$.
In order for $X_i$ and $X_{i+1}$ to have the same phase, we need that $s_i = s_{i+1} \in \{1,2\}$; and
\begin{equation} \label{same level}
\ell_{i+1} -k_{i+1} = \ell_i - k_i.
\end{equation}
Now using \cref{graded hom space}, we see that $\Hom(X_{i+1}, X_{i}[1]) \neq 0$ only if
\[
\ell_{i+1} - \ell_i = 1 \quad \text{ and } \quad k_{i+1} - k_i = 0 \text{ or }2,
\]
which contradicts \cref{same level}.
Hence we have $\Hom(X_{i+1}, X_{i}[1]) = 0$ as required.
\end{proof}

\begin{corollary} \label{cor: conditions of automaton odd}
Suppose $A \in [\gamma^j P_1, \gamma^j P_2]$ for $j \neq k+ \frac{n-1}{2}$ with indecomposable $\tau_R$-semistable pieces in the HN filtration given by $E_i$, i.e. $\supp_{\tau_R}(A) = \bigoplus_{i=1}^m E_i$.
Then
\begin{enumerate}[(i)]
\item $\supp_{\tau_R}(\sigma_{\gamma^k(P_1)}(A)) = \bigoplus_{i=1}^m \supp_{\tau_R}(\sigma_{\gamma^k(P_1)}(E_i))$, and
\item $\sigma_{\gamma^k(P_1)}(A) \in [\gamma^k P_1, \gamma^k P_2]$.
\end{enumerate}
\end{corollary}
\begin{proof}
Let $A$ and $E_i$ be as given.
Using \cref{gamma sends semistable pieces} we know that the indecomposable semistable pieces of the HN filtration of $\gamma^{-k}(A)$ are given by $\gamma^{-k}(E_i)$.
Hence $\gamma^{-k}(A) \in [\gamma^{j-k} P_1, \gamma^{j-k} P_2]$.
We can now apply \cref{sigma1 HN preserving} to $\gamma^{-k}(A)$, which gives us
\[
\supp_{\tau_R}(\sigma_{P_1}\gamma^{-k}(A)) = \bigoplus_i \supp_{\tau_R}(\sigma_{P_1}\gamma^{-k}(E_i) ) 
\]
and $\sigma_{P_1}\gamma^{-k}(A) \in [P_1,P_2]$.
The latter statement combined with \cref{gamma loops V} tells us that $\gamma^k \sigma_{P_1}\gamma^{-k}(A) \in [\gamma^k P_1, \gamma^k P_2]$, which shows (ii).
Applying \cref{gamma sends semistable pieces} to the first statement again gives us
\begin{align*}
\supp_{\tau_R}(\gamma^k\sigma_{P_1}\gamma^{-k}(A)) 
&= \gamma^k \left( \supp_{\tau_R}(\sigma_{P_1}\gamma^{-k}(A)) \right) \\
&= \gamma^k \left( \bigoplus_i   \supp_{\tau_R}(\sigma_{P_1}\gamma^{-k}(E_i)) \right) \\
&= \bigoplus_i \gamma^k \left(   \supp_{\tau_R}(\sigma_{P_1}\gamma^{-k}(E_i)) \right) \in \< \cV_k \>^\oplus_{\tau_R} \\
&= \bigoplus_i \supp_{\tau_R}(\gamma^k\sigma_{P_1}\gamma^{-k}(E_i)).
\end{align*}
as required for (i).
\end{proof}

\subsubsection{Even $n$ case}
In this subsection fix $n$ to be even instead.
We shall construct a $\tau_R$-mass automaton $\Lambda_{\tau_R}(n)$ for $\B(I_2(n))$ as follows:
the underlying $\B(I_2(n))$-automaton $\Theta$ is defined by the quiver with $n$ vertices where half of them are labelled by $v$ and another half are labelled by $u$: $v_0, v_1, ..., v_{\frac{n}{2}-1}$ and $u_0, u_1, ..., u_{\frac{n}{2}-1}$.
All the incoming braid-labelled arrows into each $v_j$ and $u_j$ are given in \Cref{fig:mass automaton for even n V} and \Cref{fig:mass automaton for even n U} respectively.
The complete $\tau_R$-mass automaton for $n=4$ can be found in \cref{sect: example n=4}.

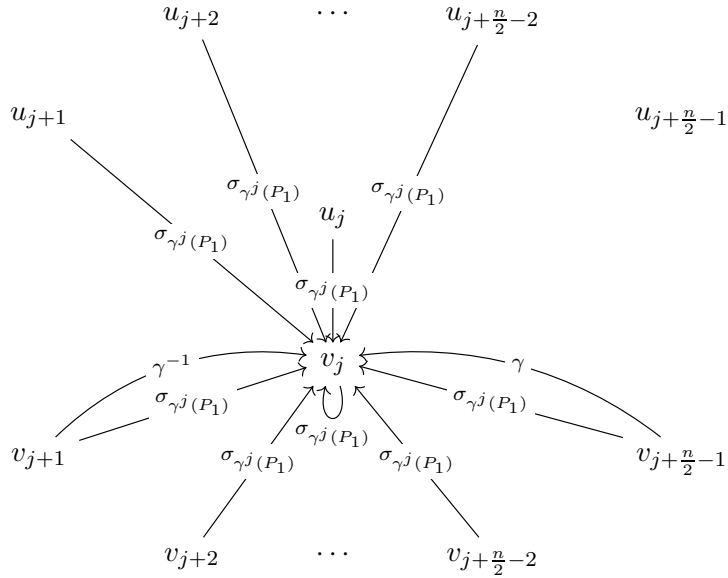
\begin{figure}[h]
\[
\begin{tikzcd}
	& {u_{j+2}} & {\cdots} & {u_{j+ \frac{n}{2}-2}} \\
	{u_{j+1}} &&&& {u_{j+\frac{n}{2}-1}} \\
	&& {u_j} \\
	\\
	&& {v_j} \\
	{v_{j+1}} &&&& {v_{j+\frac{n}{2}-1}} \\
	& {v_{j+2}} & {\cdots} & {v_{j+\frac{n}{2}-2}}
	\arrow["{\gamma^{-1}}" description, curve={height=-20pt}, from=6-1, to=5-3]
	\arrow["{\gamma}" description, curve={height=20pt}, from=6-5, to=5-3]
	\arrow["{\sigma_{\gamma^j(P_1)}}" description, from=6-1, to=5-3]
	\arrow["{\sigma_{\gamma^j(P_1)}}" description, from=6-5, to=5-3]
	\arrow["{\sigma_{\gamma^j(P_1)}}" description, from=7-2, to=5-3]
	\arrow["{\sigma_{\gamma^j(P_1)}}" description, from=7-4, to=5-3]
	\arrow["{\sigma_{\gamma^j(P_1)}}" description, from=2-1, to=5-3]
	\arrow["{\sigma_{\gamma^j(P_1)}}" description, from=1-2, to=5-3]
	\arrow["{\sigma_{\gamma^j(P_1)}}" description, from=1-4, to=5-3]
	\arrow["{\sigma_{\gamma^j(P_1)}}" description, from=3-3, to=5-3]
	\arrow["{\sigma_{\gamma^j(P_1)}}" , from=5-3, to=5-3, loop below]
\end{tikzcd}
\]
\caption{The $I_2(n)$ root automaton for even $n$, showing all of the incoming edges of $v_j$ only. All the subscripts of the vertices are taken modulo $n$. Note that there is indeed no arrow with the label $\sigma_{\gamma^j(P_1)}$ from $u_{j+\frac{n}{2}-1}$ to $v_j$.}
\label{fig:mass automaton for even n V}
\end{figure}

\begin{figure}[h]
\[
\begin{tikzcd}
	& {v_{j+2}} & {\cdots} & {v_{j+\frac{n}{2}-2}} \\
	{v_{j+1}} &&&& {v_{j+\frac{n}{2}-1}} \\
	&& {v_j} \\
	\\
	&& {u_j} \\
	{u_{j+1}} &&&& {u_{j+\frac{n}{2}-1}} \\
	& {u_{j+2}} & {\cdots} & {u_{j+\frac{n}{2}-2}}
	\arrow["{\gamma^{-1}}" description, curve={height=-20pt}, from=6-1, to=5-3]
	\arrow["{\gamma}" description, curve={height=20pt}, from=6-5, to=5-3]
	\arrow["{\sigma_{\gamma^j(P_2)}}" description, from=6-1, to=5-3]
	\arrow["{\sigma_{\gamma^j(P_2)}}" description, from=6-5, to=5-3]
	\arrow["{\sigma_{\gamma^j(P_2)}}" description, from=7-2, to=5-3]
	\arrow["{\sigma_{\gamma^j(P_2)}}" description, from=7-4, to=5-3]
	\arrow["{\sigma_{\gamma^j(P_2)}}" description, from=2-5, to=5-3]
	\arrow["{\sigma_{\gamma^j(P_2)}}" description, from=2-1, to=5-3]
	\arrow["{\sigma_{\gamma^j(P_2)}}" description, from=1-2, to=5-3]
	\arrow["{\sigma_{\gamma^j(P_2)}}" description, from=1-4, to=5-3]
	\arrow["{\sigma_{\gamma^j(P_2)}}" , from=5-3, to=5-3, loop below]
\end{tikzcd}
\]
\caption{The $I_2(n)$ root automaton for even $n$, showing all of the incoming edges of $u_j$ only. All the subscripts of the vertices are taken modulo $n$. Note that there is indeed no arrow with the label $\sigma_{\gamma^j(P_2)}$ from $v_j$ to $u_j$.}
\label{fig:mass automaton for even n U}
\end{figure}

The $\Theta$-subset $\cS: \Theta \ra $ Sets of $ob(\cK)\circ \Theta$ is defined by 
\[
\cS(v_j) := [\gamma^j P_1, \gamma^j P_2], \quad \cS(u_j) = [\gamma^j P_2, \gamma^j \sigma_2 P_1], \quad (\text{see \cref{defn: bracket vertex automaton}})
\]
The $\Theta$-representation $\cM: \Theta \ra \R[e^{\pm t}]$-mod is defined as follows.
Each $\cM(v_j)$ and $\cM(u_j)$ is a rank two free module over $ \R[e^{\pm t}]$ defined by
\[
\cM(v_j) 
:= \R[e^{\pm t}]\cdot m_{\tau_R,t} (\gamma^j P_1) \oplus  \R[e^{\pm t}]\cdot m_{\tau_R,t} (\gamma^j P_2)
\]
and
\[
\cM(u_j) 
:= \R[e^{\pm t}] \cdot m_{\tau_R,t} (\gamma^j P_2) \oplus \R[e^{\pm t}]\cdot m_{\tau_R,t} (\gamma^j \sigma_2 P_1).
\]
For each arrow $a : v_j \ra v_k$ labelled by $\cS(a)$, $\cM(a)$ is the linear map uniquely defined on the basis elements $m_{\tau_R,t} (B)$ by:
\[
m_{\tau_R,t} (B) \mapsto m_{\tau_R,t} \left(\cS(a)(B) \right).
\]
Once again elements in $\cM(v_j)$ and $\cM(u_j)$ are sometimes written as column vectors and $\cM(a)$ for each arrow $a$ is sometimes written as a two by two matrix.
The natural transformation $\iota: \cS \ra \cM$ is given by 
\[
\iota_w (X) := m_{\tau_R,t} (X)
\]
for all vertices $w$ in $\Theta$ and all $X \in \cS(v)$.
For each vertex $v_j$ the map $\mathfrak{m}_{v_j}: \cM(v_j) \hookrightarrow \R^\R$ is again the standard inclusion.
\comment{
defined by
\[
\begin{bmatrix}
a_1 \\
a_2
\end{bmatrix} \mapsto 
a_1 m_{\tau_R, t} (\gamma^j P_1) + a_2 m_{\tau_R, t} (\gamma^j P_2).
\]
and similarly for vertex $u_j$, $\mathfrak{m}_{u_j}: \cM(u_j) \ra \R^\R$ is defined by
\[
\begin{bmatrix}
a_1 \\
a_2
\end{bmatrix} \mapsto 
a_1 m_{\tau_R, t} (\gamma^j P_2) + a_2 m_{\tau_R, t} (\gamma^j \sigma_2 P_1).
\]
}
One can make a similar checklist as in the odd $n$ case to make sure that these define a $\tau_R$-mass automaton.
The required (similar) results occupy the rest of this subsection.

\begin{proposition}[compare w. \cref{gamma loops V}] \label{gamma loops V even}
Suppose $A \in \cS(v_j) = [\gamma^jP_1, \gamma^jP_2]$. Then 
\[
\gamma^{\pm 1}(A) \in \cS(v_{(j \pm 1 \mod \frac{n}{2})}).
\]
Suppose $A \in \cS(u_j) = [\gamma^jP_2, \gamma^j\sigma_2 P_1]$. Then 
\[
\gamma^{\pm 1}(A) \in \cS(u_{(j \pm 1 \mod \frac{n}{2})}).
\]
\end{proposition}

\begin{lemma}[compare w. \cref{lemma: phases of object odd}] \label{lemma: phases of object even sigma1}
For $n$ even, we have
\begin{align*}
\sigma_1(\gamma^j(P_1) \otimes \Pi_a) 
&\in
	\begin{cases}
	\cP(-1), &\text{ for } j = 0; \\
	\cP([-1, -\frac{1}{n}]), &\text{ for } 1 \leq j \leq \frac{n}{2}-1 \\
	\end{cases}
\\
\sigma_1(\gamma^j(P_2) \otimes \Pi_a)
&\in 
	\begin{cases}
	\cP([0, \frac{n-1}{n}]), &\text{ for } 0 \leq j \leq \frac{n}{2}-2; \\
	\cP(\frac{n-1}{n}), &\text{ for } j = \frac{n}{2}-1; \\
	\end{cases}
\\
\sigma_1(\gamma^j\sigma_2(P_1) \otimes \Pi_a)
&\in 
	\begin{cases}
	\cP([0, \frac{n-1}{n}]), &\text{ for } 0 \leq j \leq \frac{n}{2}-2; \\
	\cP(-1), &\text{ for } j = \frac{n}{2}-1; \\
	\end{cases}
\end{align*}
\end{lemma}
\begin{proof}
One can compute the following using \cref{lemma for braid relation}:
\begin{equation}
\begin{split}
&\sigma_1(\gamma^j(P_1)) \\
&\cong 
	\begin{cases}
	P_1 \<2\>[1], & j = 0; \\
	P_1 \otimes \Pi_{2j} \<2j+2\>[2j+1] \ra P_2 \otimes \Pi_{2j-1} \<2j+1\>[2j], & 1 \leq j \leq \frac{n}{2} -1
;	
	\end{cases}
\end{split}
\end{equation},
\begin{equation}
\begin{split}
&\sigma_1(\gamma^j(P_2)) \\
&\cong 
	\begin{cases}
	P_1 \otimes \Pi_{2j+1} \<2j+1\>[2j+1] \ra P_2 \otimes \Pi_{2j} \<2j\>[2j], & 0 \leq j \leq \frac{n}{2}-2; \\
	P_2 \otimes \Pi_{n-2} \<n-2\>[n-2],  &j = \frac{n}{2}-1.
	\end{cases}
\end{split}
\end{equation}
and
\begin{equation}
\begin{split}
&\sigma_1(\gamma^j\sigma_2(P_1)) \\
&\cong 
	\begin{cases}
	P_1 \otimes \Pi_{2j+2} \<2j+2\>[2j+2] \ra P_2 \otimes \Pi_{2j+1} \<2j+1\>[2j+1], & 0 \leq j \leq \frac{n}{2}-2; \\
	P_1 \otimes \Pi_{n-2} \<n\>[n-1],  &j = \frac{n}{2}-1.
	\end{cases}
\end{split}
\end{equation}
The result now follows as in  \cref{lemma: phases of object odd}.
\end{proof}

\begin{lemma}[compare w. \cref{lemma: phases of object odd}] \label{lemma: phases of object even sigma2}
For $n$ even, we have
\begin{align*}
\sigma_2(\gamma^j(P_1) \otimes \Pi_a) 
&\in
	\begin{cases}
	\cP(1-\frac{2}{n}), &\text{ for } j = 0; \\
	\cP([-1 - \frac{1}{n}, -\frac{2j}{n}]), &\text{ for } 1 \leq j \leq \frac{n}{2}-1 \\
	\end{cases}
\\
\sigma_2(\gamma^j(P_2) \otimes \Pi_a)
&\in 
	\begin{cases}
	\cP(-\frac{1}{n}), &\text{ for } j = 0; \\
	\cP([- \frac{1}{n}, 1-\frac{1+2j}{n}]), &\text{ for } 1 \leq j \leq \frac{n}{2}-1 \\
	\end{cases}
\\
\sigma_2(\gamma^j\sigma_2(P_1) \otimes \Pi_a)
&\in 
	\begin{cases}
	\cP([-\frac{1}{n}, 1- \frac{2j+2}{n}]), &\text{ for } 0 \leq j \leq \frac{n}{2}-2; \\
	\cP(1-\frac{2}{n}), &\text{ for } j = \frac{n}{2}-1; \\
	\end{cases}
\end{align*}
\end{lemma}
\begin{proof}
One can compute the following using \cref{lemma for braid relation}:
\begin{equation}
\begin{split}
&\sigma_2(\gamma^j(P_1)) \\
&\cong 
	\begin{cases}
	P_2\otimes\Pi_1 \<1\>[1] \ra P_1, & j = 0; \\
	P_2\otimes \Pi_{2j-1} \<2j+1\>[2j] \\
		\quad \ra P_2 \otimes \Pi_{2j-1} \<2j-1\>[2j-1] 
			\ra P_1 \otimes \Pi_{2j-2} \<2j-2\>[2j-2], & 1 \leq j \leq \frac{n}{2} -1
;	
	\end{cases}
\end{split}
\end{equation},
\begin{equation}
\begin{split}
&\sigma_2(\gamma^j(P_2)) \\
&\cong 
	\begin{cases}
	P_2 \<2\>[1], & j = 0; \\
	P_2 \otimes \Pi_{2j} \<2j+2\>[2j+1] \\
		\quad \ra P_2 \otimes \Pi_{2j} \<2j\>[2j] 
			\ra P_1 \otimes \Pi_{2j-1} \<2j-1\>[2j-1], & 1 \leq j \leq \frac{n}{2} -1;	
	\end{cases}
\end{split}
\end{equation}
and
\begin{equation}
\begin{split}
&\sigma_2(\gamma^j\sigma_2(P_1)) \\
&\cong 
	\begin{cases}
	P_2 \otimes \Pi_{2j+1} \<2j+3\>[2j+2] \\
		\quad \ra P_2 \otimes \Pi_{2j+1} \<2j+1\>[2j+1] 
			\ra P_1 \otimes \Pi_{2j} \<2j\>[2j], & 0 \leq j \leq \frac{n}{2} -2;	\\
	\left( P_2\otimes\Pi_1 \<1\>[1] \ra P_1 \right) \otimes \Pi_{n-2}\<n-2\>[n-2],  &j = \frac{n}{2}-1.
	\end{cases}
\end{split}
\end{equation}
The result now follows from looking at the HN filtration of the objects, where all of the three term complexes in each $\sigma_2(X)$ above have $\tau_R$-HN semistable pieces given by $X$ itself followed by the first term in the complex.
\end{proof}

\begin{lemma}[compare w. \cref{sigma1 phase non-overlapping}]
Let $\{E,E'\}$ be a pair of indecomposable objects in $\< \gamma^j P_1\>^\oplus_{\tau_R} \oplus \< \gamma^j P_2 \>^\oplus_{\tau_R}$.
\begin{itemize}
\item With respect to $\sigma_1$, one of the following must hold:
\begin{enumerate}[(1)]
\item $\phi(E) < \phi(E')$,
\item $\bigoplus_{k \in \Z} \Hom(\sigma_1(E'), \sigma_1(E)\<k\>[k][1]) \cong \bigoplus_{k \in \Z} \Hom(E', E\<k\>[k][1]) \cong 0$, or
\item $\floor{\sigma_1(E)} \geq \ceil{\sigma_1(E')}$.
\end{enumerate}
\item With respect to $\sigma_2$, one of the following must hold:
\begin{enumerate}[(1)]
\item $\phi(E) < \phi(E')$,
\item $\bigoplus_{k \in \Z} \Hom(\sigma_2(E'), \sigma_2(E)\<k\>[k][1]) \cong \bigoplus_{k \in \Z} \Hom(E', E\<k\>[k][1]) \cong 0$,
\item $\floor{\sigma_2(E)} \geq \ceil{\sigma_2(E')}$, or
\item $j = 0$.
\end{enumerate}
\end{itemize}
\end{lemma}
\begin{proof}
The proof method is the same as in \cref{sigma1 phase non-overlapping}, where the three possible pairs of $(E,E')$ (up to linear shifts $\<\ell\>[\ell]$ of $E$) that do not satisfy (1) and (2) are also the same:
\begin{enumerate}
\item $(E,E') = (\gamma^j(P_1) \otimes \Pi_a\<-1\> , \gamma^j(P_1) \otimes \Pi_a)$;
\item $(E,E') = (\gamma^j(P_2) \otimes \Pi_a\<-1\> , \gamma^j(P_2) \otimes \Pi_a)$; and
\item $(E,E') = (\gamma^j(P_2) \otimes \Pi_{a\pm 1}, \gamma^j(P_1) \otimes \Pi_a)$.
\end{enumerate}
We now use \cref{lemma: phases of object even sigma1} to check that $\floor{\sigma_1(E)} \geq \ceil{\sigma_1(E')}$ (no condition on $j$ is required), and we use \cref{lemma: phases of object even sigma2} to check that $\floor{\sigma_2(E)} \geq \ceil{\sigma_2(E')}$ whenever $j \neq 0$.
\end{proof}

\begin{lemma}[compare w. \cref{sigma1 phase non-overlapping}]
Let $\{E,E'\}$ be a pair of indecomposable objects in $\< \gamma^j P_2\>^\oplus_{\tau_R} \oplus \< \gamma^j \sigma_2 P_1 \>^\oplus_{\tau_R}$.
\begin{itemize}
\item With respect to $\sigma_1$, one of the following must hold:
\begin{enumerate}[(1)]
\item $\phi(E) < \phi(E')$,
\item $\bigoplus_{k \in \Z} \Hom(\sigma_1(E'), \sigma_1(E)\<k\>[k][1]) \cong \bigoplus_{k \in \Z} \Hom(E', E\<k\>[k][1]) \cong 0$,
\item $\floor{\sigma_1(E)} \geq \ceil{\sigma_1(E')}$, or
\item $j = \frac{n}{2}-1$.
\end{enumerate}
\item With respect to $\sigma_2$, one of the following must hold:
\begin{enumerate}[(1)]
\item $\phi(E) < \phi(E')$,
\item $\bigoplus_{k \in \Z} \Hom(\sigma_2(E'), \sigma_2(E)\<k\>[k][1]) \cong \bigoplus_{k \in \Z} \Hom(E', E\<k\>[k][1]) \cong 0$, or
\item $\floor{\sigma_2(E)} \geq \ceil{\sigma_2(E')}$.
\end{enumerate}
\end{itemize}
\end{lemma}
\begin{proof}
Once again the proof method is the same, but the three possible pairs of $(E,E')$ (up to linear shifts $\<\ell\>[\ell]$ of $E$) that do not satisfy (1) and (2) are instead given by:
\begin{enumerate}
\item $(E,E') = (\gamma^j \sigma_2(P_1) \otimes \Pi_a\<-1\> , \gamma^j\sigma_2 (P_1) \otimes \Pi_a)$;
\item $(E,E') = (\gamma^j(P_2) \otimes \Pi_a\<-1\> , \gamma^j(P_2) \otimes \Pi_a)$; and
\item $(E,E') = (\gamma^j \sigma_2 (P_1) \otimes \Pi_{a\pm 1} \<-1\>, \gamma^j(P_2) \otimes \Pi_a)$.
\end{enumerate}
We again check that $\floor{\sigma_1(E)} \geq \ceil{\sigma_1(E')}$ whenever $j \neq \frac{n}{2}-1$ using \cref{lemma: phases of object even sigma1}, whereas we use \cref{lemma: phases of object even sigma2} to check that $\floor{\sigma_2(E)} \geq \ceil{\sigma_2(E')}$ (no condition on $j$ required).
\end{proof}

\begin{proposition}[compare w. \cref{sigma1 HN preserving}]
Suppose $A \in \cS(u_j)$ for $j \neq \frac{n}{2}-1$, or $A \in \cS(v_i)$ for $i$ arbitrary, with indecomposable $\tau_R$-semistable pieces in the HN filtration given by $E_i$, i.e. $\supp_{\tau_R}(A) = \bigoplus_{i=1}^m E_i$.
Then
\begin{enumerate}[(i)]
\item $\supp_{\tau_R}(\sigma_1(A)) = \bigoplus_{i=1}^m \supp_{\tau_R}(\sigma_1(E_i))$, and
\item $\sigma_1(A) \in \cS(v_0)=[P_1, P_2]$.
\end{enumerate}
Similarly, suppose $B \in \cS(v_j)$ for $j \neq 0$, or $B \in \cS(u_i)$ for $i$ arbitrary, with indecomposable $\tau_R$-semistable pieces in the HN filtration given by $F_i$, i.e. $\supp_{\tau_R}(B) = \bigoplus_{i=1}^m F_i$.
Then
\begin{enumerate}[(i)]
\item $\supp_{\tau_R}(\sigma_2(B)) = \bigoplus_{i=1}^m \supp_{\tau_R}(\sigma_1(F_i))$, and
\item $\sigma_2(B) \in \cS(u_0)=[P_2, \sigma_2 P_1]$.
\end{enumerate}
\end{proposition}
\begin{proof}
The proof is essentially the same for the first part  (which involves $\sigma_1$) using the appropriate corresponding lemma used in the proof of \cref{sigma1 HN preserving}.

For the second part (which involves $\sigma_2$), the only difference is that the indecomposable semistable objects $X_j$'s in the acquired weak HN filtration of $\sigma_2(A)$ are now indecomposable objects in $\<P_2\>^\oplus_{\tau_R} \oplus \<\sigma_2 P_1\>^\oplus_{\tau_R}$ instead.
But checking the fact that they still satisfy
\[
\phi(X_{i}) = \phi(X_{i+1}) \implies \Hom(X_{i+1}, X_{i}[1]) = 0
\]
follows from a similar proof used in \cref{sigma1 HN preserving}.
\end{proof}

\begin{corollary}[compare w. \cref{cor: conditions of automaton odd}]
Suppose $A \in \cS(u_j)$ for $j \neq k+ \frac{n}{2}-1$ or $A\in \cS(v_i)$ for $i$ arbitrary, with indecomposable semistable pieces in the HN filtration given by $E_i$, i.e. $\supp_{\tau_R}(A) = \bigoplus_{i=1}^m E_i$.
Then
\begin{enumerate}[(i)]
\item $\supp_{\tau_R}(\sigma_{\gamma^k(P_1)}(A)) = \bigoplus_{i=1}^m \supp_{\tau_R}(\sigma_{\gamma^k(P_1)}(E_i))$, and
\item $\sigma_{\gamma^k(P_1)}(A) \in [\gamma^k P_1, \gamma^k P_2]$.
\end{enumerate}
Suppose $B \in \cS(v_j)$ for $j \neq k$ or $B\in \cS(u_i)$ for $i$ arbitrary, with indecomposable semistable pieces in the HN filtration given by $F_i$, i.e. $\supp_{\tau_R}(B) = \bigoplus_{i=1}^m F_i$.
Then
\begin{enumerate}[(i)]
\item $\supp_{\tau_R}(\sigma_{\gamma^k(P_2)}(B)) = \bigoplus_{i=1}^m \supp_{\tau_R}(\sigma_{\gamma^k(P_2)}(F_i))$, and
\item $\sigma_{\gamma^k(P_2)}(B) \in [\gamma^k P_2, \gamma^k \sigma_2 P_1]$.
\end{enumerate}
\end{corollary}
\begin{proof}
The proof is essentially the same as in \cref{cor: conditions of automaton odd}.
\end{proof}

\comment{
\begin{definition}
For each $j \neq k + \frac{n-1}{2}$, define a linear transformation $M_{\sigma_{\gamma^k(P_1)}}^{v_j \ra v_k} : \R[e^{\pm t}]^{\{\gamma^j P_1, \gamma^j P_2\}} \ra \R[e^{\pm t}]^{\{\gamma^k P_1, \gamma^k P_2\}}$ by the matrix
\[
M_{\sigma_{\gamma^k(P_1)}}^{v_j \ra v_k} = 
	\begin{bmatrix}
	i_{[\gamma^k P_1, \gamma^k P_2]}(\sigma_{\gamma^k(P_1)}(\gamma^jP_1)) & i_{[\gamma^k P_1, \gamma^k P_2]}(\sigma_{\gamma^k(P_1)}(\gamma^jP_2))
	\end{bmatrix}
\]
\end{definition}

\begin{corollary}
For each $j \neq k + \frac{n-1}{2}$, we have that for all $A \in [\gamma^j P_1, \gamma^j P_2]$,
\[
M_{\sigma_{\gamma^k(P_1)}}^{v_j \ra v_k} i_{[\gamma^j P_1, \gamma^j P_2]} (A) = 
i_{[\gamma^k P_1, \gamma^k P_2]}(\sigma_{\gamma^k(P_1)}(A))
\]
\end{corollary}

We construct an automaton $\Lambda_n(\tau_R)$ (for $n$ odd) as follows.
\begin{enumerate}
\item There are a total of $n$ vertices, labelled by $\cV_j$ for $0 \leq j \leq n-1$, which are finitely free subsets of semistable objects over $K_0(\gTLJ_n)[q,q^{-1}, s,s^{-1}]$.
\item For each $j$, the labelled arrows ending in $\cV_j$ are as in \cref{fig:HN automaton for odd n}.
In particular, the set of letters $\Sigma$ in the whole automaton is
\[
\Sigma=
\{
\gamma, \gamma^{-1}, \sigma_{P_1}, \sigma_{\gamma(P_1)}, ..., \sigma_{\gamma^{n-1}(P_1)}
\}.
\]
\end{enumerate}

\begin{corollary} \label{odd HN automaton}
The $I_2(n)$ root automaton $\Lambda_n(\tau_R)$ for odd $n$ is a $\tau_R$-HN automaton.
\end{corollary}
\begin{proof}
By definition, all the vertices are labelled by finitely free subsets of semistable objects over $K_0(\gTLJ_n)[q,q^{-1}, s,s^{-1}]$.

The conditions (i) and (ii) for the arrows labelled by $\gamma$ and $\gamma^{-1}$ are easily checked.
For arrows labelled by $\sigma_{\gamma^k(P_1)}$, \cref{cor: conditions of HN automaton odd} shows directly that condition (i) is satisfied.
For condition (ii), just apply the definition of HN multiplicity vector to the equation
\[
\supp_{\tau_R}(\sigma_{\gamma^k(P_1)}(A)) = \bigoplus_i \supp_{\tau_R}(\sigma_{\gamma^k(P_1)}(S_i)).
\]

We have shown that all the edges labelled with $\gamma$, $\gamma^{-1}$ and $\sigma_{P_1} = \sigma_1$ satisfies the required properties of a $\tau_R$-HN automaton (cf. \cref{gamma loops V} and \cref{sigma1 HN preserving}).
So we are left with the edges labelled by the braids 
\[
\sigma_{\gamma^k(P_1)} = \gamma^k\sigma_{P_1}\gamma^{-k}
\]
for $0 \leq k \leq n-1$.

Let $A$ be an object with $\supp(A) \in \<\cV_j\>^\oplus_{\tau_R}$ for $j \neq k + \frac{n-1}{2} \mod n$ and let $E_i$ be the $\tau_R$-semistable pieces of $A$.
Using \cref{gamma sends semistable pieces} we know that the semistable pieces of the HN filtration of $\gamma^{-k}(A)$ are given by $\gamma^{-k}(E_i)$.
By \cref{sigma1 HN preserving}, we have that
\[
\supp_{\tau_R}(\sigma_{P_1}\gamma^{-k}(A)) = \bigoplus_i \supp_{\tau_R}(\sigma_{P_1}\gamma^{-k}(E_i) ) \in \< \cV_0 \>^\oplus_{\tau_R}.
\]
Applying \cref{gamma sends semistable pieces} again, we obtain
\begin{align*}
\supp_{\tau_R}(\gamma^k\sigma_{P_1}\gamma^{-k}(A)) 
&= \gamma^k \left( \supp_{\tau_R}(\sigma_{P_1}\gamma^{-k}(A)) \right) \\
&= \gamma^k \left( \bigoplus_i   \supp_{\tau_R}(\sigma_{P_1}\gamma^{-k}(E_i)) \right) \\
&= \bigoplus_i \gamma^k \left(   \supp_{\tau_R}(\sigma_{P_1}\gamma^{-k}(E_i)) \right) \in \< \cV_k \>^\oplus_{\tau_R}.
\end{align*}
This shows condition (i) of a HN automaton.

For condition (ii), we applying \cref{gamma sends semistable pieces} to the equation above again, which gives us
\[
\supp_{\tau_R}(\gamma^k\sigma_{P_1}\gamma^{-k}(A)) 
= \bigoplus_i \supp_{\tau_R}(\gamma^k\sigma_{P_1}\gamma^{-k}(E_i)).
\]
The required statement follows directly from the definition of $HN_{\tau_R}$.

\begin{enumerate}[(i)]
\item $
\supp(\sigma_{\gamma^k(P_1)}(A)) \in ob(\<\cV_k\>^\oplus_\tau),
$
and
\item $HN^{\cV_k}_\tau(\sigma_{\gamma^k(P_1)}(A)) = M^{\cV_j \ra \cV_k}_{\sigma_{\gamma^k(P_1)}} (HN^{\cV_j}_\tau(A))$.
\end{enumerate}
\[
HN_{\tau_R}(\gamma^j\sigma_1\gamma^{-j}(A)) = HN_{\tau_R}(\gamma^j\sigma_1\gamma^{-j}(\oplus_k E_k)).
\] 
Firstly, we know that $\gamma^{-j}(A)$ and $\gamma^{-j}(\oplus_k E_k)$ are both supported by $\cV_{i-j \mod n}$ (cf. \cref{gamma loops V}), and hence $\sigma_1\gamma^{-j}(A)$ and $\sigma_1\gamma^{-j}(\oplus_k E_k)$ are both supported by $\cV_0$ (cf. \cref{sigma1 HN preserving}).
Moreover, \cref{gamma sends semistable pieces} tells us that $\gamma^{-j}(\oplus_k E_k)$ is the direct sum of all the semistable pieces of $\gamma^{-j}(A)$.
Using \cref{sigma1 HN preserving} again, we know that $\sigma_1$ is HN preserving on both $\gamma^{-j}(A)$ and $\gamma^{-j}(\oplus_k E_k)$, which together imply
\[
HN_{\tau_R}(\sigma_1\gamma^{-j}(A)) = HN_{\tau_R}(\sigma_1\gamma^{-j}(\oplus_k E_k)).
\]
Let us denote the direct sum of the semistable pieces of $\sigma_1\gamma^{-j}(A)$ and $\sigma_1\gamma^{-j}(\oplus_k E_k)$ by $V$ and $V'$ respectively, so that
\[
[V] = HN_{\tau_R}(\sigma_1\gamma^{-j}(A)) = HN_{\tau_R}(\sigma_1\gamma^{-j}(\oplus_k E_k)) = [V'].
\]

Applying \cref{gamma sends semistable pieces} again, we know the direct sum of the $\tau_R$-semistable pieces of $\gamma^j\sigma_1\gamma^{-j}(A)$ and $\gamma^j\sigma_1\gamma^{-j}(\oplus_k E_k)$ are given by $\gamma^j(V)$ and $\gamma^j(V')$.
In particular,
\begin{equation} \label{gamma V and V'}
\begin{cases}
HN_{\tau_R}(\gamma^j\sigma_1\gamma^{-j}(A)) = [\gamma^j(V)] \\
HN_{\tau_R}(\gamma^j\sigma_1\gamma^{-j}(\oplus_k E_k)) = [\gamma^j(V')]
\end{cases}.
\end{equation}
Since $\gamma^j$ restricts to an autoequivalence on $\mathfrak{S}$ (cf. \cref{gamma phase reducing}), we have that $K_0(\gamma^j) : K_0(\mathfrak{S}) \ra K_0(\mathfrak{S})$.
With $V$ and $V'$ being objects in $\mathfrak{S}$, we know that
\[
[\gamma^j(V)] = K_0(\gamma^j)[V], \quad [\gamma^j(V')] = K_0(\gamma^j)[V'].
\]
But we have shown that $[V] = [V']$.
Together with \cref{gamma V and V'} we obtain
\[
HN_{\tau_R}(\gamma^j\sigma_1\gamma^{-j}(A)) = HN_{\tau_R}(\gamma^j\sigma_1\gamma^{-j}(\oplus_k E_k)),
\]
which shows $\gamma^j\sigma_1\gamma^{-j}$ is HN preserving on $A$ as required.
The fact that $\gamma^j\sigma_1\gamma^{-j}(A)$ is $\tau_R$-supported by $\cV_j$ just follows from $\sigma_1\gamma^{-j}(A)$ being supported by $\cV_0$ and that $\gamma^j$ sends objects supported by $\cV_0$ to objects supported by $\cV_j$.
\end{proof}
}

\subsection{Properties of $\Lambda_{\tau_R}(n)$ and $\Lambda_{\tau_R}(n)$-normal forms}
In this subsection, we study some extra properties of the mass automaton $\Lambda_{\tau_R}(n)$ that we will use in the algorithm.
Throughout, we shall view $\B(I_2(n))$ as a subgroup of autoequivalence of $\cK$ using its (faithful) action on $\cK$. 

Fixing a set of letters, we shall always assume that a word is reduced; namely all subwords of the form $ww^{-1}$ and $w^{-1}w$ will be necessarily deleted.
By convention, we allow the empty word, which represents the identity element.
As such, the \emph{length} of a (reduced) word is just given by the number of letters in the word (for which the empty word has length zero).

The set of letters that we will be using are $\left\{ \sigma_{\gamma^j P_1}^{\pm 1}, \sigma_{\gamma^j P_2}^{\pm 1}, \gamma, \gamma^{-1} \right\}$, where we do not differentiate between the letters $\sigma_{\gamma^j P_a}^{\pm 1}$ and $\sigma_{\gamma^k P_b}^{\pm 1}$ whenever $\gamma^j P_a$ is isomorphic to $\gamma^k P_b$ up to shifts by $\<-\>, [-]$ and up to tensoring with $\Pi_{n-2}$.
However, note that we do differentiate between $\gamma$ and $\sigma_{P_2} \sigma_{P_1}$ when viewing them as words, even though by definition $\gamma = \sigma_{P_2} \sigma_{P_1}$ in $\B(I_2(n))$.
For any braid element $\beta \in \B(I_2(n))$, we will use an over bar $\overline{\beta}$ when we want to specifically refer to its underlying word.
In particular, $\overline{\gamma} \neq \overline{\sigma_{P_2}\sigma_{P_1}}$ as words, but $\gamma = \sigma_{P_1} \sigma_{P_2} \in \B(I_2(n))$.

\vspace{5pt}
\begin{proposition} \label{prop: properties of tau_r mass automaton}
Let $b_1, b_2 \in \{ \sigma_{\gamma^j P_1}, \sigma_{\gamma^j P_2} \}$.
\begin{enumerate}[(1)]
\item Suppose that for any choice of arrows $a_1$ and $a_2$ labelled by $b_1$ and $b_2$ respectively, we have that $(a_1, a_2)$ is not a well-defined path in $\Lambda_{\tau_R}(n)$.
Then $b_1 b_2 = \gamma \in \B(I_2(n))$.
\item Suppose $b_1, b_2$ satisfies $
b_1 \gamma^{\pm 1} = \gamma^{\pm 1} b_2 \in \B(I_2(n))
$.
Then for any given path $p_1: v \ra v'$ in $\Lambda_{\tau_R}(n)$ recognising $\overline{b_1 \gamma^{\pm 1}}$, there exists a path $p_2: v \ra v'$ recognising and $\overline{\gamma^{\pm 1} b_2}$.
Vice versa, for any given path $p_2: v \ra v'$ recognising $\overline{\gamma^{\pm 1} b_2}$, there exists a path $p_1: v \ra v'$ in $\Lambda_{\tau_R}(n)$ recognising $\overline{b_1 \gamma^{\pm 1}}$.
\end{enumerate}
\end{proposition}
\begin{proof}
For statement (1), we shall prove it for the case where $n$ is even; the reader can check the case where $n$ is odd using a similar method.
So consider the case where $n$ is even and we are given $b_2 =\sigma_{\gamma^k(P_2)}$.
Any arrow $a_2$ that is labelled by $b_2$ must end at vertex $u_k$ by construction.
The only braid $g \in \{\sigma_{\gamma^j P_2}, \sigma_{\gamma^j P_1}, \gamma, \gamma^{\pm 1} \}$ which does not have an arrow labelled by $g$ coming out from $u_k$ is $\sigma_{\gamma^{k - \frac{n}{2}+1} P_1}$, and so it must be the case that $b_2 = \sigma_{\gamma^{k - \frac{n}{2} + 1} P_1}$.
It follows from a simple algebraic manipulation that $b_1b_2 = \sigma_{\gamma^{k - \frac{n}{2} + 1} P_1} \sigma_{\gamma^k P_2} = \gamma$.
Now consider the case where $b_2 =\sigma_{\gamma^k(P_1)}$, so any arrow $a_2$ labelled by $b_2$ must end at $v_k$.
Once again the only letter without an arrow coming out from $v_k$ must be $\sigma_{\gamma^k P_2}$, and so $b_1 = \sigma_{\gamma^k P_2}$.
It is easy to see that $b_1b_2 = \sigma_{\gamma^k P_2} \sigma_{\gamma^k(P_1)} = \gamma$.

For statement (2), we shall prove it for the case where $n$ is odd instead; once again the proof for $n$ even can be obtained from a similar argument.
Assuming $n$ is odd, we can let $b_2 = \sigma_{\gamma^j P_1}$ for some $j$, where by the relation given we must have $b_1 = \sigma_{\gamma^{j \pm 1} P_1}$.
Suppose we are given a path recognising $\overline{\gamma^{\pm 1} b_2}$ (subscript of vertices taken modulo $n$):
\[
v_i \xra{b_2} v_j \xra{\gamma^{\pm 1}} v_{j \pm 1}.
\]
Then by construction of $\Lambda_{\tau_R}(n)$, we must have $i \neq j + \frac{n-1}{2}$.
However,
\[
i \neq j + \frac{n-1}{2} \iff i \pm 1 \neq j \pm 1 + \frac{n-1}{2},
\]
so there must exist a path
\[
v_i \xra{\gamma^{\pm 1}} v_{i \pm 1}  \xra{b_2} v_{j \pm 1}
\]
as required.
The vice versa statement follows similarly.
\end{proof}

We can now prove that every braid has an expression which is recognised by $\Lambda_{\tau_R}(n)$:
\begin{proposition}\label{normal form}
Every braid $\beta \in \B(I_2(n))$ can be expressed as a word using the letters $\{\sigma_{\gamma^j P_2}, \sigma_{\gamma^j P_1}, \gamma, \gamma^{- 1} \}$ in the following form:
\[
\overline{\beta} = b^{m_k}_k b^{m_{k-1}}_{k-1} \cdots b^{m_1}_1\gamma^s,
\]
where $s \in \Z$, $k \in \Z_{\geq 0}$ and $m_i \in \Z_{> 0}$, such that
\begin{enumerate}
\item $b_i  \in \{\sigma_{\gamma^j P_1}, \sigma_{\gamma^j P_2}\}$ for all $i$, and
\item there exists arrows $e_1,e_2, ...,e_k$ in $\Lambda_{\tau_R}(n)$ labelled by $b_1,b_2,...,b_k$ respectively, which together form a well-defined path $(e_1,e_2, ...., e_k)$ in $\Lambda_{\tau_R}(n)$.
\end{enumerate}
In particular, every braid in $\B(I_2(n))$ can be written as a word recognised by $\Lambda_{\tau_R}(n)$.
\end{proposition}
\begin{proof}
Let $\beta \in \B(I_2(n))$, where using the standard generators, can be expressed using the letters $\{\sigma_{P_1}, \sigma_{P_1}^{-1}, \sigma_{P_2}, \sigma_{P_2}^{-1} \}$.
By definition, we have that
\[
\gamma := \sigma_{P_2} \sigma_{P_1}  \in \B(I_2(n)),
\]
giving us the relations:
\[
\sigma_{P_2}^{-1} = \sigma_{P_1} \gamma^{-1}, \quad
\sigma_{P_1}^{-1} = \gamma^{-1} \sigma_{P_2} \in \B(I_2(n)).
\]
Using these relations, we can rewrite $\beta$ into a word using only the letters $\{\sigma_{P_1}, \sigma_{P_2}, \gamma, \gamma^{-1} \}$.
Now using the relation
\[
\gamma^{\pm 1} \sigma_{P_i} = \sigma_{\gamma^{\pm 1} P_i} \gamma^{\pm 1}  \in \B(I_2(n)),
\]
we can move all of the $\gamma^{\pm 1}$ to the right, so that the braid $\beta$ is then expressed as
\[
\overline{\beta} = b^{m_k}_k b^{m_{k-1}}_{k-1} \cdots b^{m_1}_1\gamma^s,
\]
for some $s \in \Z$, $k \in \Z_{\geq 0}$ and $m_i \in \Z_{> 0}$, such that each $b_i \in \{\sigma_{\gamma^j P_1}, \sigma_{\gamma^j P_2}\}$.
Now between each pair $b_i$ and $b_{i+1}$, whenever there do not exist arrows $a_i$ and $a_{i+1}$ labelled by $b_i$ and $b_{i+1}$ respectively that connect to form a well-defined path $(a_i, a_{i+1})$, \cref{prop: properties of tau_r mass automaton} (1) tells us that $b_ib_{i+1} = \gamma \in \B(I_2(n))$.
As such, we can rewrite $\overline{\beta}$ for each such occurrence of $b_ib_{i+1}$, and move all of the new $\gamma$ to the right as before.
This process strictly reduces the word length of $\beta$ with respect to the letters $\{\sigma_{\gamma^j P_2}, \sigma_{\gamma^j P_1}, \gamma, \gamma^{-1} \}$ and must therefore terminate.

It is easy to see that the resulting word representing $\beta$ is recognised by $\Lambda_{\tau_R}(n)$, where the loop arrows labelled by $b_i$ are used when $m_i >1$.
\end{proof}

\comment{
The following two propositions show that every braid $\beta \in \B(I_2(n))$ can be expressed using the letters $\{\sigma_{\gamma^j P_1}, \sigma_{\gamma^j P_2}, \gamma^{\pm 1} \}$, and moreover, a word recognised by $\Lambda_{\tau_R}(n)$:
\begin{proposition}\label{normal form odd}
Suppose that $n$ is odd.
Denote $Q_j = \gamma^j(P_1)$.
Every braid $\beta \in \B(I_2(n))$ is expressible as a word
\[
\overline{\beta} = \sigma^{m_1}_{Q_{a_1}}\sigma^{m_2}_{Q_{a_2}} \cdots \sigma^{m_k}_{Q_{a_k}}\gamma^s
\]
for some $s \in \Z$, $k \in \Z_{\geq 0}$, $m_i \in \Z_{> 0}$ and $0 \leq a_i \leq n-1$, such that $a_{i+1} \neq a_i + \frac{n+1}{2} \mod n$ for all $i$. 
In particular, every braid in $\B(I_2(n))$ can be written as a word recognised by $\Lambda_{\tau_R}(n)$.
\end{proposition}
\begin{proof}
By definition, we have that
\[
\gamma := \sigma_{P_2} \sigma_{P_1} = \sigma_{Q_{\frac{n+1}{2}}} \sigma_{Q_0}.
\]
Note that $\sigma_{Q_n} = \sigma_{P_1 \<2n\>[2n-2]} = \sigma_{P_1}$, and more generally $\sigma_{Q_m} = \sigma_{Q_{m'}}$ for $m = m' \mod n$.
Thus from now on, any subscript of $Q$ will be taken modulo $n$.
By applying the relation $\gamma \sigma_Q = \sigma_{\gamma(Q)} \gamma$ repeatedly, we get that
\begin{equation} \label{gamma relation}
\sigma_{Q_{\frac{n+1}{2}+j}} \sigma_{Q_j }= \gamma^j\gamma\gamma^{-j} = \gamma.
\end{equation}
Using this we obtain the relation
\begin{equation} \label{replace inverse}
\sigma_{Q_j}^{-1} = \gamma^{-1}\sigma_{Q_{\frac{n+1}{2}+j}},
\end{equation}
where $\sigma_{P_1}^{-1} = \sigma_{Q_0}^{-1}$ and $\sigma_{P_2}^{-1} = \sigma_{Q_{\frac{n+1}{2}}}^{-1}$.
With \cref{replace inverse}, we may rewrite any braid $\beta$ in $\B(I_2(n))$ as a word consists solely of the terms $\gamma$, $\gamma^{-1}$ and $\sigma_{Q_j}$.
Applying the relations
\[
\gamma \sigma_{Q_j} = \sigma_{Q_{j+1}} \gamma, \quad \gamma^{-1}\sigma_{Q_{j+1}} = \sigma_{Q_j} \gamma^{-1},
\]
we may move all the $\gamma$ and $\gamma^{-1}$ terms to the right, so that $\beta$ is now a word given by $w\gamma^i$, with $w$ a word involving only the terms $\sigma_{Q_j}$.

Replace any occurence of $\sigma_{Q_{\frac{n+1}{2}+j}} \sigma_{Q_j}$ in $w$ with $\gamma$ (cf. \cref{gamma relation}), and move all $\gamma$ to the right as before.
Note that this strictly reduces the word length with respect to the terms $\{ \gamma, \gamma^{-1}, \sigma_{Q_j} \}$, so the process must terminate.

It is easy to see that the resulting word representing $\beta$ is recognised by $\Lambda_{\tau_R}(n)$.
\end{proof}

The case where $n$ is even is similar, using the mass automaton for $n$ even instead.
The proof is omitted as it follows using a similar proof to the odd case:
\begin{proposition}[compare with \cref{normal form odd}] \label{normal form even}
Suppose that $n$ is even.
Denote 
\[
\mathfrak{T} = \{\gamma^j(P_1)\}_{j=0}^{\frac{n}{2}} \cup \{\gamma^j(P_2)\}_{j=0}^{\frac{n}{2}}.
\]
Every braid $\beta \in \B(I_2(n))$ is expressible as a word
\[
\overline{\beta} = \sigma_{T_1}^{m_1} \sigma_{T_2}^{m_2} \cdots \sigma_{T_k}^{m_k} \gamma^s
\] 
for some $s\in \Z$, $k \in \Z_{\geq 0}$, $m_i \in \Z_{> 0}$ and $T_i \in \mathfrak{T}$, such that for each $i$,
\begin{enumerate}
\item if $T_i = \gamma^k(P_1)$, then $T_{i-1} \neq \gamma^{k + \frac{n}{2} -1}P_2$;
\item if $T_i = \gamma^k(P_2)$, then $T_{i-1} \neq \gamma^{k}P_1$.
\end{enumerate}
In particular, every braid in $\B(I_2(n))$ can be written as a word recognised by $\Lambda^{\text{even}}_{\tau_R}(n)$.
\end{proposition}
}

We can then define the following:
\begin{definition}
The expression of $\beta \in \B(I_2(n))$ in \cref{normal form} is called a \emph{$\Lambda_{\tau_R}(n)$-normal form of $\beta$}.
\end{definition}
Note that if $\beta$ is not just some power of $\gamma$, then all the possible paths in the automaton recognising the obtained $\Lambda_{\tau_R}(n)$-normal form of $\beta$ ends at the same vertex; it is the unique vertex that has incoming arrows labelled by $b_k$.

\begin{remark}
The identity element of $\B(I_2(n))$ has $\Lambda_{\tau_R}(n)$-normal form given by the empty word.
\end{remark}

\subsection{The algorithm and the main theorem}
We are now ready to tackle the main theorem of this thesis.
Let us start with the following proposition, which will allow us to compute the mass growth (equivalently level-entropy) using the mass matrices in our mass automaton $\Lambda_{\tau_R}(n)$ for some braids in $\B(I_2(n))$ under certain assumptions.

\begin{proposition} \label{prop: mass growth computed by PF}
Suppose $p$ is a closed path (starts and ends at the same vertex) in $\Lambda_{\tau_R}(n)$ and suppose the corresponding mass matrix $\cM(p)$ is an irreducible matrix for all $t \in \R$.
Then the mass growth of $\beta := \cS(p)$ for each $t \in \R$ is given by
\[
h_{\tau,t}(\beta) = \log PF \left( \cM(p) \right) = h_t(\beta),
\]
with $PF(\cM(p))$ denoting the Perron-Frobenius eigenvalue of $\cM(p)$.
\end{proposition}
\begin{proof}
Firstly, note all sets $[\gamma^j P_1, \gamma^j P_2]$ and $[\gamma^j P_2, \gamma^j \sigma_2 P_1]$ contain a $\X$-split generator $G$, namely $G := \left( \gamma^j P_1 \oplus \gamma^j P_2 \right) \otimes \left( \bigoplus_{i=0}^{n-2} \Pi_i \right)$ and $\left( \gamma^j P_2 \oplus \gamma^j \sigma_2 P_1 \right) \otimes \left( \bigoplus_{i=0}^{n-2} \Pi_i \right)$ respectively.
By \cref{prop: compute entropy from simpler object}, we can compute mass growth using
\[
h_{\tau,t}(\sigma_p) = \lim_{N \ra \infty} \frac{1}{N} \log( m_{\tau,t}(\sigma_p^N (X) )) = L,
\]
with $X = \gamma^j P_1 \oplus \gamma^j P_2 $ (resp. $X= \gamma^j P_2 \oplus \gamma^j \sigma_2 P_1$) and $A = \bigoplus_{i=0}^{n-2} \Pi_i $.
As such, regardless of the parity of $n$, if $p$ is a closed path starting and ending at $v$, we shall compute mass growth using $X := X_1 \oplus X_2$ where $\cS(v) = [X_1,X_2]$.
By construction of $\Lambda_{\tau_R}(n)$, we have that
\begin{align*}
m_{\tau,t}(\sigma_p^N (X)) &= \mathfrak{m}_{v}(\iota_{v}(\sigma_p^N (X)) ) \\
&= \mathfrak{m}_{v}( \cM(p)^N(\iota_{v}(X)) ) \\
&= w \cM(p)^N 
\begin{bmatrix}
1 \\ 1
\end{bmatrix},
\end{align*}
with $w$ the row vector $\begin{bmatrix}
m_{\tau_R,t}(X_1) & m_{\tau_R,t}(X_2)
\end{bmatrix}$.

Note that $\cM(p)$ is by definition a non-negative matrix for all $t \in \R$ (since for each arrow $a$, the mass matrix $\cM(a)$ is non-negative for all $t$).
By the irreducibility assumption on $\cM(p)$, we have a Perron-Frobenius eigenvalue $\lambda(t) := PF(\cM(p))$ for each $t$, with a corresponding positive eigenvector $u_t = \begin{bmatrix}
a_t \\ b_t
\end{bmatrix}$ with $a_t,b_t >0$ for all $t$.
For each $t$, denote $K_\text{max}(t) = \max\{a_t,b_t\}, K_\text{min}(t) = \min \{a_t,b_t\}$.
Then we see that 
\[
w \cM(p)^N 
\begin{bmatrix}
1 \\ 1
\end{bmatrix} \leq 
\frac{1}{K_\text{min}(t)} w \cM(p)^N 
u_t, 
	\quad
\frac{1}{K_\text{max}(t)} w \cM(p)^N 
u_t,
\leq w \cM(p)^N 
\begin{bmatrix}
1 \\ 1
\end{bmatrix} 
\]
Combining them gives us the following inequality:
\[
\frac{1}{K_\text{max}(t)} \lambda(t)^N w u_t 
	\leq  
w \cM(p)^N 
\begin{bmatrix}
1 \\ 1
\end{bmatrix}  
	\leq
\frac{1}{K_\text{min}(t)} \lambda(t)^N  w u_t.
\]
The result now follows by applying $\frac{1}{N}\log(-)$ to the inequality and taking the limit as $N \ra \infty$.
\end{proof}
Even though \cref{normal form} shows that every braid can be expressed as a word recognised by $\Lambda_{\tau_R}(n)$, it is clear that not all braids are expressible as words recognised by \emph{closed} paths.
Moreover, even if a braid is expressible as a word recognised by a closed path, we still require the corresponding mass matrix to be irreducible.

Luckily for us, our mass growth (being equal to level entropy) is invariant under conjugation, and with the exception of the periodic braids (which we already know how to compute, and are easy!), all other braids turn out to be \emph{conjugate} to braids which are expressible as words recognised by closed paths in our automaton.
Moreover, if the corresponding mass matrix of the closed path turns out to be reducible, the braid must be reducible too (of which, again, we already know how to compute its mass growth).
As such, we are able to compute the mass growth for any given braid!
All these will be achieved through an algorithm which we shall present later (see pg.\pageref{sec: algorithm}).

We shall first prove a few results that will come in handy:
\begin{proposition} \label{conjugate odd}
Let $\beta = b^{m_k}_k b^{m_{k-1}}_{k-1} \cdots b^{m_1}_1\gamma^s$ be in $\Lambda_{\tau_R}(n)$-normal form.
Suppose that $\sum_{i=1}^k m_i \geq 2$.
Then $\overline{\beta}$ is a word recognised by a closed path if and only if 
\[
\overline{b_k^{-1}} \cdot \overline{\beta} \cdot \overline{b_k} = b^{m_k-1}_k b^{m_{k-1}}_{k-1} \cdots b^{m_1}_1 \gamma^s b_k
\]
is a recognised word in $\Lambda_{\tau_R}(n)$.
\end{proposition}
\begin{proof}
Suppose $\overline{b_k^{-1}} \cdot \overline{\beta} \cdot \overline{b_k}$ is a recognised word, with the corresponding path recognising it given by a sequence of arrows $(e_0, e_1, ..., e_\ell)$ in the automaton.
Since the first letter in the expression is $b_k \neq \gamma^{\pm 1}$, $e_0$ must be an arrow with label $b_k$, which ends at some unique vertex $v$ by the construction of $\Lambda_{\tau_R}(n)$ (for both odd and even $n$).
Thus $e_1$ must be an arrow which starts at $v$.
Since $\sum_{i=1}^k m_i \geq 2$, the arrow $e_\ell$ is either labelled by $b_k$ (when $m_k = 1$), or $b_{k-1} \neq \gamma^{\pm 1}$.
Regardless, by the definition of the $\Lambda_{\tau_R}(n)$-normal form, $e_\ell$ must end at a vertex which has an outwards arrow $e$ labelled by $b_k$.
It follows that the path in the automaton given by the sequence of arrows $(e_1, ..., e_\ell, e)$ accepts the word $\overline{\beta}$, which is a loop that starts and ends at vertex $v$.

Conversely, suppose $\overline{\beta}$ is recognised by the path $(e_1, \cdots ,e_\ell)$ which starts and ends at the same vertex $v$.
Then let $e$ be the loop arrow on vertex $v$, which must be labelled by $b_k$ by the construction of $\Lambda_{\tau_R}(n)$.
The path $(e, e_1, \cdots, e_{\ell-1})$ must then be a path that accepts the word $\overline{b_k^{-1}} \cdot \overline{\beta} \cdot \overline{b_k}$.
\end{proof}

\begin{remark}
Note that if $\sum_{i=1}^k m_i = 1$, i.e. $\overline{\beta} = b_1 \gamma^s$, then 
\[
\overline{b^{-1}_1} \overline{\beta} \overline{b_1} = \gamma^s b_1
\]
is always a recognised word.
This case will be treated separately in \cref{beta^2 not recognised gives periodic}.
The case where $\sum_{i=1}^k m_i = 0$ is vacuous and $\overline{\beta} = \gamma^s$ is always a recognised word.
\end{remark}

\begin{proposition} \label{reduce word length odd}
Let $\beta = b^{m_k}_k b^{m_{k-1}}_{k-1} \cdots b^{m_1}_1 \gamma^s$ be in $\Lambda_{\tau_R}(n)$-normal form with $\sum_{i=1}^k m_i \geq 2$.
Suppose that
\[
\overline{b_k^{-1}} \cdot \overline{\beta} \cdot \overline{b_k} = b^{m_k-1}_k b^{m_{k-1}}_{k-1} \cdots b^{m_1}_1 \gamma^s b_k
\]
is not a recognised word.
Then $b^{-1}_k \beta b_k$ can be rewritten into the following $\Lambda_{\tau_R}(n)$-normal form:
\[
\overline{ b^{-1}_k \beta b_k } = b^{m_k-1}_k b^{m_{k-1}}_{k-1} \cdots b^{m_1-1}_1 \gamma^{s+1}.
\]
In particular, the word length of the obtained $\Lambda_{\tau_R}(n)$-normal form of $b^{-1}_k \beta b_k$ is strictly less than the word length of the original $\Lambda_{\tau_R}(n)$-normal form of $\beta$.
\end{proposition}
\begin{proof}
Suppose $b_k = \sigma_X$ for some $X \in \cK$, so that
$
\gamma^s b_k = \sigma_{\gamma^s X} \gamma^s.
$
Then 
\[
b_k^{-1}\beta b_k = b^{m_k-1}_k b^{m_{k-1}}_{k-1} \cdots b^{m_1}_1 \sigma_{\gamma^s X} \gamma^s \in \B(I_2(n)).
\]
By \cref{prop: properties of tau_r mass automaton} (2), the assumption that the word $\overline{b_k^{-1}} \cdot \overline{\beta}\cdot \overline{b_k}$ is not recognised implies that the word $b^{m_k-1}_k b^{m_{k-1}}_{k-1} \cdots b^{m_1}_1 \sigma_{\gamma^s X} \gamma^s$ is also not recognised.
In particular, the subword $b_1\sigma_{\gamma^s X}$ cannot be recognised; otherwise we could have constructed a path that recognises the whole word.
Applying \cref{prop: properties of tau_r mass automaton} (1), we get that
$b_1\sigma_{\gamma^s X} = \gamma \in \B(I_2(n))$, giving us 
\[
\overline{ b^{-1}_k \beta b_k } = b^{m_k-1}_k b^{m_{k-1}}_{k-1} \cdots b^{m_1-1}_1 \gamma^{s+1}
\]
as required.
\end{proof}

\begin{proposition} \label{beta^2 not recognised gives periodic}
Let $\overline{\beta} = b_1 \gamma^s$ be in $\Lambda_{\tau_R}(n)$-normal form. Then
\begin{enumerate}[(1)]
\item If $\overline{\beta} \cdot \overline{\beta}$ is a recognised word, then $\overline{\beta}$ is recognised by a closed path.
\item If $\overline{\beta} \cdot \overline{\beta}$ is not a recognised word, then $\beta^2 = \gamma^{2s+1}$, i.e. $\beta$ is periodic.
\end{enumerate}
\end{proposition}
\begin{proof}
If $\overline{\beta} \cdot \overline{\beta}$ is recognised, then there is a path $p = (e_1,...,e_s,f_1,e_{s+1},...,e_{2s}, f_2)$ such that $\cS(p) = b_1\gamma^s b_1 \gamma^s$.
In particular, $\cS(f_i) = b_1$ for both $f_i$'s, which implies that both $f_i$'s must end at the same vertex $v$.
Since $p$ is a well-defined path, this implies that $e_{s+1}$ must be an arrow that starts at $v$.
It now follows that the subpath $p' = (e_{s+1},...,e_{2k}, f_2)$ is a closed path starting and ending at the same vertex $v$ which recognises $\overline{\beta}$.

If $\overline{\beta} \cdot \overline{\beta} = b_1 \gamma^s b_1 \gamma^s$ is not a recognised word, we shall rewrite $\beta^2$ into a $\Lambda_{\tau_R}(n)$-normal form:
\[
\overline{\beta^2} = b_1 b_2 \gamma^{2s},
\] 
where if $b_1 = \sigma_{\gamma^j P_i}$, then $b_2 = \sigma_{\gamma^{j+s} P_i}$.
Note that the only relation required throughout the rewriting process is $\gamma^{\pm 1} \sigma_{\gamma^j P_i} = \sigma_{\gamma^{j\pm 1} P_i} \gamma^{\pm 1}$.
Using (2) of \cref{prop: properties of tau_r mass automaton}, it must be the case that $b_1 b_2 \gamma^{2s}$ is not a recognised word, or else it would contradict the fact that $\overline{\beta} \cdot \overline{\beta} = b_1 \gamma^s b_1 \gamma^s$ is not a recognised word.
In particular, no arrows $a_1$ and $a_2$ labelled by $b_1$ and $b_2$ respectively can form a well-defined path $(a_1, a_2)$, and so by (1) of \cref{prop: properties of tau_r mass automaton}, we have that $b_1b_2 = \gamma$.
Hence
$
\beta^2 = \gamma^{2s+1}
$ as required.
\end{proof}
\begin{remark}
Note that only (1) in \cref{beta^2 not recognised gives periodic} will be relevant when $n$ is even.
This is because $\beta \in \B(I_2(n))$ of the form $b_1 \gamma^s$ is always recognised by a closed path when $n$ is even and the converse of (1) is trivially true.
\end{remark}

\begin{proposition} \label{reducible matrix iff reducible or periodic}
Let $\overline{\beta}$ be in $\Lambda_{\tau_R}(n)$-normal form and is recognised by some closed path $p$ in $\Lambda_{\tau_R}(n)$.
With $*$ denoting a (real valued) function which is strictly positive for all $t \in \R$, we have that
\begin{enumerate}[(i)]
\item $\cM(p) = 
\begin{bmatrix}
* & 0 \\
0 & *
\end{bmatrix}$ implies $\beta$ is periodic;
\item $\cM(p) = 
\begin{bmatrix}
* & * \\
0 & *
\end{bmatrix}$ or 
$ \begin{bmatrix}
* & 0 \\
* & *
\end{bmatrix}$
implies $\beta$ is reducible; and
\item $\cM(p) = 
\begin{bmatrix}
* & * \\
* & *
\end{bmatrix}$ implies $\beta$ is pseudo-Anosov.
\end{enumerate}
In particular, the associated mass matrix $\cM(p)$ is irreducible if and only if $\beta$ is pseudo-Anosov;
or equivalently $\cM(p)$ is reducible if and only if $\beta$ is reducible or periodic.
\end{proposition}
\begin{proof}
Let us start with the case where $n$ is odd.
We shall look for the possible forms of the mass matrices that occur in our mass automaton.
By definition of $\cM$, it is easy to see that all arrows in the automaton labelled by $\gamma$ has corresponding mass matrix of the form
$
\begin{bmatrix}
* & 0 \\
0 & *
\end{bmatrix}.
$
Using \cref{eqn: sigma1 on V_j generator 1} and \cref{eqn: sigma1 on V_j generator 2}, we see that the arrow
\[
v_{n-1} \xra{\sigma_{P_1}} v_0
\]
has corresponding mass matrix of the form
\[
\cM( v_{n-1} \xra{\sigma_{P_1}} v_0 )= 
\begin{bmatrix}
* & 0 \\
* & *
\end{bmatrix}.
\]
Similarly, the loop arrow
\[
v_0 \xra{\sigma_{P_1}} v_0 
\]
has corresponding mass matrix of the form
\[
\cM( v_0 \xra{\sigma_{P_1}} v_0 )= 
\begin{bmatrix}
* & * \\
0 & *
\end{bmatrix},
\]
whereas all the other arrows labelled by $\sigma_{P_1}$ have corresponding mass matrices of the form 
\[
\cM( v_{j} \xra{\sigma_{P_1}} v_0 )= 
\begin{bmatrix}
* & * \\
* & *
\end{bmatrix}, \quad \text{for } j \neq \frac{n-1}{2}.
\]

Using the fact that $\gamma^j \sigma_{P_1} \gamma^{-j} = \sigma_{\gamma^j(P_1)}$ and (1) from \cref{prop: properties of tau_r mass automaton},
we can classify all of the mass matrices of the automaton into the following 4 possible forms:
\begin{enumerate}
\item form 1 -- 
$
\begin{bmatrix}
* & 0 \\
0 & *
\end{bmatrix}
$.
All the arrows labelled by $\gamma$ and $\gamma^{-1}$ have corresponding mass matrices of this form.
\item form 2 -- 
$
\begin{bmatrix}
* & 0 \\
* & *
\end{bmatrix}
$.
All the arrows 
\[
v_{j+n-1} \xra{\sigma_{\gamma^j(P_1)}} v_j
\]
have corresponding mass matrices of this form.
\item form 3 --
$
\begin{bmatrix}
* & * \\
0 & *
\end{bmatrix}
$.
All the arrows (these are the loops)
\[
v_{j} \xra{\sigma_{\gamma^j(P_1)}} v_j
\]
have corresponding mass matrices of this form.
\item form 4 -- 
$
\begin{bmatrix}
* & * \\
* & *
\end{bmatrix}
$.
All other arrows have corresponding mass matrices of this form.
\end{enumerate}

Now let $\overline{\beta}$ be the given $\Lambda_{\tau_R}(n)$-normal form recognised by a closed path $p = (e_1, e_2, ..., e_k)$, with corresponding mass matrix given by
\[
\cM(p) = \cM(e_k)\cM(e_{k-1})...\cM(e_1).
\]
Since each $\cM(e_i)$ is of one of the four forms above with entries always non-negative, it follows that $\cM(p)$ must also be one of the four forms above.
As such, if we can prove (i), (ii), and (iii) in the statement of the proposition, then the final statement ``$\cM(p)$ is irreducible if and only if $\beta$ is pseudo-Anosov'' follows directly from the definitions.

For part (i), suppose $\cM(p)$ is of the form $\begin{bmatrix}
* & 0 \\
0 & *
\end{bmatrix}$.
Then it follows that all the $\cM(e_i)$ are of the same form $\begin{bmatrix}
* & 0 \\
0 & *
\end{bmatrix}$.
In particular, each edge $e_i$ must be labelled by $\gamma^{\pm 1}$, and hence $\beta$ is a periodic braid by definition.

Now for part (ii), let us first suppose that $\cM(p)$ is of the form $\begin{bmatrix}
* & 0 \\
* & *
\end{bmatrix}$.
Then each $\cM(e_i)$ must either be of the form $\begin{bmatrix}
* & 0 \\
0 & *
\end{bmatrix}$ or $\begin{bmatrix}
* & 0 \\
* & *
\end{bmatrix}$, with at least one of them having the latter form.
Using the fact that $p$ is a closed path, we can deduce that $\overline{\beta}$ must be of the form:
\[
\overline{\beta} = \sigma_{\gamma^{j+k}(P_1)}\cdots\sigma_{\gamma^{j+1}(P_1)}\sigma_{\gamma^{j}(P_1)}\gamma^{\ell n - k - 1}
\]
for some $0 \leq j \leq n-1$, $k \in \mathbb{N}_0$ and $\ell \in \Z$.
A simple calculation using $\sigma_1\gamma^{-1} = \sigma_2^{-1}$ shows that this can be simplified to 
\[
\beta = \gamma^{j+k}\sigma_2^{-k}\sigma_1\gamma^{\ell n - k - 1 - j}.
\]
Conjugating by $\gamma^{-j-k-1}$, we see that
\[
\gamma^{-j-k-1}\beta \gamma^{j+k+1} = \sigma_1^{-1} \sigma_2^{-k-1} \sigma_1 \gamma^{\ell n}.
\]
Since $\gamma^{\ell n}$ is central, we see that $\beta$ is conjugate to $\sigma_2^{-k-1} \gamma^{\ell n}$, which is reducible.

Now suppose instead that the associated mass matrix $\cM(p)$ is of the form
$
\begin{bmatrix}
* & * \\
0 & *
\end{bmatrix}
$.
Then each $\cM(e_i)$ must either be of the form $\begin{bmatrix}
* & 0 \\
0 & *
\end{bmatrix}$ or $\begin{bmatrix}
* & * \\
0 & *
\end{bmatrix}$, with at least one of them having the latter form.
Once again by rewriting $\beta$ into $\Lambda_{\tau_R}(n)$-normal form if necessary, we may assume without loss of generality that
\[
\beta = \sigma^k_{\gamma^j(P_1)} \gamma^{\ell n}
\]
with $0 \leq j \leq n-1$, for some $k \in \mathbb{N}$ and $\ell \in \Z$, which again shows that $\beta$ is reducible.

Lastly, for part (iii), suppose that $\cM(p)$ is irreducible.
Then $\cM(p)$ satisfies the requirements of \cref{prop: mass growth computed by PF}, which allows us to compute the mass growth of $\beta$ using the Perron-Frobenius eigenvalue of $\cM(p)$:
\[
h_0(\beta) = \log PF(\cM(p)|_{t=0}).
\]
We have shown previously in \cref{periodic reducible entropy 0 when t=0} that the periodic braids and reducible braids have zero mass growth at $t=0$.
Therefore to show $\beta$ is pseudo-Anosov, it is sufficient to show that $\log PF(\cM(p)|_{t=0})>0$, or equivalently $PF(\cM(p)|_{t=0}) > 1$.

Since we are given that $\cM(p)$ is irreducible, so it must be of the form $\begin{bmatrix}
* & * \\
* & *
\end{bmatrix}$.
Using the fact that
\[
PF\dim(\Pi_a) = \Delta_a \left(2\cos \left( \frac{\pi}{n} \right) \right)  \geq 1,
\]
we see that $\cM(p)$ evaluated at $t=0$ is given by
\[
M_p|_{t=0} = 
\begin{bmatrix}
a & b \\
c & d
\end{bmatrix}, \quad \text{ where } a,b,c,d \geq 1.
\]
Now one can show that $PF(\cM(p)|_{t=0}) > 1$ by using the theorem of Perron-Frobenius:
\[
PF(\cM(p)|_{t=0}) \geq \min_i \sum_j (M_p|_{t=0})_{i,j} \geq 2,
\]
or on a direct analysis on the roots of the characteristic polynomial of $\cM(p)|_{t=0}$.
This completes the proof for the case where $n$ is odd.

The analysis for the case where $n$ is even is similar, where we can classify all of the mass matrices of the automaton into the following 4 possible forms:
\begin{enumerate}
\item form 1 -- 
$
\begin{bmatrix}
* & 0 \\
0 & *
\end{bmatrix}
$.
All the arrows labelled by $\gamma$ and $\gamma^{-1}$ have corresponding mass matrices of this form.
\item form 2 -- 
$
\begin{bmatrix}
* & 0 \\
* & *
\end{bmatrix}
$.
All the arrows 
\[
v_{j+\frac{n}{2}-1} \xra{\sigma_{\gamma^j(P_1)}} v_j, \quad u_{j+\frac{n}{2}-1} \xra{\sigma_{\gamma^j(P_2)}} u_j
\]
have corresponding mass matrices of this form.
\item form 3 --
$
\begin{bmatrix}
* & * \\
0 & *
\end{bmatrix}
$.
All the arrows (these are the loops)
\[
v_{j} \xra{\sigma_{\gamma^j(P_1)}} v_j, \quad u_{j} \xra{\sigma_{\gamma^j(P_2)}} u_j
\]
have corresponding mass matrices of this form.
\item form 4 -- 
$
\begin{bmatrix}
* & * \\
* & *
\end{bmatrix}
$.
All the other arrows have corresponding mass matrices of this form.
\end{enumerate}
The rest of the proof is now similar to the odd case.
\comment{
The characteristic polynomial is then $\lambda^2 - (a+d) \lambda + ad - bc$, and so the roots are
\[
\lambda^\pm = \frac{a+d}{2} \pm \frac{\sqrt{(a+d)^2 - 4(ad-bc)}}{2}.
\]
Using the fact that $a, b, c, d \geq 1$, one can show that $\lambda^+ = PF_0(M_\beta) > 1$, and by \cref{prop: mass growth computed by PF}, we get that
\[
h_0(\beta) = \log PF_0(M_\beta) > 0.
\]
Since the reducible braids and periodic braids have level entropy 0 at $t = 0$ (cf. \cref{periodic reducible entropy 0 when t=0}), $\beta$ must be pseudo-Anosov.
}
\end{proof}

We are now ready to present our algorithm:

\subsubsection*{The algorithm that decides whether a braid is periodic, reducible or pseudo-Anosov} \label{sec: algorithm}
\begin{enumerate}[$\bullet$]
\item Input: A braid $\beta$ in $\B(I_2(n))$.
\item Output: (i) return some conjugate of $\beta$, denoted as out($\beta$),
(ii) return the type of $\beta$: periodic, reducible or pseudo-Anosov, which is the same as the type of out($\beta$), and
(iii) return the mass growth $h_{\tau_R,t}(\beta)$, which again is the same as $h_{\tau_R,t}$(out$(\beta)$).
\end{enumerate}
Start of algorithm:
\begin{enumerate}
\item Write $\beta$ in $\Lambda_{\tau_R}(n)$-normal form, if necessary (using \cref{normal form}):
\[
\overline{\beta} = b^{m_k}_k b^{m_{k-1}}_{k-1} \cdots b^{m_1}_1 \gamma^s.
\]
\item If $\sum_{i=1}^k m_i = 0$, $\beta$ is periodic by definition.\\
Set out($\beta$) $:= \beta$ and compute its entropy $h_t(\beta)$ using \cref{prop: entropy of periodic}.
The algorithm halts.
\item If $\sum_{i=1}^k m_i = 1$, check if $\overline{\beta}\cdot \overline{\beta}$ is a recognised word.

If it is, then there is a closed path $p$ recognising the $\Lambda_{\tau_R}(n)$-normal form of $\beta$ in step (1) (cf. \cref{beta^2 not recognised gives periodic} (1)), and we jump to step (5).\\
Otherwise, $\beta$ must be periodic with $\beta^2 = \gamma^{2s+1}$ (cf. \cref{beta^2 not recognised gives periodic} (2)).\\
Set out($\beta$) $:= \beta$ and compute its mass growth $h_t(\beta) = \frac{1}{2}h_t(\beta^2) = \frac{1}{2}h_t(\gamma^{2s+1})$ using \cref{prop: entropy of periodic}.
The algorithm halts.
\item We are now left with the case where $\sum_{i=1}^k m_i \geq 2$. \\
If the word
\[
\overline{b_k^{-1}} \cdot \overline{\beta} \cdot \overline{b_k} = b^{m_k-1}_k b^{m_{k-1}}_{k-1} \cdots b^{m_1}_1 \gamma^s b_k
\]
is a recognised word, \cref{conjugate odd} tells us that $\overline{\beta}$ is recognised by some closed path $p$ and we jump to step (5). \\
Else, we rewrite $b_k^{-1} \beta b_k$ into a $\Lambda_{\tau_R}(n)$-normal form as in \cref{reduce word length odd} and set it as our new $\beta$.
Repeat from (1).
\item If we arrive at this step, then we must have a braid $\beta$ with $\overline{\beta}$ in $\Lambda_{\tau_R}(n)$-normal form that is recognised by a closed path $p$.
Moreover, $\sum_{i=1}^k m_i \geq 1$.
Using \cref{reducible matrix iff reducible or periodic}, the irreducibility of  the mass matrix $\cM(p)$ associated to $p$ decides the type of $\beta$ as follows:
\begin{enumerate}[-]
\item If $\cM(p)$ is reducible, then $\beta$ must be reducible. This is because $\sum_{i=1}^k m_i \geq 1$, which implies that $\cM(p) =
\begin{bmatrix}
* & * \\
0 & *
\end{bmatrix}$ or
$
\begin{bmatrix}
* & 0 \\
* & *
\end{bmatrix}$.
We proceed as in \cref{reducible matrix iff reducible or periodic} to arrive at a conjugate of $\beta$ of the form $\sigma_i^k \chi^{\ell}$ for some $i\in \{1,2\}, k \neq 0$ and $\ell \in \Z$, where $\chi$ is the generator of the centre of $\B(I_2(n))$.
Set out($\beta) := \sigma_i^k \chi^{\ell}$.
Compute the mass growth $h_t(\beta) = h_t($out$(\beta))$ using \cref{entropy reducible}.
The algorithm halts.
\item If $\cM(p)$ is irreducible, then $\beta$ must be pseudo-Anosov.
Set out($\beta) := \beta$, and we have that $h_t(\beta) = \log PF(\cM(p))$ (cf. \cref{prop: mass growth computed by PF}).
The algorithm halts.
\end{enumerate}
\end{enumerate}
Note that \cref{reduce word length odd} guarantees the termination of this algorithm, since the word length of the $\Lambda_{\tau_R}(n)$-normal forms used are strictly decreasing.

All in all, we obtain the following main theorem of this thesis:
\begin{theorem}\label{theorem: classification}
The algorithm in pg.\pageref{sec: algorithm} provides a way to decide the type of any braid element $\beta \in \B(I_2(n))$: either periodic, reducible or pseudo-Anosov.
Moreover, it allows us to compute the mass growth of any braid, where in particular for $\beta$ pseudo-Anosov, the mass growth for each $t\in \R$ can be computed from the Perron-Frobenius eigenvalue of $\cM(p)$:
\[
h_{\tau_R,t}(\beta) = \log PF(\cM(p)) = h_t(\beta),
\]
with $p$ some closed path in $\Lambda_{\tau_R}(n)$ recognising a word representing some conjugate of $\beta$, and $\cM(p)$ is the corresponding mass matrix of $p$.
\end{theorem}

Recall that when specialised to $t=0$, the mass growth of reducible and periodic braids are known to be zero (cf. \cref{periodic reducible entropy 0 when t=0}).
From the proof of \cref{reducible matrix iff reducible or periodic}, we see that for $\beta$ pseudo-Anosov,
\[
h_t(\beta) \geq \log (2) >0.
\]
Combining with \cref{entropy reducible} and \cref{prop: entropy of periodic}, we obtain the following result:
\begin{corollary} \label{cor: categorical entropy t=0}
The mass growth at $t=0$ of $\beta \in \B(I_2(n))$ is given by:
\[
h_0(\beta) = \begin{cases}
	0, &\text{ if $\beta$ is periodic or reducible}; \\
	\log(PF(\cM(p)|_{t=0})) \geq \log(2), &\text{ if $\beta$ is pseudo-Anosov}.
	\end{cases}
\]
Moreover, $h_t(\beta)$, as a whole function, completely determines the type of $\beta$:
\begin{enumerate}[(i)]
\item $\beta$ is periodic if and only if $h_0(\beta) = 0$ and $h_t(\beta)$ is linear over $t$;
\item $\beta$ is reducible if and only if $h_0(\beta) = 0$ and $h_t(\beta)$ is strictly piece-wise linear over $t$; and
\item $\beta$ is pseudo-Anosov if and only if $h_0(\beta) > 0$.
\end{enumerate}
\end{corollary}

\subsection{Example: $n = 5$} \label{example n=5}
\comment{
The Temperley-Lieb-Jones category in place is $\TLJ_5$, which is a fusion category with four simple objects $\{ \Pi_0, \Pi_1, \Pi_2, \Pi_3 \}$.

The Burau representation $\rho: \B(I_2(5)) \ra \GL_2(\Z[2\cos\left( \frac{\pi}{5}\right)][q,q^{-1}])$ is a $q$-deformation of the standard geometric representation, and can be described explicitly by the following matrices on the generators:
\begin{align*}
\rho(\sigma_1) = 
	\begin{bmatrix}
	-q^2 & -\delta q \\
	0    & 1
	\end{bmatrix}, \quad
\rho(\sigma_2) = 
	\begin{bmatrix}
	-1        & 0 \\
	-\delta q & -q^2
	\end{bmatrix}
\end{align*}
where $\delta := 2\cos\left( \frac{\pi}{5} \right) \in \R$ is the golden ratio, satisfying $\delta^2 = 1 + \delta$.
}

\comment{
To simplify calculations, we shall fix $t=0$ throughout this example.
As such, the two $\Z$-gradings $\<-\>$ and $[-]$ will not matter, hence they will be left out in all of the upcoming calculations.
Furthermore, all HN matrices of the root automaton will have its entries in $\R_{\geq 0}$, and 
}

This subsection contains the case $n=5$ as a working example.
We remind the reader that
\[
PF\dim: K_0(\TLJ_5) \ra \R
\]
is defined by $[\Pi_1], [\Pi_2] \mapsto 2\cos\left( \frac{\pi}{5} \right)=: \delta$ (which is the golden ratio), and $[\Pi_0] , [\Pi_3] \mapsto 1$.

The complete $I_2(5)$ $\tau_R$-mass automaton is given in \Cref{fig: mass automaton 5}, where all arrows of the same colour are labelled by the same braid.
We have that $\cS(v_j) = [\gamma^j P_1, \gamma^j P_2]$ and we will describe the mass matrices associated to each arrow below.

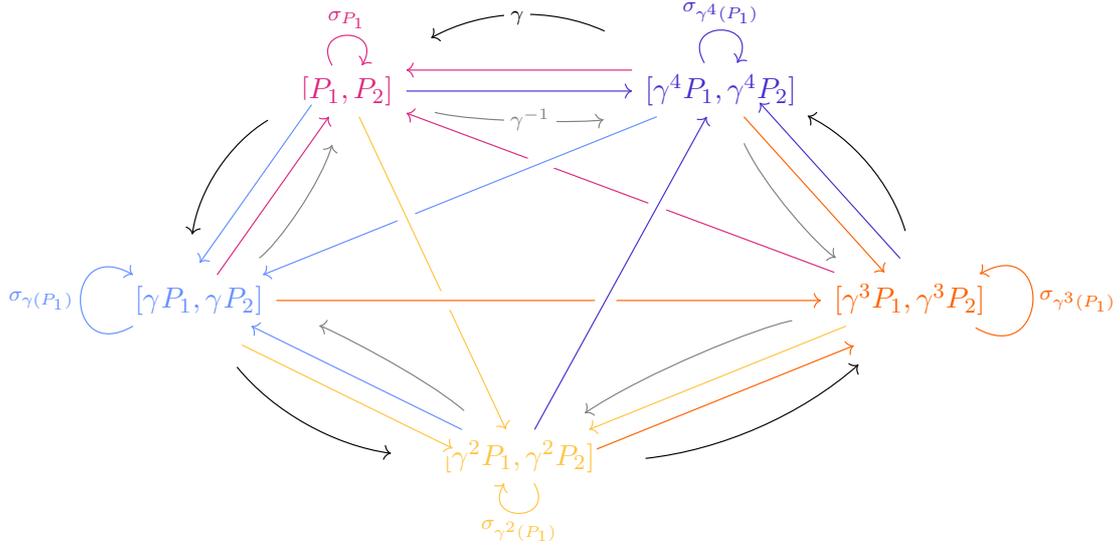
\begin{figure}[H]
\[\begin{tikzcd}[column sep = 1 ex]
	& {\color{CBF_pink} [P_1,P_2]} &&&& {\color{CBF_purple} [\gamma^{4}P_1, \gamma^{4}P_2]} \\
	\\
	\\
	{\color{CBF_blue} [\gamma^{}P_1, \gamma^{}P_2] } &&&&&& {\color{CBF_orange} [\gamma^{3}P_1, \gamma^{3}P_2]} \\
	\\
	&&& {\color{CBF_yellow} [\gamma^{2}P_1, \gamma^{2}P_2]}
	%
	%
	\arrow["\gamma" description, from=1-6, to=1-2, shift right=3, curve={height=15pt}, shorten >= 1em, shorten <= 1em]
	\arrow[from=4-1, to=1-2, curve={height=12pt}, shorten >= 1em, shorten <= 1em, gray]
	\arrow["\sigma_{P_1}", from=1-2, to=1-2, CBF_pink, loop, out=120, in=60, looseness=4]
	\arrow[from=1-6, to=1-2, crossing over, shift right=3, CBF_pink]
	\arrow[from=4-7, to=1-2, crossing over, CBF_pink]
	\arrow[from=4-1, to=1-2, crossing over, CBF_pink]
	%
	%
	\arrow[from=1-2, to=4-1, shift right=3, curve={height=12pt}, shorten >= 1em, shorten <= 1em]
	\arrow[from=6-4, to=4-1, curve={height=12pt}, shorten >= 1.5em, shorten <= 1em, gray]
	\arrow["\sigma_{\gamma(P_1)}", from=4-1, to=4-1, CBF_blue, loop, out=-150, in=150, looseness=4,  shift left=3]
	\arrow[from=1-2, to=4-1, crossing over, shift right=3 ,CBF_blue]
	\arrow[from=1-6, to=4-1, crossing over, CBF_blue]
	\arrow[from=6-4, to=4-1, crossing over, CBF_blue]
	%
	%
	\arrow[from=4-1, to=6-4, shift right=3, curve={height=12pt}, shorten >= 1em, shorten <= 1em]
	\arrow[from=4-7, to=6-4, curve={height=12pt}, shorten >= 1em, shorten <= 1em, gray]
	\arrow["\sigma_{\gamma^2(P_1)}", from=6-4, to=6-4, CBF_yellow, loop, out=-60, in=-120, looseness=4]
	\arrow[from=4-1, to=6-4, crossing over, shift right=3, CBF_yellow]
	\arrow[from=1-2, to=6-4, crossing over, CBF_yellow]
	\arrow[from=4-7, to=6-4, crossing over, CBF_yellow]
	%
	%
	\arrow[from=6-4, to=4-7, shift right=3, curve={height=12pt}, shorten >= 1em, shorten <= 1em]
	\arrow[from=1-6, to=4-7, curve={height=12pt}, shorten >= 1em, shorten <= 1em, gray]
	\arrow["\sigma_{\gamma^3(P_1)}", from=4-7, to=4-7, CBF_orange, loop, out=-30, in=30, looseness=4,  shift right=3, swap]
	\arrow[from=6-4, to=4-7, crossing over, shift right=3, CBF_orange]
	\arrow[from=1-6, to=4-7, crossing over, CBF_orange]
	\arrow[from=4-1, to=4-7, crossing over, CBF_orange]
	%
	%
	\arrow[from=4-7, to=1-6, shift right=3, curve={height=12pt}, shorten >= 1em, shorten <= 1em]
	\arrow["{\color{gray} \gamma^{-1}}" description, from=1-2, to=1-6, curve={height=12pt}, shorten >= 1em, shorten <= 1em, gray]
	\arrow["\sigma_{\gamma^4(P_1)}", from=1-6, to=1-6, CBF_purple, loop, out=120, in=60, looseness=4]
	\arrow[from=6-4, to=1-6, crossing over, CBF_purple]
	\arrow[from=4-7, to=1-6, crossing over, shift right=3, CBF_purple]
	\arrow[from=1-2, to=1-6, crossing over, CBF_purple]
\end{tikzcd}\]

\caption{The $\tau_R$-mass automaton $\Lambda_{\tau_R}(5)$ of $\B(I_2(5))$, with color-coded arrows.}
\label{fig: mass automaton 5}
\end{figure}

Any arrow labelled by $\gamma^{\pm 1}$ that is \emph{not} ${\color{gray}\gamma^{-1}}: {\color{CBF_pink}v_0} \leftrightarrows {\color{CBF_purple}v_4} : \gamma$ will have the same mass matrix given by the identity matrix:
\[
\cM(\gamma) = \begin{bmatrix}
1 & 0 \\
0 & 1
\end{bmatrix}
=
\cM({\gamma^{-1}}),
\]
whereas
\[
\cM( {\color{CBF_pink} v_0} {\color{gray} \xra{\gamma^{-1}}} {\color{CBF_purple}v_4}) = \begin{bmatrix}
e^{2t} & 0 \\
0 & e^{2t}
\end{bmatrix}
\]
and
\[
\cM( {\color{CBF_purple}v_4} \xra{\gamma} {\color{CBF_pink}v_0}) = \begin{bmatrix}
e^{-2t} & 0 \\
0 & e^{-2t}
\end{bmatrix}.
\]
A simple application of \cref{gamma sends semistable pieces} tells us that by conjugating with the appropriate $\cM(v_i \xra{\gamma^{\pm 1}} v_{i\pm 1})$, the mass matrices $\cM(a)$ for each arrow $a$ can be obtained from just the four mass matrices corresponding to the {\color{CBF_pink} pink} arrows (labelled by {\color{CBF_pink} $\sigma_{P_1}$}).
To illustrate, we start by using \cref{lemma for braid relation} to compute the following (or refer to \cref{eqn: sigma1 on V_j generator 1} and \cref{eqn: sigma1 on V_j generator 2}):
\begin{align*}
\sigma_{P_1}(\gamma(P_1)) &\cong 0 \ra P_1\otimes \Pi_2\<4\>[3] \ra P_2 \otimes \Pi_1\<3\>[2] \ra 0 \\
\sigma_{P_1}(\gamma(P_2)) &\cong 0 \ra P_1\otimes \Pi_3\<3\>[3] \ra P_2 \otimes \Pi_2\<2\>[2] \ra 0,
\end{align*}
which shows that the $\tau_R$-HN semistable pieces of $\sigma_{P_1}(\gamma(P_1))$ and $\sigma_{P_1}(\gamma(P_2))$ are given by $(P_2 \otimes \Pi_1\<3\>[2], P_1\otimes \Pi_2\<4\>[3])$ and $(P_2 \otimes \Pi_2\<2\>[2], P_1\otimes \Pi_3\<3\>[3])$ respectively.
Thus, the corresponding mass matrix of the arrow ${\color{CBF_blue} v_1} {\color{CBF_pink} \xra{\sigma_{P_1}} v_0}$ is then given by 
\[
\cM( {{\color{CBF_blue} v_1} {\color{CBF_pink} \xra{\sigma_{P_1}} v_0}})
=
\begin{bmatrix}
\delta e^{-t} & 1 \\
\delta e^{-t} & \delta
\end{bmatrix}.
\]
All mass matrices of the arrows $v_{i+1} \xra{\sigma_{\gamma^i(P_i)}} v_{i}$ can then be obtained from conjugating with the appropriate mass matrices of arrows labelled by $\gamma$; we remind the reader again that this does depend on the arrows and not just $\gamma^{\pm 1}$.
For example, $\cM( {\color{CBF_yellow} v_2} {\color{CBF_blue} \xra{\sigma_{\gamma(P_1)} } v_1})$ agrees with $\cM( {{\color{CBF_blue} v_1} {\color{CBF_pink} \xra{\sigma_{P_1}} v_0}})$ since the conjugation only involves the identity matrices:
\[
\cM( {\color{CBF_yellow} v_2} {\color{CBF_blue} \xra{\sigma_{\gamma(P_1)} } v_1}) = 
\cM( {\color{CBF_pink} v_0} \xra{\gamma} {\color{CBF_blue}v_1})
\cM( {{\color{CBF_blue} v_1} {\color{CBF_pink} \xra{\sigma_{P_1}} v_0}})
\cM( {\color{CBF_yellow} v_2} {\color{gray} \xra{\gamma^{-1}} } {\color{CBF_blue}v_1})
=\begin{bmatrix}
\delta e^{-t} & 1 \\
\delta e^{-t} & \delta
\end{bmatrix},
\]
whereas one of the mass matrix of $\gamma^{-1}$ (namely $\cM( {\color{CBF_pink} v_0} {\color{gray} \xra{\gamma^{-1}}} {\color{CBF_purple}v_4})$) in obtaining $\cM( {\color{CBF_pink} v_0} {\color{CBF_purple} \xra{\sigma_{\gamma^4(P_1)} } v_4})$ is not the identity matrix:
\begin{align*}
&\cM( {\color{CBF_pink} v_0} {\color{CBF_purple} \xra{\sigma_{\gamma^4(P_1)} } v_4}) \\
=
&\cM( {\color{CBF_orange} v_3} \xra{\gamma} {\color{CBF_purple}v_4}) \cdots
\cM( {\color{CBF_pink} v_0} \xra{\gamma} {\color{CBF_blue}v_1}) \\
&\cM( {\color{CBF_blue}v_1} {\color{CBF_pink} \xra{\sigma_{P_1}} v_0}) \\
&\cM( {\color{CBF_yellow}v_2} {\color{gray} \xra{\gamma^{-1}}} {\color{CBF_blue}v_1})
\cM( {\color{CBF_orange}v_3} {\color{gray} \xra{\gamma^{-1}}} {\color{CBF_yellow}v_2})
\cM( {\color{CBF_purple} v_4} {\color{gray} \xra{\gamma^{-1}}} {\color{CBF_orange}v_3})
\cM( {\color{CBF_pink} v_0} {\color{gray} \xra{\gamma^{-1}}} {\color{CBF_purple}v_4}) \\
= &\begin{bmatrix}
\delta e^{t} & e^{2t} \\
\delta e^{t} & \delta e^{2t}
\end{bmatrix}.
&
\end{align*}
Note however that all these matrices evaluated at $t=0$ are indeed the same, since all mass matrices of arrows labelled by $\gamma$ are the identity matrix when $t=0$.

We shall now illustrate the algorithm with some examples.
Consider the braid given by $\beta := \sigma_{P_1}^2 \sigma_{P_2}^2 \in \B(I_2(5))$.
Note that this is already in $\Lambda_{\tau_R}(5)$-normal form, and $\sum_{i=1}^k m_i = 4$.
We see that its conjugate
$
\overline{\sigma_{P_1}^{-1}} \cdot \overline{\sigma_{P_1}^2\sigma_{P_2}^2} \cdot \overline{\sigma_{P_1}} = \sigma_{P_1} \sigma_{P_2}^2 \sigma_{P_1}
$
is not a recognised word,  so we rewrite the conjugate into $\Lambda_{\tau_R}(n)$-normal form:
$
\sigma_{P_1} \sigma_{P_2} \gamma = \sigma_{P_1} \sigma_{\gamma^2P_1} \gamma.
$
We now have $\sum_{i=1}^k m_i = 2$, and taking conjugation by $\sigma_{P_1}^{-1}$ again gives us the expression
$
\overline{\sigma_{P_1}^{-1}} \cdot \overline{\sigma_{P_1} \sigma_{\gamma^2P_1} \gamma} \cdot \overline{\sigma_{P_1}} = \sigma_{\gamma^2P_1} \gamma \sigma_{P_1},
$
which is now a recognised word.
Hence $\overline{\beta} = \sigma_{P_1} \sigma_{\gamma^2P_1} \gamma$ is recognised in the automaton by a closed path $p$, which is given by
\[
{\color{CBF_pink} v_0} \xra{\gamma} {\color{CBF_blue} v_1} {\color{CBF_orange} \xra{\sigma_{\gamma^2P_1}} v_3} {\color{CBF_pink} \xra{\sigma_{P_1}} v_0},
\]
where the corresponding mass matrices are
\[
\cM({\color{CBF_pink}v_0}) \xra{
	\begin{bmatrix}
	1 & 0 \\
	0 & 1
	\end{bmatrix}		
	} 
\cM({\color{CBF_blue}v_1}) \xra{
	\begin{bmatrix}
	\delta   & \delta e^t \\
	1        & \delta e^t
	\end{bmatrix}
	} 
\cM({\color{CBF_orange}v_3}) \xra{
	\begin{bmatrix}
	\delta e^{-2t} & \delta e^{-t}\\
	e^{-2t}        & \delta e^{-t}
	\end{bmatrix}
	}
\cM({\color{CBF_pink}v_0}).
\]
Thus
\[
\cM(p) = \begin{bmatrix}
	\delta e^{-2t} & \delta e^{-t}\\
	e^{-2t}        & \delta e^{-t}
	\end{bmatrix}
	\begin{bmatrix}
	\delta   & \delta e^t \\
	1        & \delta e^t
	\end{bmatrix}
 = \delta
 	\begin{bmatrix}
	e^{-t} + \delta e^{-2t} & \delta(e^{-t} + 1) \\
	e^{-t} + e^{-2t}       & e^{-t} + \delta
	\end{bmatrix},
\]
which is irreducible.
This shows that $\sigma_{P_1}^2 \sigma_{P_2}^2$ is pseudo-Anosov and the mass growth can be calculated from the Perron-Frobenius eigenvalue of the matrix above (apply the quadratic formula and take the the root obtained from the positive sum).
In particular, the mass growth at $t=0$ is given by
\begin{align*}
h_{\tau_R, 0}(\sigma_{P_1}^2 \sigma_{P_2}^2) 
&= \log \left( PF \left( \begin{bmatrix}
	1 + 2\delta & 2(1+\delta) \\
	2\delta       & 1 + 2\delta
	\end{bmatrix} \right) \right)  \\
&= \log \left(2\sqrt{\sqrt{5}+2} + \sqrt{5} + 2 \right) >1.
\end{align*}

Now consider the (rather complicated) braid $\beta:= \sigma_{P_2} \sigma_{P_1} \sigma_{P_2} \sigma_{P_1}^{-1} \sigma_{P_2} \sigma_{P_1} \sigma_{P_2}^{-1} \sigma_{P_1} \sigma_{P_2}^3 \sigma_{P_1}$.
We start by rewriting it into $\Lambda_{\tau_R}(n)$-normal form as follows:
\begin{align*}
\beta &= \sigma_{P_2} \sigma_{P_1} \sigma_{P_2} \gamma^{-1} \sigma_{P_2}^2 \sigma_{P_1}^2 \gamma^{-1} \sigma_{P_1} \sigma_{P_2}^3 \sigma_{P_1} \quad (\text{removing } \sigma_i^{-1}) \\
&= \gamma \sigma_{P_2} \gamma^{-1} \sigma_{P_2} \gamma \sigma_{P_1} \gamma^{-1} \sigma_{P_1} \sigma_{P_2}^2 \gamma \quad (\text{removing non-viable paths with } \gamma) \\
&= \sigma_{\gamma(P_2)} \sigma_{P_2} \sigma_{\gamma(P_1)} \sigma_{P_1} \sigma_{P_2}^2 \gamma \quad (\text{pushing all } \gamma^{\pm1 } \text{ to the right}).
\end{align*}
Now that $\beta$ is in $\Lambda_{\tau_R}(n)$-normal form, we take its conjugate:
\[
\overline{\sigma_{\gamma(P_2)}^{-1}} \cdot \overline{\beta} \cdot \overline{\sigma_{\gamma(P_2)}} = \sigma_{P_2} \sigma_{\gamma(P_1)} \sigma_{P_1} \sigma_{P_2}^2 \gamma \sigma_{\gamma(P_2)},
\]
which is not a recognised word.
So we rewrite it into $\Lambda_{\tau_R}(n)$-normal form:
\begin{align*}
\sigma_{P_2} \sigma_{\gamma(P_1)} \sigma_{P_1} \sigma_{P_2}^2 \gamma \sigma_{\gamma(P_2)} 
&= \sigma_{P_2} \sigma_{\gamma(P_1)} \sigma_{P_1} \sigma_{P_2}^2 \sigma_{1} \gamma \\
&= \sigma_{P_2} \sigma_{\gamma(P_1)} \sigma_{P_1} \sigma_{P_2} \gamma^2.
\end{align*}
Conjugating again, we see that 
\[
\overline{\sigma_{P_2}^{-1}} \cdot \overline{\sigma_{P_2} \sigma_{\gamma(P_1)} \sigma_{P_1} \sigma_{P_2} \gamma^2} \cdot \overline{\sigma_{P_2}} = \sigma_{\gamma(P_1)} \sigma_{P_1} \sigma_{P_2} \gamma^2\sigma_{P_2}
\]
is again not a recognised word, and can be rewritten into $\Lambda_{\tau_R}(n)$-normal form given by
\begin{align*}
\sigma_{\gamma(P_1)} \sigma_{P_1} \sigma_{P_2} \gamma^2\sigma_{P_2} 
&= \sigma_{\gamma(P_1)} \sigma_{P_1} \sigma_{P_2} \sigma_{P_1} \gamma^2 \\
&=\sigma_{\gamma(P_1)} \sigma_{P_1} \gamma^3.
\end{align*}
Conjugating again shows that this is a recognised word with closed path $p$ given by
\[
{\color{CBF_blue} v_1} \xra{\gamma} {\color{CBF_yellow} v_2} \xra{\gamma} {\color{CBF_orange} v_3} \xra{\gamma} {\color{CBF_purple} v_4} {\color{CBF_pink} \xra{\sigma_{P_1}} v_0} {\color{CBF_blue} \xra{\sigma_{\gamma(P_1)}} v_1}.
\]
We claim that this braid is reducible, which can be sufficiently checked from the mass matrix of $p$ at $t=0$, given by
\[
\cM(p)|_{t=0} = 
	\begin{bmatrix}
	1 & 0 \\
	\delta & 1 
	\end{bmatrix}
	\begin{bmatrix}
	1 & 0 \\
	\delta & 1 
	\end{bmatrix}
	\begin{bmatrix}
	1 & 0 \\
	0 & 1 
	\end{bmatrix}^3
= 	\begin{bmatrix}
	1 & 0 \\
	2\delta & 1 
	\end{bmatrix}.
\]
We conclude that $\beta$ is reducible and the mass growth $h_{\tau_R, 0}(\beta)$ at $t=0$ is $0$.
The motivated reader should try and work out what $h_{\tau_R, t}(\beta)$ is, by rewriting $\beta$ into a conjugate of $\sigma_{i}^k \chi^\ell$ following the steps in the proof of \cref{reducible matrix iff reducible or periodic}.

\subsection{Example: $n = 4$} \label{sect: example n=4}
This subsection contains the working example for $n=4$.
We remind the reader that
$
PF\dim: K_0(\TLJ_4) \ra \R,
$
is defined by $[\Pi_1] \mapsto 2\cos\left( \frac{\pi}{4} \right)= \sqrt{2}$, and $[\Pi_0] , [\Pi_2] \mapsto 1$.

The complete $I_2(4)$ $\tau_R$-mass automaton is given in \cref{fig: mass automaton 4}, where all arrows of the same colour are labelled by the same braid.

\begin{figure}[H]
\[
\begin{tikzcd}[column sep = 20ex, row sep = 20ex]
{\color{CBF_pink} [P_1, P_2]}
	\ar[d, "\gamma^{\pm 1}" description, shift right =3]
	\ar[d, color = CBF_purple, curve={height=12pt}, shift right = 8] 
	\ar[out = 120, in = 60, loop,  "\sigma_{P_1}", color = CBF_pink] 
	\ar[rd, color = CBF_yellow, crossing over]
&  
{\color{CBF_blue} [P_2, \sigma_{P_2} P_1]} 
	\ar[d, "\gamma^{\pm 1}" description, shift right =3] 
	\ar[d, color = CBF_yellow, curve={height=-12pt}, shift left = 8]
	\ar[out = 120, in = 60, loop, "\sigma_{P_2}", color=CBF_blue] 
	\ar[l, color = CBF_pink]
\\	
{\color{CBF_purple} [\gamma P_1, \gamma P_2]} 
	\ar[u, "\gamma^{\pm 1}" description, shift right = 3] 
	\ar[u, curve={height= 12pt}, shift right = 8, crossing over, color = CBF_pink]
	\ar[out = 300, in = 240, loop, "\sigma_{\gamma P_1}", color = CBF_purple] 
	\ar[ru, color=CBF_blue, crossing over]
&
{\color{CBF_yellow} [\gamma P_2, \gamma \sigma_{P_2} P_1]}
 	\ar[u, "\gamma^{\pm 1}" description, shift right =3] 
	\ar[u, color=CBF_blue, curve={height=-12pt}, shift left = 8, crossing over]
	\ar[out = 300, in = 240, loop,  "\sigma_{\gamma P_2}", color = CBF_yellow] 
	\ar[l, color = CBF_purple]
\end{tikzcd}
\]
\caption{The $\tau_R$-mass automaton $\Lambda_{\tau_R}(4)$ of  $\B(I_2(4))$, with color-coded arrows. The left column consists of the vertices ${\color{CBF_pink} v_0}, {\color{CBF_purple} v_1}$ and the right column consists of the vertices ${\color{CBF_blue} u_0}, {\color{CBF_yellow} u_1}$.}
\label{fig: mass automaton 4}
\end{figure}
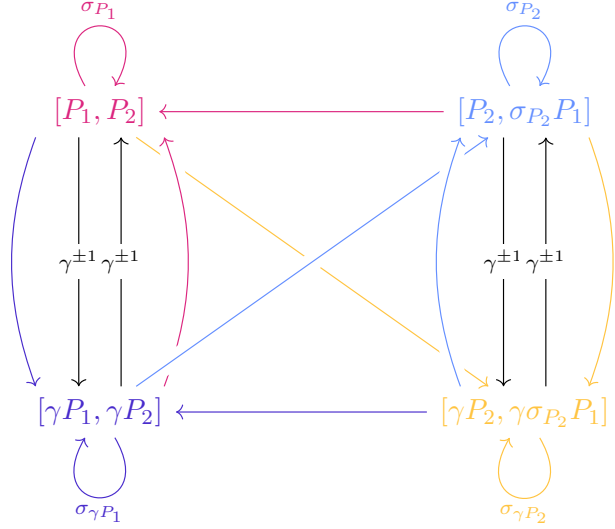
The mass matrices at $t=0$ are given by
\begin{enumerate}[$\bullet$]
\item $\cM({\color{CBF_blue} u_0 \xra{\sigma_{P_2}} u_0})|_{t=0} 
	= \begin{bmatrix}
	1 & \sqrt{2} \\
	0 & 1
	\end{bmatrix} = 
	\cM({\color{CBF_pink} v_0 \xra{\sigma_{P_1}} v_0})|_{t=0}$
	.
\item $\cM({\color{CBF_yellow} u_1} {\color{CBF_blue}  \xra{\sigma_{P_2}} u_0})|_{t=0} = 
	\begin{bmatrix}
	1        & 0 \\
	\sqrt{2} & 1
	\end{bmatrix} =
	\cM({\color{CBF_purple} v_1} {\color{CBF_pink}  \xra{\sigma_{P_1}} v_0})|_{t=0}$
	.
\item $\cM({\color{CBF_purple} v_1} {\color{CBF_blue}  \xra{\sigma_{P_2}} u_0})|_{t=0} = 
	\begin{bmatrix}
	\sqrt{2} & 1 \\
	1        & \sqrt{2}
	\end{bmatrix} =
	\cM({\color{CBF_blue} u_0} {\color{CBF_pink}  \xra{\sigma_{P_1}} v_0})|_{t=0}$
	.
\end{enumerate}
The rest of the mass matrices can be obtained through conjugation with the mass matrix of $\gamma$ (for $t=0$ they are all the identity matrix).
We leave it to the reader to work out what the mass matrices for generic $t$ are.

We shall list a few braids here, together with their corresponding types and mass growths, where the reader may use to compare with their own calculation.
\begin{enumerate}[$\bullet$]
\item $\beta:= \sigma_{P_2}^2 \sigma_{P_1}^3$, pseudo-Anosov, $h_0(\beta) = \log (5+2\sqrt{6})$.
\item $\beta:= \sigma_{P_2}\sigma_{P_1}\sigma_{P_2}\sigma_{P_1}^{-1}\sigma_{P_2}^{-1}\sigma_{P_1}$, periodic (conjugate to $\gamma$), $h_0(\beta) = 0$.
\item $\beta:= \sigma_{P_2}^{-2} \sigma_{P_1} \sigma_{P_2}$, reducible (conjugate to $\sigma_{P_1}^2\gamma^{-1}$), $h_0(\beta) = 0$.
\end{enumerate}

\chapter*{Afterword: categorical vs. mapping class representations}
\addcontentsline{toc}{chapter}{Afterword: categorical vs. mapping class representations}
Since the theory of triangulated categories is supposed to be a categorical analogue of the theory of surfaces, it is only natural to compare and contrast some of the known categorical representations with the known mapping class representations (homomorphisms into mapping class groups) of the generalised braid groups.

One main difference is that the well known mapping class representations -- the monodromy representations -- of generalised braid groups are not faithful in general \cite{labruere_1997}.
Even for type $E$: $E_6, E_7$ and $E_8$ (which are simply laced and spherical), the monodromy representations are not faithful.
To make things worse, it was actually shown that any homomorphism from a type $E$ generalised braid group into a mapping class group which sends standard generators to Dehn twists cannot be faithful \cite{wajnryb_1999}.
This has made it difficult to construct faithful mapping class representations for the type $E$ cases.
To the best of my knowledge, it is not known whether the type $E$ generalised braid groups have faithful mapping class representations.
On the other hand, we have seen that all the generalised braid groups of spherical types have faithful categorical actions on triangulated categories.
\begin{remark}
Attaining faithfulness is (one of) the surprising result from categorification -- the categorification of the type $A$ Burau representations, as done in \cite{khovanov_seidel_2001}, are faithful; whereas the (classical) type $A$ Burau representations of large ranks are known to be \emph{not} faithful (cf. \cite{Bige_99} and \cite{LONG1993439}).
\end{remark}

Nonetheless, there are indeed faithful mapping class representations in the rank two cases.
We have seen in \cref{sect: double A configuration} that the generalised braid group $\B(I_2(n))$ has an injection into the type $A_{n-1}$ (classical) braid group $\B(A_{n-1}$), where the latter is known to be isomorphic to the mapping class group $\mathscr{M}CG(\D_n)$ of the $n$-punctured disk $\D_n$.
Using this, one can view $\B(I_2(n))$ as a subgroup of $\mathscr{M}CG(\D_n)$, and obtain a dynamical classification of its elements through the classical Nielsen-Thurston classification of mapping class groups.

Since we now have two dynamical classification of similar flavour, it is only natural to ask if they coincide.
Indeed, one can show that the mass growth and topological entropy of $\beta$ agrees:
\begin{proposition}
Let $\beta \in \B(I_2(n))$.
Denote $R_\beta$ as its corresponding autoequivalence on $\Kom^b(\I$-prmod) and $f_\beta$ as its corresponding mapping class element of the $n$-punctured disk through the mapping class representation above.
Then the mass growth at $t=0$ agrees with the topological entropy:
\[
h_0(R_\beta) = h_\text{top}(f_\beta).
\]
\end{proposition}
\begin{proof}[Proof sketch]
We will only provide a sketch of the proof.
Recall the (faithful) action of $\B(A_{n-1})$ on the derived category of $\mathscr{A}_m$ modules $D^b(\mathscr{A}_m\text{-mod})$ defined in \cite{khovanov_seidel_2001}.
Khovanov-Seidel showed that this category is of ``Fukaya type'', namely certain objects correspond to curves and the dimension of morphism space corresponds to the intersection number.
As shown in \cref{sect: double A configuration}, our category $\Kom^b(\I$-prmod) is equivalent (as triangulated categories only) to a triangulated subcategory $\Kom^b(\mathscr{A}_{n-1}\text{-prgrmod}^{>0})$ of $D^b(\mathscr{A}_{n-1}$-mod) (more precisely, two copies of them).
Moreover our action of $\B(I_2(n))$ on $\Kom^b(\I$-prmod) intertwines the action of $\B(I_2(n))$ on $\Kom^b(\mathscr{A}_{n-1}\text{-prgrmod}^{>0})$ when viewed as a subgroup of $\B(A_{n-1})$.
The result now follows from \cite[Theorem 2.6]{DHKK13}, together with the fact that topological entropy can be computed from intersection numbers.
\end{proof}
As such, we expect that our categorical Nielsen-Thurston classification of $\beta \in \B(I_2(n))$ coincides with the classical Nielsen-Thurston classification of $\beta$ viewed as mapping class elements in $\mathscr{M}CG(\D_n)$ under the mapping class representation above.

Note, however, that our classification theorem \cref{theorem: classification} actually proves that the mass growth of pseudo-Anosov elements in $\B(I_2(n))$ are given by degree two algebraic numbers over $\Z[2\cos(\frac{\pi}{n})]$ -- they are eigenvalues of rank two matrices with entries in $\Z[2\cos(\frac{\pi}{n})]$; whereas on the topological side, the matrices assigned through the Bestvina-Handel algorithm have entries in $\Z$ and are \emph{not always} of rank two.
In some sense, the injection of $\B(I_2(n))$ into $\B(A_{n-1})$ has forgotten the extra  ``$2\cos\left(\frac{\pi}{n} \right)$ symmetry'' preserved by $\B(I_2(n))$.
More precisely, the action of $\B(I_2(n))$ on $\Kom^b(\mathscr{A}_{n-1}\text{-prgrmod}^{>0})$ is only an action on a triangulated category; not on a triangulated module category over $\TLJ_n$.
It is an interesting problem to search for a possible topological interpretation of this $\TLJ_n$ action.

Finally, it is also worth noting that mass growth, as a whole function of $t$, provides a better invariant than topological entropy, which only captures the $t=0$ part.
In particular, we see that the mass growth at $t=0$ (hence topological entropy) cannot differentiate between periodic braids and reducible braids (both have topological entropy zero). 
To the best of my knowledge, such a notion of ``$t$-deformed'' entropy does not exist in low-dimensional topology.

\comment{
Recall the injection $\varphi: \B(I_2(n)) \hookrightarrow \B(A_{n-1})$ given in \cref{sect: double A configuration}.
\begin{theorem} \label{thm: mass growth and top entropy}
Let $\beta \in \B(I_2(n))$.
The mass growth (at $t=0$) of $\beta$ as an autoequivalence on $\Kom^b(\I$-prmod) agrees with the topological entropy of $\varphi(\beta) \in \B(A_{n-1}) \cong \mathscr{M}CG(\D_n)$ as a mapping class element of the $n$-punctured disk.
\end{theorem}
{
\color{red}
To be added
}
}

\comment{
We now give a classification in terms of categorical dynamics:
\begin{definition}
Let $\cF: \cD \ra \cD$ be an autoequivalence.
We say that $\cF$ is:
\begin{enumerate}
\item periodic, if $\cF$ permutes a finite set of proper triangulated subcategories $\{ \cX_i \}$;
\item reducible, if $\cF$ restricts to an autoequivalence of a proper triangulated subcategory;
\item pseudo-Anosov, otherwise. {\color{red}??}
\end{enumerate}
\end{definition}
}

\newpage

\comment{
Throughout, $\cC$ is a fusion category, $g\cC$ is its $\Z$-graded category and $\cD$ is a module category over $g\cC$.
\begin{definition}
Denote $\overline{X} := PF\dim(X)$ for $X \in \cC$.
Let $E, F \in \Ob(\cD)$.
The \emph{$g\cC$-complexity of $F$ relative to $E$} is the function $\delta_{(q,t)}^{g\cC}(E,F) : \R\times \R \ra \R_{\geq 0} \cup \{ \infty \}$ defined by
\[
	\begin{cases}
	0, &F \cong 0; \\
	\underset{F'}{\inf} \left\{ 
		\sum_{i=1}^{m} \overline{X_i}e^{k_i q + \ell_i t} : 
			\begin{tikzcd}[scale cd = .7, column sep = .5mm, row sep = 2mm]
			0
				\ar[rr] &
			{}
				{} &
			F_1
				\ar[dl]&
			\cdots
				{} &
			F_{m-1}
				\ar[rr] &
			{}
				{} &
			F \oplus F'
				\ar[dl] \\
			{}
				{} &
			\Psi(X_1)E\<k_1\>[\ell_1]
				\ar[lu, dashed] &
			{}
				{} &
			{}
				{} &
			{}
				{} &
			\Psi(X_m)E\<k_m\>[\ell_m]
				\ar[lu, dashed] &
			\end{tikzcd}
		\right\},
	&F \in \< E \>_{g\cC}; \\
	\infty, &F \not\in \< E \>_{g\cC}.		
	\end{cases}
\] 
\end{definition}

The $g\cC$-complexity function satisfies similar properties to the usual complexity function:
\begin{proposition} \label{complexity lemma app}
For $D, E_1,E_2,E_3 \in \cD$,
\begin{enumerate}
\item $\delta_{(q,t)}^{g\cC}(E_1, E_3) \leq \delta_{(q,t)}^{g\cC}(E_1,E_2)\delta_{(q,t)}^{g\cC}(E_2,E_3)$;
\item if $E_1 \ra E_2 \ra E_3 \ra$ is an distinguished triangle, then 
\[
\delta_{(q,t)}^{g\cC}(D, E_2) \leq \delta_{(q,t)}^{g\cC}(D,E_1) + \delta_{(q,t)}^{g\cC}(D, E_3);
\]
\item for any endofunctor $\cF: \cD \ra \cD$ respecting the $g\cC$-module structure on $\cD$,
\[
\delta_{(q,t)}^{g\cC}(\cF(E_1), \cF(E_2)) \leq \delta_{(q,t)}^{g\cC}(E_1,E_2).
\]
\end{enumerate}
\end{proposition}

\begin{definition}
Let $G$ be a split generator over $g\cC$ for $\cD$ and $\Psi: \cD \ra \cD$ an  endofunctor.
Then the \emph{$g\cC$-entropy} of $\Psi$ is the function $h_{(q,t)}(\Psi): \R \times \R \ra [-\infty, \infty)$ defined by
\[
h_{(q,t)}(\Psi) := \lim_{N \ra \infty} \frac{1}{N} \log \delta_{(q,t)}^{g\cC} \left(G, \Psi^N(G) \right).
\]
\end{definition}

\begin{proposition} \label{ikeda prop 3.4}
Let $\tau = (\cP, \cZ, s)$ be a $q$-stability condition over $g\cC$ on $\cD$.
Then
\[
m_{\tau,t}(A) \leq m_{\tau,t}(B)\delta_{(\Re(s)t,t)}^{g\cC}(B,A).
\]
\end{proposition}
\begin{proof}
The proof follows exactly as in Proposition 3.4 of \cite{ikeda_2020}, using our definition of $g\cC$-complexity instead, together with the assumption that $\tau$ is a $q$-stability condition over $g\cC$ to obtain:
\[
m_{\tau,t}(\Psi(X)A\<k\>[\ell]) = PF\dim(X)m_{\tau,t}(A) e^{(k\Re(s) + \ell)t}.
\]
\end{proof}

We now obtain our version of Theorem 3.5 in \cite{ikeda_2020}.
\begin{theorem} \label{mass growth split generator and entropy}
Let $\tau$ be a $q$-stability condition over $g\cC$ on $\cD$ and $\cF : \cD \ra \cD$ be an endofunctor which respects the $g\cC$-module structure on $\cD$.
If $G$ be a split generator over $g\cC$ of $\cD$, then
\begin{enumerate}
\item the mass growth is achieved by $G$:
\[
h_{\tau,t}(\cF) = \limsup_{N\ra \infty} \frac{1}{N} \log  m_{\tau,t}(\cF^N(G)).
\]
\item We have an inequality
\[
h_{\tau,t}(\cF) \leq h_{(\Re(s)t,t)}(\cF) < \infty,
\]
with $h_{(\Re(s)t,t)}$ the $g\cC$-entropy of $\cF$.
\end{enumerate}
\end{theorem}
\begin{proof}
This follows from the same proof in \cite{ikeda_2020}, replacing Proposition 3.4 by our \cref{ikeda prop 3.4}.
\end{proof}

{\color{red}
DOESNT WORK

The following follows easily from the first part of the theorem, together with \cref{complexity lemma} and \cref{ikeda prop 3.4}:
\begin{corollary}\label{mass growth invariant under conjugation}
Let $\tau$ be a $q$-stability condition over $g\cC$ on $\cD$ and $\cF : (\cD, \Psi) \ra (\cD, \Psi)$ be an endofunctor.
Suppose $\cG: (\cD, \Psi) \ra (\cD, \Psi)$ is an autoequivalence.
Then
$
h_{\tau,t}(\cF) = h_{\tau,t}(\cG\cF\cG^{-1}). 
$
\end{corollary}
}

Let $\cH$ be a finite length abelian heart of $\cD$ such that every simple object is isomorphic to $S_i\<k\>[k]$ for some finite set of objects $\{S_1, ..., S_m\}$ and some $k\in \Z$.
It follows that the Grothendieck group $H_0$ of $\cH$ is a free $\Z[t,t^{-1}]$-module of rank $m$ generated by $[S_i]$, with $t\cdot [X] = [X\<1\>[1]]$.

For an object $E$ in $\cH$, let $[E] = \sum_{i=1}^m p_i(t)[S_i]$ be the its class in $H_0$, with $p_i(t) \in \Z[t,t^{-1}]$ for each $i$.
We define the \emph{dimension} of $E$ by 
\[
\dim (E) := \sum_{i=1}^m p_i(1) \in \Z_{\geq 0}.
\]

\begin{remark} \label{module stability function}
If we endow $\C$ with a $\Z[t,t^{-1}]$-module structure by evaluating $t=1$, note that a stability function on $\cH$ which respects the $\Z[t,t^{-1}]$-module structure on $H_0$ and $\C$ defines a levelled stability condition on $\cD$.
\end{remark}

The following is true using the same proof as in Lemma 3.12 of \cite{ikeda_2020}.

\begin{lemma}
Let $\cH$ be as defined above.
Then $G = \bigoplus_{i=1}^m S_i$ is a split level-generator, and
\[
\delta_t(G, E) \leq \dim(E)
\]
for any $E \in \cH$.
\end{lemma}

Let $Z_0: H_0 \ra \C$ be the $\Z[t,t^{-1}]$-module homomorphism (see \cref{module stability function}) defined by
\[
Z_0(S_j) := e^{i\frac{\pi}{2}} = i.
\]
This defines a stability function on $\cH$, and we shall denote the levelled stability condition it defines by $\tau_0$.
Note that the $\tau_0$-mass of all objects $E$ in $\cH$ are given by
\[
m_{\tau_0,t}(E) = \dim(E)\cdot e^{\frac{1}{2}t}.
\]
Once again we have the following proposition using the same proof as in  Proposition 3.13 of \cite{ikeda_2020}.

\begin{proposition}
For $G = \bigoplus_{i=1}^m S_i$ the split level-generator and $\tau_0$ given as above, we have
\[
\delta_t(G,A) \leq e^{-\frac{1}{2}t}\cdot m_{\tau_0,t}(A)
\]
for all objects $A$ in $\cD$.
\end{proposition}

Following Theorem 3.14 in \cite{ikeda_2020} we may conclude the following:
\begin{theorem}\label{mass growth computes entropy}
Let $\cH$ be a finite length abelian heart of $\cD$ such that every simple object is isomorphic to $S_i\<k\>[k]$ for some finite set of objects $\{S_1, ..., S_m\}$ and some $k\in \Z$.
If $G$ is a split-level generator of $\cD$, then for any levelled stability condition $\tau$ which lies in the same connected component as $\tau_0$, we have that
\[
h_t(\cF) = h_{\tau,t}(\cF) = \limsup_{N\ra \infty} \frac{1}{N} \log  m_{\tau,t}(\cF^N(G))
\]
for any endofunctor $\cF: \cD \ra \cD$ and any split level generator $G$ of $\cD$.
\end{theorem}

}

\appendix
\chapter{A brief tour into categorification} \label{appen: categorification}
The term categorification, coined by Crane and Frenkel \cite{Crane_1994}, is the general idea of lifting mathematical structures living in the world of sets to the world of categories.
A prototypical example is the lift of natural numbers $\mathbb{N}$ to the category of (finite dimensional) real vector spaces $\cV ec_\R$.
Every natural number $n$ can be lifted to the $n$-dimensional vector space $\R^n$, and the inverse process can be obtained by taking the dimension of the vector space.
Moreover, addition corresponds to direct sums, and multiplication corresponds to tensor product.
If we wanted the ring of integers $\Z$, we can instead take the Grothendieck group of $\cV ec_\R$, hence we can also say that (the monoidal, abelian category) $\cV ec_\R$ categorifies (the ring) $\Z$.
The category of vector spaces ought to have more structure then the set of natural numbers, as we now have morphisms between two objects $\R^n$ and $\R^m$, whereas ``morphisms'' between two natural numbers do not exist.
Hence, one could think of mathematical structures in the realm of sets as shadows of much richer mathematical structures in the realm of categories.

Let us now look at how this philosophy plays out in representation theory.
Classically, a group action $G$ on a set $X$ can be defined by a group homomorphism from $G$ to the group $\Aut(X)$ of automorphisms on $X$ (if $X$ has more structure, such as a vector space, the automorphisms are required to preserve the extra structure as well).
In particular, we require that $g \in G \mapsto F_g \in \Aut(X)$ such that
\begin{align*}
F_e &= \id_X \\
F_g \circ F_h &= F_{g\cdot h}.
\end{align*}
Lifting this to the categorical world by replacing equalities with isomorphisms, we say that a (weak) \emph{categorical group action} of $G$ on a category $\cC$ is an assignment of $g \in G$ with autoequivalences $\cF_g$ of $\cC$ such that
\begin{align*}
\cF_e &\cong \id_\cC \\
\cF_g \circ \cF_h &\cong \cF_{g\cdot h}.
\end{align*}
It turns out that in many examples of categorifying linear representations, the categories will be triangulated and the autoequivalences are exact (respects the triangulated structure).
In particular, the exact Grothendieck group $K_0(\cD)$ of $\cD$ will be a free module, and exact autoequivalences $\cF_g$ descend into linear automorphisms $K_0(\cF_g)$ on $K_0(\cD)$ such that the following diagram commutes:
\[
\begin{tikzcd}[column sep = large]
\cD \ar[d, dashed, "K_0"] \ar[r, "\cF_g"] & \cD \ar[d, dashed, "K_0"] \\
K_0(\cD) \ar[r, "K_0(\cF_g)"]  & K_0(\cD).
\end{tikzcd}
\]
A \emph{categorification of a linear representation} $g \in G \mapsto F_g \in GL(V)$ is a categorical action of $G$ on a triangulated category $\cD$ with $g \mapsto \cF_g$ such that
\[
\begin{tikzcd}[column sep = large]
K_0(\cD) \ar[d, "\cong"] \ar[r, "K_0(\cF_g)"] & K_0(\cD) \ar[d, "\cong"] \\
V \ar[r, "F_g"]  & V
\end{tikzcd}
\]
is commutative for all $g \in G$.

\chapter{Rectifiable triangles and geodesic filtration polygons}
The main purpose of this appendix is to collect the results from the appendix of \cite{bapat2020thurston}, which provide certain sufficient condition that allows one to obtain (weak) HN filtrations through concatenation of other (weak) HN filtrations.
As such, we shall state the results without proofs (except for \cref{sufficient condition for geodesic}, which is a slight variant of \cite[Proposition A.14]{bapat2020thurston}).
The presentation of them, however, will be slightly different: the results and statements will be stated ``dually'' in terms of triangles and polygons introduced in \cite{dyckerhoff_kapranov_2018}, which captures a nice ``two-dimensional symmetry'' built into the foundations of triangulated categories.
Throughout we fix $\cD$ to be a triangulated category with some stability condition $\tau$ fixed.

Let us start by recalling this dual representation of distinguished triangles introduced in \cite{dyckerhoff_kapranov_2018}: 
\begin{center}
\begin{tikzpicture}[scale=1.4]
\tikzset{
  arrow/.pic={\path[tips,every arrow/.try,->,>=#1] (0.1,0) -- +(.1pt,0);},
  pics/arrow/.default={triangle 90}
}

\coordinate (X) at (0,0);
\coordinate (Y) at (2,0);
\coordinate (Z) at (1,1.732);

\begin{scope}[very thick,nodes={sloped,allow upside down}]
\draw[thick] (X) -- pic{arrow=latex} (Y) ;
\draw[thick] (X) -- pic{arrow=latex} (Z) ;
\draw[thick] (Z) -- pic{arrow=latex} (Y) ;
\end{scope}

\pic [draw, <-, swap, "$\alpha$", angle radius = 15, angle eccentricity=1.5] {angle=Y--X--Z};
\pic [draw, <-, swap, "$\beta$", angle radius = 15, angle eccentricity=1.5] {angle=Z--Y--X};
\pic [draw, <-, swap, dashed,"$\gamma$", angle radius = 15, angle eccentricity=1.5] {angle=X--Z--Y};

\filldraw[color=black!, fill=black!]  (0,0) circle [radius=1.5pt];
\filldraw[color=black!, fill=black!]  (2,0) circle [radius=1.5pt];
\filldraw[color=black!, fill=black!]  (1,1.732) circle [radius=1.5pt];

\node[left] at (0.5, 0.9) {$A$};
\node[below] at (1, 0) {$B$};
\node[right] at (1.5, 0.9) {$C$};

\node at (-3,1) {$A \xra{\alpha} B \xra{\beta} C \xra{\gamma} A[1] \qquad \leftrightsquigarrow$};
\end{tikzpicture}
\end{center}
The octahedral axiom applied to the composition $A_1 \xra{f} A_2 \xra{g} A_3$ is just a flip of the 4-gon in the dual representation, switching from one triangulation to the other:
\begin{center}
\begin{tikzpicture}[scale=1.4]
\tikzset{
  arrow/.pic={\path[tips,every arrow/.try,->,>=#1] (0.1,0) -- +(.1pt,0);},
  pics/arrow/.default={triangle 90}
}

\coordinate (X) at (0,0);
\coordinate (Y) at (2,0);
\coordinate (Z) at (1,1.732);
\coordinate (W) at (1,-1.732);

\begin{scope}[very thick,nodes={sloped,allow upside down}]
\draw[thick] (X) -- pic{arrow=latex} (Y) ;
\draw[thick] (X) -- pic{arrow=latex} (Z) ;
\draw[thick] (Z) -- pic{arrow=latex} (Y) ;
\draw[thick] (X) -- pic{arrow=latex} (W) ;
\draw[thick] (Y) -- pic{arrow=latex} (W) ;
\end{scope}

\pic [draw, <-, swap, angle radius = 15, angle eccentricity=1.5, "$f$"] {angle=Y--X--Z};
\pic [draw, <-, swap, angle radius = 15, angle eccentricity=1.5] {angle=Z--Y--X};
\pic [draw, <-, swap, dashed, angle radius = 15, angle eccentricity=1.5] {angle=X--Z--Y};
\pic [draw, <-, swap, angle radius = 15, angle eccentricity=1.5, "$g$"] {angle=W--X--Y};
\pic [draw, <-, swap, angle radius = 15, angle eccentricity=1.5] {angle=Y--W--X};
\pic [draw, <-, swap, dashed, angle radius = 15, angle eccentricity=1.5] {angle=X--Y--W};

\filldraw[color=black!, fill=black!]  (0,0) circle [radius=1.5pt];
\filldraw[color=black!, fill=black!]  (2,0) circle [radius=1.5pt];
\filldraw[color=black!, fill=black!]  (1,1.732) circle [radius=1.5pt];
\filldraw[color=black!, fill=black!]  (1,-1.732) circle [radius=1.5pt];

\node[left] at (0.5, 0.9) {$A_1$};
\node[below] at (1, 0) {$A_2$};
\node[right] at (1.5, 0.9) {$A_{12}$};
\node[left] at (0.5, -0.9) {$A_3$};
\node[right] at (1.5, -0.9) {$A_{23}$};

\node at (3,0) {$\overset{\text{octahedral flip}}{\leftrightsquigarrow}$};

\coordinate (X) at (4,0);
\coordinate (Y) at (6,0);
\coordinate (Z) at (5,1.732);
\coordinate (W) at (5,-1.732);

\begin{scope}[very thick,nodes={sloped,allow upside down}]
\draw[thick] (Z) -- pic{arrow=latex} (W) ;
\draw[thick] (X) -- pic{arrow=latex} (Z) ;
\draw[thick] (Z) -- pic{arrow=latex} (Y) ;
\draw[thick] (X) -- pic{arrow=latex} (W) ;
\draw[thick] (Y) -- pic{arrow=latex} (W) ;
\end{scope}

\pic [draw, <-, swap, angle radius = 10, angle eccentricity=1.5, "$gf$"] {angle=W--X--Z};
\pic [draw, <-, swap, angle radius = 30, angle eccentricity=1.5] {angle=W--Z--Y};
\pic [draw, <-, swap, dashed, angle radius = 30, angle eccentricity=1.5] {angle=X--Z--W};

\pic [draw, <-, swap, angle radius = 25, angle eccentricity=1.5] {angle=Z--W--X};
\pic [draw, <-, swap, angle radius = 25, angle eccentricity=1.5] {angle=Y--W--Z};
\pic [draw, <-, swap, dashed, angle radius = 10, angle eccentricity=1.5] {angle=Z--Y--W};

\filldraw[color=black!, fill=black!]  (X) circle [radius=1.5pt];
\filldraw[color=black!, fill=black!]  (Y) circle [radius=1.5pt];
\filldraw[color=black!, fill=black!]  (Z) circle [radius=1.5pt];
\filldraw[color=black!, fill=black!]  (W) circle [radius=1.5pt];

\node[left] at (4.5, 0.9) {$A_1$};
\node[right] at (5, 0) {$A_{13}$};
\node[right] at (5.5, 0.9) {$A_{12}$};
\node[left] at (4.5, -0.9) {$A_3$};
\node[right] at (5.5, -0.9) {$A_{23}$};
\end{tikzpicture}
\end{center}
We shall call this the \emph{octahedral flip}.

The following is a generalisation of filtrations in this dual setting:
\begin{definition}
A \emph{filtration polygon} of an object $A \in \cD$ with filtration pieces $E_i \in \cD$ is an oriented polygon with some fixed oriented triangulation, such that:
\begin{enumerate}
\item each triangle in the polygon is a (dual) distinguished triangle; and
\item the orientation of the outer edges (together with their labels) are given by
\begin{center}
\begin{tikzpicture}
\tikzset{
  arrow/.pic={\path[tips,every arrow/.try,->,>=#1] (0.1,0) -- +(.1pt,0);},
  pics/arrow/.default={triangle 90}
}

\coordinate (A) at (0,0);
\coordinate (B) at (2,0);
\coordinate (C) at (-1,1.732);
\coordinate (D) at (3,1.732);
\coordinate (E) at (0,3.5);
\coordinate (F) at (2,3.5);
\coordinate (Z) at (1,3.5);

\begin{scope}[very thick,nodes={sloped,allow upside down}]
\draw[thick] (A) -- pic{arrow=latex} (C) ;
\draw[thick] (A) -- pic{arrow=latex} (B) ;
\draw[thick] (D) -- pic{arrow=latex} (B) ;
\draw[thick] (F) -- pic{arrow=latex} (D) ;
\draw[thick] (C) -- pic{arrow=latex} (E) ;
\end{scope}

\filldraw[color=black!, fill=black!]  (A) circle [radius=1.5pt];
\filldraw[color=black!, fill=black!]  (B) circle [radius=1.5pt];
\filldraw[color=black!, fill=black!]  (C) circle [radius=1.5pt];
\filldraw[color=black!, fill=black!]  (D) circle [radius=1.5pt];
\filldraw[color=black!, fill=black!]  (E) circle [radius=1.5pt];
\filldraw[color=black!, fill=black!]  (F) circle [radius=1.5pt];

\node[left] at (-0.5, 0.9) {$E_1$};
\node[left] at (-0.5, 2.7) {$E_2$};
\node[below] at (1, 0) {$A$};
\node[right] at (2.5, 0.9) {$E_m$};
\node[right] at (2.5, 2.7) {$E_{m-1}$};

\node at (Z) {$\cdots$};
\end{tikzpicture}
\end{center}
\end{enumerate}
\end{definition}
As an example, the following filtration polygon (maps between edges omitted)
\begin{center}
\begin{tikzpicture}
\tikzset{
  arrow/.pic={\path[tips,every arrow/.try,->,>=#1] (0.1,0) -- +(.1pt,0);},
  pics/arrow/.default={triangle 90}
}

\coordinate (A) at (0,0);
\coordinate (B) at (2,0);
\coordinate (C) at (-1,1.732);
\coordinate (D) at (3,1.732);
\coordinate (E) at (0,3.5);
\coordinate (F) at (2,3.5);
\coordinate (Z) at (1,3.5);

\begin{scope}[very thick,nodes={sloped,allow upside down}]
\draw[thick] (A) -- pic{arrow=latex} (B) ;
\draw[thick] (A) -- pic{arrow=latex} (C) ;
\draw[thick] (A) -- pic{arrow=latex} (D) ;
\draw[thick] (A) -- pic{arrow=latex} (E) ;
\draw[thick] (A) -- pic{arrow=latex} (F) ;
\draw[thick] (D) -- pic{arrow=latex} (B) ;
\draw[thick] (F) -- pic{arrow=latex} (D) ;
\draw[thick] (C) -- pic{arrow=latex} (E) ;
\end{scope}

\filldraw[color=black!, fill=black!]  (A) circle [radius=1.5pt];
\filldraw[color=black!, fill=black!]  (B) circle [radius=1.5pt];
\filldraw[color=black!, fill=black!]  (C) circle [radius=1.5pt];
\filldraw[color=black!, fill=black!]  (D) circle [radius=1.5pt];
\filldraw[color=black!, fill=black!]  (E) circle [radius=1.5pt];
\filldraw[color=black!, fill=black!]  (F) circle [radius=1.5pt];

\node[left] at (-0.5, 0.9) {$E_1$};
\node[left] at (-0.5, 2.7) {$E_2$};
\node[below] at (1, 0) {$A_m=A$};
\node[right] at (2.5, 0.9) {$E_m$};
\node[right] at (2.5, 2.7) {$E_{m-1}$};

\node at (Z) {$\cdots$};
\end{tikzpicture}
\end{center}
is just the dual representation of a standard filtration of $A$:
\[
\begin{tikzcd}[column sep = 3mm]
0
	\ar[rr] &
{}
	{} &
A_1
	\ar[rr] \ar[dl]&
{}
	{} &	
A_2
	\ar[rr] \ar[dl]&
{}
	{} &	
\cdots
	\ar[rr] &
{}
	{} &
A_{m-1}
	\ar[rr] &
{}
	{} &
A_m = A
	\ar[dl] \\
{}
	{} &
E_1
	\ar[lu, dashed] &
{}
	{} &
E_2
	\ar[lu, dashed] &
{}
	{} &
{}
	{} &
{}
	{} &
{}
	{} &
{}
	{} &
E_m
	\ar[lu, dashed] &
\end{tikzcd}
\]
As such, every filtration polygon is just the dual representation of a standard filtration up to applying octahedral flips.

We will need the following weaker version of HN filtration later on:
\begin{definition} \label{defn: weak HN filtration}
Let $A \in \cD$ with a stability condition $\tau$.
A filtration $F$ of $A$ 
\[
F = 
\begin{tikzcd}[column sep = 3mm]
0
	\ar[rr] &
{}
	{} &
A_1
	\ar[rr] \ar[dl]&
{}
	{} &	
A_2
	\ar[rr] \ar[dl]&
{}
	{} &	
\cdots
	\ar[rr] &
{}
	{} &
A_{m-1}
	\ar[rr] &
{}
	{} &
A_m = A
	\ar[dl] \\
{}
	{} &
E_1
	\ar[lu, dashed] &
{}
	{} &
E_2
	\ar[lu, dashed] &
{}
	{} &
{}
	{} &
{}
	{} &
{}
	{} &
{}
	{} &
E_m
	\ar[lu, dashed] &
\end{tikzcd}
\]
is a \emph{weak $\tau$-HN filtration of $A$} if all $E_i$ are $\tau$-semistable with non-increasing phases: $\phi(E_i) \geq \phi(E_{i+1})$.
Similarly, we say that a filtration polygon of $A$ with filtration pieces $E_i$
is a \emph{weak $\tau$-HN filtration polygon}  if all the $E_i$'s are $\tau$-semistable with non-increasing phases as before.
In particular, this means that a weak HN filtration polygon is a (dual) weak HN filtration up to octahedral flips.
\end{definition}
Note that by definition a $\tau$-HN filtration is automatically a weak $\tau$-HN filtration, where the converse is clearly not true.
In particular, we see that weak $\tau$-HN filtrations need not be unique.
However, a $\tau$-HN filtration can always be obtained from a weak $\tau$-HN filtration by taking cones of the composition 
\[
A_{i-1} \xra{f_{i-1}} A_i \xra{f_i} A_{i+1}
\]
whenever $E_{i+1}$ and $E_i$ are semistable of the same phase $\phi$, where it is easy to deduce that the cone of $f_i f_{i-1}$ is also semistable of phase $\phi$ using the following distinguished triangle induced by the octahedral axiom:
\[
E_i \ra \cone(f_i f_{i-1}) \ra E_{i+1} \ra.
\]
We call this process \emph{strictification} of a weak $\tau$-HN filtration to a $\tau$-HN filtration.
Note that this also implies that if $F$ is a weak $\tau$-HN filtration polygon of $A$ with filtration pieces $E_i$'s, the mass of $A$ can be computed from summing up the mass of the $E_i$'s:
\[
m_{\tau,t}(A) = \sum_i m_{\tau,t}(E_i).
\]

\begin{definition}
Fix a stabiltiy condition $\tau$.
A distinguished triangle $A \ra B \ra C \xra{\varphi} A[1]$ is said to be \emph{rectifiable} (with respect to $\tau$) if for all $\alpha \in \R$, the compositions of maps:
\[
C_{> \alpha} \ra C \xra{\varphi} A[1] \ra A_{\leq \alpha}[1]
\]
is 0. Equivalently, we will also say that the map $\varphi$ is \emph{rectifiable}.
\end{definition}

\begin{proposition}
If $C \xra{\varphi} A[1]$ is rectifiable, then any pre-composition and post-composition is also rectifiable, i.e.
\[
X \xra{f} C \xra{\varphi} A[1] \xra{g[1]} Y[1]
\]
is rectifiable for all $X \xra{f} C$ and $ A \xra{g} Y$.
\end{proposition}

\begin{proposition}
Suppose $A\ra B \ra C \xra{\varphi} A[1]$ is a rectifiable triangle.
Then $A_{>\alpha} \ra B_{>\alpha} \ra C_{>\alpha} \ra A_{>\alpha}[1]$ and $A_{\leq \alpha} \ra B_{\leq \alpha} \ra C_{\leq \alpha} \ra A_{\leq \alpha}[1]$ are both distinguished triangles.
Moreover, they are rectifiable.
\end{proposition}

\begin{proposition} \label{rectifiable and weak HN}
Let $A \ra B \ra C \xra{\varphi} A[1]$ be a rectifiable triangle.
Then a weak HN filtration polygon of $B$ can be obtained from concatenating weak HN filtration polygons of $A$ and $C$ to the given rectifiable triangle, rearranging the filtration pieces if necessary.
\end{proposition}

\begin{definition}
A filtration polygon is said to be \emph{geodesic} if every inner distinguished triangle is rectifiable. 
\end{definition}

The following proposition shows that a filtration polygon is geodesic if and only if any of its octahedral flipped variant is geodesic; namely being geodesic is a property invariant under octahedral flips.
\begin{proposition} \label{prop: geodesic invariant octa flip}
Let 
\begin{center}
\begin{tikzpicture}
\tikzset{
  arrow/.pic={\path[tips,every arrow/.try,->,>=#1] (0.1,0) -- +(.1pt,0);},
  pics/arrow/.default={triangle 90}
}

\coordinate (A) at (0,0);
\coordinate (B) at (2,0);
\coordinate (C) at (0,2);
\coordinate (D) at (2,2);

\begin{scope}[very thick,nodes={sloped,allow upside down}]
\draw[thick] (A) -- pic{arrow=latex} (B) ;
\draw[thick] (A) -- pic{arrow=latex} (C) ;
\draw[thick] (D) -- pic{arrow=latex} (B) ;
\draw[thick] (C) -- pic{arrow=latex} (D) ;
\draw[thick] (C) -- pic{arrow=latex} (B) ;
\end{scope}

\filldraw[color=black!, fill=black!]  (A) circle [radius=1.5pt];
\filldraw[color=black!, fill=black!]  (B) circle [radius=1.5pt];
\filldraw[color=black!, fill=black!]  (C) circle [radius=1.5pt];
\filldraw[color=black!, fill=black!]  (D) circle [radius=1.5pt];

\node[left] at (0, 1) {$A$};
\node[below] at (1, 0) {$C$};
\node[right] at (2, 1) {$E$};
\node[above] at (1, 2) {$D$};
\node[left] at (1, 1) {$F$};

\coordinate (A') at (4,0);
\coordinate (B') at (6,0);
\coordinate (C') at (4,2);
\coordinate (D') at (6,2);

\begin{scope}[very thick,nodes={sloped,allow upside down}]
\draw[thick] (A') -- pic{arrow=latex} (B') ;
\draw[thick] (A') -- pic{arrow=latex} (C') ;
\draw[thick] (D') -- pic{arrow=latex} (B') ;
\draw[thick] (C') -- pic{arrow=latex} (D') ;
\draw[thick] (A') -- pic{arrow=latex} (D') ;
\end{scope}

\filldraw[color=black!, fill=black!]  (A') circle [radius=1.5pt];
\filldraw[color=black!, fill=black!]  (B') circle [radius=1.5pt];
\filldraw[color=black!, fill=black!]  (C') circle [radius=1.5pt];
\filldraw[color=black!, fill=black!]  (D') circle [radius=1.5pt];

\node[left] at (4, 1) {$A$};
\node[below] at (5, 0) {$C$};
\node[right] at (6, 1) {$E$};
\node[above] at (5, 2) {$D$};
\node[left] at (5, 1) {$B$};

\node at (3, 1) {$\leftrightsquigarrow$};
\end{tikzpicture}
\end{center}
be two filtration polygons of $C$ related by octahedral flip.
Then one is rectifiable if and only if the other is; namely $F\ra A[1]$ and $E \ra D[1]$ are rectifiable if and only if $E \ra B[1]$ and $D \ra A[1]$ are rectifiable.
\end{proposition}

Combining \cref{rectifiable and weak HN} and \cref{prop: geodesic invariant octa flip} we obtain:
\begin{corollary}\label{cor: geodesic implies weak HN}
Let $F$ be a geodesic filtration polygon of $A$ with filtration pieces $E_i$.
Then a weak HN filtration polygon of $A$ can be obtained by concatenating weak HN filtration polygons of the $E_i$'s to $F$, rearranging the resulting filtration pieces if necessary.
\end{corollary}

\begin{proposition} \label{0 map swap geodesic}
Let $F$ be a filtration polygon of $A$ with filtration pieces $E_i$.
Suppose that $F$ contains an inner triangle given by
\begin{center}
\begin{tikzpicture}
\tikzset{
  arrow/.pic={\path[tips,every arrow/.try,->,>=#1] (0.1,0) -- +(.1pt,0);},
  pics/arrow/.default={triangle 90}
}

\coordinate (X) at (0,0);
\coordinate (Y) at (2,0);
\coordinate (Z) at (1,1.732);

\begin{scope}[very thick,nodes={sloped,allow upside down}]
\draw[thick] (X) -- pic{arrow=latex} (Y) ;
\draw[thick] (X) -- pic{arrow=latex} (Z) ;
\draw[thick] (Z) -- pic{arrow=latex} (Y) ;
\end{scope}

\pic [draw, <-, swap, angle radius = 15, angle eccentricity=1.5] {angle=Y--X--Z};
\pic [draw, <-, swap, angle radius = 15, angle eccentricity=1.5] {angle=Z--Y--X};
\pic [draw, <-, swap, dashed,"$0$", angle radius = 15, angle eccentricity=1.5] {angle=X--Z--Y};

\filldraw[color=black!, fill=black!]  (0,0) circle [radius=1.5pt];
\filldraw[color=black!, fill=black!]  (2,0) circle [radius=1.5pt];
\filldraw[color=black!, fill=black!]  (1,1.732) circle [radius=1.5pt];

\node[left] at (0.5, 0.9) {$E_i$};
\node[below] at (1, 0) {$E_i \oplus E_{i+1}$};
\node[right] at (1.5, 0.9) {$E_{i+1}$};
\end{tikzpicture}.
\end{center}
Define the filtration polygon $F'$ of $A$ with all inner triangles the same as $F$, except the above inner triangle where we swap $E_i$ and $E_{i+1}$.
Then $F$ is geodesic if and only if $F'$ is geodesic.
\end{proposition}
\comment{
\begin{proof}
This follows from the fact that all inner triangles of $F$ and $F'$ are the same except for the triangle $A_i \ra A_i \oplus A_{i+1} \ra A_{i+1} \xra{0} A_i[1]$.
But the $0$ map is always rectifiable.
\end{proof}
}

We now arrive at the following sufficient condition which allows one to obtain weak HN filtrations from concatenating weak HN filtrations:
\begin{proposition} \label{sufficient condition for geodesic}
Let $(\cD, \X)$ be a triangulated $\X$-category with a stability condition $\tau$.
Let $F$ be a filtration polygon of $A$ with filtration pieces $E_i$.
If $F$ satisfies the following properties: 
\begin{enumerate}
\item for all $j$ and $j'$, $\bigoplus_{k \in \Z} \Hom(E_j, E_{j'}\cdot \X^k[k][1]) \cong \bigoplus_{k \in \Z}\Hom(E_{j'}, E_j\cdot \X^k[k][1])$;
\item for all $i > j$, either $\bigoplus_{k \in \Z} \Hom(E_i, E_j \cdot \X^k[k][1]) \cong 0$ or $\ceil*{E_i} \leq \floor*{E_j}$,
\end{enumerate}
then $F$ must be geodesic.
In particular, a weak $\tau$-HN filtration of $A$ is obtainable from concatenating weak $\tau$-HN filtrations of $E_i$'s to $F$ and then applying octahedral flips.
\end{proposition}
\begin{proof}
The final statement is a direct result of \cref{cor: geodesic implies weak HN}.

We shall show that $F$ is geodesic through an induction on $m$.
The base case where $m=2$ follows from \cref{rectifiable and weak HN}.
Now assume for the induction hypothesis that it is true for $m$. 
If $\ceil{E_{j+1}} > \ceil{E_j}$, it must be the case that $\ceil{E_{j+1}} > \floor{E_j}$, which by our assumption on $F$ implies $\Hom(E_{j+1}, E_j[1]) = 0$.
In this case, we shall swap $E_{j+1}$ and $E_j$.
By \cref{0 map swap geodesic}, this new filtration is geodesic if and only if the original filtration is.
Moreover, after swapping $E_j$ and $E_{j+1}$, the required assumptions on the new filtration still hold; assumption 1 comes for free and the first part of assumption 2 is always true for $i=j+1$:
\[
\bigoplus_{k\in \Z} \Hom(E_j, E_{j+1}\cdot \X^k[k][1])
\cong \bigoplus_{k\in \Z} \Hom(E_{j+1}, E_j\cdot \X^k[k][1])
\cong 0.
\]
Hence we may assume without loss of generality that our $E_j$'s satisfy
\[
\ceil{E_1} \geq \ceil{E_2} \geq \cdots \geq \ceil{E_{m+1}}. 
\]
Note that this guarantees that if $\floor{E_1} \geq \ceil{E_k}$ for some $k$, then $\floor{E_1} \geq \ceil{E_{k'}}$ for all $k' \geq k$.
We shall now proceed via a case analysis.
We remind the reader that we may apply octahedral flips freely whilst checking the property of being geodesic since it is invariant under octahedral flips (cf. \cref{prop: geodesic invariant octa flip}).
\begin{enumerate}[-]
\item Case 1: For all $j > 1$, $\floor{E_1} < \ceil{E_j}$. \\
By assumption 2, it must be that 
\begin{equation}\label{j to 1 0}
\bigoplus_{k\in \Z} \Hom(E_j, E_1\cdot \X^k[k][1]) \cong 0
\end{equation}
for all $j > 1$.
By applying octahedral flips, consider the following triangulation:
\begin{center}
\begin{tikzpicture}
\tikzset{
  arrow/.pic={\path[tips,every arrow/.try,->,>=#1] (0.1,0) -- +(.1pt,0);},
  pics/arrow/.default={triangle 90}
}

\coordinate (A) at (0,0);
\coordinate (B) at (2,0);
\coordinate (C) at (-1,1.732);
\coordinate (D) at (3,1.732);
\coordinate (E) at (0,3.5);
\coordinate (F) at (2,3.5);
\coordinate (Z) at (1,3.5);

\begin{scope}[very thick,nodes={sloped,allow upside down}]
\draw[thick] (A) -- pic{arrow=latex} (C) ;
\draw[thick] (A) -- pic{arrow=latex} (B) ;
\draw[thick] (D) -- pic{arrow=latex} (B) ;
\draw[thick] (F) -- pic{arrow=latex} (D) ;
\draw[thick] (C) -- pic{arrow=latex} (E) ;
\draw[thick] (C) -- pic{arrow=latex} (B) ;
\end{scope}

\filldraw[color=black!, fill=black!]  (A) circle [radius=1.5pt];
\filldraw[color=black!, fill=black!]  (B) circle [radius=1.5pt];
\filldraw[color=black!, fill=black!]  (C) circle [radius=1.5pt];
\filldraw[color=black!, fill=black!]  (D) circle [radius=1.5pt];
\filldraw[color=black!, fill=black!]  (E) circle [radius=1.5pt];
\filldraw[color=black!, fill=black!]  (F) circle [radius=1.5pt];

\node[left] at (-0.5, 0.9) {$E_1$};
\node[left] at (-0.5, 2.7) {$E_2$};
\node[below] at (1, 0) {$A$};
\node[right] at (2.5, 0.9) {$E_m$};
\node[right] at (2.5, 2.7) {$E_{m-1}$};
\node[right] at (0.7,1) {$E_{2,m}$};

\node at (Z) {$\cdots$};
\end{tikzpicture}
\end{center}
Using \cref{j to 1 0}, we get that $\Hom(E_{2,m}, E_1[1]) \cong 0$, and so $E_{2,m} \ra E_1[1]$ is rectifiable.
Since the sub-polygon
\begin{center}
\begin{tikzpicture}
\tikzset{
  arrow/.pic={\path[tips,every arrow/.try,->,>=#1] (0.1,0) -- +(.1pt,0);},
  pics/arrow/.default={triangle 90}
}

\coordinate (A) at (0,0);
\coordinate (B) at (2,0);
\coordinate (C) at (-1,1.732);
\coordinate (D) at (3,1.732);
\coordinate (E) at (0,3.5);
\coordinate (F) at (2,3.5);
\coordinate (Z) at (1,3.5);

\begin{scope}[very thick,nodes={sloped,allow upside down}]
\draw[thick] (D) -- pic{arrow=latex} (B) ;
\draw[thick] (F) -- pic{arrow=latex} (D) ;
\draw[thick] (C) -- pic{arrow=latex} (E) ;
\draw[thick] (C) -- pic{arrow=latex} (B) ;
\end{scope}

\filldraw[color=black!, fill=black!]  (B) circle [radius=1.5pt];
\filldraw[color=black!, fill=black!]  (C) circle [radius=1.5pt];
\filldraw[color=black!, fill=black!]  (D) circle [radius=1.5pt];
\filldraw[color=black!, fill=black!]  (E) circle [radius=1.5pt];
\filldraw[color=black!, fill=black!]  (F) circle [radius=1.5pt];

\node[left] at (-0.5, 2.7) {$E_2$};
\node[right] at (2.5, 0.9) {$E_m$};
\node[right] at (2.5, 2.7) {$E_{m-1}$};
\node[right] at (0.7,1) {$E_{2,m}$};

\node at (Z) {$\cdots$};
\end{tikzpicture}
\end{center}
is geodesic by our induction hypothesis, we conclude that our filtration must be geodesic.
\item Case 2: There is some $k > 1$ such that $\floor{E_1} \geq \ceil{E_k}$. \\
In this case, pick $k$ to be the smallest such that $\floor{E_1} \geq \ceil{E_k}$ and use octahedral flips to obtain the following triangulation:
\begin{center}
\begin{tikzpicture}
\tikzset{
  arrow/.pic={\path[tips,every arrow/.try,->,>=#1] (0.1,0) -- +(.1pt,0);},
  pics/arrow/.default={triangle 90}
}

\coordinate (A') at (4,0);
\coordinate (B') at (6,0);
\coordinate (C') at (4,2);
\coordinate (D') at (6,2);
\coordinate (E') at (4.5,3.5);
\coordinate (F') at (5.5,3.5);
\coordinate (G') at (7.5,1.5);
\coordinate (H') at (7.5,0.5);

\begin{scope}[very thick,nodes={sloped,allow upside down}]
\draw[thick] (A') -- pic{arrow=latex} (B') ;
\draw[thick] (A') -- pic{arrow=latex} (C') ;
\draw[thick] (D') -- pic{arrow=latex} (B') ;
\draw[thick] (C') -- pic{arrow=latex} (D') ;
\draw[thick] (A') -- pic{arrow=latex} (D') ;
\draw[thick] (C') -- pic{arrow=latex} (E') ;
\draw[thick] (F') -- pic{arrow=latex} (D') ;
\draw[thick] (D') -- pic{arrow=latex} (G') ;
\draw[thick] (H') -- pic{arrow=latex} (B') ;
\end{scope}

\filldraw[color=black!, fill=black!]  (A') circle [radius=1.5pt];
\filldraw[color=black!, fill=black!]  (B') circle [radius=1.5pt];
\filldraw[color=black!, fill=black!]  (C') circle [radius=1.5pt];
\filldraw[color=black!, fill=black!]  (D') circle [radius=1.5pt];
\filldraw[color=black!, fill=black!]  (E') circle [radius=1.5pt];
\filldraw[color=black!, fill=black!]  (F') circle [radius=1.5pt];
\filldraw[color=black!, fill=black!]  (G') circle [radius=1.5pt];
\filldraw[color=black!, fill=black!]  (H') circle [radius=1.5pt];

\node[left] at (4, 1) {$E_1$};
\node[left] at (4.2, 2.7) {$E_2$};
\node[right] at (4.62,3.5) {$\cdots$};
\node[right] at (5.7, 2.7) {$E_{k-1}$};
\node[above] at (5, 2) {$E_{2,k-1}$};
\node[below] at (5.3, .9) {$E_{1,k-1}$};
\node[below] at (5, 0) {$A$};
\node[right] at (6, 1) {$E_{k,m}$};

\node[above] at (7,1.7) {$E_k$};
\node[above] at (7.5,0.7) {$\vdots$};
\node[below] at (7.1,0.2) {$E_m$};
\end{tikzpicture}
\end{center}
Note that if $k=2$, we set $E_{2,k-1}=0$ and $E_{1,k-1} =E_1$.
By induction hypothesis, the sub-filtration polygons of $E_{1,k-1}$ and $E_{k,m}$ are both geodesic, hence we are only required to show that the triangle 
\begin{center}
\begin{tikzpicture}
\tikzset{
  arrow/.pic={\path[tips,every arrow/.try,->,>=#1] (0.1,0) -- +(.1pt,0);},
  pics/arrow/.default={triangle 90}
}

\coordinate (A') at (4,0);
\coordinate (B') at (6,0);
\coordinate (D') at (6,2);

\begin{scope}[very thick,nodes={sloped,allow upside down}]
\draw[thick] (A') -- pic{arrow=latex} (B') ;
\draw[thick] (D') -- pic{arrow=latex} (B') ;
\draw[thick] (A') -- pic{arrow=latex} (D') ;
\end{scope}

\filldraw[color=black!, fill=black!]  (A') circle [radius=1.5pt];
\filldraw[color=black!, fill=black!]  (B') circle [radius=1.5pt];
\filldraw[color=black!, fill=black!]  (D') circle [radius=1.5pt];

\node[below] at (5.3, .9) {$E_{1,k-1}$};
\node[below] at (5, 0) {$A$};
\node[right] at (6, 1) {$E_{k,m}$};
\end{tikzpicture}
\end{center}
is rectifiable.
To show this, note that the same argument in case 1 tells us that the map $E_{2,k-1} \ra E_1$ is 0, hence $E_{1,k-1} \cong E_1 \oplus E_{2,k-1}$.
As such, it is sufficient to show that the two induced maps $E_{k,m}\ra E_1[1]$ and $E_{k,m}\ra E_{2,k-1}[1]$ are rectifiable.
But since $\floor{E_1} \geq \ceil{E_{k'}}$ for all $k' \geq k$, we have that $\floor{E_1} \geq \ceil{E_{k,m}}$, which shows that any map $E_{k,m}\ra E_1[1]$ is always rectifiable.
On the other hand, the filtration polygon
\begin{center}
\begin{tikzpicture}
\tikzset{
  arrow/.pic={\path[tips,every arrow/.try,->,>=#1] (0.1,0) -- +(.1pt,0);},
  pics/arrow/.default={triangle 90}
}

\coordinate (B') at (6,0);
\coordinate (C') at (4,2);
\coordinate (D') at (6,2);
\coordinate (E') at (4.5,3.5);
\coordinate (F') at (5.5,3.5);
\coordinate (G') at (7.5,1.5);
\coordinate (H') at (7.5,0.5);

\begin{scope}[very thick,nodes={sloped,allow upside down}]
\draw[thick] (D') -- pic{arrow=latex} (B') ;
\draw[thick] (C') -- pic{arrow=latex} (D') ;
\draw[thick] (C') -- pic{arrow=latex} (E') ;
\draw[thick] (F') -- pic{arrow=latex} (D') ;
\draw[thick] (D') -- pic{arrow=latex} (G') ;
\draw[thick] (H') -- pic{arrow=latex} (B') ;
\draw[thick] (C') -- pic{arrow=latex} (B') ;
\end{scope}

\filldraw[color=black!, fill=black!]  (B') circle [radius=1.5pt];
\filldraw[color=black!, fill=black!]  (C') circle [radius=1.5pt];
\filldraw[color=black!, fill=black!]  (D') circle [radius=1.5pt];
\filldraw[color=black!, fill=black!]  (E') circle [radius=1.5pt];
\filldraw[color=black!, fill=black!]  (F') circle [radius=1.5pt];
\filldraw[color=black!, fill=black!]  (G') circle [radius=1.5pt];
\filldraw[color=black!, fill=black!]  (H') circle [radius=1.5pt];

\node[left] at (4.2, 2.7) {$E_2$};
\node[right] at (4.62,3.5) {$\cdots$};
\node[right] at (5.7, 2.7) {$E_{k-1}$};
\node[above] at (5, 2) {$E_{2,k-1}$};
\node[left] at (5.2, .9) {$E_{2,m}$};
\node[right] at (6, 1) {$E_{k,m}$};

\node[above] at (7,1.7) {$E_k$};
\node[above] at (7.5,0.7) {$\vdots$};
\node[below] at (7.1,0.2) {$E_m$};
\end{tikzpicture}
\end{center}
is geodesic by the induction hypothesis.
In particular, the induced map $E_{k,m} \ra E_{2,k-1}[1]$ is rectifiable.
This concludes the proof.
\end{enumerate}
\end{proof}

\chapter{Gaussian Elimination Lemma} 
\begin{lemma}[cf. Lemma 3.2 in \cite{bar-natan_burgos-soto_2014}]\label{gaussian elimination}
Let $\cA$ be an additive category and consider the complex in $\Com(\cA)$ given by
\[
\cdots 
	\ra 
C 
	\xra{\begin{bmatrix}
		\alpha \\
		\beta
		\end{bmatrix}
		} 
b_1 \oplus D 
	\xra{\begin{bmatrix}
		\phi & \delta \\
		\gamma & \epsilon
		\end{bmatrix}
		}
b_2 \oplus E 
	\xra{\begin{bmatrix}
		\mu & \nu
		\end{bmatrix}
		}
F 
	\ra 
\cdots
\]
with $\phi: b_1 \ra b_2$ an isomorphism in $\cA$.
Then the complex is homotopy equivalent to
\[
\cdots 
	\ra 
C 
	\xra{\beta
		} 
D 
	\xra{\epsilon - \gamma  \phi^{-1} \delta
		}
E 
	\xra{\nu
		}
F 
	\ra 
\cdots.
\]
\end{lemma}

\printbibliography[]

\end{document}